\documentclass[12pt]{article}
\usepackage[utf8]{inputenc}
\usepackage{amsmath}  
\usepackage{amssymb}   
\usepackage{amsthm}
\usepackage{amsfonts}
\usepackage{mathtools}
\usepackage{bbm}
\usepackage{geometry}
\usepackage{enumerate}
\theoremstyle{definition}
\usepackage{setspace}
\usepackage[ragged]{sidecap}
\sidecaptionvpos{figure}{t} 

\newtheorem{Theorem}{Theorem}
\newtheorem*{Theorem*}{Theorem}
\newtheorem{Lemma}{Lemma}
\newtheorem{Proposition}{Proposition}

\newtheorem*{Lemma*}{Lemma}
\newtheorem{Remark}{Remark}

\newtheorem{Corollary}{Corollary}

\newtheorem*{Definition*}{Definition}
\newtheorem{Condition}{Condition}

\usepackage{xcolor}
\usepackage[normalem]{ulem}
\usepackage{hyperref}
\usepackage{url}
\usepackage[maxnames=99]{biblatex}
\usepackage{ulem}
\usepackage{multirow}
\usepackage{longtable}
\usepackage{lscape}
\usepackage{booktabs}
\usepackage{threeparttable}
\usepackage{caption}
\usepackage[skip=0.5ex]{subcaption}
\usepackage{array}
\usepackage{svg}

\makeatletter
\def\blfootnote{\xdef\@thefnmark{}\@footnotetext}
\makeatother

\title{Statistical Guarantees for Data-driven Posterior Tempering}
\author{Ruchira Ray \and  Marco Avella Medina  \and Cynthia Rush  \\ \\ Department of Statistics, Columbia University \thanks{This research was partly
supported by  NSF grants DMS-2310973 (Avella Medina) and DMS-2413828 (Rush).}}
\date{}

\begin{document}

\maketitle

\begin{abstract}
Posterior tempering reduces the influence of the likelihood in the calculation of the posterior by raising the likelihood to a fractional power $\alpha$. The resulting power posterior — also known as an $\alpha$-posterior  or fractional posterior — has been shown to exhibit appealing properties, including robustness to model misspecification and asymptotic normality (Bernstein-von Mises theorem). However, practical recommendations for selecting the tempering parameter  and statistical guarantees for the resulting power posterior remain open questions. Cross-validation-based approaches to tuning this  parameter suggest interesting asymptotic regimes  for the selected $\alpha$, which can either vanish or behave like a mixture distribution with a point mass at infinity and the remaining mass converging to zero. We formalize the asymptotic properties of  the power posterior in these regimes. In particular, we provide sufficient conditions for (i) consistency of the power posterior moments and (ii) asymptotic normality of the power posterior mean. Our analysis required us to  establish a new Laplace approximation that is interesting in its own right and is the key technical tool for showing a critical threshold $\alpha\asymp 1/\sqrt{n}$ where the asymptotic  normality of the posterior mean breaks. Our results allow for the power to depend on the data in an arbitrary way.\blfootnote{Code to replicate the experiments and plots on a smaller scale is available at \url{https://github.com/rray123/data-dependent-tempering}.}
\end{abstract}


\clearpage
\begin{doublespace}
\section{Introduction}\label{sec:intro}
Posterior tempering has gained popularity in Bayesian statistics because, in the presence of model misspecification, it offers increased robustness and more reliable inference compared to the standard posterior  \cite{grunwald_safe_2012, grunwald_inconsistency_2017, holmes_assigning_2017, miller_robust_2019, avella_medina_robustness_2022}. In particular, the tempering reduces the influence of the (potentially misspecified) likelihood in the construction of the posterior by raising the likelihood to a fractional exponent. The resulting posterior is known as a power posterior, fractional posterior \cite{bhattacharya_bayesian_2019}, or $\alpha$-posterior \cite[Chapter 8.6]{yang_alpha-variational_2020,ghosal_fundamentals_2017}.
A recent body of work has studied the statistical properties of $\alpha$-posteriors, where $\alpha$ is specified as a fixed constant \cite{bhattacharya_bayesian_2019, alquier_concentration_2020, yang_alpha-variational_2020, avella_medina_robustness_2022, ray_asymptotics_2023}. The authors in \cite{avella_medina_robustness_2022} established the asymptotic normality of the $\alpha$-posterior and showed that, relative to the standard posterior, the limiting variance is rescaled by $\alpha$ while the mean is unchanged. Under similar regularity conditions, \cite{ray_asymptotics_2023} generalized this result by establishing the consistency of the $\alpha$-posterior moments and asymptotic normality of the $\alpha$-posterior mean.

In practice, $\alpha$ is not a constant value, as theoretical prescriptions rely on unknown population quantities \cite{avella_medina_robustness_2022}, but is rather tuned from the data. Optimal tuning of $\alpha$ is critical for $\alpha$-posterior performance and to this end, there are many methods suggested in the literature for tuning $\hat\alpha_n$ from the data \cite{grunwald_safe_2012, grunwald_inconsistency_2017, wang_robust_2017, holmes_assigning_2017, miller_robust_2019, perrotta_practical_2020, syring_calibrating_2019, martin_chapter_2022, bu_towards_2024, mclatchie_predictive_2025}.  Some of these methods suggest that appropriate asymptotic regimes for considering $\alpha$-posterior performance are not consistent with a fixed parameter $\alpha$, but rather one that changes with the sample size, $n$. In this work, we generalize the asymptotic analysis of \cite{avella_medina_robustness_2022,ray_asymptotics_2023} to the case where $\alpha$ is random and data-dependent.

In Section \ref{sec:motivating_examples}, we empirically study this phenomenon of how data-driven selections of $\alpha$ change with $n$. We investigate the asymptotic behavior of $\hat\alpha_n$ values that are tuned using cross-validation-based methods in the context of regularized linear regression. We present extensive numerical experiments, which reveal that the selected $\hat\alpha_n$ can be vanishing as $n$ increases or, for some methods like Bayesian cross-validation and train-test splitting, can have a mixture distribution. This mixture behavior is consistent with previous findings that these methods tend to select $\hat\alpha_n=\infty$ with some fixed probability \cite{perrotta_practical_2020, mclatchie_predictive_2025,sengupta_asymptotic_2020}. Overall, these numerical results motivate the idea that statistical analyses of data-driven $\hat\alpha_n$-posteriors should consider the until-now-unexplored regimes where $\hat{\alpha}_n$ is vanishing or diverging in probability as the sample size increases.

Our main contributions can be summarized as follows:
\begin{enumerate}[(a)]
    \item When $\hat\alpha_n$ is a sequence satisfying $\frac{1}{n} \ll \hat\alpha_n \ll 1$ in probability, we prove consistency of all $\hat\alpha_n$-posterior moments (Theorem \ref{thm:moments_alt}), showing they converge to the moments of a normal distribution centered at the maximum likelihood estimator (MLE). Our result implies a Bernstein-von Mises (BvM) theorem for the $\hat\alpha_n$-posterior (Corollary \ref{cor:BvM}) that is sharp, meaning  that $\frac{1}{n} \ll \hat\alpha_n$ in probability is a necessary condition (Proposition \ref{prop:counter_example_exp_family}). In the regime where $\hat\alpha_n \to \infty$ in probability, we prove that the limiting distribution of the $\hat\alpha_n$-posterior is a point mass at the MLE (Proposition \ref{prop:alpha_inf}). We leverage this result to establish a BvM-type theorem for a mixed setting where  $\hat\alpha_n$ either tends to infinity with a fixed probability or vanishes (Theorem \ref{thm:BvM_mix_data-dep}). To the best of our knowledge, none of these regimes have been studied in the literature but are strongly motivated by empirical evidence. Our work advances the understanding of asymptotic properties of power posteriors in more realistic settings and our theoretical analysis reveals important differences in asymptotic behavior between $\hat{\alpha}_n$-posteriors and $\alpha$-posteriors with a fixed $\alpha$, which practitioners may need to consider in constructing methods to tune $\alpha$ from the data.

    \item We prove asymptotic normality of the $\hat\alpha_n$-posterior mean when $\frac{1}{\sqrt{n}} \ll \hat{\alpha}_n \ll 1$ in probability (Corollary \ref{cor:BayesEstimator}). Our proof relies critically on a new Laplace approximation result that is interesting in its own right (Lemma \ref{lem:lap_int_lem}). The Laplace approximation enables us to $(i)$ quantify accurately the distance between the posterior mean and MLE (Theorem \ref{thm:BayesEstimator}); $(ii)$ demonstrate that $\frac{1}{\sqrt{n}} \ll \hat\alpha_n$ in probability is a necessary condition for the asymptotic equivalence between the posterior mean and MLE.
\end{enumerate}

\subsection{Related work}\label{sec:related_work}

There is a vast body of work in Bayesian statistics that is related to our contribution. Perhaps the most obvious connections are to the  literature on the asymptotic properties of  Bayesian posteriors and to the literature on methods for choosing the tempering parameter, $\alpha$. In particular, our work builds on existing BvM theorems for standard posteriors with misspecified likelihoods \cite{kleijn_bernstein-von-mises_2012,avella_medina_robustness_2022}, but we  note that there is a broader class  of results for generalized posterior distributions \cite{chernozhukov_mcmc_2003, ghosh_robust_2016, miller_asymptotic_2021, lee_liu_nicholls2025,marusic:avellamedina:rush2025}. Tempering parameter selection methods include the SafeBayesian algorithm of \cite{grunwald_safe_2012, grunwald_inconsistency_2017}, frequentist calibration approaches \cite{syring_calibrating_2019,lyddon_general_2019,altamirano:briol:knoblauch2023,matsubara_generalized_2024,wu_comparison_2023}, expected information matching \cite{holmes_assigning_2017}, expected generalization error \cite{bu_towards_2024} and Bayesian prediction methods \cite{mclatchie_predictive_2025,lee_liu_nicholls2025}.

The popularity of $\alpha$-posteriors stems largely from the robustness properties they have been shown to enjoy in both theory and practice. Robustness of $\alpha$-posteriors was initially explored in the context of misspecified  models \cite{grunwald_inconsistency_2017,holmes_assigning_2017,miller_robust_2019, avella_medina_robustness_2022,lin_tarp_evans2025}. 
There have also been proposals to robustify standard posterior constructions to the presence of outliers, like Bayesian data reweighting \cite{wang_robust_2017} or the construction of generalized posteriors with robust losses \cite{hooker:vidyashankar2014, ghosh_robust_2016,altamirano_robust_2024,matsubara_generalized_2024,marusic:avellamedina:rush2025}  and test statistics \cite{baraud_robust_2024}. Finally, we note that $\alpha$-posteriors have also been studied in the context of variational inference by \cite{huang_improving_2018, yang_alpha-variational_2020} and variational autoencoders \cite{higgins_beta-vae_2016, burgess_understanding_2018}. 

We provide a more detailed discussion of related work in Appendix \ref{App:related_work}.

\subsection{Statistical model and $\alpha$-posteriors}

We now introduce the statistical framework used throughout the paper.
To model the random sample $X^n = (X_1,\dots,X_n) \in \mathcal{X}^n$, where $\mathcal{X}\subset\mathbb{R}^d$ denotes the sample space, let $\mathcal{F}_n = \{f_n(\cdot|\theta): \theta \in\Theta\subset \mathbb{R}^p \}$ be a parametric family of densities. This model is allowed to be misspecified in the sense that $f_{0,n}(X^n)$, the true density of the random sample $X^n$, may not belong to $\mathcal{F}_n$. We will use the notion of the pseudo-true parameter $\theta^*$, which is the true data-generating $\theta$ if the model is well specified, meaning $f_{0,n}(X^n) = f_n(X^n|\theta^*) \in\mathcal{F}_n$. If the model is misspecified, then $\theta^*$ best approximates the true data-generating process in terms of KL divergence. We denote the (pseudo) MLE by $\hat{\theta} = \arg\max_\theta f_n(X^n|\theta)$. 

For a sequence of distributions $f_{0,n}$ on random vectors $Y_n$, we say $Y_n=o_{f_{0,n}}(1)$ if for every $\epsilon > 0$ we have $\lim_n \mathbb{P}_{f_{0,n}}(\|Y_n\|_2 >\epsilon ) = 0$. We say $Y_n$ is ``bounded in $f_{0,n}$-probability,'' or $O_{f_{0,n}}(1)$, if for every $\epsilon > 0$ there exists $M\equiv M(\epsilon) > 0$ such that $\mathbb{P}_{f_{0,n}}(\|Y_n\|_2>M) < \epsilon$.  

Given a statistical model $f_n(X^n|\theta)\in\mathcal{F}_n$, a prior density $\pi(\theta)$ for $\theta$, and a scalar $\alpha > 0$, the $\alpha$-posterior is defined by the density
\begin{equation}\label{alphaposterior}
    \pi_{n,\alpha}(\theta|X^n) \equiv \frac{f_n(X^n|\theta)^{\alpha}\pi(\theta)}{\int_{\mathbb{R}^p} f_n(X^n|\theta)^{\alpha}\pi(\theta) d\theta}.
\end{equation}
We subscript $\alpha$ with $n$ to emphasize the dependence of $\alpha$ on the sample size, $n$. We also differentiate $\alpha_n$-posteriors, where the sequence, $\alpha_n$ is chosen a priori, and $\hat{\alpha}_n$-posteriors, where $\hat{\alpha}_n$ is learned from the data or depends on the data in any way. Throughout, all  $\alpha$-posteriors are assumed to be proper (i.e., the normalization constant in the denominator of \eqref{alphaposterior} is finite). Indeed, this is the case when $\pi(\theta)$ is proper \cite[Theorem 1]{carvalho_normalized_2021}.

\subsection{Notation and general definitions} \label{sec:notation}
We use $\phi(\cdot|\mu,\Sigma)$ to denote a multivariate normal density with mean $\mu$ and covariance $\Sigma$.  The determinant of a matrix $B$ is $|B|$ and the indicator function of an event $A$ is $\mathbbm{1}\{A\}$. For sequences $a_n$ and $b_n$, we say $a = O(b_n)$ if there exists $M > 0$ and $N$ such that $|a_n|\leq Mb_n$ for all $n > N$. We say $a_n\asymp b_n$ if $a_n = O(b_n)$ and $b_n = O(a_n)$. We use $\|v\|_p = \left(\sum_i|v_i|^p\right)^{1/p}$ to denote the $p$-norm of a vector $v$. Let $\bar{B}_{v}(\delta) \equiv \{\theta: \|\theta - v\|_2 \leq \delta\}$ denote the  closed Euclidean ball of radius $\delta$ around vector $v \in \mathbb{R}^p$ and $B_{v}(\delta) \equiv \{\theta: \|\theta - v\|_2 < \delta \}$ the open Euclidean ball. For any two probability densities $f$ and $g$ that are absolutely continuous with respect to the Lebesgue measure in $\mathbb{R}^p$, the total variation distance, $d_{\text{TV}}(f,g)$, between them is
\begin{equation}
    \label{eq:dtv}
    d_{\text{TV}}(f,g) = \frac{1}{2}\int_{\mathbb{R}^p}|f(x)-g(x)|dx.
\end{equation}
For any two probability measures, $\mu$ and $\nu$, on $\mathbb{R}^p$ with the Euclidian norm, the  $p$-Wasserstein distance between them is defined as
\begin{equation}\label{eq:Wasserstein_def}
    d_{\text{W}_p}(\mu, \nu) = \Big(\inf_{\gamma\in\Gamma(\mu,\nu)}\int_{\mathbb{R}^p\times\mathbb{R}^p}\|x-y\|_2^pd\gamma(x,y)\Big)^{1/p},
\end{equation}
where $\Gamma(\mu,\nu)$ is the set of joint distributions of $\mu$ and $\nu$ on the product space $\mathbb{R}^p\times\mathbb{R}^p$. We use tensor product notation to represent multivariate polynomials of vectors, $v\in\mathbb{R}^p$. Define the ``tensor inner product" between a $\text{k}^{\text{th}}$-order tensor, $T\in\mathbb{R}^{p \times \cdots \times p}$, and $\text{k}$ vectors $v^{(1)}, v^{(2)},\dots,v^{(k)}$ (all $\in\mathbb{R}^p$) to be
\begin{equation}\label{eq:tensor_inner_product_def}
    \left< T, v^{(1)}\otimes\ldots\otimes v^{(k)} \right> = \sum_{i_1=1}^p\cdots\sum_{i_k=1}^p [T]_{i_1,\ldots,i_k}\times v_{i_1}^{(1)}\times\cdots \times v_{i_k}^{(k)},   
\end{equation}
where $[T]_{i_1,\ldots,i_k}$ denotes the entries of $T$ for $1 \leq i_1,\ldots,i_k \leq p$. 
We also define the 1-norm of the $\textrm{k}^{\textrm{th}}$-order outer product of a vector $v\in\mathbb{R}^p$ to be
\begin{equation}\label{eq:tensor_norm_def}
    \begin{split}
        \|v^{\otimes k}\|_1 
        &= \sum_{i_1=1}^{p} \dots \sum_{i_k=1}^{p}\left|v_{i_1}\times\dots\times v_{i_k}\right|.\
    \end{split}
\end{equation}

\subsection{Organization}
The rest of this paper is organized as follows. In Section \ref{sec:motivating_examples} we present an extensive discussion of the connections between $\alpha$-posteriors and regularized estimation in the context of linear regression. These connections and extensive numerical experiments motivate our theoretical analysis. Section \ref{sec:main_results} presents the main asymptotic analysis; namely, for regimes of $\hat{\alpha}_n$ that arise in the numerical experiments in Section \ref{sec:motivating_examples}, we provide $\hat{\alpha}_n$-posterior moment consistency/BvM-type results as well as asymptotic normality for the posterior mean. Section \ref{sec:discussion} concludes the paper with a discussion. We relegate all our proofs, supplementary discussions, and examples to the Appendix. 

\section{A motivating example for tuning $\alpha$}\label{sec:motivating_examples}

In this section, we consider tuning the regularization parameter in ridge regression, which can be equivalently formulated as tuning the tempering parameter in Bayesian linear regression. We numerically investigate the limiting distributions of the tuned parameter values selected using various methods in both settings: ridge and Bayesian linear regression. While this is a simple example, it motivates and provides intuition for the asymptotic analysis in the general setting we study later in the paper.

\subsection{Model specification, $\mathbf{\boldsymbol{\alpha}_n}$-posterior and ridge regression}\label{sec:model_spec}

We conduct numerical experiments in the following regression setting. Consider paired data $\{x_i, y_i\}_{i=1}^n$, where $x_i = [x_{i,1}, x_{i,2}, x_{i,3}]^\top\in\mathbb{R}^3$ and $y_i\in\mathbb{R}$ for all $i = 1,\ldots,n$. Let $ x_i \sim \mathcal{N}(0, I_3)$ and  $\epsilon_i \sim \mathcal{N}(0, 1)$, independently. We simulate each sample according to the generative model
\begin{equation}\label{eq:gen_model}
        y_i = x_i^\top  \beta^* + \epsilon_i.
\end{equation}
In our experiments, for $\epsilon_i \sim \mathcal{N}(0, 1)$ and $\tilde{x}_i = [x_{i,2}, x_{i,3}]^\top\in\mathbb{R}^2$, we consider the following working model:
\begin{equation}\label{eq:spec_model}
    y_i = \tilde{x}_i^\top  \tilde{\beta} + \epsilon_i.
\end{equation}
The true parameter is set to be  $\beta^* = \left[\beta_1^*, -0.5, 0.1\right]$, and therefore \eqref{eq:spec_model} is  ``well specified" if $\beta_1^* = 0$, whereas it is ``misspecified" if $\beta_1^* \neq 0$ due to the omitted variable $x_{i,1}$. In our experiments, we set $\beta_1^* = 1$ in the misspecified case.

\begin{table}[h]
\begin{center}
\begin{tabular}{ c|c|c|c|c  }
    \multicolumn{5}{c}{Behavior of $\sqrt{n}(\hat{\beta}^{\text{Ridge}} - b(\lambda) - \beta_0)$}\\
    \hline
    \multicolumn{2}{c|}{Regime} 
    &    
    \multicolumn{1}{c}{$\alpha_n$-posterior}
    &                                       
    \multicolumn{2}{|c}{$\hat{\beta}^{\text{Ridge}}$ ($\alpha_n$-posterior mean)} \\\hline  
    $\lambda_n$& $\alpha_n$ &Variance& Bias $b(\lambda)$&$\sqrt{n}(\hat{\beta}^{\text{Ridge}}-\hat{\beta})$ \\
    \hline
    $\frac{1}{\sqrt{n}}\ll\lambda_n\ll 1$&$\frac{1}{n}\ll\alpha_n \ll \frac{1}{\sqrt{n}}$&$O_{f_{0,n}}(\frac{1}{\alpha_nn})$&$O(\lambda_n)$&Diverges\\
    $\lambda_n\asymp\frac{1}{\sqrt{n}}$&$\alpha_n\asymp\frac{1}{\sqrt{n}}$
    &$O_{f_{0,n}}(\frac{1}{\sqrt{n}})$&$O(\frac{1}{\sqrt{n}})$&$O_{f_{0,n}}(1)$\\
    $\frac{1}{n} \ll\lambda_n\ll\frac{1}{\sqrt{n}}$ &$\frac{1}{\sqrt{n}}\ll\alpha_n \ll 1$ &$O_{f_{0,n}}(\frac{1}{\alpha_nn})$&$O(\lambda_n)$&$o_{f_{0,n}}(1)$\\
    $\lambda_n\asymp\frac{1}{n}$&$\alpha_n\asymp1$ &$O_{f_{0,n}}(\frac{1}{n})$&$O(\frac{1}{n})$&$o_{f_{0,n}}(1)$\\
\end{tabular}
\end{center}
\caption{\label{table:ridge} Asymptotic behavior of $\hat{\beta}^{\text{Ridge}}$ and the corresponding $\alpha$-posterior. We denote the bias of the ridge estimator/$\alpha$-posterior mean by $b(\lambda)$. The last three columns denote the variance of the $\alpha_n$-posterior, the bias of the ridge estimator, and the order of $\sqrt{n}(\hat{\beta}^{\text{Ridge}}-\hat{\beta})$, where $\hat{\beta}$ is the MLE of $\beta$ (i.e., the OLS estimator). 
}
\end{table}

Note the equivalence between ridge regression to estimate $\beta$ and $\alpha$-posterior inference on $\beta$ in the linear regression model with a Gaussian prior on the slope parameters. Indeed,  the  model in \eqref{eq:gen_model} gives rise to a likelihood for the data $\{x_i, y_i\}_{i=1}^n$, given the coefficient vector $\beta$.  Combining this likelihood with a  $\mathcal{N}(0,I_3)$ prior for $\beta$  yields the $\alpha_n$-posterior 
\begin{equation}\label{eq:ridge_post}
    \pi_{n,\alpha_n}\left(\beta|{X},Y\right) = \phi\Big(\beta \, \big| \,\hat{\beta}^{\text{Ridge}},~\frac{1}{\alpha_n n}\Big(\frac{1}{n}X^\top  X+\frac{1}{\alpha_n n} I_p\Big)^{-1}\Big),
\end{equation}
where, using $\lambda_n =\frac{1}{n \alpha_n}$,
\begin{equation}\label{eq:ridge_obj}
    \hat{\beta}^{\text{Ridge}} = \arg\min_\beta \frac{1}{n}\|Y-X\beta\|_2^2 + \lambda_n\|\beta\|_2^2,
\end{equation}
with $X = [x_1 \cdots x_n]^\top$ and $Y = [y_1,\ldots,y_n]$.
Standard frequentist theory shows that, under some regularity conditions,  $\hat{\beta}^{\text{Ridge}}$ is consistent as long as $\lambda_n=o(1)$ and $\hat{\beta}^{\text{Ridge}}$ is asymptotically normally distributed if $\lambda_n=o(1/\sqrt{n})$ \cite{hoerl:kennard1970}.  In Table \ref{table:ridge}, we compare the asymptotic behavior of $\hat{\beta}^{\text{Ridge}}$ to that of the $\alpha_n$-posterior in \eqref{eq:ridge_post} in these regimes of $\lambda_n$ (and the corresponding regimes of $\alpha_n$). Table \ref{table:ridge} suggests that the regimes $1/n \ll\alpha_n\ll 1$ (rows 1-3) and $1/\sqrt{n} \ll\alpha_n\ll 1$ (row 3) are particularly interesting. In the regime where $1/n \ll\alpha_n\ll 1$, the variance of the $\alpha_n$-posterior and the bias of the ridge regression estimator both converge to zero. In the regime where $1/\sqrt{n} \ll\alpha_n\ll 1$, the $\alpha_n$-posterior mean -- which coincides with the ridge regression estimator -- is asymptotically equivalent to the MLE. We will see in Section \ref{sec:main_results} that these regimes are indeed most relevant for establishing BvM thereoms (when $1/n \ll\alpha_n\ll 1$) and asymptotic normality of the $\alpha$-posterior mean estimator (when $1/\sqrt{n} \ll\alpha_n\ll 1$). Before giving our theoretical results, we provide empirical evidence showing that natural data-driven choices of $\alpha$  lead us to these types of vanishing $\hat\alpha_n$ regimes as well. 

\subsection{Data-driven choices of the tempering parameter}\label{sec:sel_methods}

In both the well-specified and misspecified model settings (Section \ref{sec:model_spec}), we investigate the distribution of the data-driven tempering parameter $\hat\alpha_n$ selected according to the following commonly-used methods for parameter tuning: 

\begin{enumerate}
    \item \textbf{Bayesian cross-validation (BCV)}: We choose $\alpha$ to maximize the leave-one-out log pointwise predictive density \cite{gelman_understanding_2014, vehtari_practical_2017,gelmanetal2013} given by 
\begin{equation}\label{eq:elpd}
    \overline{\text{lppd}}_{\text{LOO}}(\alpha) = \frac{1}{n}\sum_{i=1}^n \log \int p(x_i,y_i|\beta)\pi_{n,\alpha}\left(\beta|X_{-i}Y_{-i}\right)d\beta,  
\end{equation}
where $X_{-i}$ and $Y_{-i}$ are the design matrix and dependent variable, respectively, with the $i^{\text{th}}$ entries removed. The $\overline{\text{lppd}}_{\textrm{LOO}}$ loss function was discussed in the context of tuning hyperparameters in Bayesian inference in \cite{yang_alpha-variational_2020}. See Appendix \ref{sec:elpd_calc} for the calculation of $\overline{\textrm{lppd}}_{\textrm{LOO}}(\alpha)$ for the posterior in \eqref{eq:ridge_post}.

\item \noindent \textbf{Bayesian CV + VI (BCV + VI)}: We also consider \eqref{eq:elpd} with the $\alpha$-posterior replaced with its Gaussian mean field variational approximation
\begin{equation}
\label{eq:elpd_vi}
    \overline{\text{lppd}}_{\textrm{LOO-VI}}(\alpha) = \frac{1}{n}\sum_{i=1}^n \log \int p\left(x_i,y_i|\beta\right)\tilde{\pi}_{n,\alpha}\left(\beta|X_{-i}Y_{-i}\right)d\beta,
\end{equation}
where $\tilde{\pi}_{n,\alpha}(\cdot)$ is the density of the Gaussian mean field variational approximation of \eqref{eq:ridge_post}: 
\begin{equation}\label{eq:ridge_post_VI}
    \tilde{\pi}_{n,\alpha}\left(\beta|X,Y\right) = \phi\Big(\beta\big|\hat{\beta}^{\text{Ridge}},~\frac{1}{\alpha n}\Big[\text{diag}\Big(\frac{1}{n}X^\top  X+\frac{1}{\alpha n} I\Big)\Big]^{-1}\Big).
\end{equation}
See Appendix \ref{sec:elpdvi_calc} for the $\overline{\textrm{lppd}}_{\text{LOO-VI}}(\alpha)$ calculation for the posterior in \eqref{eq:ridge_post_VI}.
      
\item \noindent \textbf{LOOCV}: We choose $\alpha$ to minimize Allen's PRESS statistic \cite{allen_mean_1971}, given by $\text{PRESS}(\alpha) = \frac{1}{n}\|B(I_n - H(\alpha))Y\|_2^2,$
where $B$ is diagonal with entries $B_{ii} = (1-H_{ii})^{-1}$ and $H(\alpha) = X(X^\top X + \frac{1}{\alpha} I_p)^{-1}X^\top$. Allen's PRESS statistic is equivalent to the leave-one-out cross-validation (LOOCV) squared error of $\hat{\beta}^{\text{Ridge}}$. 

\item \noindent \textbf{Train-test split (Train-test)}: Given data split into a training set and test set, we choose $\alpha$ to minimize the squared test error of $\hat{\beta}^{\textrm{Ridge}}$ estimated on the training dataset. The results that we report correspond to a 70-30 training-test data split. We omit the results we obtained for the splits 50-50 and 90-10 as they were very similar.

\item \noindent \textbf{SafeBayes}: The SafeBayesian algorithm was proposed in \cite{grunwald_safe_2011} and explored in the context of misspecified linear models in \cite{grunwald_inconsistency_2017}. This procedure chooses $\alpha \in [0,1]$ that minimizes the cumulative posterior expected log loss. 
Since we assume our noise variance is known, we use the R-square-Safe-Bayesian Ridge Regression algorithm \cite{safebayes}.  
\end{enumerate}

We use grid searches to select the optimal parameter according to each method. For methods selecting $\lambda$, we map the selected value, denoted by $\hat{\lambda}_n$, to its equivalent $\hat{\alpha}_n$ value. See Table \ref{tab:sim_grid} in Appendix \ref{app:grid_searches} for these mappings and the grids used.

\subsection{Simulation results}\label{sec:sim-results}

Figures \ref{fig:lim_dist_constant}--\ref{fig:lim_dist_mix} display the distribution of $\hat{\alpha}_n$, over 1,000 independent replications of the data, selected according to the five methods discussed in Section \ref{sec:sel_methods} in the model specification settings described in Section \ref{sec:model_spec}. To understand the asymptotic behavior of $\hat{\alpha}_n$, we compare the distribution of $\hat{\alpha}_n$ when $n=100$, $n=1,000$, and $n=5,000$. 

We observe the following classes of asymptotic behavior of $\hat{\alpha}_n$:
\begin{enumerate}
    \item \noindent\textbf{Constant limit (Figure \ref{fig:lim_dist_constant}):} In both model specification settings, LOOCV and SafeBayes suggest that $\hat{\alpha}_n$ approaches a constant. The asymptotics of the resulting constant $\alpha$-posterior were studied in \cite{avella_medina_robustness_2022} and \cite{ray_asymptotics_2023}. In particular, a BvM theorem for the $\alpha$-posterior was established in \cite{avella_medina_robustness_2022}. The moment constency of the $\alpha$-posterior and asymptotic normality of the $\alpha$-posterior mean were established in \cite{ray_asymptotics_2023}.
    \begin{figure}
        \centering
        \includegraphics{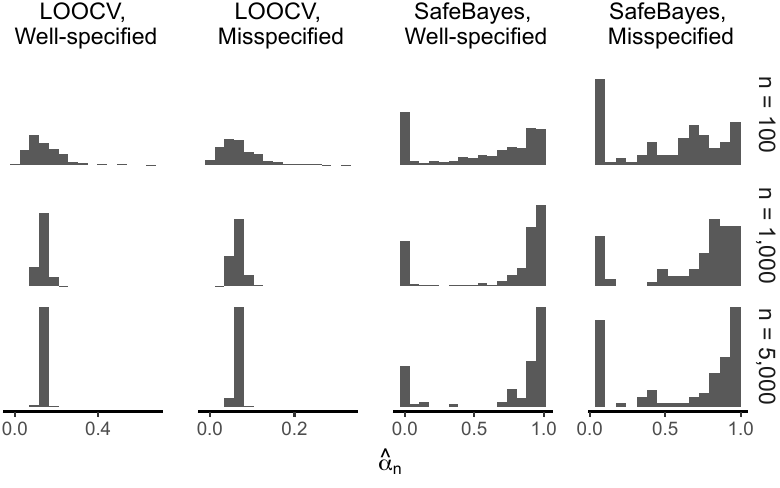}
        \caption{ \small Distribution of selected $\hat{\alpha}_n$ over 1,000 replications.$^*$ Column titles denote the tuning method (LOOCV, SafeBayes) and the model specification (Well-specified, Misspecified). Within each column, rows correspond to sample size per replication ($n=100$, $n=1,000$, $n=5,000$).\\ \small{$^*$The SafeBayes, $n=5,000$ plots have $100$ replications due to computational constraints}
        }
        \label{fig:lim_dist_constant}
    \end{figure}

   \item \noindent\textbf{Quickly vanishing $\mathbf{\hat{\boldsymbol{\alpha}}_n}$ (Figure \ref{fig:lim_dist_quick}):} In the misspecified model setting, BCV and BCV + VI suggest that $\hat{\alpha}_n$ decreases to zero. The convergence rate is roughly of the order $\frac{1}{n}$, as can be seen in Table \ref{tab:conv_rates}. In Section \ref{sec:BvM_discussion}, we study the $\alpha_n$-posterior when $\alpha_n\asymp \frac{1}{n}$ and show that a BvM result does not hold in general, as we provide a counterexamaple with exponential family conjugate models (Propositon \ref{prop:counter_example_exp_family}).

\begin{SCfigure}[0.95]
        \centering
        \includegraphics{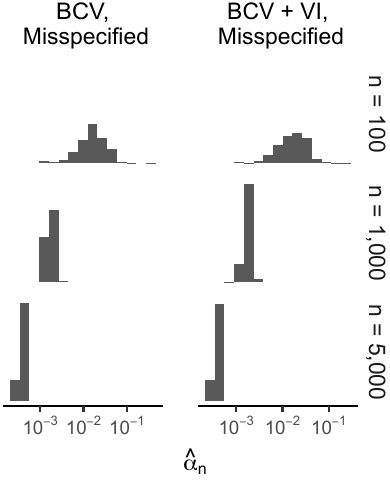}
        \caption{ \small Distribution of selected $\hat{\alpha}_n$ over 1,000 replications.$^*$ Column titles denote the tuning method (BCV, BCV+VI) and the model specification (Misspecified). Within each column, rows correspond to sample size per replication ($n=100$, $n=1,000$, $n=5,000$). \\ \small{$^*$The $\hat{\alpha}_n$ values are plotted on a log scale to more easily show the convergence of $\hat{\alpha}_n$ to zero. See Figure \ref{fig:lim_dist_quick_linear} in Appendix \ref{app:lim_dist_linear} for histograms of $\hat{\alpha}_n$ plotted on a linear scale.}
        }
        \label{fig:lim_dist_quick}
    \end{SCfigure}
    
    \item \noindent\textbf{Mixed $\mathbf{\hat{\boldsymbol{\alpha}}_n}$ (Figure \ref{fig:lim_dist_mix}):} The remaining settings (BCV and BCV + VI in the well-specified setting and Train-test in both model specification settings) suggest a distribution of $\hat{\alpha}_n$ that has a point mass at what is effectively $\alpha=\infty$, which we will call a ``corner solution" as it corresponds essentially to a mass point at the MLE, with remaining mass appearing to converge to zero. For more details on how we recode $\hat{\alpha}_n$ values to be $\infty$, see Appendix \ref{app:grid_searches}. Table \ref{tab:conv_rates} shows that  the rate of convergence to zero of the non-degenerate component of the distribution is estimated to be slower than $\frac{1}{n}$. Furthermore, approximately half of the $\hat{\alpha}_n$ are corner solutions (see Figure \ref{fig:conv-rate-prop}), asymptotically. We refer to estimated $\alpha$ in this regime as $\hat{\alpha}_n^{\text{mix}}$ and study the asymptotic properties of such $\hat{\alpha}^{\text{mix}}$-posteriors in  Section \ref{sec:alpha_mix} (Theorem \ref{thm:BvM_mix_data-dep}).
    
\begin{figure}[!htb]
    \centering
    \includegraphics{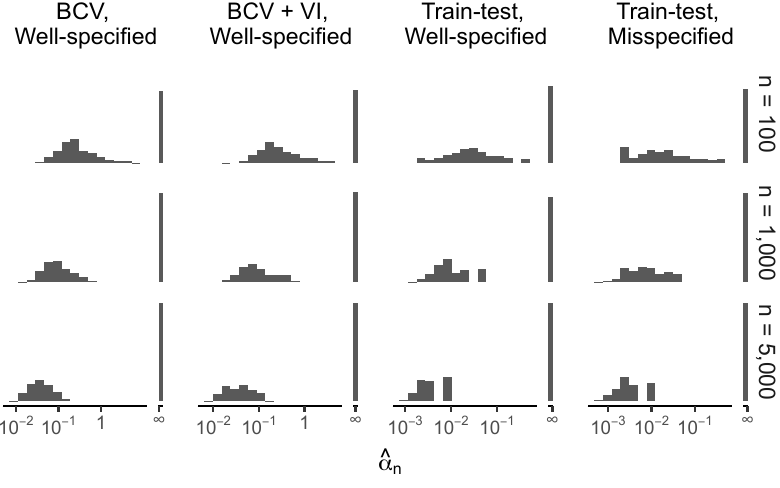}
    \caption{ \small Distribution of selected $\hat{\alpha}_n$ over 1,000 replications.$^*$ Column titles denote the tuning method (BCV, BCV+VI, Train-test) and the model specification (Well-specified, Misspecified). Within each column, rows correspond to sample size per replication ($n=100$, $n=1,000$, $n=5,000$). \small{$^*$The $\hat{\alpha}_n$ values are plotted on a log scale to more easily show the convergence of $\hat{\alpha}_n$ values that are not corner solutions to zero. See Figure \ref{fig:lim_dist_mix_linear} in Appendix \ref{app:lim_dist_linear} for histograms of $\hat{\alpha}_n$ plotted on a linear scale.}
    }
    \label{fig:lim_dist_mix}
\end{figure}
\end{enumerate}

We note that all the regimes for data-driven choices of $\alpha$ have appeared in previous work. The constant $\hat\alpha_n$ for SafeBayes is consistent with the results in \cite[Figure 5]{grunwald_inconsistency_2017} and the mixed corner solutions have been observed in other works as well. When $\alpha$ is selected to optimize the expected log predictive density (a quantity similar to the log point predictive density), \cite{mclatchie_predictive_2025} also observe a point mass at $\alpha = \infty$. Furthermore, when the   ridge regression parameter $\lambda$ is selected according to a train-test split, \cite{sengupta_asymptotic_2020}  observe that, asymptotically, around half of the selected values are corner solutions at $\lambda = 0$ (i.e., $\alpha=\infty$) and formalize this observation into a result about the asymptotic distribution of the selected $\lambda$. In Section \ref{sec:main_results}, we provide a general theory that covers all these regimes of data-driven $\hat\alpha_n$ for  $\alpha$-posteriors constructed using possibly misspecified likelihood models. 

\begin{minipage}{\textwidth}
  \begin{minipage}[c]{0.47\textwidth}
  \vspace{0pt}
    \centering
    \begin{tabular}{ccc}
        \toprule
        Method & Specification & Exponent, 95\% CI\\
        \midrule
         & Well-specified & -0.63, [-0.66, -0.59]\\
        \cmidrule{2-3}
        \multirow{-2}{*}{\centering\arraybackslash BCV} & Misspecified & -1.04, [-1.05, -1.03]\\
        \cmidrule{1-3}
         & Well-specified & -0.63, [-0.66, -0.59]\\
        \cmidrule{2-3}
        \multirow{-2}{*}{\centering\arraybackslash BCV + VI} & Misspecified & -1.04, [-1.05, -1.02]\\
        \cmidrule{1-3}
         & Well-specified & -0.63, [-0.66, -0.6]\\
        \cmidrule{2-3}
        \multirow{-2}{*}{\centering\arraybackslash Train-test} & Misspecified & -0.61, [-0.64, -0.58]\\
        \bottomrule
    \end{tabular}
      \captionof{table}{\small Estimated stochastic order of $\hat{\alpha}_n$. We model $\hat{\alpha}_n$ (that are not corner solutions) as $\hat{\alpha}_n=Cn^{\gamma}$ and report the estimates and $95\%$ confidence intervals of $\gamma$ (Exponent). See Appendix \ref{sec:alpha_lambda} for the corresponding best fit curves (Figure \ref{fig:conv-rate}) and for details on the estimation of $\gamma$.}\label{tab:conv_rates}
    \end{minipage}
    \hfill
  \begin{minipage}[c]{0.45\textwidth}
  \vspace{0pt}
    \centering
    \includegraphics[width=\columnwidth]{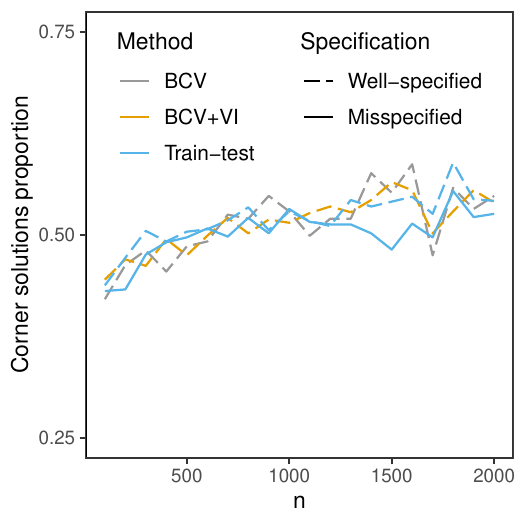}
    \captionof{figure}{\small Proportion of corner solutions vs. sample size ($n$).}\label{fig:conv-rate-prop}
  \end{minipage}
  \end{minipage}

  \vspace{5mm}

\subsection{Real data application}\label{sec:data_example}

We illustrate the practical relevance of the asymptotic regimes of $\hat{\alpha}_n$ we observed in our simulations by considering the Current Population Survey (CPS) data published by the U.S.\ Census Bureau for March 1988. This cross-sectional dataset is publicly available in the \texttt{AER} package for \texttt{R} as \texttt{CPS1988}. The dataset consists of 28,155 observations containing information about the wage, education, work experience, and ethnicity of men between the ages of 18 and 70 who are making more than $\$50$ yearly and are neither self-employed nor working without pay \cite{kleiber_applied_2008}. As a baseline, we start with the model specified in \cite{kleiber_applied_2008}, which we refer to as the ``full model"; namely,
\begin{equation}\label{eq:full_model}
    \log(\textrm{Wage}) = \beta_0 + \beta_1\textrm{Education} + \beta_2\textrm{Ethnicity} + \beta_3\textrm{Experience} + \beta_4\textrm{Experience}^2 + \epsilon.\
\end{equation}
We consider two additional model specifications that remove covariates from \eqref{eq:full_model}. A first model with the Ethnicity variable removed 
\begin{equation}\label{eq:noeth}
    \log(\textrm{Wage}) = \beta_0 + \beta_1\textrm{Education} + \beta_2\textrm{Experience} + \beta_3\textrm{Experience}^2 + \epsilon,
\end{equation}
and another model with the $\textrm{Experience}^2$ term removed
\begin{equation}\label{eq:noexp2}
    \log(\textrm{Wage}) = \beta_0 + \beta_1\textrm{Education} + \beta_2\textrm{Ethnicity} + \beta_3\textrm{Experience} + \epsilon.\
\end{equation}
We note that eliminating the $\textrm{Experience}^2$ term will greatly affect the estimation of the coefficient of the linear term. While we do not claim that \eqref{eq:full_model} is the correct or true model,  it does seem to be a better model or closer to correctly specified than either \eqref{eq:noeth} or \eqref{eq:noexp2}.

First, we repeat the experiments in Figures \ref{fig:lim_dist_constant}--\ref{fig:lim_dist_mix} and Table \ref{tab:conv_rates} by taking subsamples from the full dataset. The results of these experiments are in Figure \ref{fig:lim-dist-data} in Appendix \ref{app:lim_dist_data}. Note that these subsampling experiments are not equivalent to the experiments in Figures \ref{fig:lim_dist_constant}--\ref{fig:lim_dist_mix} and Table \ref{tab:conv_rates} because those experiments involve independent replications of the data, whereas subsampling creates dependence. Nevertheless, we find that similar ``asymptotic'' regimes arise (see Figure \ref{fig:lim-dist-data} for more details). In particular, $(i)$ under all model specifications, SafeBayes suggests $\hat{\alpha}_n$ approaching a constant; $(ii)$ we do not observe the ``quickly vanishing" regime under any of the model specifications; $(iii)$ the distributions of $\hat{\alpha}_n$ selected according to both BCV + VI (under the full model and model excluding ethnicity) and Train-test (under all model specifications) exhibit the mixture behavior discussed in our simulation example (Table \ref{tab:conv_rates-data}). More specifically, Table \ref{tab:conv_rates-data} reports the non-corner solutions selected according to BCV + VI and Train-test are estimated to be of order slower than $\frac{1}{n}$ under all model specifications. We also note that LOOCV appears to exhibit the same mixture behavior to some extent under all model specifications (Figure \ref{fig:lim-dist-data}). However, we find that the values of the ``non-corner" $\hat{\alpha}_n$ appear to remain constant or grow with $n$, rather than shrink (Table \ref{tab:conv_rates-data}; see Figure \ref{fig:conv-rate-data-non-corner} in Appendix \ref{sec:alpha_lambda} for more details). Finally, BCV selects corner solutions under all model specifications (Figure \ref{fig:lim-dist-data}). We formalize the behavior of the resulting $\alpha$-posterior, which we refer to as the $\infty$-posterior, in Proposition \ref{prop:alpha_inf}. LOOCV and BCV are the only selection methods for which the distribution of $\hat{\alpha}_n$ is different in the data example from the distribution of $\hat{\alpha}_n$ in the simulation (Figures \ref{fig:lim_dist_constant} and \ref{fig:lim_dist_mix}).

Under all model specifications, all methods tend to favor higher values of $\hat{\alpha}_n$, including $\alpha = \infty$. SafeBayes constrains $\alpha\in[0,1]$ and overwhelmingly selects $\alpha$ to be close to 1 under all model specifications. Moreover, Train-test exhibits mixture behavior similar to that in Figure \ref{fig:lim_dist_mix}, where a proportion of values of $\hat{\alpha}_n$ are effectively $\infty$.

\begin{table}[]
    \centering
    \begin{tabular}{ccc}
        \toprule
        Method & Model Specification & Exponent, 95\% CI\\
        \midrule
         & Full Model & -0.82, [-0.95, -0.68]\\
        \cmidrule{2-3}
        \multirow{-2}{*}{\centering\arraybackslash BCV + VI} & No Ethnicity & -0.85, [-0.98, -0.71]\\
        \cmidrule{1-3}
         & Full Model & 0.36, [0.27, 0.44]\\
        \cmidrule{2-3}
         & No Ethnicity & 0.47, [0.3, 0.64]\\
        \cmidrule{2-3}
        \multirow{-3}{*}{\centering\arraybackslash LOOCV} & No $\text{Experience}^2$ & 0.00, [-0.04, 0.04]\\
        \cmidrule{1-3}
         & Full Model & -0.66, [-0.75, -0.56]\\
        \cmidrule{2-3}
         & No Ethnicity & -0.70, [-0.77, -0.64]\\
        \cmidrule{2-3}
        \multirow{-3}{*}{\centering\arraybackslash Train-test} & No $\text{Experience}^2$ & -0.59, [-0.69, -0.5]\\
        \bottomrule
    \end{tabular}
    \caption{Estimated stochastic order of $\hat{\alpha}_n$. We model the $\hat{\alpha}_n$ (that are not corner solutions) as $\hat{\alpha}_n=Cn^{\gamma}$ and report the estimates and confidence intervals of $\gamma$ (Exponent). We also note that the ``non-corner" solutions according to LOOCV appear to be growing, rather than shrinking, with $n$. We include this result for completeness. See Appendix \ref{sec:alpha_lambda} for the corresponding curves of best fit (Figure \ref{fig:conv-rate-data-non-corner}) and for details on the estimation of $\gamma$.}
    \label{tab:conv_rates-data}
\end{table}

Taking subsamples of size $n=5,000$, one could expect that given the relatively large subsample size, the $\hat{\alpha}_n$ values would ``agree" across subsamples. Indeed, across all methods and model specifications, the $\hat{\alpha}_n$ values chosen according to BCV, BCV+VI, and SafeBayes are not sensitive to the choice of subsample. In particular, we observe that  BCV always chooses $\hat\alpha_n$ equals ``infinity'', while  BCV+VI gives  $\hat\alpha_n$ equals ``infinity'' a larger probability as $n$ grows large. However, the mixture proportion of the distribution of the $\hat{\alpha}_n$ chosen according to Train-Test remains stable as $n$ grows. The $\hat{\alpha}_n$ chosen according to LOOCV are somewhat sensitive to choice of subsample under the full model, but all the selected $\hat{\alpha}_n$ are relatively large (Figure \ref{fig:lim-dist-data}).

We also compute the mean of the resulting $\alpha$-posterior according to \eqref{eq:ridge_post} (which is the same as the mean of its variational approximation according to \eqref{eq:ridge_post_VI}) when $n=5,000$. We compare the distribution of $\hat{\alpha}_n$-posterior mean values to the least squares estimates of the full model coefficients on the full dataset. We call these the ``full coefficient values." The results of these experiments are displayed in  Appendix \ref{app:data_experiment} in Figure \ref{fig:data_fig_full}--\ref{fig:data_fig_noexp2}. Overall, they demonstrate that the posterior mean estimator is stable across the random subsamples.

\section{Asymptotic results for the $\hat{\alpha}_n$-posterior}\label{sec:main_results}
In this section, we study the asymptotic properties of $\hat{\alpha}_n$-posteriors in the three asymptotic growth regimes of $\hat{\alpha}_n$ that were observed in the simulated and real data experiments from  Section \ref{sec:motivating_examples}. Our first result establishes moment consistency in the regime where $\frac{1}{n}\ll\hat\alpha_n\ll 1$ in probability, which is part of the ``mixed $\hat{\alpha}_n$" regime discussed in Section \ref{sec:motivating_examples}. We also discuss how $\frac{1}{n} \ll \hat{\alpha}_n$ is the critical threshold for this result to hold. Our second result shows that the $\hat\alpha_n$-posterior mean estimator is asymptotically normal in the regime where $\frac{1}{\sqrt{n}}\ll\hat\alpha_n\ll 1$ in probability. To show that $\frac{1}{\sqrt{n}} \ll \hat{\alpha}_n$ is the critical threshold for this result, we establish a novel Laplace approximation. Our final result formalizes the convergence of the $\hat{\alpha}_n$-posterior in the ``mixed $\hat{\alpha}_n$" regime where $\hat{\alpha}_n$ has a mixed distribution characterized by a point mass at infinity with remaining mass vanishing slower than $\frac{1}{n}$. This regime can arise when the tempering parameter is selected with Bayesian cross-validation as discussed previously.

\subsection{BvM for the $\hat{\alpha}_n$-posterior and consistency of its moments}\label{sec:BvM+moments}

The main result of this section establishes the consistency of the $\hat\alpha_n$-posterior moments for any random sequence, $\hat{\alpha}_n$, such that $\frac{1}{n}\ll\hat\alpha_n\ll1$ with probability tending to 1. More precisely, we show that the $L_1$ difference between the (centered and rescaled) moments of the $\hat\alpha_n$-posterior and those of a multivariate Gaussian converges in probability to $0$.

\subsubsection{Main result}
 We require the following regularity conditions on the model. Throughout, $\alpha_n$ denotes an arbitrary positive sequence satisfying $\frac{1}{n}\ll\alpha_n\ll1$ unless otherwise specified.
\begin{itemize}
     \item[\textbf{(A0)}] The MLE, $\hat{\theta}$, is unique and there exists a $\theta^*\in\Theta$, where $\Theta$ is an open subset of  $\mathbb{R}^p$, such that  $\hat{\theta} \to \theta^*$ in $f_{0,n}$-probability and $\sqrt{n}(\hat{\theta}-\theta^*)$ is asymptotically normal. 
    \item[\textbf{(A1)}] The prior, $\pi(\theta)$, is continuous and positive in an open ball of $\theta^*$. 
    \item[\textbf{(A2)}]
    Define $\Delta_{n,\theta^*}=\sqrt{n}(\hat{\theta}-\theta^*$). There exists a positive definite matrix $V_{\theta^*}$ such that $\sup_{h\in \bar{B}_0(r)}|R_n(h)|\rightarrow 0$ in $f_{0,n}$-probability for any  $r > 0$ where
    \begin{equation}\label{eq:Rn(h)_def}
        \begin{split}
        R_n(h) \equiv \alpha_n\Big[\log f_n\Big(X^n\Big|\theta^*+\frac{h}{\sqrt{\alpha_nn}}\Big)-\log f_n(X^n|\theta^*)\Big] -\sqrt{\alpha_n}h^\top V_{\theta^*}\Delta_{n,\theta^*}+\frac{1}{2}h^\top V_{\theta^*}h.\
        \end{split}
    \end{equation}
    \item[\textbf{(A3)}] For all $\epsilon > 0$ and all $k_0 \geq 0$, there exist a constant $M \equiv M(k_0, \epsilon) < \infty$ and an  integer $N\equiv N(k_0,\epsilon)$ such that for $\alpha_nn\geq N$, 
    \begin{equation*}    
        \mathbb{P}_{f_{0,n}}\Big(\int_{\mathbb{R}^p} \|\sqrt{\alpha_nn}(\theta-\theta^*)\|_2^{k_{0}}\pi_{n,\alpha_n} (\theta|X^n)  d\theta > M\Big) < \epsilon.
    \end{equation*}
\end{itemize}
Assumptions \textbf{(A0)} and \textbf{(A1)} are minimal regularity conditions on the MLE and prior, respectively. They are standard in the literature and also assumed in \cite{lehmann_asymptotic_1998, kleijn_bernstein-von-mises_2012, avella_medina_robustness_2022, ray_asymptotics_2023}. Assumption \textbf{(A2)} is a variant of the standard local asymptotic normality (LAN) condition in \cite[eq.\ (2.1)]{kleijn_bernstein-von-mises_2012}, \cite[Assumption 1]{avella_medina_robustness_2022}, and \cite[\textbf{(A2)}]{ray_asymptotics_2023}. We note that our condition is a generalization of LAN, which can be recovered by setting $\alpha_n=1$. However, the same usual regularity conditions involving three derivatives of the log-likelihood that are used to check LAN \cite{lehmann_asymptotic_1998, chernozhukov_mcmc_2003, van_der_vaart_asymptotic_1998} can also check our variant of LAN when $\alpha_n\to0$. See Appendix \ref{app:suff_cond_A2} for a more detailed discussion and a formal statement. Assumption \textbf{(A3)} states that all moments of the rescaled $\alpha_n$-posterior are finite with high probability for $\alpha_nn$ sufficiently large. This condition is related to posterior concentration assumptions used to establish BvM theorems \cite{kleijn_bernstein-von-mises_2012, avella_medina_robustness_2022, ray_asymptotics_2023}. We also note that other works put conditions on the prior and log-likelihood that imply posterior concentration \cite{lehmann_asymptotic_1998, chernozhukov_mcmc_2003, syring_gibbs_2023}. We complement these assumptions with  conditions on the likelihood and prior that are sufficient to check that \textbf{(A3)} holds, as shown in Appendix \ref{app:post_conc_suff}. We also require them to establish moment convergence for the random $\hat\alpha_n$-posterior.
\begin{Condition}
\label{condition1}
   Let $L, c_0, c_1 > 0$ and $c_2 \geq 0$ be constants. Then for some $r > \frac{e^{c_2/c_0}}{c_1}$, there exists a function $\gamma:[0,\infty)\to[0,\infty)$ with the property $\gamma(v) \geq c_0\log(1+c_1v) - c_2$ when $v \geq r$,  such that uniformly on $B_{\theta^*}(r)^\mathsf{c}$ with $f_{0,n}$-probability tending to one, we have
        \begin{align}\label{eq:cond2_part1}
            \log f_n(X^n|\theta)-\log f_n(X^n|\theta^*) \leq -Ln\gamma(\|\theta-\theta^*\|_2).
        \end{align}
\end{Condition}
\begin{Condition}
\label{condition2}
    For $c_0$, $c_1$, $c_2$ and $r$ defined in Condition \ref{condition1}, there exist positive constants, $c_3 \equiv c_3(r) > 0$ and $c_4 \equiv c_4(r) > 0$, such that uniformly on $B_{\hat{\theta}}(2r)$, we have $-nc_3\|\theta-\hat{\theta}\|_2^2 \leq \log f_n(X^n|\theta)-\log f_n(X^n|\hat{\theta}) \leq -nc_4\|\theta-\hat{\theta}\|_2^2$ with $f_{0,n}$-probability tending to one.
    
\end{Condition}
Condition \ref{condition1} requires the pseudo-true parameter to be ``salient" enough, in the sense that the suboptimiality of the log-likelihood evaluated at $\theta$ outside of a neighborhood of $\theta^*$ is lower bounded by an increasing function $\gamma$ of $\|\theta-\theta^*\|_2$. Condition \ref{condition2} assumes that with probability tending to one the log-likelihood is locally strongly convex and smooth in a neighborhood of the MLE.

We are now ready to state our first main result. The proof is in Appendix \ref{app:main_results_thm2}. 

\begin{Theorem}\label{thm:moments_alt} 
Assume \textbf{(A0)}, \textbf{(A1)}, \textbf{(A2)}, and \textbf{(A3)} hold. Then, for any sequence $\alpha_n$ such that $\frac{1}{n}\ll\alpha_n\ll1$ and for any integer $k \geq 0$,
\begin{equation}\label{eq:moment_result}
    \int_{\mathbb{R}^p}\|\sqrt{\alpha_nn}(\theta-\theta^*)\|_2^k \, \Big|\pi_{n,\alpha_n}\left(\theta|X^n\right) - \phi\Big(\theta\big|\hat{\theta},\frac{1}{\alpha_nn}V_{\theta^*}^{-1}\Big)\Big|d\theta \rightarrow 0,
\end{equation}
in $f_{0,n}$-probability.
Furthermore, suppose that for a positive, data-dependent sequence $\hat{\alpha}_n$ there exists a positive sequence $\alpha_n$ with $\frac{1}{n}\ll\alpha_n\ll1$ such that $\frac{\hat{\alpha}_n}{\alpha_n}\to1$ in $f_{0,n}$-probability and that $f_n(X^n|\theta)$  satisfies Conditions \ref{condition1}-\ref{condition2}. Furthermore, assume  that $\pi(\theta) \leq \kappa$ for some $\kappa > 0$. Then the result in \eqref{eq:moment_result} holds with $\pi_{n,\alpha_n}(\theta|X^n)$ replaced with $\pi_{n,\hat{\alpha}_n}(\theta|X^n)$. 
\end{Theorem}

Note that $k=0$ in Theorem \ref{thm:moments_alt} immediately implies a BvM result that generalizes \cite[Theorem 1]{avella_medina_robustness_2022}, which we state below.
\begin{Corollary} \label{cor:BvM}
Assume the conditions of Theorem \ref{thm:moments_alt} hold. Then, in $f_{0,n}$-probability,
\begin{equation}\label{eq:BvM_result}
    \textrm{d}_{\textrm{TV}}\Big(\pi_{n,\hat{\alpha}_n}(\theta|X^n), \phi\Big(\theta \, \big| \, \hat{\theta}, \frac{1}{\alpha_nn}V_{\theta^*}^{-1}\Big)\Big) \rightarrow 0.
\end{equation}
\end{Corollary}

\subsubsection{Discussion}\label{sec:BvM_discussion}

Theorem \ref{thm:moments_alt} extends the moment convergence result of \cite[Theorem 1]{ray_asymptotics_2023} to a vanishing and random sequence $\hat\alpha_n$ and Corollary \ref{cor:BvM} extends the BvM theorem of \cite[Theorem 1]{avella_medina_robustness_2022} in a similar way. Our results recover these known results when $\hat\alpha_n$ is a fixed constant $\alpha$. 
Since we now assume that $\hat\alpha_nn \ll  n$  with high probability, our rate of concentration in the limiting variance is slower than that in the existing results for standard posteriors \cite{lehmann_asymptotic_1998, kleijn_bernstein-von-mises_2012, avella_medina_robustness_2022}.

We emphasize that the condition $\frac{1}{n}\ll\alpha_n$ is necessary for \eqref{eq:moment_result} to hold, so a BvM result does not hold for the $\hat{\alpha}_n$-posterior if $\hat{\alpha}_n$ is in the ``quickly vanishing" regime observed in Section \ref{sec:motivating_examples}. 
Indeed, we can show explicit counterexamples  when $\alpha_n \asymp \frac{1}{n}$. Consider the statistical model
\begin{equation}
\label{eq:exponential_family_likelihood}
    f_n(X^n|\eta) = \exp\left(\eta^\top T(X^n) - nA(\eta) - B(X^n)\right),
\end{equation}
where $\eta \in H$, some natural parameter space, and $B(X^n)$ is chosen such that \eqref{eq:exponential_family_likelihood} integrates to one. Furthermore, consider the prior density
\begin{equation}\label{eq:exponential_family_prior}
    \pi(\eta) = \exp\left(\eta^\top\xi - \nu A(\eta) - \psi(\xi, \nu)\right),
\end{equation}
where $\psi(\xi, \nu)$ is chosen such that \eqref{eq:exponential_family_prior} integrates to one. Simple manipulations show that the density of the $\alpha_n$-posterior has the form
\begin{align}
        \pi_{n,\alpha_n}(\eta|X^n) 
        \label{eq:alpha_post_exponential_short}&= \exp\big(\eta^\top [\alpha_nT(X^n) + \xi] -[\alpha_nn + \nu]A(\eta) - \psi(\alpha_nT(X^n) + \xi, \alpha_nn + \nu)\big).
\end{align}

With this simple construction, we show that when $\alpha_n \asymp \frac{1}{n}$, a BvM result like that of \eqref{eq:BvM_result} in Corollary \ref{cor:BvM} does not hold. We state this in the next proposition, which shows that even the weaker notion of covergence in distribution cannot hold in this case.
\begin{Proposition}
    \label{prop:counter_example_exp_family}
    Suppose the model in \eqref{eq:exponential_family_likelihood} is well specified with true parameter value given by $\eta^*$ and  that the characteristic function of the $\alpha_n$-posterior is continuous in $\xi_n$ and $\nu_n$. Then, if $\alpha_nn\to \alpha_0 \geq0$ and $\frac{1}{n}T(X^n) \rightarrow g(\eta^*)$ in $f_{0,n}$-probability, the $\alpha_n$-posterior \eqref{eq:alpha_post_exponential_short} converges weakly to a distribution belonging to the same family as \eqref{eq:exponential_family_prior}. 
\end{Proposition}
The proof of Proposition \ref{prop:counter_example_exp_family} is in Appendix \ref{app:counterexample_expfam_proof}.  Interpreting Proposition \ref{prop:counter_example_exp_family}, if the prior distribution \eqref{eq:exponential_family_prior} is not Gaussian, then the $\alpha_n$-posterior converges weakly to a distribution that is not Gaussian. Hence, a BvM-type theorem does not hold. Moreover,  when  \eqref{eq:exponential_family_likelihood} and \eqref{eq:exponential_family_prior} are Gaussian, even though  \eqref{eq:alpha_post_exponential_short} and its limiting distribution are Gaussian, the limiting distribution does not match the limiting Gaussian given in Corollary \ref{cor:BvM}. See 
Appendix \ref{app:counterexamples} for details and some examples of explicit limiting calculations for conjugate priors.

\subsection{Asymptotic normality of the $\hat{\alpha}_n$-posterior mean estimator}\label{sec:BayesEstimator}

The main result in this section, Theorem \ref{thm:BayesEstimator}, quantifies the difference between the $\hat\alpha_n$-posterior mean and the MLE, which, in turn, implies asymptotic normality of the $\hat{\alpha}_n$-posterior mean for any random sequence, $\hat{\alpha}_n$, such that $\frac{1}{\sqrt{n}} \ll \hat{\alpha}_n \ll 1$ with probability tending to one (Corollary \ref{cor:BayesEstimator}).

\subsubsection{Main result}\label{sec:BayesEstimator_result}
 We will require the following assumptions. Throughout, $\alpha_n$ denotes an arbitrary positive sequence satisfying $\frac{1}{n}\ll\alpha_n\ll1$ unless otherwise specified.
\begin{itemize}
    \item[\textbf{(A1')}] The prior, $\pi(\theta)$, is positive and twice continuously differentiable in  an open ball of $\theta^*$. Furthermore, the prior mean is finite.
    
    \item[\textbf{(A2')}] 
    For $n$ sufficiently large, there exist a matrix $H_n\in\mathbb{R}^{p\times p}$ that is symmetric and positive definite with high probability, a tensor $S_n\in\mathbb{R}^{p\times p\times p}$, and a constant $M < \infty$ such that for all $\epsilon > 0$ and some $\delta>0$, 
    \begin{equation}\label{eq:tildeR_remainder_control}
        \mathbb{P}_{f_{0,n}}\Big(\sup_{h\in\mathcal{H}_n}\tilde{R}_n(h) - \frac{M}{\alpha_n n}\|h\|_2^4 > 0\Big) < \epsilon,
    \end{equation}
    where $\mathcal{H}_n\equiv\{h\in\mathbb{R}^p|\frac{\|h\|_2}{\sqrt{\alpha_nn}} \leq \delta\}$ and
    \begin{equation*}
        \begin{split}
            \tilde{R}_n(h)
            &\equiv\alpha_n\Big[\log f_n\Big(X^n\Big|\hat{\theta}+\frac{h}{\sqrt{\alpha_nn}}\Big)-\log f_n(X^n|\hat{\theta})\Big] +\frac{1}{2}h^\top H_nh - \frac{1}{6\sqrt{\alpha_n n}}\langle S_n, h^{\otimes 3}\rangle.
        \end{split}
    \end{equation*}
    
    \item[\textbf{(A3')}] For all $\epsilon > 0$ and all $\delta' > 0$, there exists a constant $c\equiv c(\epsilon,\delta') > 0$ such that for $n$ sufficiently large, 
    \begin{equation}\label{eq:A3'}
       \mathbb{P}_{f_{0,n}}\Big(\sup_{\theta\in B_{\hat{\theta}}(\delta')^\mathsf{c}}\frac{1}{n}\log f_n(X^n|\theta) - \frac{1}{n}\log f_n(X^n|\hat{\theta}) > -c \Big) < \epsilon.
    \end{equation}
\end{itemize}

Assumption \textbf{(A1')} adds an additional differentiability requirement on the prior on top of Assumption \textbf{(A1)}. Assumption \textbf{(A2')} can be thought of as a higher order LAN condition. The standard LAN condition in \textbf{(A2)} is a quadratic approximation of the log-likelihood ratio centered at the pseudo-true parameter. Assumption \textbf{(A2')} is a cubic approximation of the log-likelihood ratio centered at the MLE. We provide sufficient conditions for verifying \textbf{(A2')} in Appendix \ref{app:likelihood_suff}. Both assumptions \textbf{(A1')} and \textbf{(A2')} are needed for the Laplace approximation that underpins our main result, as discussed below in Sections \ref{sec:postmean_discussion} and \ref{sec:laplace_approx}. These sort of smoothness assumptions are typical for Laplace approximation arguments \cite{kass_validity_1990, shun1995laplace, katsevich_laplace_2025, de2014asymptotic, erdelyi1956asymptotic}. Assumption \textbf{(A3')} is a condition on the salience of the MLE that is related to standard regularity conditions used to establish consistency of M-estimators \cite{van_der_vaart_asymptotic_1998}. A similar condition was stated in \cite{lehmann_asymptotic_1998} (pg.\ 489, assumption (B3)) to establish a BvM theorem for standard posteriors under classical conditions.  It has also been employed to establish posterior concentration (see \cite[Condition 1, eq.\ (8)]{syring_gibbs_2023}, \cite[Assumption S.1]{mclatchie_predictive_2025}, and \cite[Assumption (2)]{miller_asymptotic_2021}), which, in turn, implies that the posterior moments are bounded in probability as required in Assumption \textbf{(A3)}. See Appendix \ref{app:post_conc_suff} for more details.

We are now ready to state the main result of this section. The proof is in Appendix~\ref{app:BayesEstimatorProof}.

\begin{Theorem} \label{thm:BayesEstimator} Assume \textbf{(A0)}, \textbf{(A1')}, \textbf{(A2')}, and \textbf{(A3')} hold. Denote the $\alpha_n$-posterior mean estimator by $\hat{\theta}^\text{B} = [\hat{\theta}^\text{B}_1,\ldots,\hat{\theta}^\text{B}_p]^\top$, where $\hat{\theta}^{\text{B}}_j = \int_{\mathbb{R}^p}\theta_j\pi_{n,\alpha_n}(\theta|X^n)d\theta$ for $j=1,\ldots,p$. Then, 
\begin{align}
    \label{eq:diff}\hat{\theta}^{\text{B}}-\hat{\theta} &= O_{f_{0,n}}\Big(\frac{1}{\alpha_nn}\Big).
\end{align}
Furthermore, suppose that for a data-dependent, positive sequence $\hat{\alpha}_n$ there exists a positive sequence $\alpha_n\to0$ with $\frac{1}{n}\ll\alpha_n\ll1$ such that ${\hat{\alpha}_n}/{\alpha_n}\to1$ in $f_{0,n}$-probability. Then, the result in \eqref{eq:diff} holds with $\pi_{n,\alpha_n}(\theta|X^n)$ replaced with $\pi_{n,\hat{\alpha}_n}(\theta|X^n)$.
\end{Theorem}

A direct corollary of Theorem \ref{thm:BayesEstimator} is asymptotic normality of the $\hat{\alpha}_n$-posterior mean if $\hat{\alpha}_n$ decays sufficiently slowly (Corollary \ref{cor:BayesEstimator}). Indeed, the result of Theorem \ref{thm:BayesEstimator} tells us that $\sqrt{n}(\hat{\theta}^{\text{B}} - \theta^*) = \sqrt{n}(\hat{\theta}^{\text{B}} - \hat{\theta}) + \sqrt{n}(\hat{\theta} - \theta^*)=O_{f_{0,n}}(\frac{1}{\alpha_n\sqrt{n}}) + \sqrt{n}(\hat{\theta} - \theta^*)$. If $\frac{1}{\sqrt{n}}\ll\alpha_n\ll1$, then the first term is $o_{f_{0,n}}(1)$ and Assumption \textbf{(A0)} guarantees that the dominating term is asymptotically normal.

\begin{Corollary}  \label{cor:BayesEstimator} 
Assume the conditions of Theorem \ref{thm:BayesEstimator} hold. Then, for any sequence $\alpha_n$ such that $\frac{1}{\sqrt{n}}\ll\alpha_n\ll 1$, the $\alpha_n$-posterior mean estimator is asymptotically equivalent to $\hat{\theta}$. That is, $\sqrt{n}(\hat{\theta}^{\text{B}}-\theta^*) \overset{d}{\rightarrow} N(0,\tilde{V}_\theta),$ where $\tilde{V}_\theta$ is the asymptotic variance of the MLE. Furthermore, this result remains valid for the $\hat{\alpha}_n$-posterior mean estimator as described in Theorem \ref{thm:BayesEstimator}. 
\end{Corollary}

\begin{Remark}
    If $\log f_n(X^n|\theta^*)$ is differentiable at $\theta^*$, then $\tilde{V}_{\theta^*} = V_{\theta^*}^{-1}M_{\theta^*}V_{\theta^*}^{-1}$, where $V_{\theta^*}$ is as defined in \textbf{(A2)} and  $M_{\theta^*}$ is the limit in $f_{0,n}$-probability as $n\rightarrow\infty$ of $\frac{1}{n}\dot{\ell}_{n,\theta^*}\dot{\ell}_{n,\theta^*}^T$ and $\dot{\ell}_{n,\theta^*} = \nabla_{\theta}\log f_n(X^n|\theta)|_{\theta=\theta^*}$ is the score function of the model $f_n(X^n|\theta)$ at $\theta^*$.
\end{Remark}

\subsubsection{Discussion}\label{sec:postmean_discussion}

Corollary \ref{cor:BvM} shows that the total variation distance between the $\hat\alpha_n$-posterior and a normal distribution centered at the MLE converges to zero in $f_{0,n}$-probability when  $\frac{1}{n}\ll\hat\alpha_n\ll1$ with probability tending to one. This suggests that the posterior mean might be asymptotically equivalent to the MLE in this regime. However,  this is only the case in the more restricted regime where $\frac{1}{\sqrt{n}}\ll\hat\alpha_n\ll1$ with probability tending to one, which is precisely the result stated in Corollary \ref{cor:BayesEstimator}.

To see that Corollary \ref{cor:BayesEstimator} is sharp, it suffices to consider the Bayesian linear regression example of Section \ref{sec:motivating_examples} with a deterministic tempering parameter $\alpha_n$. In that example, the $\alpha_n$-posterior mean is the ridge regression estimator, which is  consistent when $\frac{1}{n}\ll\alpha_n\ll1$ but only equivalent to the MLE when $\frac{1}{\sqrt{n}}\ll\alpha_n\ll1$. We also illustrate this fact numerically in an  example that cannot be worked out in closed form. We consider a well-specified Bayesian logistic regression model with a  Gaussian prior with mean zero and standard deviation 100 on the regression coefficients. 
We sample from the $\alpha_n$-posterior under three settings of $\alpha_n$: $\alpha_n \asymp n^{-3/4}$, $\alpha_n \asymp n^{-1/2}$, and $\alpha_n \asymp n^{-1/4}$. We note that $\alpha_n \asymp n^{-1/4}$ is compatible with the regime specified by Theorem \ref{thm:BayesEstimator}, while $\alpha_n \asymp n^{-3/4}$ falls outside this regime with $\alpha_n \asymp n^{-1/2}$ on the border. We see in the left panel of Figure \ref{fig:post-mean-diff} that the $\sqrt{n}$-scaled difference between the posterior mean and the MLE diverges when $\alpha_n \asymp n^{-3/4}$, vanishes when $\alpha_n \asymp n^{-1/4}$, and is constant order when $\alpha_n \asymp n^{-1/2}$. These behaviors are also reflected in the right panel of Figure \ref{fig:post-mean-diff}, which shows the $\sqrt{n}$-rescaled difference between the $\alpha_n$-posterior mean and the true parameter value.

\begin{figure}
    \centering
    \includegraphics{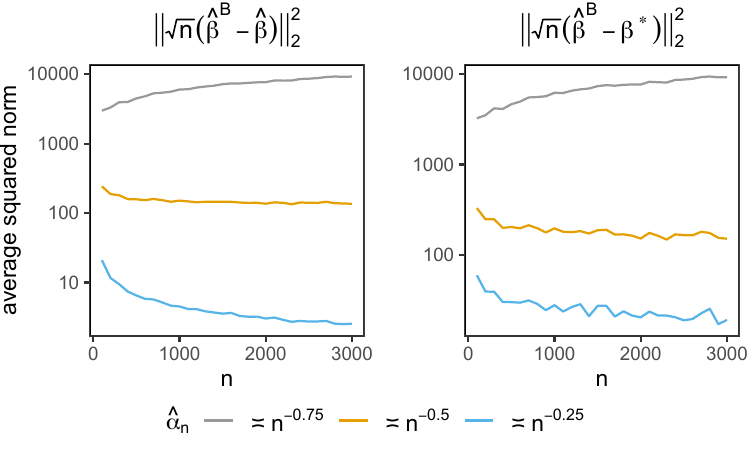}
    \caption{The squared norm of $\sqrt{n}(\hat{\beta}^\text{B} - \hat{\beta})$ (left) and $\sqrt{n}(\hat{\beta}^\text{B} - \beta^*)$ (right) averaged over 100 replications, as a function of sample size. The $\alpha_n$-posterior mean was obtained by samples from the $\alpha_n$-posterior, where $\alpha_n = 0.5n^{-3/4}$ (gray), $\alpha_n = 0.5n^{-1/2}$ (orange), and $\alpha_n = 0.5n^{-1/4}$ (blue). The data was generated from the i.i.d.\ model $\mathbb{P}(Y_i = 1|x_i) = [1 + \exp(-x_i^\top  \beta^*)]^{-1}$, where $\{x_i\} \overset{iid}{\sim}\mathcal{N}(0, I_3)$ and $\beta^* =[1, -0.5, 0.1]$. }
    \label{fig:post-mean-diff}
\end{figure}

Usually, given a moment convergence result like Theorem \ref{thm:moments_alt}, a result along the lines of Theorem \ref{thm:BayesEstimator} is a simple consequence requiring only slight changes in proof technique (see \cite[Ch.\ 6, Theorems 8.2-8.3]{lehmann_asymptotic_1998} for an example for standard posteriors). However, this approach no longer works when $\alpha_n\to0$  as it would only lead to $\sqrt{n}(\hat{\theta}^{\text{B}}-\hat{\theta}) = o_{f_{0,n}}(\frac{1}{\sqrt{\alpha_n}})$, which is too loose.  Indeed, using Theorem \ref{thm:moments_alt} setting $k=1$, it is easy to see that
\begin{align*}
\sqrt{\alpha_nn}\|\hat{\theta}^B-\hat{\theta}\|_1
    &= \sqrt{\alpha_nn}\Big\|\int(\theta-\theta^*)\pi_{n,\alpha_n}(\theta|X^n)d\theta-\int(\theta-\theta^*)\phi\Big(\theta\Big|\hat{\theta},\frac{1}{\alpha_nn}V_{\theta^*}^{-1}\Big)d\theta\Big\|_1 \\
    &\le \int \|\sqrt{\alpha_nn}(\theta-\theta^*)\|_1\Big|\pi_{n,\alpha_n}(\theta|X^n) - \phi\Big(\theta|\hat{\theta},\frac{1}{\alpha_nn}V_{\theta^*}^{-1}\Big)\Big|d\theta = o_{f_{0,n}}(1).
\end{align*}
In contrast, the corresponding result from Corollary \ref{cor:BayesEstimator} is that $\sqrt{n}(\hat{\theta}^{\text{B}}-\hat{\theta}) = O_{f_{0,n}}(\frac{1}{\alpha_n\sqrt{n}})$, which is $o_{f_{0,n}}(1)$ if $\alpha_n\sqrt{n}\to\infty$ (i.e., $\frac{1}{\sqrt{n}}\ll\alpha_n$).

The key ingredient in the proof of Theorem \ref{thm:BayesEstimator} is a new Laplace approximation. This is a natural approach since the coordinates of the posterior mean estimator are defined by 
\begin{equation}\label{eq:post_mean_ratio}
    \begin{split}
        \hat{\theta}^{\text{B}}_j
        &= \frac{\int_{\mathbb{R}^p} \theta_j\pi(\theta)\exp\left(\alpha_n\log f_n(X^n|\theta)\right)d\theta}{\int_{\mathbb{R}^p}\pi(\theta)\exp\left(\alpha_n\log f_n(X^n|\theta)\right)d\theta}, \quad \mbox{ for } j=1,\dots,p,
    \end{split}
\end{equation}
which, for some function $b:\mathbb{R}^p\rightarrow\mathbb{R}$, is a ratio of two integrals of the form 
\begin{equation}\label{eq:laplace_gen}
    \int_{\mathbb{R}^p} b(\theta)\exp\Big(\alpha_n\log f_n(X^n|\theta)\Big)d\theta.
\end{equation}
In particular, \eqref{eq:post_mean_ratio}  uses $b(\theta) = \theta_j\pi(\theta)$ in the numerator and $b(\theta) = \pi(\theta)$ in the denominator. Laplace approximations are asymptotic expansions to Laplace-type integrals such as \eqref{eq:laplace_gen}. In this work, we refer to ``Laplace approximation" as an approximation of Laplace-type integrals \cite{lindley_use_1961, lindley_approximate_1980, kass_approximate_1989, kass_validity_1990} and not a BvM \cite{katsevich_improved_2025} or an approximation of the normalizing constant of the posterior \cite{bhattacharya_nonasymptotic_2020, miller_asymptotic_2021}, although these are all interrelated concepts.

\subsubsection{Laplace approximation}\label{sec:laplace_approx}

We provide a general Laplace approximation for integrals of the form  \eqref{eq:laplace_gen} in Lemma \ref{lem:lap_int_lem}. For this we require the following  regularity condition on the function $b$ in \eqref{eq:laplace_gen}.  

\begin{itemize}
    \item [\textbf{(ALap)}] For a function 
    $b:\mathbb{R}^p\to\mathbb{R}$,  
    assume the following conditions. \textbf{1.\ } $b$ is integrable on $\mathbb{R}^p$.
        \textbf{2.\ } $b(\theta)$ is finite in an open ball around $\theta^*$.
        \textbf{3.\ } There exist a vector $v\in\mathbb{R}^p$ and a constant $M < \infty$ such that for all $\epsilon > 0$ and some $\delta > 0$,
        \begin{equation}\label{eq:b_remainder_control}
            \mathbb{P}_{f_{0,n}}\Big(\sup_{h\in\mathcal{H}_{n,b}} R_b(h) - \frac{M}{\alpha_n n}\|h\|_2^2 > 0\Big) < \epsilon,
        \end{equation}
        where $\mathcal{H}_{n,b}\equiv\{h\in\mathbb{R}^p|\frac{\|h\|_2}{\sqrt{\alpha_nn}} \leq \delta\}$ and 
        $R_b(h) \equiv [b(\hat{\theta} + \frac{h}{\sqrt{\alpha_n n}}) - b(\hat{\theta})] - \frac{v^\top h}{\sqrt{\alpha_n n}}.$
\end{itemize}

\begin{Lemma}\label{lem:lap_int_lem}
    Suppose assumptions \textbf{(A0)}, \textbf{(A1')}, \textbf{(A2')}, and \textbf{(A3')} hold. Given a function $q:\mathbb{R}^p\to\mathbb{R}$, define $b(\theta)\equiv q(\theta)\pi(\theta)$, where this choice of $b$ is assumed to satisfy \textbf{(ALap)}. For any sequence $\alpha_nn\rightarrow\infty$, 
    the following hold for $n$ sufficiently large:
\begin{enumerate}
    \item For any $\alpha_n > 0$, with $H_n$ defined in Assumption \textbf{(A2')}, the following expansion holds 
\begin{equation}
\label{eq:lap_lem_result}
        \int_{\mathbb{R}^p}  b(\theta) e^{\alpha_n  \log f_n(X^n|\theta)  } d\theta =e^{\alpha_n \log f_n(X^n|\hat{\theta}) }|H_n|^{-\frac{1}{2}}\Big(\frac{2\pi}{\alpha_n n}\Big)^{\frac{p}{2}}\Big( b(\hat{\theta}) + O_{f_{0,n}}\Big(\frac{1}{\alpha_n n}\Big)\Big).
    \end{equation}
    \item The expectation of $q(\theta)$  with respect to the $\alpha_n$-posterior satisfies the  approximation
    \begin{equation}\label{eq:cond_exp_ratio}
        \begin{split}
            \int_{\mathbb{R}^p} q(\theta)\pi_{n,\alpha_n}(\theta|X^n)d\theta
            &= q(\hat{\theta}) + O_{f_{0,n}}\Big(\frac{1}{\alpha_n n}\Big).
        \end{split}
    \end{equation}
\end{enumerate}
    Furthermore, suppose for a data-dependent, positive sequence, $\hat{\alpha}_n$, there exists a positive sequence $\alpha_n$ with $\alpha_nn\to\infty$ such that $\frac{\hat{\alpha}_n}{\alpha_n}\to1$ in $f_{0,n}$-probability. Then the result in \eqref{eq:cond_exp_ratio} holds with $\pi_{n,\alpha_n}(\theta|X^n)$ replaced with $\pi_{n,\hat{\alpha}_n}(\theta|X^n)$.
\end{Lemma}

The proof of Lemma \ref{lem:lap_int_lem} can be found in Appendix \ref{app:laplace_lemma_proof}. It follows most closely the arguments advanced in \cite[Theorem 1]{kass_validity_1990} while accounting for the tempering parameter $\alpha_n$ and making fewer differentiability assumptions, which are sufficient for our purposes. Expansions of Laplace-type integrals are found in \cite{erdelyi1956asymptotic, de2014asymptotic, kirwin2010} and are obtained specifically for integrals with respect to Bayesian posteriors in \cite{lindley_use_1961, lindley_approximate_1980, kass_approximate_1989, shun1995laplace, katsevich_laplace_2025}. 
We note that even though \textbf{(ALap)} looks different from the differentiability assumption on $b$ stated in \cite{kass_validity_1990}, both assumptions can be verified with the regularity conditions made in Proposition \ref{prop:prior_suff}.

Finally, the conditions in Lemma \ref{lem:lap_int_lem} are also roughly similar to those in \cite{katsevich_laplace_2025}, as we require smoothness of the prior and log-likelihood up to the order of the expansion with bounded derivatives (Propositions \ref{prop:prior_suff} and \ref{prop:likelihood_suff}) and control of the growth of both outside a neighborhood of $\hat{\theta}$ to ensure integrability of the left hand side of \eqref{eq:lap_lem_result} (implicitly Assumption \textbf{(A1')} and Assumption \textbf{(A3')}). However, the conditions in \cite{katsevich_laplace_2025} involve more careful consideration of the dimension dependence and growth rate of the derivatives in order to obtain a tight approximation in the high dimensional setting considered in that work.

We conclude this section with the proposition below (proof in Appendix \ref{app:prior_suff}), which gives sufficient conditions on $q$ that are easy to check and imply Condition 3 of Assumption \textbf{(ALap)}.   Indeed, to apply the results of Lemma \ref{lem:lap_int_lem} in the proof of Theorem \ref{thm:BayesEstimator}, we need to verify Assumption \textbf{(ALap)} for functions of the form $b(\theta) = q(\theta)\pi(\theta)$.

\begin{Proposition}\label{prop:prior_suff}
    Assume \textbf{(A0)} and \textbf{(A1')} hold. Suppose $b:\mathbb{R}^p\to\mathbb{R}$ is of the form $b(\theta)=q(\theta)\pi(\theta)$ for some $q:\mathbb{R}^p\to\mathbb{R}$ and that $q$ is twice continuously differentiable on $B_{\theta^*}(\gamma)$ for some $\gamma > 0$. 
    Then, Condition 3 of Assumption \textbf{(ALap)} holds with $\delta=\gamma/2$, $v=\nabla b(\hat{\theta}) = \pi(\hat\theta)\nabla q(\hat\theta) + q(\hat\theta)\nabla\pi(\hat\theta)$, and $M = p\sup_{\theta\in B_{\theta^*}(\gamma)}  \max_{1 \leq i,j \leq p}| [\nabla^2b(\theta)]_{i,j}|$ for $\alpha_nn$ large enough. 
\end{Proposition}

\subsection{BvM-type theorem for $\hat{\alpha}_n^{\text{mix}}$-posterior}\label{sec:alpha_mix}

One observation gathered from the empirical studies in Section~\ref{sec:motivating_examples} is that, in some settings, the limiting distributions of $\alpha$'s selected in data-driven ways will have a mixed distribution, characterized by a point mass at infinity with the remaining probability mass converging to zero slower than $1/n$ (see Figure \ref{fig:lim_dist_mix}, Table \ref{tab:conv_rates}, and Figure \ref{fig:conv-rate} in Appendix \ref{sec:alpha_lambda}). Our aim is to formalize the limiting behavior of the resulting posterior, which we will refer to as the $\hat{\alpha}_n^{\text{mix}}$-posterior, in this section.  The object of interest is  $ \pi_{n,\hat{\alpha}_n^{\text{mix}}}(\theta|X^n)$, where letting $q_n \equiv q + o_{f_{0,n}}(1)$ for some $q \in [0,1]$, we  write 
\begin{equation}
\label{eq:mix}
    \pi_{n,\hat{\alpha}_n^{\text{mix}}}(\theta|X^n) = q_n\pi_{n,\hat{\alpha}_n}(\theta|X^n) + (1-q_n)\pi_{n,\hat{\gamma}_n}(\theta|X^n),
\end{equation}
where for all $\epsilon > 0$, for all $B > 0$, and a sequence $\frac{1}{n}\ll\alpha_n\ll1$,
\begin{equation}
\label{eq:mix_alpha}
\mathbb{P}_{f_{0,n}}\Big(1 - \epsilon < \frac{\hat{\alpha}_n}{\alpha_n} < 1 + \epsilon\Big) \rightarrow 1 \quad \text{ and } \quad \mathbb{P}_{f_{0,n}}\big(\hat{\gamma}_n > B\big) \to 1.
\end{equation}
Essentially what is being studied in \eqref{eq:mix}, is a mixture posterior distribution $\pi_{n,\hat{\alpha}_n^{\text{mix}}}(\theta|X^n)$, where the tempering $\alpha$ is a point mass at infinity with probability $1-q_n$ and is shrinking slowly enough with probability $q_n$. 
In what follows, we investigate the limiting distribution of the mixed posterior, $\pi_{n,\hat{\alpha}_n^{\text{mix}}}(\theta|X^n)$.

We first need to establish a key intermediate result formalizing the behavior of the $\infty$-posterior defined as $\pi_{n,\infty}(\theta|X^n)\equiv\lim_{n \rightarrow\infty}\pi_{n, \hat{\gamma}_n}(\theta|X^n)$ where $\hat{\gamma}_n$ diverges in probability. The limiting distribution of the  tempered posterior given by \cite[Theorem 1]{avella_medina_robustness_2022} suggests that when driving the tempering parameter to infinity, the resulting $\infty$-posterior should converge to a mass point at $\hat{\theta}$, which we denote by $\delta_{\hat{\theta}}$. Similar results have been claimed informally in the literature, e.g.\ in \cite{mclatchie_predictive_2025}, or established for Gibbs posteriors in terms of weak convergence \cite{hwang_laplaces_1980}. This suggests that the limiting distribution of the mixed posterior $\pi_{n,\hat{\alpha}_n^{\text{mix}}}(\theta|X^n)$ from \eqref{eq:mix} will also be a mixture distribution between a point mass at the MLE with probability $1-q$ and a limiting Gaussian (characterized by the BvM results from Section~\ref{sec:BvM+moments}; see Corollary \ref{cor:BvM}) with probability $q$. This result is formalized in Theorem~\ref{thm:BvM_mix_data-dep} below.

Unlike the convergence results of the previous sections, if $\pi_{n,\hat{\gamma}_n}(\theta|X^n)$  is absolutely continuous with respect to the Lebesgue measure for each $n$, then convergence in total variation cannot be achieved. Indeed, in this case $d_{\text{TV}}(\pi_{n,\hat{\gamma}_n}(\theta|X^n), \delta_{\hat{\theta}})=1$ for each fixed $n$. Although convergence of the $\infty$-posterior to $\delta_{\hat{\theta}}$ does not hold in total variation distance, it does hold in $p$-Wasserstein distance. We formalize this in the result below.
\begin{Proposition}\label{prop:alpha_inf}
Assume \textbf{(A0)}, \textbf{(A1)}, and  \textbf{(A3')} hold. Moreover, assume that $(i)$ $f_n(X^n|\theta)$ is continuous at $\hat\theta$ and $(ii)$ $\int \|\theta\|^p \,\pi(\theta) d\theta < \infty$. Then, for any $t > 0$ and any deterministic sequence $\gamma_n \rightarrow \infty$ as $n \rightarrow \infty$, in $f_{0,n}$-probability,
\begin{equation}
\label{eq:conv_1}
    \textrm{d}_{\textrm{W}_p} \big(\pi_{n,\gamma_n},\delta_{\hat\theta}\big) \to 0.
\end{equation}
Furthermore, if a random sequence $\hat\gamma_n$ is growing (i.e., it satisfies \eqref{eq:mix_alpha}), then the result in \eqref{eq:conv_1} remains true with $\hat{\gamma}_n$ replacing $\gamma_n$.
\end{Proposition}

The proof of Proposition \ref{prop:alpha_inf} is in Appendix \ref{app:BvM_inf}. 
We now state our main result regarding the limiting distribution of the  $\hat{\alpha}_n^{\text{mix}}$-posterior and the proof is in Appendix \ref{app:BvM_mix_data-dep}.

\begin{Theorem}\label{thm:BvM_mix_data-dep}
    Consider the mixed posterior $\pi_{n,\hat{\alpha}_n^{\text{mix}}}(\theta|X^n)$ described in \eqref{eq:mix} and \eqref{eq:mix_alpha}. Assume the conditions of Theorem \ref{thm:moments_alt} (such that the result in \eqref{eq:moment_result} holds with $\pi_{n,\alpha_n}(\theta|X^n)$ replaced with $\pi_{n,\hat{\alpha}_n}(\theta|X^n)$) and Proposition \ref{prop:alpha_inf} hold. Assume further that $\int \|\theta\|^m \,\pi(\theta) d\theta < \infty$ for some $m > p$.   
    Then, letting $\phi_n := \phi(\theta | \hat{\theta}, \frac{1}{\alpha_nn}V_{\theta^*}^{-1})$, we have in $f_{0,n}$-probability 
    \begin{equation}\label{eq:conv_2}
        \textrm{d}_{\textrm{W}_p} \Big(\pi_{n,\hat{\alpha}_n^{\text{mix}}}, \, q  \phi_n +(1-q)\delta_{\hat\theta}\Big)\to 0.
    \end{equation}
\end{Theorem}
\section{Discussion}\label{sec:discussion}
The utility of posterior tempering depends on optimally tuning the tempering parameter, $\alpha$. While practical recommendations have been suggested for tuning $\hat{\alpha}_n$ from the data, no statistical guarantees exist to guide inference for the resulting tempered $\hat{\alpha}_n$-posterior. In this work, we prove (i) consistency of the $\hat{\alpha}_n$-posterior moments (Theorem \ref{thm:moments_alt}) (ii) asymptotic normality of the $\hat{\alpha}_n$-posterior mean (Theorem \ref{thm:BayesEstimator}/Corollary \ref{cor:BayesEstimator}) and (iii) a BvM-type theorem for the mixed $\hat{\alpha}_n$-posterior (Theorem \ref{thm:BvM_mix_data-dep}) in asymptotic regimes of $\hat{\alpha}_n$ that had not been considered in previous theoretical analyses of tempered posteriors.

These novel regimes --  vanishing random sequences, diverging random sequences, and mixture distributions — are validated by the numerical experiments we conduct on simulated and real-world data. By establishing guarantees in these regimes, we bridge the gap between theory and practice. Our results reveal that $\hat{\alpha}_n$-posteriors behave differently than $\alpha$-posteriors with a fixed, nonrandom $\alpha$. For any fixed $\alpha > 0$, the $\alpha$-posterior is asymptotically normal \cite{avella_medina_robustness_2022}. In contrast, there are regimes of $\hat{\alpha}_n$ that arise empirically where the $\hat{\alpha}_n$-posterior is not asymptotically normal (e.g., Theorem \ref{thm:BvM_mix_data-dep}, Proposition \ref{prop:counter_example_exp_family}). Our results in these regimes inform the type of inferences that can be made from such $\hat{\alpha}_n$-posteriors.

While our asymptotic guarantees hold for tempered posteriors, we believe our proof techniques can be used to extend our results to modern posterior constructions like generalized posteriors or variational approximations of $\hat{\alpha}_n$-posteriors. The latter extension would be particularly relevant, as variational inference approximates $\hat{\alpha}_n$-posteriors that cannot be computed in closed form.

Our work provides the foundation for several interesting future directions. Our empirical work on a Bayesian linear regression model revealed that Bayesian cross-validation selects a tempering parameter that has a mixture distribution, a finding that was also reported in \cite{mclatchie_predictive_2025} for a normal location model. One line of work could be to theoretically understand why this mixture distribution arises, which may reveal other statistical models where cross-validation is unstable. Another direction is to generalize our results about the $\hat\alpha_n$-posterior mean to point estimators that are risk function minimizers (e.g., the $\hat\alpha_n$-posterior median). The Laplace approximation that we developed in Lemma \ref{lem:lap_int_lem} could be utilized for this purpose. Finally, our results hold for fixed-dimensional parametric models. It would be interesting to investigate the asymptotic behavior of $\hat\alpha_n$-posteriors in high-dimensional parametric settings where the parameter dimension increases with the sample size, or in semiparametric settings.
\end{doublespace}

\printbibliography

\newpage
\appendix

\section{Proofs of main results}\label{app:main_results}

\subsection{Proof of Theorem \ref{thm:moments_alt}}\label{app:main_results_thm2}
\begin{proof}
Our proof approach follows closely the proof of \cite[Theorem 1]{ray_asymptotics_2023}, which adapts arguments from the BvM results of \cite[Theorem 1]{avella_medina_robustness_2022} and \cite[Theorem 2.1]{kleijn_bernstein-von-mises_2012}, to show convergence of the posterior moments instead.
Let $Z$ denote the integral we are trying to control; namely,
\begin{align}
        Z &\equiv \int_{\mathbb{R}^p} \hspace{-4pt}\|\sqrt{\alpha_nn}(\theta-\theta^*)\|_2^k\Big|\pi_{n,\alpha_n}(\theta|X^n) - \phi\Big(\hspace{-1.1pt}\theta\Big|\hat{\theta},\frac{1}{\alpha_nn}V_{\theta^*}^{-1} \hspace{-1.1pt} \Big)\Big|d\theta 
        \label{eq:Z0_def}= \int_{\mathbb{R}^p}\hspace{-4pt} \|h\|_2^k|\pi^{LAN}_{n,\alpha_n}(h) - \phi_n(h)|dh,
\end{align}
where \eqref{eq:Z0_def} follows from the change of variable $h = \sqrt{\alpha_nn}(\theta-\theta^*)$ using the transformed densities 
\begin{equation}\label{eq:scaled_densities}
    \begin{split}
        \pi_{n,\alpha_n}^{\text{LAN}}(h) &\equiv \frac{1}{(\alpha_nn)^{\frac{p}{2}}} \pi_{n,\alpha_n}\Big(\theta^*+\frac{h}{\sqrt{\alpha_nn}}  \Big|  X^n\Big) \, \text{ and } \,
        \phi_n(h) \equiv \frac{1}{(\alpha_nn)^{\frac{p}{2}}}\phi\left(h \big| \sqrt{\alpha_nn}(\hat{\theta}-\theta^*),V_{\theta^*}^{-1}\right).
    \end{split}
\end{equation}
To establish \eqref{eq:moment_result}, it suffices to show that $\mathbb{P}_{f_{0,n}}(Z > t)$ converges to 0 for any $t > 0$.

Let  $\{x\}^+ = \max\{0,x\}$. For vectors $g,h\in \mathbb{R}^p$, define the random variable
\begin{equation}\label{fn+}
    \begin{split}
        & f_n(g,h)=\Big\{1-\frac{\phi_n(h)\,\pi_{n,\alpha_n}^{\text{LAN}}(g|X^n) }{\pi_{n,\alpha_n}^{\text{LAN}}(h|X^n) \,\phi_n(g)}\Big\}^+.
    \end{split}
\end{equation}
For arbitrary $\epsilon > 0$ and $\eta > 0$,  define an event
\begin{equation}\label{eq:events_moments}
    \begin{split}
         \mathcal{A} \equiv \mathcal{A}_n &= \Big\{ \sup_{g,h\in \bar{B}_{0}(r_n)}f_n(g,h) \leq \eta \Big\},
    \end{split}
\end{equation} 
where $r_n$ is the growing sequence in Lemma \ref{lem:f(g,h)} for which there exists a $N_0 \equiv N(\eta,\epsilon)$ such that for $\alpha_nn > N_0$, the random variable $f_n(g,h)$ is well defined with probability at least $1-\frac{\epsilon}{2}$ over $\bar{B}_{0}(r_n)$.

Notice that
\begin{align}\label{Zbound}
    \mathbb{P}_{f_{0,n}}(Z > t) \leq \mathbb{P}_{f_{0,n}}(\{Z > t\}\cap\mathcal{A}) + \mathbb{P}_{f_{0,n}}(\mathcal{A}^\mathsf{c}),
\end{align}
so to complete the proof, we aim to show that the first term in \eqref{Zbound} converges to $0$, since by Lemma \ref{lem:f(g,h)}, we have $\mathbb{P}_{f_{0,n}}(\mathcal{A}^\mathsf{c}) <  \epsilon$ for  all $\alpha_nn > N_0$.
In what follows, we focus on the first term of \eqref{Zbound}. Throughout the proof, for constants $0 < M_1 < \infty$ and $0 < M_2 < \infty$ to be specified later, we let
\begin{equation}\label{eq:Lem3_eta_ep}
    \eta \equiv \eta(t, M_1, M_2) = \frac{t}{8} \min\Big\{ \frac{1}{M_1}, \frac{1}{M_2}\Big\}.
\end{equation}
Notice $\eta$ is positive for any $t > 0$, allowing us to appeal to Lemma \ref{lem:f(g,h)}, and we will show that this choice of $\eta$ results in $\mathbb{P}_{f_{0,n}}(\{Z > t\}\cap\mathcal{A}) \to 0$ for any $t > 0$.

By Lemma \ref{lem:f(g,h)}, the density $\pi_{n,\alpha_n}^{\text{LAN}}$ is positive uniformly on $\bar{B}_0(r_n)$ with probability at least $1-\frac{\epsilon}{2}$ whenever $\alpha_nn > N_0$. Furthermore, $\phi_n(h)$ is positive everywhere, as it is a Gaussian density. Hence, whenever $\alpha_nn \geq N_0$, we can apply Lemma~\ref{lem:TVbound} with $s(h) = \|h\|_2^k$ and $K =  \bar{B}_0(r_n)$ -- which is compact for any fixed $\alpha_nn$ -- to obtain
\begin{align}
    Z
    \nonumber&= \int_{\mathbb{R}^p}\|h\|_2^k\left|\pi_{n,\alpha_n}^{\text{LAN}}(h|X^n) - \phi_n(h)\right|dh \\
    \label{lem3bd}&\leq \Big[ \sup_{g',h'\in \bar{B}_{0}(r_n)}f_n(g',h')\Big]\int_{\mathbb{R}^p}\|h\|_2^k\pi_{n,\alpha_n}^{\text{LAN}}(h|X^n)dh 
    + \Big[ \sup_{g',h'\in \bar{B}_{0}(r_n)}f_n(g',h')\Big]\int_{\mathbb{R}^p}\|h\|_2^k\phi_n(h)dh \\
    \nonumber& \qquad + \int_{\bar{B}_{0}(r_n)^\mathsf{c}}\|h\|_2^k \pi_{n,\alpha_n}^{\text{LAN}}(h|X^n)dh 
    + \int_{\bar{B}_{0}(r_n)^\mathsf{c}}\|h\|_2^k\phi_n(h)dh.
\end{align}
Label the four terms in \eqref{lem3bd} as $T_1$, $T_2$, $T_3$, and $T_4$. By the fact that 
$$\{T_1 + T_2 + T_3 + T_4 > t\} \cap \mathcal{A} \quad  \subseteq \quad \cup_{j=1}^4 \{T_j > \frac{t}{4}\} \cap \mathcal{A} \quad \subseteq \quad \cup_{j=1}^4 \Big\{\{T_j > \frac{t}{4}\} \cap \mathcal{A}\Big\},$$ 
along with a union bound, we find
\begin{align}
    \nonumber \mathbb{P}_{f_{0,n}}(\{Z > t\}\cap\mathcal{A})
    \leq \mathbb{P}_{f_{0,n}}(\{T_1 + T_2 + T_3 + T_4 > t\}\cap\mathcal{A}) 
    &\leq \mathbb{P}_{f_{0,n}}\Big(\cup_{j=1}^4\Big\{\{T_j > \frac{t}{4}\}\cap\mathcal{A}\Big\}\Big) \\
    \label{EZbd}&\leq \sum_{j=1}^4\mathbb{P}_{f_{0,n}}\Big(\{T_j > \frac{t}{4}\}\cap\mathcal{A}\Big).
\end{align}
Let $P_j \equiv \mathbb{P}_{f_{0,n}}(\{T_j > \frac{t}{4}\}\cap\mathcal{A})$. In what follows, we argue that, for any $\epsilon > 0$ and any $t > 0$, each can be  upper bounded by $\epsilon/4$ (possibly for $\alpha_nn$ large enough). \

\textbf{Term $\mathbf{P_1}$ of \eqref{EZbd}.} We have 
\begin{align}\label{eq:tildeT1}
    P_1
    = \mathbb{P}_{f_{0,n}}\Big(\Big\{ \Big[ \sup_{g',h'\in \bar{B}_{0}(r_n)}f_n(g',h')\Big]\int_{\mathbb{R}^p}\|h\|_2^k\pi_{n,\alpha_n}^{\text{LAN}}(h|X^n)dh > \frac{t}{4}\Big\}\cap\mathcal{A}\Big).
\end{align}
Notice that the supremum in \eqref{eq:tildeT1} is bounded by $\frac{t}{8} \min\{ \frac{1}{M_1}, \frac{1}{M_2}\}$ under the event $\mathcal{A}$. Hence, 
\begin{align*}\label{eq:tildeT1bd1}
    &P_1
    \leq \mathbb{P}_{f_{0,n}}\Big(\Big[ \frac{t}{8} \min\Big\{ \frac{1}{M_1}, \frac{1}{M_2}\Big\}\Big]\int_{\mathbb{R}^p}\|h\|_2^k\pi_{n,\alpha_n}^{\text{LAN}}(h|X^n)dh > \frac{t}{4}\Big) \nonumber \\
    &\leq \mathbb{P}_{f_{0,n}}\hspace{-1pt}\Big(   \int_{\mathbb{R}^p}\hspace{-1pt}\|h\|_2^k\pi_{n,\alpha_n}^{\text{LAN}}(h|X^n)dh > 2\max\{M_1, M_2\}\hspace{-1pt}\Big) \leq \mathbb{P}_{f_{0,n}}\hspace{-1pt}\Big(   \int_{\mathbb{R}^p}\hspace{-1pt}\|h\|_2^k\pi_{n,\alpha_n}^{\text{LAN}}(h|X^n)dh > M_1\hspace{-1pt}\Big),
\end{align*}
as $2\max\{M_1, M_2\} \geq 2M_1 \geq M_1$.
By Assumption \textbf{(A3)}, for all $\epsilon > 0$, there exists a constant $M_1 \equiv M_1(k) < \infty$ and an integer $N_1\equiv N_1(M_1, \epsilon)$ such that for all $\alpha_nn > N_1$, 
\begin{align*}
\mathbb{P}_{f_{0,n}}\Big(\int_{\mathbb{R}^p}\|h\|_2^{k}\pi_{n,\alpha_n}^{\text{LAN}}(h|X^n)dh > M_1\Big) < \frac{\epsilon}{4}.
\end{align*}
Choosing this $M_1$ in our definition of $\eta$ from \eqref{eq:Lem3_eta_ep} gives $P_1 \leq \frac{\epsilon}{4}$
 whenever $\alpha_nn > N_1$.

\textbf{Term $\mathbf{P_2}$ of \eqref{EZbd}.} Recall the definition of $P_2$:
\begin{equation}\label{eq:tildeT2}
    P_2 
    = \mathbb{P}_{f_{0,n}}\Big(\Big\{\Big[ \sup_{g',h'\in \bar{B}_{0}(r_n)}f_n(g',h')\Big]\int_{\mathbb{R}^p}\|h\|_2^k\phi_n(h)dh > \frac{t}{4}\Big\}\cap\mathcal{A}\Big).
\end{equation}
As in our bound for $P_1$, notice that the supremum in \eqref{eq:tildeT2} is bounded by $\frac{t}{8} \min\{ \frac{1}{M_1}, \frac{1}{M_2}\}$ under the event $\mathcal{A}$. Hence, 
$
    P_2
    \leq \mathbb{P}_{f_{0,n}}(\int_{\mathbb{R}^p}\|h\|_2^k\phi_n(h)dh > M_2) \leq \frac{\epsilon}{4},$
where the final inequality follows
by Assumption \textbf{(A0)} and Lemma \ref{lem:gauss_int_bound2},
from which we have that for all $\epsilon > 0$, there exist a constant $M_2 \equiv M_2(k) < \infty$ and an integer $N_2\equiv N_2(M_2, \epsilon)$ such that
\begin{equation}
    \begin{split}\label{eq:tildeT2E2c}
        &\mathbb{P}_{f_{0,n}}\Big(\int_{\mathbb{R}^p} \|h\|_2^k \phi_n(h) dh > M_2\Big) < \frac{\epsilon}{4},
    \end{split}
\end{equation}
for all $\alpha_nn > N_2$.
This is the $M_2$ we use in our definition of $\eta$ from \eqref{eq:Lem3_eta_ep}.

\textbf{Term $\mathbf{P_3}$ of \eqref{EZbd}.} We bound $P_3$, by appealing to Lemma~\ref{MomentKc2}. Under Assumption \textbf{(A3)}, letting $k_0 = k(1+\gamma)$ for arbitrary $\gamma > 0$, we have that $\int_{\mathbb{R}^p} \|h\|_2^{k(1+\gamma)}\pi_{n,\alpha_n}^{\text{LAN}}(h|X^n)dh = O_{f_{0,n}}(1)$. Hence, we can apply Lemma~\ref{MomentKc2} with $f_Z(z)$ being $\pi_{n,\alpha_n}^{\text{LAN}}(z|X^n)$ and conclude that there exists $N_3\equiv  N_3(\gamma, k, t, \epsilon)$ such that for all $\alpha_nn > N_3$,
\begin{equation*} 
\begin{split}
    P_3 &= \mathbb{P}_{f_{0,n}} \Big(\Big\{\int_{\bar{B}_{0}(r_n)^\mathsf{c}} \|h\|_2^k\pi_{n,\alpha_n}^{\text{LAN}}(h|X^n)dh > \frac{t}{4}\Big\}\cap\mathcal{A}\Big) \\
    &\leq \mathbb{P}_{f_{0,n}} \Big(\int_{\bar{B}_{0}(r_n)^\mathsf{c}} \|h\|_2^k\pi_{n,\alpha_n}^{\text{LAN}}(h|X^n)dh > \frac{t}{4}\Big) < \frac{\epsilon}{4}.
    \end{split}
\end{equation*}

\textbf{Term $\mathbf{P_4}$ of \eqref{EZbd}.} 
Recall that
\begin{align}\label{eq:tildeT4}
    P_4
    = \mathbb{P}_{f_{0,n}}\Big(\Big\{\int_{\bar{B}_{0}(r_n)^\mathsf{c}}\|h\|_2^k\phi_n(h)dh > \frac{t}{4}\Big\}\cap\mathcal{A}\Big) \leq \mathbb{P}_{f_{0,n}}\Big(\int_{\bar{B}_{0}(r_n)^\mathsf{c}}\|h\|_2^k\phi_n(h)dh > \frac{t}{4}\Big).
\end{align}
Note by H{\"o}lder's inequality,
\begin{equation}\label{eq:tildeT3bd1}
    \begin{split}
        \int_{\bar{B}_{0}(r_n)^\mathsf{c}}\|h\|_2^k \phi_n(h)dh 
        &= \int_{\mathbb{R}^p} \|h\|_2^k(1-\mathbbm{1}\{\|h\|_2\leq r_n\}) \phi_n(h) dh \\
        &\leq \Big(\int_{\mathbb{R}^p}\|h\|_2^{k(1+\gamma)}\phi_n(h)dh\Big)^{\frac{1}{1+\gamma}}\Big(\int_{\mathbb{R}^p}(1-\mathbbm{1}\{\|h\|_2\leq r_n\})\phi_n(h)dh \Big)^{\frac{\gamma}{1+\gamma}}.
    \end{split}
\end{equation}
To bound $P_4$, define the following event $\mathcal{E}_4 = \{\int_{\mathbb{R}^p} \|h\|_2^{k(1+\gamma)} \phi_n(h) dh < M_4^{1+\gamma}\}$ for some $M_4$ (to be specified later).
By Assumption \textbf{(A0)} and Lemma \ref{lem:gauss_int_bound2} applied with $k=k(1+\gamma)$, for all $\epsilon > 0$, there exist a constant $M_4 \equiv M_4(k, \gamma) < \infty$ and an integer $N_4\equiv N_4(\epsilon, M_4)$ such that for all $\alpha_nn > N_4$,
\begin{align*}
    \mathbb{P}_{f_{0,n}}(\mathcal{E}_4^\mathsf{c}) = \mathbb{P}_{f_{0,n}}\Big(\int_{\mathbb{R}^p} \|h\|_2^{k(1+\gamma)} \phi_n(h) dh > M_4^{1+\gamma}\Big) < \frac{\epsilon}{4}.
\end{align*}
Furthermore, note that because $r_n\rightarrow\infty$ and $\phi_n(h) \leq |V_{\theta^*}|^{1/2}$ for all $h\in\mathbb{R}^p$, we see that $(1-\mathbbm{1}\{\|h\|_2\leq r_n\})\phi_n(h) \rightarrow 0$ for any fixed $h$. Furthermore, $(1-\mathbbm{1}\{\|h\|_2\leq r_n\})\phi_n(h) \leq \phi_n(h)$, which is a proper density; hence, integrable. Thus, by the dominated convergence theorem, there exists $N_5\equiv N_5(t, M_4, \gamma)$ such that for all $\alpha_nn > \max(1, N_5)$,
\begin{align}\label{eq:tildeT3bd2}
    \int_{\mathbb{R}^p} (1-\mathbbm{1}\{\|h\|_2\leq r_n\}) \phi_n(h) dh < \Big(\frac{t}{8}\frac{1}{M_4}\Big)^{\frac{1+\gamma}{\gamma}}.\
\end{align}
Hence, by \eqref{eq:tildeT4}, \eqref{eq:tildeT3bd1}, and \eqref{eq:tildeT3bd2} we can bound $P_4$ as follows: whenever $\alpha_nn > \max(1, N_4, N_5)$,
\begin{align}
    &P_4
    \nonumber\leq \mathbb{P}_{f_{0,n}}\Big(\Big\{\int_{\bar{B}_{0}(r_n)^\mathsf{c}}\|h\|_2^k\phi_n(h)dh > \frac{t}{4}\Big\}\cap\mathcal{E}_4\Big) + \mathbb{P}_{f_{0,n}}(\mathcal{E}_4^\mathsf{c}) \\
    \nonumber&\leq \mathbb{P}_{f_{0,n}}\Big(\Big\{\Big(\int_{\mathbb{R}^p}\|h\|_2^{k(1+\gamma)}\phi_n(h)dh\Big)^{\frac{1}{1+\gamma}}\Big(\int_{\mathbb{R}^p} \hspace{-1.2pt}(1-\mathbbm{1}\{\|h\|_2\leq r_n\})\phi_n(h)dh\Big)^{\frac{\gamma}{1+\gamma}} > \frac{t}{4}\Big\} \cap \mathcal{E}_4\Big) + \frac{\epsilon}{4} \\
    \nonumber&\leq \mathbb{P}_{f_{0,n}}\Big(M_4\Big(\frac{t}{8}\frac{1}{M_4}\Big) > \frac{t}{4}\Big) + \frac{\epsilon}{4} = \frac{\epsilon}{4}
\end{align}

Now that we have bounded terms $P_1 - P_4$, using \eqref{Zbound} we find the result in equation \eqref{eq:moment_result}.

To show that the result in \eqref{eq:moment_result} holds with $\pi_{n,\hat{\alpha}_n}(\theta|X^n)$ replacing $\pi_{n,\alpha_n}(\theta|X^n)$, we rewrite the $\hat{\alpha}_n$-posterior as follows:
\begin{equation*}
    \pi_{n,\hat{\alpha}_n}(\theta|X^n) = \frac{f_n(X^n|\theta)^{\hat{\alpha}_n}\pi(\theta)}{\int f_n(X^n|\theta)^{\hat{\alpha}_n}\pi(\theta)d\theta} 
    = \frac{[f_n(X^n|\theta)^{\frac{\hat{\alpha}_n}{\alpha_n}}]^{\alpha_n}\pi(\theta)}{\int [f_n(X^n|\theta)^{\frac{\hat{\alpha}_n}{\alpha_n}}]^{\alpha_n}\pi(\theta) d\theta} = \frac{\tilde{f}_n(X^n|\theta)^{\alpha_n}\pi(\theta)}{\int \tilde{f}_n(X^n|\theta)^{\alpha_n}\pi(\theta)d\theta},
\end{equation*}
where $\tilde{f}_n(X^n|\theta)^{\alpha_n} = f_n(X^n|\theta)^{\frac{\hat{\alpha}_n}{\alpha_n}}$. Hence, we obtain the desired result by applying Theorem \ref{thm:moments_alt} with $\tilde{f}_n(X^n|\theta)$ and $\pi(\theta)$, from which the following converges to 0 in $f_{0,n}$-probability:
\begin{equation}\label{eq:moment_result_random}
    \int_{\mathbb{R}^p}\|\sqrt{\alpha_nn}(\theta-\theta^*)\|_2^k\left|\frac{\tilde{f}_n(X^n|\theta)^{\alpha_n}\pi(\theta)}{\int \tilde{f}_n(X^n|\theta)^{\alpha_n}\pi(\theta)d\theta} - \phi\Big(\theta\big|\hat{\theta},\frac{1}{\alpha_nn}V_{\theta^*}^{-1}\Big)\right|d\theta.
\end{equation}
To apply Theorem \ref{thm:moments_alt}, we need to verify that $\tilde{f}_n(X^n|\theta)$ and $\pi(\theta)$ satisfy assumptions \textbf{(A0)}--\textbf{(A3)}. Note that \textbf{(A1)} is already assumed to hold.
First notice that the objects $\hat{\theta}$, $\theta^*$, and $V_{\theta^*}$ appear in \eqref{eq:moment_result_random}. Hence, we must show not only that \textbf{(A0)} and \textbf{(A2)} hold for $\tilde{f}_n(X^n|\theta)$, but that they hold with the same $\hat{\theta}$, $\theta^*$, and $V_{\theta^*}$ for which assumptions \textbf{(A0)} and \textbf{(A2)} hold for $f_n(X^n|\theta)$. We verify this in Proposition \ref{prop:random_alpha_(A0)(A2)}. By Proposition \ref{prop:random_alpha(A3)}, Assumption \textbf{(A3)} holds for $\tilde{f}_n(X^n|\theta)$ (since we have assumed $f_n(X^n|\theta)$ and $\pi(\theta)$ satisfy the conditions of Proposition \ref{prop:A3_suff}). 
\end{proof}

\subsection{Proof of Theorem \ref{thm:BayesEstimator}}\label{app:BayesEstimatorProof}
\begin{proof}
Recall that the posterior mean is $\hat{\theta}^{\text{B}} = [\hat{\theta}^{\text{B}}_1,\ldots,\hat{\theta}^{\text{B}}_p]^\top$ where $\hat{\theta}^{\text{B}}_j = \int_{\mathbb{R}^p}\theta_j\pi_{n,\alpha_n}(\theta|X^n)d\theta$.
Suppose, for a moment, that the conditions of Lemma \ref{lem:lap_int_lem} hold (we will verify them later). Then, appealing to the second result of Lemma \ref{lem:lap_int_lem} for each $j=1,\ldots,p$, we have
$
    \hat{\theta}^{\text{B}}_j = \hat{\theta}_j + O_{f_{0,n}}(\frac{1}{\alpha_nn}).$
This implies that there exists a constant $M_j > 0$ and an integer $N_j\equiv N_j(\epsilon,p)$ such that for $\alpha_nn > N_{j}$,
\begin{equation}\label{eq:post_mean_coordinate_diff}
    \mathbb{P}_{f_{0,n}}\Big(\Big|\hat{\theta}^{\text{B}}_j-\hat{\theta}_j\Big| > \frac{M_j}{\alpha_nn}\Big) < \frac{\epsilon}{p}.
\end{equation}
We will use \eqref{eq:post_mean_coordinate_diff} to show that 
\begin{equation}\label{eq:post_mean_diff}
    \hat{\theta}^{\text{B}} - \hat{\theta} = O_{f_{0,n}}\Big(\frac{1}{\alpha_nn}\Big),
\end{equation}
or in other words, there exists a constant $M$ and an integer $N$ such that for $\alpha_nn > N$, we have
$
\mathbb{P}_{f_{0,n}}(\|\hat{\theta}^{\text{B}}-\hat{\theta}\|_2 > \frac{M}{\alpha_nn}) < \epsilon.
$
To show this, notice
$
\|\hat{\theta}^{\text{B}} - \hat{\theta}\|^2_2 = \sum_{j=1}^p(\hat{\theta}^{\text{B}}_j - \hat{\theta}_j)^2 \leq p  \max_{j \in [p]}\{(\hat{\theta}^{\text{B}}_j-\hat{\theta}_j)^2\}.
$
Hence, from a union bound we find
\begin{align}
\nonumber \mathbb{P}_{f_{0,n}}\Big(\|\hat{\theta}^{\text{B}} - \hat{\theta}\|_2 > \frac{M}{\alpha_nn} \Big) &\leq \mathbb{P}_{f_{0,n}}\Big( \max_{j \in [p]}\{(\hat{\theta}^{\text{B}}_j-\hat{\theta}_j)^2\} > \frac{M^2}{p(\alpha_nn)^2}\Big) \\
&= \nonumber \mathbb{P}_{f_{0,n}}\Big(\cup_{j=1}^p \hspace{-1pt}\Big\{(\hat{\theta}^{\text{B}}_j-\hat{\theta}_j)^2 > \frac{M^2}{p(\alpha_nn)^2}\Big\} \Big) \\
&\leq  \sum_{j=1}^p \mathbb{P}_{f_{0,n}}\Big(\Big|\hat{\theta}^{\text{B}}_j-\hat{\theta}_j\Big| > \frac{M}{\sqrt{p}\alpha_nn}\Big). \label{eq:union_bd}
\end{align}
Now, let $M = \sqrt{p}\max\{M_1,\ldots,M_p\}$, for the $M_j$ coming from  \eqref{eq:post_mean_coordinate_diff}. 

By \eqref{eq:union_bd} and \eqref{eq:post_mean_coordinate_diff}, for $\alpha_nn > \max(N_1,\ldots,N_p)$, we have 
\begin{equation*}
    \begin{split}
    \mathbb{P}_{f_{0,n}}\Big(\|\hat{\theta}^{\text{B}} - \hat{\theta}\|_2 > \frac{M}{\alpha_nn}\Big) &\leq \sum_{j=1}^p \mathbb{P}_{f_{0,n}}\Big(\Big|\hat{\theta}^{\text{B}}_j-\hat{\theta}_j\Big| > \frac{M}{\sqrt{p}\alpha_nn}\Big) \\
        &= \sum_{j=1}^p\mathbb{P}_{f_{0,n}}\Big(\Big|\hat{\theta}^{\text{B}}_j-\hat{\theta}_j\Big| > \frac{\sqrt{p}\max\{M_1,\ldots,M_p\}}{\sqrt{p}\alpha_nn}\Big) \\
        &\leq \sum_{j=1}^p \mathbb{P}_{f_{0,n}}\Big(\Big|\hat{\theta}^{\text{B}}_j-\hat{\theta}_j\Big| > \frac{M_j}{\alpha_nn}\Big) < \sum_{j=1}^p \frac{\epsilon}{p} = \epsilon.\
        \end{split}
\end{equation*}
It remains to verify the conditions of Lemma \ref{lem:lap_int_lem}. We have  assumed \textbf{(A0)}, \textbf{(A1')}, \textbf{(A2')}, and \textbf{(A3')} to hold, so it remains to verify that Assumption \textbf{(ALap)} holds for $b(\theta) = \theta_j\pi(\theta)$ for each $j=1,\ldots,p$:
\begin{enumerate}
    \item We need to verify $\int_{\mathbb{R}^p}|\theta_j\pi(\theta)|d\theta < \infty$. This follows from Assumption \textbf{(A1')}. Indeed,
    \[\int_{\mathbb{R}^p}|\theta_j\pi(\theta)|d\theta = \int_{\mathbb{R}^p}|\theta_j|\pi(\theta)d\theta \leq \int_{\mathbb{R}^p}\|\theta\|_1\pi(\theta)d\theta < \infty.\] 
    \item We need to verify that there exists a $\delta > 0$ such that $\theta_j\pi(\theta)$ is finite on $B_{\theta^*}(\delta)$. Indeed,
    \[|\theta_j\pi(\theta)|\leq |\theta_j|\pi(\theta) \leq \|\theta\|_1\pi(\theta) \leq \sqrt{p}\|\theta\|_2\pi(\theta) \leq \sqrt{p}(\|\theta-\theta^*\|_2 + \|\theta^*\|_2)\pi(\theta). 
    \] 
    By Assumption \textbf{(A1')}, there exists a $\delta > 0$ such that $\pi(\theta)$ is finite on $B_{\theta^*}(\delta)$; hence, the right side of the above is finite. Indeed,  
    $\|\theta-\theta^*\|_2 < \delta$ on $B_{\theta^*}(\delta)$ and $\|\theta^*\|_2$ is finite for all $\theta^*\in\mathbb{R}^p$.
    \item We verify condition \eqref{eq:b_remainder_control} by appealing to Proposition \ref{prop:prior_suff} with $q(\theta)=\theta_j$. This choice of $q$ is twice differentiable everywhere; hence, Assumption \textbf{(ALap)} holds with the $\delta$ from the above.
\end{enumerate}

By Lemma \ref{lem:lap_int_lem}, the result in \eqref{eq:cond_exp_ratio} holds with $\pi_{n,\hat{\alpha}_n}(\theta|X^n)$ replacing $\pi_{n,\alpha_n}(\theta|X^n)$. Hence, the result in \eqref{eq:diff} also holds with $\pi_{n,\hat{\alpha}_n}(\theta|X^n)$ replacing $\pi_{n,\alpha_n}(\theta|X^n)$, using a similar argument to that just given.
\end{proof}

\subsection{Proof of Laplace approximation, Lemma \ref{lem:lap_int_lem}}\label{app:laplace_lemma_proof}
\begin{proof}
    We will first prove the approximation in \eqref{eq:cond_exp_ratio}, assuming that \eqref{eq:lap_lem_result} holds, and then prove \eqref{eq:lap_lem_result}. We first establish results with $\pi_{n,\alpha_n}$ and then verify them for $\pi_{n,\hat\alpha_n}$. Notice that the expectation in \eqref{eq:cond_exp_ratio} is a ratio of integrals of the form \eqref{eq:lap_lem_result}, since
    \begin{equation}\label{eq:cond_exp_ratio_laplace}
        \int_{\mathbb{R}^p} q(\theta)\pi_{n,\alpha_n}(\theta|X^n)d\theta
        =\frac{\int_{\mathbb{R}^p} q(\theta)\pi(\theta)e^{\alpha_n \log f_n(X^n|\theta)}d\theta}{\int_{\mathbb{R}^p} \pi(\theta)e^{\alpha_n \log f_n(X^n|\theta)}d\theta}.
    \end{equation}
   First, we apply the approximation \eqref{eq:lap_lem_result} to the numerator of \eqref{eq:cond_exp_ratio_laplace} taking $b(\theta) = q(\theta)\pi(\theta)$, which we have assumed meets the conditions of Assumption \textbf{(ALap)}. Doing so yields
    \begin{equation}\label{eq:numerator_approx}
        \begin{split}
            &\int_{\mathbb{R}^p} q(\theta)\pi(\theta)e^{\alpha_n \log f_n(X^n|\theta)}d\theta = e^{\alpha_n\log f_n(X^n|\hat{\theta})}|H_n|^{-\frac{1}{2}}\Big(\frac{2\pi}{\alpha_n n}\Big)^{\frac{p}{2}}\Big[ q(\hat{\theta})\pi(\hat{\theta}) + O_{f_{0,n}}\Big(\frac{1}{\alpha_n n}\Big)\Big].\
        \end{split}
    \end{equation}
    Next, we can apply the approximation \eqref{eq:lap_lem_result} to the denominator of \eqref{eq:cond_exp_ratio_laplace} taking $b(\theta) = \pi(\theta)$ to yield
    \begin{equation}\label{eq:denominator_approx}
        \begin{split}
           & \int_{\mathbb{R}^p} \pi(\theta)e^{\alpha_n \log f_n(X^n|\theta)}d\theta = e^{\alpha_n\log f_n(X^n|\hat{\theta})}|H_n|^{-\frac{1}{2}}\Big(\frac{2\pi}{\alpha_n n}\Big)^{\frac{p}{2}}\Big[ \pi(\hat{\theta}) + O_{f_{0,n}}\Big(\frac{1}{\alpha_n n}\Big)\Big].\
        \end{split}
    \end{equation}
    We note that $b(\theta) = \pi(\theta)$ satisfies the three conditions of Assumption \textbf{(ALap)}, which allows us to apply the approximation in \eqref{eq:denominator_approx}:
    \begin{enumerate}
        \item We need to verify that $\int_{\mathbb{R}^p}|\pi(\theta)|d\theta < \infty$. Indeed, $\pi$ is a density; hence, is nonnegative everywhere and integrates to one.
        \item By Assumption \textbf{(A1')} there exists $r > 0$ such that $\pi$ is finite on $B_{\theta^*}(r)$.

        \item Note that in $\int_{\mathbb{R}^p} \pi(\theta)e^{\alpha_n \log f_n(X^n|\theta)}d\theta$, we have that  $b(\theta) = 1\times\pi(\theta)$. We apply Proposition \ref{prop:prior_suff}, setting $q(\theta)=1$, which is twice continuously differentiable everywhere. Hence, Condition 3 of Assumption \textbf{(ALap)} holds. 
    \end{enumerate}
    Having verified (or assumed) Assumption \textbf{(ALap)} for both the numerator or denominator of \eqref{eq:cond_exp_ratio_laplace}, we apply \eqref{eq:numerator_approx} and \eqref{eq:denominator_approx} to \eqref{eq:cond_exp_ratio_laplace} to obtain, 
    \begin{align}
          \int_{\mathbb{R}^p} q(\theta)&\pi_{n,\alpha_n}(\theta|X^n)d\theta 
            \nonumber=\frac{\int_{\mathbb{R}^p} q(\theta)\pi(\theta)e^{\alpha_n \log f_n(X^n|\theta)}d\theta}{\int_{\mathbb{R}^p} \pi(\theta)e^{\alpha_n \log f_n(X^n|\theta)}d\theta} \\
    &=\frac{e^{\alpha_n\log f_n(X^n|\hat{\theta})}|H_n|^{-\frac{1}{2}}(\frac{2\pi}{\alpha_n n})^{\frac{p}{2}}[ q(\hat{\theta})\pi(\hat{\theta}) + O_{f_{0,n}}(\frac{1}{\alpha_n n})]}{e^{\alpha_n\log f_n(X^n|\hat{\theta})}|H_n|^{-\frac{1}{2}}(\frac{2\pi}{\alpha_n n})^{\frac{p}{2}}[\pi(\hat{\theta}) + O_{f_{0,n}}(\frac{1}{\alpha_n n})]} 
    \label{eq:cond_exp_ratio_sub} =\frac{q(\hat{\theta})\pi(\hat{\theta}) + O_{f_{0,n}}(\frac{1}{\alpha_n n})}{\pi(\hat{\theta}) + O_{f_{0,n}}(\frac{1}{\alpha_n n})}.
    \end{align}
    We argue that \eqref{eq:cond_exp_ratio_sub} is well defined (i.e., $0 < \pi(\hat{\theta}) < \infty$ so that we are not dividing by $0$) with high probability as follows. First, $\pi(\theta)$ is continuous and positive in a neighborhood of $0$ by Assumption \textbf{(A1')} and by Assumption \textbf{(A0)}, for all $\epsilon > 0$, we may choose some $N'' \equiv N''(\epsilon)$ large enough such that $\hat{\theta}$ belongs to the neighborhood specified by \textbf{(A1')} with probability at least $1-\epsilon$ whenever $\alpha_nn > N''$.
 Next, notice that
 \[\frac{O_{f_{0,n}}(\frac{1}{\alpha_n n})}{\pi(\hat{\theta}) + O_{f_{0,n}}(\frac{1}{\alpha_n n})}  = O_{f_{0,n}}\Big(\frac{1}{\alpha_n n}\Big), \]
and by a Taylor expansion of the function $f(z) = \frac{a}{a+z}$ around $z=0$,
\[\frac{q(\hat{\theta})\pi(\hat{\theta})}{\pi(\hat{\theta}) + O_{f_{0,n}}(\frac{1}{\alpha_n n})}  = q(\hat{\theta}) - \frac{q(\hat{\theta})O_{f_{0,n}}(\frac{1}{\alpha_n n})}{\pi(\hat\theta)} = q(\hat{\theta}) + O_{f_{0,n}}\Big(\frac{1}{\alpha_n n}\Big).\]
Therefore, we have what we wanted to show: using \eqref{eq:cond_exp_ratio_sub},
    \begin{align}
          \int_{\mathbb{R}^p} q(\theta)\pi_{n,\alpha_n}(\theta|X^n)d\theta =\frac{q(\hat{\theta})\pi(\hat{\theta}) + O_{f_{0,n}}(\frac{1}{\alpha_n n})}{\pi(\hat{\theta}) + O_{f_{0,n}}(\frac{1}{\alpha_n n})} \nonumber &=\frac{q(\hat{\theta})\pi(\hat{\theta})}{\pi(\hat{\theta}) + O_{f_{0,n}}(\frac{1}{\alpha_n n})} 
            + \frac{O_{f_{0,n}}(\frac{1}{\alpha_n n})}{\pi(\hat{\theta}) + O_{f_{0,n}}(\frac{1}{\alpha_n n})} \\
            \nonumber &= q(\hat{\theta}) + O_{f_{0,n}}\Big(\frac{1}{\alpha_n n}\Big).
    \end{align}

    We will now argue that the result in \eqref{eq:cond_exp_ratio} holds with $\pi_{n,\hat{\alpha}_n}(\theta|X^n)$ replacing $\pi_{n,\alpha_n}(\theta|X^n)$. First, we rewrite the expectation with $\pi_{n,\hat{\alpha}_n}$ instead of $\pi_{n,\alpha_n}$:
    \begin{align}
        \int q(\theta)\pi_{n,\hat{\alpha}_n}(\theta|X^n)d\theta = \frac{\int q(\theta)\pi(\theta)f_n(X^n|\theta)^{\hat{\alpha}_n}d\theta}{\int \pi(\theta)f_n(X^n|\theta)^{\hat{\alpha}_n}d\theta} 
        \nonumber&= \frac{\int q(\theta)\pi(\theta)\big[f_n(X^n|\theta)^{\frac{\hat{\alpha}_n}{\alpha_n}}\big]^{\alpha_n}d\theta}{\int \pi(\theta)\big[f_n(X^n|\theta)^{\frac{\hat{\alpha}_n}{\alpha_n}}\big]^{\alpha_n} d\theta} \\
        \label{eq:cond_exp_ratio_tilde}&= \frac{\int q(\theta)\pi(\theta)\exp(\alpha_n\log\tilde{f}_n(X^n|\theta))d\theta}{\int \pi(\theta)\exp(\alpha_n\log\tilde{f}_n(X^n|\theta))d\theta},
    \end{align}
    where $\tilde{f}_n(X^n|\theta) = f_n(X^n|\theta)^{\frac{\hat{\alpha}_n}{\alpha_n}}$. Notice that the above is a ratio of integrals of the form of \eqref{eq:lap_lem_result} with $f_n(X^n|\theta)$ replaced by $\tilde{f}_n(X^n|\theta)$. Hence, from arguments similar to those around \eqref{eq:cond_exp_ratio_sub} and the equality in \eqref{eq:cond_exp_ratio_tilde}, the desired result is true if \eqref{eq:lap_lem_result} holds with $f_n(X^n|\theta)$ replaced by $\tilde{f}_n(X^n|\theta)$. We prove this in what follows. 
    
    In what follows, we establish the approximation in \eqref{eq:lap_lem_result}.
    Our proof consists of two parts. First, we identify the leading order term in the approximation
  and then analyze the remainder terms, showing they are order $\frac{1}{\alpha_nn}$ with high probability for $\alpha_nn$ sufficiently large. To do so, we use an approach similar to that in the Theorem \ref{thm:moments_alt} proof, where we analyzed remainder terms on a (growing) neighborhood of $\hat{\theta}$ and its complement. However, unlike the arguments in the Theorem \ref{thm:moments_alt} proof, the growing neighborhood in this proof is centered at $\hat{\theta}$ instead of $\theta^*$. We change the centering to use the representation of the log-likelihood in Assumption \textbf{(A2')} and obtain the desired order of the approximation error.

    \textbf{Step 1: Identifying the leading term.}
    Define the multiplicative constant in the approximation \eqref{eq:lap_lem_result} as
    \begin{equation}\label{eq:Ldef}
        L \equiv e^{\alpha_n\log f_n(X^n|\hat{\theta})}|H_n|^{-\frac{1}{2}}\Big(\frac{2\pi}{\alpha_n n}\Big)^{\frac{p}{2}}.
    \end{equation}
    Using \eqref{eq:Ldef}, we rewrite the integral in \eqref{eq:lap_lem_result} as
    \begin{equation}
    \begin{split}
        I&=\int_{\mathbb{R}^p} b(\theta)e^{\alpha_n \log f_n(X^n|\theta)} d\theta \\
        &= L\int_{\mathbb{R}^p} |H_n|^{\frac{1}{2}}\Big(\frac{2\pi}{\alpha_n n}\Big)^{-\frac{p}{2}} b(\theta)e^{\alpha_n [\log f_n(X^n|\theta) - \log f_n(X^n|\hat{\theta})]} d\theta \\
        &= L\int_{\mathbb{R}^p}|H_n|^{\frac{1}{2}}\Big(\frac{2\pi}{\alpha_n n}\Big)^{-\frac{p}{2}}  b(\hat{\theta}) e^{-\frac{1}{2}(\theta-\hat{\theta})^\top (\alpha_nnH_n) (\theta-\hat{\theta})}d\theta \\
        &\quad + L\int_{\mathbb{R}^p}|H_n|^{\frac{1}{2}}\Big(\frac{2\pi}{\alpha_n n}\Big)^{-\frac{p}{2}}  \Big[b(\theta)e^{\alpha_n [\log f_n(X^n|\theta) - \log f_n(X^n|\hat{\theta})]} - b(\hat{\theta}) e^{-\frac{1}{2}(\theta-\hat{\theta})^\top (\alpha_nnH_n) (\theta-\hat{\theta})} \Big] d\theta\\
        &\equiv L(T_1+T_2).
    \label{eq:laplace_int_split}
    \end{split}
    \end{equation}
   The leading term of the approximation in \eqref{eq:lap_lem_result} will come from $T_1$ because
    \begin{equation*}
        \begin{split}
            T_1
            &=\int_{\mathbb{R}^p}|H_n|^{\frac{1}{2}}\Big(\frac{2\pi}{\alpha_n n}\Big)^{-\frac{p}{2}}b(\hat{\theta}) e^{-\frac{1}{2}(\theta-\hat{\theta})^\top (\alpha_nnH_n) (\theta-\hat{\theta})}d\theta \\
            &= b(\hat{\theta}) \int_{\mathbb{R}^p}(2\pi)^{-\frac{p}{2}}|\alpha_nn H_n|^{\frac{1}{2}}e^{-\frac{1}{2}(\theta-\hat{\theta})^\top (\alpha_nnH_n) (\theta-\hat{\theta})}d\theta 
            =b(\hat{\theta}).
        \end{split}
    \end{equation*}
The final equality above follows as the multivariate Gaussian density integrates to one.

\textbf{Step 2: Analyzing the remainder term, $\mathbf{T_2}$.} We will show that the remainder term, $T_2$ in \eqref{eq:laplace_int_split}, is comprised of terms of order $\frac{1}{\alpha_nn}$ or lower with high probability for $\alpha_nn$ sufficiently large. To do so, we analyze the integral $T_2$ on a set, $K$, and its complement, $K^\mathsf{c} = \mathbb{R}^p\setminus K$. We will detail our choice of $K$ below. We begin by noting that
\begin{align}
    \nonumber T_2&=\int_{K\cup K^\mathsf{c}} |H_n|^{\frac{1}{2}}\Big(\frac{2\pi}{\alpha_n n}\Big)^{-\frac{p}{2}}  \Big[b(\theta)e^{\alpha_n [\log f_n(X^n|\theta) - \log f_n(X^n|\hat{\theta})]} - b(\hat{\theta}) e^{-\frac{1}{2}(\theta-\hat{\theta})^\top (\alpha_nnH_n) (\theta-\hat{\theta})} \Big] d\theta \\
    \nonumber&= \int_{K^\mathsf{c}} |H_n|^{\frac{1}{2}}\Big(\frac{2\pi}{\alpha_n n}\Big)^{-\frac{p}{2}} b(\theta)e^{\alpha_n [\log f_n(X^n|\theta) - \log f_n(X^n|\hat{\theta})]}d\theta \\
    \nonumber&\hspace{20pt} -\int_{K^\mathsf{c}} |H_n|^{\frac{1}{2}}\Big(\frac{2\pi}{\alpha_n n}\Big)^{-\frac{p}{2}} b(\hat{\theta}) e^{-\frac{1}{2}(\theta-\hat{\theta})^\top (\alpha_nnH_n) (\theta-\hat{\theta})}   d\theta\\
    &\hspace{20pt}+\int_{K} |H_n|^{\frac{1}{2}}\Big(\frac{2\pi}{\alpha_n n}\Big)^{-\frac{p}{2}} \Big[b(\theta)e^{\alpha_n [\log f_n(X^n|\theta) - \log f_n(X^n|\hat{\theta})]} -  b(\hat{\theta}) e^{-\frac{1}{2}(\theta-\hat{\theta})^\top (\alpha_nnH_n) (\theta-\hat{\theta})}  \Big] d\theta. \label{eq:T2}
\end{align}
We label the three terms on the right side of the above as $T_{21}$, $T_{22}$, and $T_{23}$.
The aim for the rest of the proof is to show that each of these terms is $O_{f_{0,n}}(\frac{1}{\alpha_n n})$.
Before we do this, we present and discuss our choice of $K$.

\noindent\textbf{Choosing $\mathbf{K}$:} 
Let $\delta_1$ and $\delta_2$ denote the values of $\delta$ such that assumptions \textbf{(ALap)} point (3) and \textbf{(A2')} hold, respectively. Define $\delta\equiv\min(\delta_1, \delta_2)$. We will show that the choice $K = B_{\hat{\theta}}(\delta)$ in \eqref{eq:T2} leads to the desired approximation \eqref{eq:lap_lem_result}.

To do so, we define the following events, where the constant $c > 0$ is to be specified later:
\begin{equation}\label{eq:eventsABC}
    \begin{split}
        \mathcal{A}_K &= \Big\{\sup_{\theta\in K^\mathsf{c}}\Big[\frac{1}{n}\log f_n(X^n|\theta) - \frac{1}{n}\log f_n(X^n|\hat{\theta})\Big] < -c\Big\}, \\
        \mathcal{B} &= \left\{H_n~\text{is symmetric and}~\lambda_{\text{min}}(H_n) < -c\right\}.\
    \end{split}
\end{equation}
In this section, we will show that these events occur with high probability. In more detail, we will show that for each event, the probability of the complement of the event is bounded by $\epsilon/4$ for arbitrary $\epsilon > 0$. First, $K^\mathsf{c} = B_{\hat{\theta}}(\delta)^\mathsf{c}$, so
\begin{align}
    \mathbb{P}_{f_{0,n}}(\mathcal{A}_{K}^{\mathsf{c}})
    \label{eq:BKc_ub}&= \mathbb{P}_{f_{0,n}}\Big(\sup_{\theta\in B_{\hat{\theta}}(\delta)^\mathsf{c}}\Big[\frac{1}{n}\log f_n(X^n|\theta) - \frac{1}{n}\log f_n(X^n|\hat{\theta})\Big] > -c\Big).
\end{align}
By Assumption \textbf{(A3')}, there exists a constant $c \equiv c(\epsilon,\delta)$ and $N_{0} \equiv N_{0}(\epsilon, \delta)$ such that \eqref{eq:BKc_ub} is upper bounded by $\epsilon/4$ whenever $\alpha_nn > N_{0}$. We use this $c$ in the definition of $\mathcal{A}_K$ in \eqref{eq:eventsABC}.

Next, since $\mathcal{B}$ corresponds to $H_n$ being symmetric and positive definite, by Assumption \textbf{(A2')}, there exists $N_1\equiv N_1(\epsilon)$ such that $\mathbb{P}_{f_{0,n}}(\mathcal{B}^{\mathsf{c}}) <\frac{\epsilon}{4}$ whenever $\alpha_nn > N_{1}$.
We conclude that the events in \eqref{eq:eventsABC} hold with high probability. Furthermore, by a union bound, for any $\epsilon > 0$ and $\alpha_nn > \max(N_0,N_1)$, we have
$\mathbb{P}_{f_{0,n}}\left(\left(\mathcal{A}_{K}\cap\mathcal{B}\right)^{\mathsf{c}}\right) \leq \mathbb{P}_{f_{0,n}}(\mathcal{A}_{K}^{\mathsf{c}}) + \mathbb{P}_{f_{0,n}}(\mathcal{B}^{\mathsf{c}}) \leq \frac{\epsilon}{2}.$
Using the above, it will be useful to note that given any event, $\mathcal{C}$, by the law of total probability,
    \begin{align}
       \mathbb{P}_{f_{0,n}}(\mathcal{C}) 
            \leq \mathbb{P}_{f_{0,n}}(\mathcal{C}\,|\,\mathcal{A}_{K}\cap\mathcal{B}) + \mathbb{P}_{f_{0,n}}((\mathcal{A}_{K}\cap\mathcal{B})^{\mathsf{c}}) 
            \label{eq:prob_arg}&\leq \mathbb{P}_{f_{0,n}}(\mathcal{C}\,|\,\mathcal{A}_{K}\cap\mathcal{B}) + \frac{\epsilon}{2}.
     \end{align}
    
    We are now ready to turn our attention to the three terms in \eqref{eq:T2}.

    \noindent\textbf{Analysis of Term $\mathbf{T_{21}}$ of \eqref{eq:T2}:} Note that for $p \geq 1$,
    \begin{equation*}
        \begin{split}
            |T_{21}| 
            &= \Big|\int_{K^\mathsf{c}} |H_n|^{\frac{1}{2}}\Big(\frac{2\pi}{\alpha_n n}\Big)^{-\frac{p}{2}}b(\theta)e^{\alpha_n [\log f_n(X^n|\theta) - \log f_n(X^n|\hat{\theta})]}d\theta\Big| \\
            &\leq |H_n|^{\frac{1}{2}}(\alpha_nn)^{\frac{p}{2}} \exp\Big(\alpha_n n \sup_{\theta\in K^\mathsf{c}}\Big[\frac{1}{n}\log f_n(X^n|\theta) - \frac{1}{n}\log f_n(X^n|\hat{\theta})\Big]\Big)\int_{\mathbb{R}^p} |b(\theta)|d\theta \\
            &\leq |H_n|^{\frac{1}{2}}(\alpha_nn)^{\frac{p}{2}} \exp(-c\alpha_n n)\int_{\mathbb{R}^p} |b(\theta)|d\theta,
        \end{split}
    \end{equation*}
    where the final inequality follows from conditioning on $\mathcal{A}_K$. For any given $M_0 > 0$, by \eqref{eq:prob_arg}, we have that for $\alpha_nn > \max(N_0,N_1)$,
    \begin{equation*}
        \begin{split}
            \mathbb{P}_{f_{0,n}}\Big(|T_{21}| > \frac{M_0}{\alpha_nn}\Big) 
            &\leq  \mathbb{P}_{f_{0,n}}\Big(|T_{21}| > \frac{M_0}{\alpha_nn} \, \Big| \, \mathcal{A}_{K}\cap\mathcal{B}\Big) +\frac{\epsilon}{2}.
        \end{split}
    \end{equation*}
    Furthermore, there exists $N_2 \equiv N_2(M_0, \delta)$ such that $\mathbb{P}_{f_{0,n}}(|T_{21}| > \frac{M_0}{\alpha_nn} \, | \, \mathcal{A}_{K}\cap\mathcal{B}) = 0$ whenever $\alpha_nn > N_2$, since $T_{21}$ decreases exponentially fast (in $\alpha_n n$) to zero on the event $\mathcal{A}_{K}$ by the assumed integrability of $b(\theta)$.
    For any $\epsilon >0$ and $\alpha_nn > \max(N_0, N_1, N_2)$, we conclude that
    $\mathbb{P}_{f_{0,n}}(|T_{21}| > \frac{M_0}{\alpha_nn}) < \frac{\epsilon}{2}$ from which it follows $T_{21} = O_{f_{0,n}}(\frac{1}{\alpha_n n}).$

   \noindent\textbf{Analysis of Term $\mathbf{T_{22}}$  of \eqref{eq:T2}:} Consider the  the change of variable $h = \sqrt{\alpha_nn}(\theta-\hat{\theta})$. We note that 
  $
        \theta \in K \equiv B_{\hat{\theta}}(\delta)$ if and only if $h\in\tilde{K} \equiv B_{0}(\delta \sqrt{\alpha_n n}).$
   In particular, it follows that $\tilde{K}^{\mathsf{c}} = B_{0}(\delta \sqrt{\alpha_n n})^{\mathsf{c}}$ and therefore, through the change of variable,
   \begin{align}
        |T_{22}| &= b(\hat{\theta})\int_{\tilde{K}^\mathsf{c}} |H_n|^{\frac{1}{2}}\Big(\frac{2\pi}{\alpha_n n}\Big)^{-\frac{p}{2}} (\alpha_nn)^{-\frac{p}{2}}e^{-\frac{1}{2}h^\top H_n h}dh = b(\hat{\theta})\int_{\tilde{K}^\mathsf{c}} |H_n|^{\frac{1}{2}}(2\pi)^{-\frac{p}{2}}e^{-\frac{1}{2}h^\top H_n h} dh.
        \label{eq:gauss_int}
    \end{align}
    We notice that, under event $\mathcal{B}$, we have that $H_n$ is symmetric and positive definite. Hence, the integral in \eqref{eq:gauss_int} says for $X \sim \mathcal{N}(0, H_n^{-1})$
    \begin{align}
        & |T_{22}| = b(\hat{\theta}) \mathbb{P}(X \in \tilde{K}^\mathsf{c}) = b(\hat{\theta}) \mathbb{P}(X \in B_{0}(\delta\sqrt{\alpha_n n})^{\mathsf{c}}) = b(\hat{\theta}) \mathbb{P}(\|X\| > \delta\sqrt{\alpha_n n}).
        \label{eq:T_22_triangle}
    \end{align}
    Applying Lemma \ref{lem:Gaussian_concentration} with $s = \delta\sqrt{\alpha_n n}$ to \eqref{eq:T_22_triangle} gives
    \begin{equation}\label{eq:T_22_conc}
        |T_{22}| \leq 2b(\hat{\theta})\exp\Big(-\frac{(\delta\sqrt{\alpha_nn} - \sqrt{\text{tr}(H_n^{-1})})^2}{2\|H_n^{-1}\|_{\text{op}}}\Big),
    \end{equation}
    where $\|H_n^{-1}\|_{\text{op}} \equiv \sup_{\|v\|_2=1}\|H_n^{-1} v\|_2$ denotes the operator norm of $H_n^{-1}$. 
   For any given $M_1 > 0$, by \eqref{eq:prob_arg} for $\alpha_nn > \max(N_0,N_1)$, we have that $\mathbb{P}_{f_{0,n}}(|T_{22}| > \frac{M_1}{\alpha_nn}) \leq  \mathbb{P}_{f_{0,n}}(|T_{22}| > \frac{M_1}{\alpha_nn} \, | \, \mathcal{A}_K\cap\mathcal{B}) +\frac{\epsilon}{2}.$ Furthermore, $\|H_n^{-1}\|_{\text{op}} = \lambda_{\text{min}}(H_n)^{-1}$, which is finite as $H_n$ is positive definite on the event $\mathcal{B}$. Hence, \eqref{eq:T_22_conc} decreases exponentially (in $\alpha_n n$) on the event $\mathcal{B}$, so there exists $N_3 \equiv N_3(M_1, \delta)$ such that $\mathbb{P}_{f_{0,n}}(|T_{22}| > \frac{M_1}{\alpha_nn} \, | \, \mathcal{A}_K\cap\mathcal{B}) = 0$ whenever $\alpha_nn > N_3$.
    We conclude that for $\alpha_nn > \max(N_0, N_1, N_3)$, we have 
    $\mathbb{P}_{f_{0,n}}(|T_{22}| > \frac{M_1}{\alpha_nn}) < \frac{\epsilon}{2}$ from which it follows $T_{22} = O_{f_{0,n}}(\frac{1}{\alpha_n n}).$

    \noindent \textbf{Analysis of Term $\mathbf{T_{23}}$ of \eqref{eq:T2}:} We will start by analyzing the integrand of $T_{23}$. Recall that
    \begin{equation*}
        \begin{split}
            T_{23} &= \int_{K} |H_n|^{\frac{1}{2}}\Big(\frac{2\pi}{\alpha_n n}\Big)^{-\frac{p}{2}}\Big[b(\theta)e^{\alpha_n [\log f_n(X^n|\theta) - \log f_n(X^n|\hat{\theta})]} - b(\hat{\theta})e^{-\frac{1}{2}(\theta-\hat{\theta})^\top (\alpha_nnH_n) (\theta-\hat{\theta})} \Big] d\theta.
        \end{split}
    \end{equation*}
    Let $\tilde{T}_{23}(\theta) \equiv b(\theta)e^{\alpha_n [\log f_n(X^n|\theta) - \log f_n(X^n|\hat{\theta})]}$ denote the first term of intergrand of $T_{23}$ in the brackets.
    We apply the the change of variable $h = \sqrt{\alpha_n n}(\theta-\hat{\theta})$ to $\tilde{T}_{23}(\theta)$ and use the representation of $\tilde{R}_n(h)$ from Assumption \textbf{(A2')} to obtain,
    \begin{equation}\label{eq:prime_T21_expand}
        \begin{split}
            \tilde{T}^h_{23}(h)    
&=b\Big(\hat{\theta}+\frac{h}{\sqrt{\alpha_n n}}\Big)e^{\alpha_n[\log f_n(X^n|\hat{\theta}+\frac{h}{\sqrt{\alpha_n n}})-\log f_n(X^n|\hat{\theta})]} \\
            &=b\Big(\hat{\theta}+\frac{h}{\sqrt{\alpha_n n}}\Big)e^{-\frac{1}{2}h^\top H_nh} e^{\frac{1}{6\sqrt{\alpha_n n}}\langle S_n, h^{\otimes 3}\rangle + \tilde{R}_n(h)},
        \end{split}
    \end{equation}
    where $\langle S_n, h^{\otimes 3}\rangle$ uses the tensor product notation defined in \eqref{eq:tensor_inner_product_def}. Furthermore, defining $A(h)\equiv-\frac{1}{6\sqrt{\alpha_n n}}\langle S_n, h^{\otimes 3}\rangle - \tilde{R}_n(h)$, we see that a Taylor expansion gives
    \begin{equation}\label{eq:A_expand}
        \exp(-A(h))= 1 - A(h) + \frac{1}{2}A(h)^2  -\frac{\exp(-\tilde{A})}{6}A(h)^3,
    \end{equation}
    where $\tilde{A}\in (0, A(h))$. We use the representation of $R_b(h)$ from Assumption \textbf{(ALap)} to obtain
    \begin{equation}
    \label{eq:b_expand}
        b\Big(\hat{\theta}+\frac{h}{\sqrt{\alpha_n n}}\Big) = b(\hat{\theta}) + \frac{v^\top h}{\sqrt{\alpha_n n}} + R_b(h).
    \end{equation}
    Combining \eqref{eq:prime_T21_expand}--\eqref{eq:b_expand} yields
    \begin{align}
\nonumber           
            \tilde{T}^h_{23}(h)
            &=e^{-\frac{1}{2}h^\top H_nh}  % 
            \Big[1 - A(h) + \frac{1}{2}A(h)^2 - \frac{\exp(-\tilde{A})}{6}A(h)^3\Big]\Big[b(\hat{\theta}) + \frac{1}{\sqrt{\alpha_n n}}v^\top h + R_b(h)\Big]\\
            &=e^{-\frac{1}{2}h^\top H_nh} B_n(h), % 
        \label{eq:lap_int_gb_expand}
    \end{align}
    where
    \begin{equation*}
        B_n(h)\equiv\Big[1 - A(h) + \frac{1}{2}A(h)^2 - \frac{\exp(-\tilde{A})}{6}A(h)^3 \Big]\Big[b(\hat{\theta}) + \frac{v^\top h}{\sqrt{\alpha_n n}} + R_b(h)\Big]. 
    \end{equation*}
   We rewrite $T_{23}$ using \eqref{eq:lap_int_gb_expand} after applying the change of variable $h=\sqrt{\alpha_n n}(\theta-\hat{\theta})$. Denote the $\mathcal{N}(0, H_n^{-1})$ density as $\rho_{H}(h) = |H_n|^{\frac{1}{2}}(2\pi)^{-\frac{p}{2}}\exp(-\frac{1}{2}h^\top H_nh)$.  Since $\mathcal{B}$ corresponds to $H_n$ being symmetric and positive definite, given $\mathcal{B}$, the density is well defined. This yields
    \begin{equation}
        \begin{split}
        \label{eq:T23}
            &T_{23}
            =\int_K|H_n|^{\frac{1}{2}}\Big(\frac{2\pi}{\alpha_n n}\Big)^{-\frac{p}{2}}\Big[ \tilde{T}_{23}(\theta) - b(\hat{\theta})e^{-\frac{1}{2}(\theta-\hat{\theta})^\top(\alpha_nnH_n)(\theta-\hat{\theta})}\Big]d\theta \\
            &=\int_{\tilde{K}}|H_n|^{\frac{1}{2}}\Big(\frac{2\pi}{\alpha_n n}\Big)^{-\frac{p}{2}}\Big[ \tilde{T}_{23}^h(h) - b(\hat{\theta})e^{-\frac{1}{2}h^\top H_nh}\Big](\alpha_nn)^{-\frac{p}{2}}dh = \int_{\tilde{K}}\rho_{H}(h) [B_n(h) - b(\hat{\theta})] dh. 
        \end{split}
    \end{equation}
Next, recall that $A(h)\equiv-\frac{1}{6\sqrt{\alpha_n n}}\langle S_n, h^{\otimes 3}\rangle - \tilde{R}_n(h)$. Letting 
    \begin{equation}\label{eq:R1n_def}
        \begin{split}
            R_{1,n}(h)
            &\equiv \tilde{R}_n(h) + \frac{1}{2}A(h)^2  -\frac{\exp(-\tilde{A})}{6}A(h)^3\\
            &=\tilde{R}_n(h) + \frac{1}{2}\Big(-\frac{\langle S_n,h^{\otimes 3}\rangle}{6\sqrt{\alpha_n n}} - \tilde{R}_n(h)\Big)^2 -\frac{\exp(-\tilde{A})}{6}\Big(-\frac{\langle S_n,h^{\otimes 3}\rangle}{6 \sqrt{\alpha_n n}} - \tilde{R}_n(h)\Big)^3,
        \end{split}
    \end{equation}
    we see that
    \begin{equation}\label{eq:B_mult}
        \begin{split}
            B_n(h) - b(\hat{\theta})
            &=\Big[1 + \frac{\langle S_n,h^{\otimes 3}\rangle}{6\sqrt{\alpha_n n}} + R_{1,n}(h)\Big]\Big[b(\hat{\theta}) + \frac{v^\top h}{\sqrt{\alpha_n n}} + R_b(h)\Big] - b(\hat{\theta}) \\
            &=  \frac{\langle S_n,h^{\otimes 3}\rangle\cdot b(\hat{\theta})}{6\sqrt{\alpha_n n}} + \frac{v^\top h}{\sqrt{\alpha_n n}} + R_{n,n}(h),
        \end{split}
    \end{equation}
    where $R_{n,n}(h)$ includes terms depending on $R_b(h)$, $R_{1,n}(h)$, and $h^{\otimes s}$ with $s\geq 4$; namely,
    \begin{align*}
          \nonumber  R_{n,n}(h)
            & =R_{1,n}(h)\Big[b(\hat{\theta}) + \frac{v^\top h}{\sqrt{\alpha_n n}}\Big] 
            + R_{1,n}(h)R_b(h)
            + R_b(h)\Big[1 + \frac{\langle S_n,h^{\otimes 3}\rangle}{6\sqrt{\alpha_n n}} \Big] + \frac{\langle S_n,h^{\otimes 3}\rangle}{6\alpha_n n} \cdot v^\top h.
    \end{align*}
    We rewrite $T_{23}$ using \eqref{eq:T23} and \eqref{eq:B_mult}:
    \begin{equation*}
        \begin{split}
            T_{23}
           & = \int_{\tilde{K}}\rho_{H}(h) [B_n(h) - b(\hat{\theta})]dh \\
            &= \int_{\tilde{K}} \rho_{H}(h)R_{n,n}(h)dh + \int_{\tilde{K}} \rho_{H}(h)\Big[\frac{v^\top h}{\sqrt{\alpha_n n}} + \frac{\langle S_n,h^{\otimes 3}\rangle}{6\sqrt{\alpha_n n}}\cdot b(\hat{\theta}) \Big] dh \equiv \tilde{V}_1+\tilde{V}_2.
        \end{split}
    \end{equation*}
    We show that $\tilde{V}_1$ and $\tilde{V}_2$ are both of order $\frac{1}{\alpha_nn}$ or lower, which implies $T_{23} = O_{f_{0,n}}(\frac{1}{\alpha_nn})$.

  \textbf{Analysis of $\mathbf{\tilde{V}_1}$:} First, by the triangle inequality, we have
    \begin{equation}\label{eq:Rnn_triangle}
        \begin{split}
            |R_{n,n}(h)|  &\leq 
            \Big|R_{1,n}(h)\Big[b(\hat{\theta}) + \frac{v^\top h}{\sqrt{\alpha_n n}}\Big]  \Big|
            + |R_{1,n}(h)R_b(h)| \\
            &\qquad 
            + \Big|R_b(h) \Big[1 + \frac{\langle S_n,h^{\otimes 3}\rangle}{6\sqrt{\alpha_n n}} \Big]\Big| + \Big|\frac{\langle S_n,h^{\otimes 3}\rangle v^\top h }{6\alpha_n n} \Big| .
        \end{split}            
    \end{equation}
    We use the inequality in \eqref{eq:Rnn_triangle} to upper bound $|\tilde{V}_1|$.  Then, we apply the Cauchy-Schwarz inequality to each of the resulting terms. 
    \begin{equation}\label{eq:V1_tilde}
        \begin{split}
            |\tilde{V}_1|
            \leq \int_{\tilde{K}} \rho_{H}(h) |R_{n,n}(h)|dh &\leq \Big[\int_{\tilde{K}} \rho_{H}(h) R_{1,n}(h)^2  dh \Big]^{\frac{1}{2}} \Big[\int_{\tilde{K}} \rho_{H}(h) \Big|b(\hat{\theta}) + \frac{v^\top h}{\sqrt{\alpha_n n}}\Big|^2  dh \Big]^{\frac{1}{2}} \\
            &\quad + \Big[\int_{\tilde{K}}  \rho_{H}(h) R_{1,n}(h)^2 dh\Big]^{\frac{1}{2}} \Big[\int_{\tilde{K}}  \rho_{H}(h) R_b(h)^2 dh\Big]^{\frac{1}{2}} \\
            &\quad+ \Big[\int_{\tilde{K}} \rho_{H}(h)  R_b(h)^2 dh\Big]^{\frac{1}{2}} \Big[\int_{\tilde{K}} \rho_{H}(h)  \Big|1 + \frac{\langle S_n,h^{\otimes 3}\rangle }{6\sqrt{\alpha_n n}}\Big|^2 dh\Big]^{\frac{1}{2}} \\
            &\quad + \frac{1}{\alpha_nn}\int_{\tilde{K}} \rho_{H}(h) \Big|\frac{1}{6}\langle S_n,h^{\otimes 3}\rangle \cdot v^\top h\Big|dh.
        \end{split}
    \end{equation}
    The first three terms of \eqref{eq:V1_tilde} are multiples of 
    \begin{equation}\label{eq:Rb_V1}
        \Big(\int_{\tilde{K}}\rho_{H}(h) R_b(h)^2 dh\Big)^{\frac{1}{2}}, \quad \text{ and } \quad 
        \Big(\int_{\tilde{K}}\rho_{H}(h)  R_{1,n}(h)^2 dh\Big)^{\frac{1}{2}}.
    \end{equation}
    Given $\mathcal{B}$ and noting that $\tilde{K}\subseteq\mathbb{R}^p$, the other objects besides those in \eqref{eq:Rb_V1} in the first three terms of \eqref{eq:V1_tilde} are upper bounded by expectations of polynomials under a mean zero Gaussian with covariance $H_{n}^{-1}$. These terms are of constant order or lower. Furthermore, the last term of \eqref{eq:V1_tilde} is $O_{f_{0,n}}(\frac{1}{\alpha_nn})$. Hence, we will conclude that $\tilde{V}_1 = O_{f_{0,n}}(\frac{1}{\alpha_nn})$ by showing that the terms in \eqref{eq:Rb_V1} are $O_{f_{0,n}}(\frac{1}{\alpha_nn})$.

    We analyze the first term in \eqref{eq:Rb_V1}. Let $M_2 > 0$ be the constant in Assumption \textbf{(ALap)}  point (3) and define an event
    $\mathcal{E}_0 = \{R_b(h) \leq \frac{M_2}{\alpha_nn}\|h\|_2^2\text{ for all }  h\in\tilde{K}\}.$ Recall the definition of the set $\mathcal{H}_{n,b}\equiv\{h\in\mathbb{R}^p|\frac{\|h\|_2}{\sqrt{\alpha_nn}} \leq \delta_1\}$ in Assumption \textbf{(ALap)}. Since $\tilde{K} = B_{0}(\delta \sqrt{\alpha_nn})$ and $\delta \leq \delta_1$ by definition (see just above \eqref{eq:eventsABC}), we have that $\tilde{K} \subseteq \mathcal{H}_{n,b}$.
    
    By this and Assumption \textbf{(ALap)}, there exists $N_4\equiv N_4(\epsilon,\delta_1)$ such that for $\alpha_nn > N_4$,
    \begin{equation}
    \begin{split}
       & \mathbb{P}_{f_{0,n}}(\mathcal{E}_0^\mathsf{c}) 
      = \mathbb{P}_{f_{0,n}}\Big(R_b(h) > \frac{M_2}{\alpha_nn}\|h\|_2^2  \text{ for some } h\in\tilde{K}\Big) \\
          &\leq \mathbb{P}_{f_{0,n}}\Big(R_b(h) > \frac{M_2}{\alpha_nn}\|h\|_2^2 \text{ for some } h\in\mathcal{H}_{n,b}\Big) = \mathbb{P}_{f_{0,n}}\Big(\sup_{h\in\mathcal{H}_{n,b}} R_b(h) - \frac{M_2}{\alpha_nn}\|h\|_2^2>0\Big) 
        \label{eq:Rb_complement} < \frac{\epsilon}{2}.
    \end{split}
    \end{equation}
       Now, given a constant $\tilde{M}_2 > 0$ to be specified later, for $\alpha_nn > \max(N_0, N_1, N_4)$, by \eqref{eq:prob_arg}
    \begin{equation}
         \begin{split}
            \mathbb{P}_{f_{0,n}}
            &\Big(\Big(\int_{\tilde{K}} \rho_{H}(h)  R_b(h)^2 dh\Big)^{\frac{1}{2}} > \frac{\tilde{M}_2}{\alpha_nn}\Big) \\
            &\leq \mathbb{P}_{f_{0,n}}\Big(\Big(\int_{\tilde{K}} \rho_{H}(h) R_b(h)^2 dh\Big)^{\frac{1}{2}} > \frac{\tilde{M}_2}{\alpha_nn} \, \big| \, \mathcal{A}_{K}\cap\mathcal{B}\Big) + \frac{\epsilon}{2} \\
            &\leq \mathbb{P}_{f_{0,n}}\Big(\frac{M_2}{\alpha_nn}\Big(\int_{\tilde{K}}\rho_{H}(h)\|h\|_2^4 dh\Big)^{\frac{1}{2}} > \frac{\tilde{M}_2}{\alpha_nn} \, \big| \, \mathcal{A}_{K}\cap\mathcal{B}\cap\mathcal{E}_0\Big)+ \mathbb{P}_{f_{0,n}}(\mathcal{E}_0^\mathsf{c}) +\frac{\epsilon}{2} \\
            & \leq \mathbb{P}_{f_{0,n}}\Big(\frac{M_2}{\alpha_nn}\Big(\int_{\mathbb{R}^p}\rho_{H}(h)\|h\|_2^4 dh\Big)^{\frac{1}{2}} > \frac{\tilde{M}_2}{\alpha_nn}\, \big| \, \mathcal{A}_{K}\cap\mathcal{B}\cap\mathcal{E}_0\Big) + \mathbb{P}_{f_{0,n}}(\mathcal{E}_0^\mathsf{c}) + \frac{\epsilon}{2} < \epsilon.
         \label{eq:Rb_arg_eventE}
    \end{split}
    \end{equation}
   The second inequality in \eqref{eq:Rb_arg_eventE} follows from conditioning on $\mathcal{E}_0$ and the final inequality follows from \eqref{eq:Rb_complement} and choosing $\tilde{M}_2 = 2M_2(\int_{\mathbb{R}^p}\rho_{H}(h) \|h\|_2^4 dh)^{\frac{1}{2}}$ so that
    \begin{equation*}
        \mathbb{P}_{f_{0,n}}\Big(\frac{M_2}{\alpha_nn}\Big(\int_{\mathbb{R}^p}\rho_{H}(h) \|h\|_2^4 dh\Big)^{\frac{1}{2}} > \frac{\tilde{M}_2}{\alpha_nn}\, \big|\, \mathcal{A}_{K}\cap\mathcal{B}\cap\mathcal{E}_0\Big) = 0.\
    \end{equation*}
    We note that $\tilde{M}_2$ exists (i.e., is finite) since, conditional on $\mathcal{B}$, the integral term is the expectation of a polynomial with respect to a  Gaussian. We have therefore shown, for $\alpha_nn$ sufficiently large, that $(\int_{\tilde{K}}\rho_{H}(h) R_b(h)^2dh)^{\frac{1}{2}}$ is $O_{f_{0,n}}(\frac{1}{\alpha_nn})$.

    We use a similar argument for the second term in \eqref{eq:Rb_V1}, namely $(\int_{\tilde{K}}\rho_{H}(h) R_{1,n}(h)^2 dh)^{\frac{1}{2}}$. Inspecting the definition of $R_{1,n}(h)$ in \eqref{eq:R1n_def}, we see that its highest order terms in $\alpha_n n$ are $(\frac{\langle S_n,h^{\otimes 3}\rangle}{6\sqrt{\alpha_n n}})^2$ (which is of order $\frac{1}{\alpha_nn}$ when integrated with respect to $\rho_H(h)$) and a polynomial of $\tilde{R}_n(h)$. The highest order term that depends on $\tilde{R}_n(h)$ is $(\int_{\tilde{K}}\rho_{H}(h) \tilde{R}_n(h)^2 dh)^{\frac{1}{2}}$ (meaning that $\tilde{R}_n(h)^p$ for $p > 2$ decay more quickly to zero in $\alpha_n n$), which we will show is $O_{f_{0,n}}(\frac{1}{\alpha_nn})$. We mention that the cross term $\frac{\langle S_n,h^{\otimes 3}\rangle}{6\sqrt{\alpha_n n}} \tilde{R}_n(h)$ can be handled similarly.\\

    We use a similar argument as that for $R_b(h)$ above. Let $M_3 > 0$  be the constant specified by Assumption \textbf{(A2')} and define the event $\mathcal{E}_1 = \{\tilde{R}_n(h) \leq \frac{M_3}{\alpha_nn}\|h\|_2^4 \text{ for all } h \in \tilde{K}\}.$ Recall the definition of the set $\mathcal{H}_{n}\equiv\{h\in\mathbb{R}^p|\frac{\|h\|_2}{\sqrt{\alpha_nn}} \leq \delta_2\}$ in Assumption \textbf{(A2')} where $\delta_2$ is the value of $\delta$ such that \textbf{(A2')} holds. As in the argument around \eqref{eq:Rb_complement}, using that $\tilde{K} \subseteq \mathcal{H}_n$ and Assumption \textbf{(A2')}, there exists $N_5\equiv N_5(\epsilon,\delta_2)$ such that for $\alpha_nn > N_5$,
    \begin{equation}
    \begin{split}
        \mathbb{P}_{f_{0,n}} (\mathcal{E}_1^\mathsf{c}) 
        &= \mathbb{P}_{f_{0,n}}\Big(\tilde{R}_n(h) - \frac{M_3}{\alpha_nn}\|h\|_2^4>0 \text{ for some } h\in\tilde{K}\Big) \\
        &\leq \mathbb{P}_{f_{0,n}}\Big(\sup_{h\in\mathcal{H}_n}\tilde{R}_n(h) - \frac{M_3}{\alpha_nn}\|h\|_2^4>0\Big)
        < \frac{\epsilon}{2}.
        \label{eq:tilde_Rn_complement}
        \end{split}
    \end{equation}
    Now, given a constant $\tilde{M}_3 > 0$ to be specified later, we apply a similar argument as in \eqref{eq:Rb_arg_eventE} to conclude for $\alpha_nn > \max(N_0, N_1, N_5)$, using \eqref{eq:prob_arg},
    \begin{align}
        \nonumber&\mathbb{P}_{f_{0,n}}\Big(\Big(\int_{\tilde{K}}\rho_{H}(h)\tilde{R}_n(h)^2 dh\Big)^{\frac{1}{2}} > \frac{\tilde{M}_3}{\alpha_nn}\Big) \\
        \label{eq:tilde_Rn_arg_prob_arg} &\leq \mathbb{P}_{f_{0,n}}\Big(\frac{M_3}{\alpha_nn}\Big(\int_{\tilde{K}}\rho_{H}(h) \|h\|_2^8 \, dh\Big)^{\frac{1}{2}} > \frac{\tilde{M}_3}{\alpha_nn} \, \big| \, \mathcal{A}_{K}\cap\mathcal{B}\cap\mathcal{E}_1\Big) + \mathbb{P}_{f_{0,n}}(\mathcal{E}_1^\mathsf{c}) + \frac{\epsilon}{2} < \epsilon.
    \end{align}
    \normalsize
    The first inequality in \eqref{eq:tilde_Rn_arg_prob_arg}  follows from conditioning on $\mathcal{E}_1$ and  the final inequality from \eqref{eq:tilde_Rn_complement} and choosing
  $
        \tilde{M}_3 = 2M_3(\int_{\mathbb{R}^p}\rho_{H}(h) \|h\|_2^8\, dh)^{\frac{1}{2}},$
    which is finite since the integral term is the expectation of a polynomial with respect to a Gaussian (conditional on $\mathcal{B}$). This choice of $\tilde{M}_3$ implies that
    \begin{equation*}
        \mathbb{P}_{f_{0,n}}\Big(\frac{M_3}{\alpha_nn}\Big(\int_{\mathbb{R}^p}\rho_{H}(h)\|h\|_2^8 dh\Big)^{\frac{1}{2}} > \frac{\tilde{M}_3}{\alpha_nn} \,\big|\,\mathcal{A}_{K}\cap\mathcal{B}\cap\mathcal{E}_1\Big) = 0.
    \end{equation*}
    We have shown $(\int_{\tilde{K}}\rho_{H}(h)\tilde{R}_n(h)^2 dh)^{\frac{1}{2}}$ is $O_{f_{0,n}}(\frac{1}{\alpha_nn})$ for $\alpha_nn$ sufficiently large. Hence,
    $\tilde{V_1} = O_{f_{0,n}}(\frac{1}{\alpha_n n}).$

    \noindent\textbf{Analysis of $\mathbf{\tilde{V}_2}$:} We write $\tilde{V}_2$ as,
    \begin{equation}\label{eq:V2_tilde}
        \begin{split}
        \tilde{V}_2
        &= \int_{\tilde{K}}\rho_{H}(h) \Big[\frac{v^\top h}{\sqrt{\alpha_n n}}+ \frac{\langle S_n,h^{\otimes 3}\rangle}{6\sqrt{\alpha_n n}}\cdot b(\hat{\theta}) \Big] dh = - \int_{\tilde{K}^\mathsf{c}} \rho_{H}(h) \Big[\frac{v^\top h}{\sqrt{\alpha_n n}}+ \frac{\langle S_n,h^{\otimes 3}\rangle}{6\sqrt{\alpha_n n}}\cdot b(\hat{\theta}) \Big] dh,
        \end{split}
    \end{equation}
    where we have used $\int_{\mathbb{R}^p} \rho_{H}(h) [\frac{v^\top h}{\sqrt{\alpha_n n}}+ \frac{\langle S_n,h^{\otimes 3}\rangle}{6\sqrt{\alpha_n n}}\cdot b(\hat{\theta})]dh = 0$, which follows since, conditional on $\mathcal{B}$, the integral is the expectation of a polynomial in $h$ with respect to a mean-zero Gaussian density with covariance $H_n^{-1}$, so the odd moments are zero.
    Next, by Cauchy-Schwarz
    \begin{equation*}
        \begin{split}
            &|\tilde{V}_2|\leq \int_{\tilde{K}^\mathsf{c}}\rho_{H}(h) \Big|\frac{v^\top h}{\sqrt{\alpha_n n}}+ \frac{\langle S_n,h^{\otimes 3}\rangle}{6\sqrt{\alpha_n n}}\cdot b(\hat{\theta}) \Big|dh \\
            &\leq \Big[\int_{\tilde{K}^\mathsf{c}}\rho_{H}(h) \Big|\frac{v^\top h}{\sqrt{\alpha_n n}}+ \frac{\langle S_n,h^{\otimes 3}\rangle}{6\sqrt{\alpha_n n}}\cdot b(\hat{\theta}) \Big|^2dh\Big]^{\frac{1}{2}} \Big[\int_{\tilde{K}^\mathsf{c}}\rho_{H}(h)dh\Big]^{\frac{1}{2}} \\
            &\leq \Big[\int_{\mathbb{R}^p} \hspace{-4pt} \rho_{H}(h) \Big|\frac{v^\top h}{\sqrt{\alpha_n n}}+ \frac{\langle S_n,h^{\otimes 3}\rangle}{6\sqrt{\alpha_n n}}\cdot b(\hat{\theta}) \Big|^2dh\Big ]^{\frac{1}{2}} \Big[\mathbb{P}\Big(X \in \bar{B}_0(\delta \sqrt{\alpha_nn})^\mathsf{c} \Big)\Big]^{\frac{1}{2}},
        \end{split}
    \end{equation*}
    where, in the final inequality, $X \sim \mathcal{N}(0, H_n^{-1})$ and we have used that $\tilde{K}^{\mathsf{c}} = B_{0}(\delta \sqrt{\alpha_n n})^{\mathsf{c}}$. Notice that, conditional on $\mathcal{B}$, the first integral on the right side of the above is upper bounded by a constant, as it is an expectation of a Gaussian polynomial. We will finally appeal to Lemma \ref{lem:Gaussian_concentration} as in \eqref{eq:T_22_conc}, to bound the above yielding, for some constant $C>0$,
    \begin{equation}\label{eq:V2_tilde_chernoff} 
        |\tilde{V}_2|
        \leq C \Big[\exp\Big(-\frac{(\delta\sqrt{\alpha_nn} - \sqrt{\text{tr}(H_n^{-1})})^2}{2\|H_n^{-1}\|_{\text{op}}}\Big)\Big]^{\frac{1}{2}}.
    \end{equation}

Using \eqref{eq:prob_arg}, for any given $M_5 > 0$,  we have that for $\alpha_nn > \max(N_0,N_1)$,
    \begin{equation*}
        \begin{split}
            \mathbb{P}_{f_{0,n}}\Big(|\tilde{V}_2| > \frac{M_5}{\alpha_nn}\Big) 
            &\leq  \mathbb{P}_{f_{0,n}}\Big(|\tilde{V}_2| > \frac{M_5}{\alpha_nn} \, \Big| \, \mathcal{A}_{K}\cap\mathcal{B}\Big) +\frac{\epsilon}{2}.
        \end{split}
    \end{equation*}
    Furthermore, there exists an $M_5 > 0$ and $N_7 \equiv N_7(\epsilon, \delta, M_5)$ such that $\mathbb{P}_{f_{0,n}}(|\tilde{V}_2| > \frac{M_5}{\alpha_nn} \, | \, \mathcal{A}_{K}\cap\mathcal{B}) =0 $ whenever $\alpha_nn > N_7$ since \eqref{eq:V2_tilde_chernoff} decreases exponentially (in $\alpha_n n$) on the event $\mathcal{B}$ by an argument similar to that near \eqref{eq:T_22_conc}. Hence, by similar arguments as those for bounding $|T_{21}|$ and $|T_{22}|$, we have that
    $\tilde{V_2} = O_{f_{0,n}}(\frac{1}{\alpha_n n})$
    and
    $T_{21} = \tilde{V}_1 + \tilde{V}_2 = O_{f_{0,n}}(\frac{1}{\alpha_nn})$.

    Now that we have established \eqref{eq:lap_lem_result}, we will now show that this result holds with $f_n(X^n|\theta)$ replaced by $\tilde{f}_n(X^n|\theta)$. To so, we verify that $\tilde{f}_n(X^n|\theta)$ satisfies assumptions \textbf{(A0)}, \textbf{(A2')}, and \textbf{(A3')}. Assumptions \textbf{(A1')} and \textbf{(ALap)} are assumptions on the prior and $q(\theta)$, so they are assumed already to hold. We check the remaining assumptions as follows: \\

\noindent\textbf{Assumption (A0):} Notice that the MLE, $\hat{\theta}$, appears in \eqref{eq:lap_lem_result}. Hence, we must show that the MLE is the same for both $f_n(X^n|\theta)$ and $\tilde{f}_n(X^n|\theta)$. We do so in Proposition \ref{prop:random_alpha_(A0)(A2)}.\\

\noindent\textbf{Assumptions (A2') and (A3'):} By Proposition \ref{prop:random_alpha_(A2')(A3')}, assumptions \textbf{(A2')} and \textbf{(A3')} hold for $\tilde{f}_n(X^n|\theta)$ (since we have assumed they hold for $f_n(X^n|\theta)$).
\end{proof}

\subsection{Proof of Proposition \ref{prop:alpha_inf}}\label{app:BvM_inf}
\begin{proof}
We first prove \eqref{eq:conv_1}. We aim to show that for any $t > 0$ and any $\epsilon > 0$, there exists an $N>0$ such that for all $n > N$, we have 
$\mathbb{P}_{f_{0,n}} (\textrm{d}_{\textrm{W}_p} (\pi_{n,\gamma_n},\delta_{\hat\theta}) > t) < \epsilon$.
Recall that for a Dirac measure,
$\textrm{d}_{\textrm{W}_p}(\pi_{n,\gamma_n},\delta_{\hat\theta})^p 
= \int_{\mathbb{R}^p} \|\theta -\hat\theta\|_2^p \,\pi_{n,\gamma_n}(\theta|X_n) d\theta.$
Fix an arbitrary $r>0$ small enough and consider a ball around the MLE, denoted $B_{\hat \theta}(r)$.
Notice that by \textbf{(A0)}, we have a unique maximizer of the likelihood; namely, there exists a unique 
$\hat\theta \in \Theta$ such that 
$f_n(X^n|\hat \theta) = \sup_{\theta\in\Theta} f_n(X^n|\theta) =: f_{\max}$.

By \textbf{(A3')} we have that for all $r > 0$, a constant $c\equiv c(r,\epsilon)$, and $n$ sufficiently large,
\begin{equation*}
    \begin{split}
        \mathbb{P}_{f_{0,n}}&\Big(\sup_{\theta\in B_{\hat{\theta}}(r)^\mathsf{c}} f_n(X^n|\theta)> f_{\max}\Big) 
        \leq \mathbb{P}_{f_{0,n}}\Big(\sup_{\theta\in B_{\hat{\theta}}(r)^\mathsf{c}} \frac{1}{n}\log f_n(X^n|\theta) -\frac{1}{n}\log f_{\max} > -c\Big) < \epsilon.
    \end{split}
\end{equation*}
This implies that there exists a $\rho(r)\in[0,1)$ and an $N_0 > 0$, such that when $n > N_0$,
\begin{equation}
    \mathbb{P}_{f_{0,n}}\Big(\sup_{B_{\hat{\theta}}(r)^\mathsf{c}}  f_n(X^n|\theta) \leq \rho(r)\,f_{\max} \Big) \geq 1 - \frac{\epsilon}{3}.
\end{equation}
Label the event $\mathcal{E}_r := \{\sup_{B_{\hat{\theta}}(r)^\mathsf{c}}  f_n(X^n|\theta) \leq \rho(r)\,f_{\max}\}$, and the above implies $\mathbb{P}_{f_{0,n}}(\mathcal{E}_r^{\mathsf{c}}) < \epsilon/3$ whenever $n > N_0$. 

By continuity of the likelihood at $\hat\theta$, for every $r>0$ small enough, there exists $c(r)\in(0,1]$ such that
$\inf_{\theta \in \bar{B}_{\hat \theta}(r)} f_n(X^n|\theta) \ge c(r)\,f_{\max}$. Notice that $c(r) \rightarrow 1$ as $r \rightarrow 0$.

Finally, we argue that for $r_0>0$ small enough, $\pi(\bar{B}_{\hat \theta}(r_0))>0$ by Assumption \textbf{(A0)} and the prior support Assumption \textbf{(A1)}. Let $\delta>0$ define the  ball of \textbf{(A1)} around $\theta^*$ where $\pi$ is positive and we take $r_0 \leq \delta/2$.
Indeed, define an event $\mathcal{E}_{r_0} := \{\|\theta^* - \hat{\theta}\|_2 < r_0\}$ and notice that \textbf{(A0)} implies that there exists $N_1$ such that $\mathbb{P}_{f_{0,n}}(\mathcal{E}_{r_0}^{\mathsf{c}}) < \epsilon/3$ whenever $n > N_1$. Therefore, on event $\mathcal{E}_{r_0}$, we have $\bar{B}_{\hat \theta}(r_0) \subset B_{\theta^*}(\delta)$,  the ball where $\pi$ is positive, and $\theta^* \in \bar{B}_{\hat \theta}(r_0)$. In particular, there is an $r'_0 \leq r_0$ such that $B_{\theta^*}(r'_0) \subset \bar{B}_{\hat \theta}(r_0)$ and therefore $\pi(\bar{B}_{\hat \theta}(r_0)) > \pi(B_{\theta^*}(r'_0)).$

Then,
\begin{equation}
\begin{split}
\label{eq:dw_bound1}
\textrm{d}_{\textrm{W}_p} \big(\pi_{n,\gamma_n},\delta_{\hat\theta}\big)^p  &= \int_{\mathbb{R}^p} \|\theta -\hat\theta\|_2^p \,\pi_{n,\gamma_n}(\theta|X_n) d\theta  \\
&= \int_{B_{\hat \theta}(r)}\|\theta -\hat\theta\|_2^p \,\pi_{n,\gamma_n}(\theta|X_n) d\theta 
   + \int_{B_{\hat \theta}(r)^\mathsf{c}}  \|\theta -\hat\theta\|_2^p \,\pi_{n,\gamma_n}(\theta|X_n) d\theta \\
&\le  r^p \,+\, 
  \frac{\int_{B_{\hat \theta}(r)^\mathsf{c}} \|\theta -\hat\theta\|_2^p f_n(X^n|\theta)^{\gamma_n}\pi(\theta) d \theta}
   {\int_{\mathbb{R}^p} f_n(X^n|\theta)^{\gamma_n}\pi(\theta) d \theta} .
\end{split}
\end{equation}
We notice that conditional on an event $\mathcal{E}_M=\{\|\hat{\theta}\|_2^p < M\}$, for some $M > 0$, by assumption (ii) we have that $\int_{B_{\hat \theta}(r)^\mathsf{c}} \|\theta -\hat\theta\|_2^p \,\pi(\theta) d \theta \leq M_p(r) < \infty$ where $M_p(r)$ is finite. 
By Assumption \textbf{(A0)} for any $\epsilon > 0$, and for some $N_2$, we have that $\mathbb{P}_{f_{0,n}}(\mathcal{E}_M^\mathsf{c}) < \epsilon/3$ whenever $n > N_2$. Next, for some fixed $r_0 > 0$ (selected so that we are inside the ball given by Assumption \textbf{(A1)}), and conditional on $\mathcal{E}\equiv\mathcal{E}_r \cap \mathcal{E}_{r_0} \cap \mathcal{E}_{M}$,
\begin{equation}
\begin{split}
\label{eq:dw_bound2}
   \frac{\int_{B_{\hat \theta}(r)^\mathsf{c}} \|\theta -\hat\theta\|_2^p f_n(X^n|\theta)^{\gamma_n}\pi(\theta) d \theta}
   {\int_{\mathbb{R}^p} f_n(X^n|\theta)^{\gamma_n}\pi(\theta) d \theta} & \le \frac{\int_{B_{\hat \theta}(r)^\mathsf{c}} \|\theta -\hat\theta\|_2^p f_n(X^n|\theta)^{\gamma_n}\pi(\theta) d \theta}
   {\int_{\bar{B}_{\hat \theta}(r_0)} f_n(X^n|\theta)^{\gamma_n}\pi(\theta) d \theta} \\
&\le 
   \frac{(\rho(r) f_{\max})^{\gamma_n} \int_{B_{\hat \theta}(r)^\mathsf{c}} \|\theta -\hat\theta\|_2^p \pi(\theta) d \theta}
   {(c(r_0) f_{\max})^{\gamma_n} \pi(\bar{B}_{\hat \theta}(r_0))} \\
   &\leq \frac{M_p(r)}{\pi(B_{\theta^*}(r'_0))} \Big(\frac{\rho(r)}{c(r_0)}\Big)^{\gamma_n}.
\end{split}
\end{equation}

We have shown $\mathbb{P}_{f_{0,n}}(\mathcal{E}^{\mathsf{c}}) \leq 3 \times (\epsilon/3) = \epsilon$ for any $\epsilon>0$ when $n > \max(N_0, N_1, N_2)$. Therefore,
\begin{align*}
&\mathbb{P}_{f_{0,n}} (\textrm{d}_{\textrm{W}_p} (\pi_{n,\gamma_n},\delta_{\hat\theta}) > t) \\
&  \qquad \le \mathbb{P}_{f_{0,n}} \Big( \Big\{ r^p + 
  \frac{\int_{B_{\hat \theta}(r)^\mathsf{c}} \|\theta -\hat\theta\|_2^p f_n(X^n|\theta)^{\gamma_n}\pi(\theta) d \theta}
   {\int_{\Theta} f_n(X^n|\theta)^{\gamma_n}\pi(\theta) d \theta} > t^p \Big\} \cap \mathcal{E}\Big)+ \mathbb{P}_{f_{0,n}}(\mathcal{E}^{\mathsf{c}}) \\
&\qquad \le \mathbb{I} \Big\{ r^p + 
  \frac{M_p(r)}{\pi(B_{\theta^*}(r'_0))} \Big(\frac{\rho(r)}{c(r_0)}\Big)^{\gamma_n} > t^p\Big\} + \epsilon.
\end{align*}
%'
Now, since for any $r > 0$, we have $\rho(r) < 1$, we can select $0 < r_0 \leq \delta/2$ small enough such that $\rho(r)< c(r_0)$ as $c(r) \rightarrow 1$ as $r$ shrinks.

For example, let us take $r = t/\sqrt[p]{2}$. Hence, 
\begin{align}
\label{eq:final_bound}
     \mathbb{P}_{f_{0,n}} \Big(\textrm{d}_{\textrm{W}_p} \big(\pi_{n,\gamma_n},\delta_{\hat\theta}\big) > t\Big)
    &\le  \mathbb{I} \Big\{ \frac{M_p(r)}{\pi(B_{\theta^*}(r'_0))} \Big(\frac{\rho(r)}{c(r_0)}\Big)^{\gamma_n} > \frac{1}{2}t^p\Big \} + \epsilon,
\end{align}
giving the desired result since $M_p(r)$ and $\pi(B_{\theta^*}(r'_0))$ are constants and
\begin{align}
\label{eq:final_bound_1}
\mathbb{I}\Big(\frac{M_p(r)}{\pi(B_{\theta^*}(r'_0))} \Big(\frac{\rho(r)}{c(r_0)}\Big)^{\gamma_n} > \frac{1}{2}t^p\Big) \overset{n \rightarrow \infty}{\rightarrow} 0 \text{ as } \Big(\frac{\rho(r)}{c(r_0)}\Big)^{\gamma_n} \overset{n \rightarrow \infty}{\rightarrow} 0 .
\end{align}

Now we are interested in proving the result in \eqref{eq:conv_1} with $\hat{\gamma}_n$ satisfying \eqref{eq:mix_alpha}. 
We first define a set $\mathcal{A} = [N_{\gamma}, \infty)$ using any value $N_{\gamma}$ for which
\[\frac{M_p(r)}{\pi(B_{\theta^*}(r'_0))} \Big(\frac{\rho(r)}{c(r_0)}\Big)^{N_{\gamma}} \leq \frac{1}{2}t^p \quad \implies \quad  N_{\gamma}  > \frac{\log \Big(\frac{t^p\pi(B_{\theta^*}(r'_0))}{2M_p(r)} \Big)}{\log\Big(\frac{\rho(r)}{c(r_0)}\Big)}.\]
By \eqref{eq:mix_alpha}, for any $\epsilon >  0$, we have that there exists an $N^*$ such that for $n > N^*$ we have 
\[\mathbb{P}_{f_{0,n}} (\hat{\gamma}_n \not\in \mathcal{A}) = \mathbb{P}_{f_{0,n}} (\hat{\gamma}_n < N_{\gamma}) \leq \epsilon.\]
Then we notice from \eqref{eq:dw_bound1} and \eqref{eq:dw_bound2} that, conditional on $\mathcal{E}\equiv\mathcal{E}_r \cap \mathcal{E}_{r_0} \cap \mathcal{E}_{M}$,
\begin{equation}
\begin{split}
\label{eq:dw_bound3}
 \sup_{\gamma \in \mathcal{A}} \textrm{d}_{\textrm{W}_p} \big(\pi_{n,\gamma},\delta_{\hat\theta}\big)^p
   \le  r^p \,+\, 
   \sup_{\gamma \in \mathcal{A}}  
    \frac{M_p(r)}{\pi(B_{\theta^*}(r'_0))} \Big(\frac{\rho(r)}{c(r_0)}\Big)^{\gamma}  \le  r^p \,+\, \frac{1}{2}t^p,
\end{split}
\end{equation}
where the final step uses the definition of the set $\mathcal{A}$.
We have shown that $\mathbb{P}_{f_{0,n}}(\mathcal{E}^{\mathsf{c}}) \leq 3 \times (\epsilon/3) = \epsilon$ for any $\epsilon>0$ when $n > \max(N_0, N_1, N_2)$. Therefore,
 when $n > \max(N_0, N_1, N_2)$,
\begin{equation*}
\begin{split}
  \mathbb{P}_{f_{0,n}} \Big( \sup_{\gamma \in \mathcal{A}}   \textrm{d}_{\textrm{W}_p} \big(\pi_{n,\gamma},\delta_{\hat\theta}\big) > t\Big)
  &\leq   \mathbb{P}_{f_{0,n}} \Big( \Big\{\sup_{\gamma \in \mathcal{A}}   \textrm{d}_{\textrm{W}_p} \big(\pi_{n,\gamma},\delta_{\hat\theta}\big)^p > t^p\Big\} \,\, \cap \,\, \mathcal{E}\Big) + \epsilon \\
    &\le   \mathbb{P}_{f_{0,n}} \Big( r^p   > \frac{1}{2}t^p \,\, \cap \,\, \mathcal{E}\Big) + \epsilon  = \epsilon,
    \end{split}
\end{equation*}
again by taking $r$ small enough.

Our aim is to show that for any $t > 0$ and any $\epsilon > 0$, there exists an $N>0$ such that for all $n > N$, we have 
$\mathbb{P}_{f_{0,n}} (\textrm{d}_{\textrm{W}_p} (\pi_{n,\hat{\gamma}_n},\delta_{\hat\theta}) > t) < 2\epsilon$. Notice that
\begin{equation}
\label{eq:wassbound1}
\mathbb{P}_{f_{0,n}} (\textrm{d}_{\textrm{W}_p} (\pi_{n,\hat{\gamma}_n},\delta_{\hat\theta}) > t) \leq \mathbb{P}_{f_{0,n}} (\textrm{d}_{\textrm{W}_p} (\pi_{n,\hat{\gamma}_n},\delta_{\hat\theta}) > t) \,\, \cap \,\, \hat{\gamma}_n \in \mathcal{A}) + \mathbb{P}_{f_{0,n}} (\hat{\gamma}_n \not\in \mathcal{A}).
\end{equation}
When $n > N^*$ is large enough, the second term on the right side of the above is upper bounded by $\epsilon$. For the first term, notice that when $n > \max(N_0, N_1, N_2)$,
\begin{equation}
\begin{split}
\label{eq:wassbound2}
\mathbb{P}_{f_{0,n}} (\textrm{d}_{\textrm{W}_p} (\pi_{n,\hat{\gamma}_n},\delta_{\hat\theta}) > t \,\, \cap \,\, \hat{\gamma}_n \in \mathcal{A}) 
& \leq \mathbb{P}_{f_{0,n}} \Big( \sup_{\gamma \in \mathcal{A}} \textrm{d}_{\textrm{W}_p} (\pi_{n,\gamma},\delta_{\hat\theta}) > t \,\, \cap \,\, \hat{\gamma}_n \in \mathcal{A}\Big) \\
&\leq    \mathbb{P}_{f_{0,n}} \Big(\sup_{\gamma \in \mathcal{A}} \textrm{d}_{\textrm{W}_p} \big(\pi_{n,\gamma},\delta_{\hat\theta}\big) > t\Big) \leq \epsilon,
\end{split}
\end{equation}
which completes the proof with $N = \max(N^*, N_0, N_1, N_2)$.
\end{proof}

\subsection{Proof of Theorem \ref{thm:BvM_mix_data-dep}}\label{app:BvM_mix_data-dep}
Before we begin the proof of Theorem \ref{thm:BvM_mix_data-dep}, we state and prove two auxiliary results bounding Wasserstein distance with total variation distance plus moments, so that we can appeal to results from Section~\ref{sec:BvM+moments} where convergence was in total variation.
\begin{Lemma}\label{lem:tv-wp}
Let $(X,d)$ be a metric space and fix $p\ge 1$ and $m>p$. For any $x_0\in X$ and probability measures $\mu,\nu$ on $X$ with finite $m$-th moments, set $\delta:=\textrm{d}_{\textrm{TV}}(\mu,\nu):=\sup_{A}|\mu(A)-\nu(A)|$ and
\begin{equation}\label{eq:M_m}
M_m(\mu,\nu):=\int d(x,x_0)^m \mu(dx)+\int d(y,x_0)^m\,\nu(dy).
\end{equation}
Then $\textrm{d}_{\textrm{W}_p}(\mu,\nu) \le 2^{1+\frac{1}{p}}M_m(\mu,\nu)^{\frac{1}{m}}\delta^{\frac{1}{p}-\frac{1}{m}}.$
\end{Lemma}
\begin{proof}
Let $(X,Y)$ be a maximal coupling of $(\mu,\nu)$ so that $\mathbb{P}(X\neq Y)=\delta$. For any $x_0 \in X$, we have $d(X,Y)\le d(X,x_0)+d(Y,x_0)$. Moreover, $(a+b)^p\le 2^{p-1}(a^p+b^p)$ and $\mathbf{1}\{\max(a,b)>c\}\leq \mathbf{1}\{a>c\}+\mathbf{1}\{b>c\}$. Using these results, we see that for any $R>0$ and any $x_0 \in X$, 
\begin{equation}
\begin{split}
&d(X,Y)^p = d(X,Y)^p\mathbf{1}\{X\neq Y, d(X,x_0)\le R, d(Y,x_0)\le R\} \\
&\qquad + d(X,Y)^p \mathbf{1}\{\max(d(X,x_0),d(Y,x_0))>R\}\\
&\le (2R)^p\mathbf{1}\{X\neq Y\}
+ 2^{p-1}(d(X,x_0)^p+d(Y,x_0)^p)\mathbf{1}\{\max(d(X,x_0),d(Y,x_0))>R\}\\
&\le (2R)^p\mathbf{1}\{X\neq Y\}
+ 2^{p}d(X,x_0)^p\mathbf{1}\{d(X,x_0)>R\}+ 2^pd(Y,x_0)^p\mathbf{1}\{d(Y,x_0)>R\}.
\label{eq:dbound}
\end{split}
\end{equation}
Hence, from \eqref{eq:dbound}, we have
\begin{align}
\nonumber \mathbb{E}[d(X,Y)^p]
&\le (2R)^p\mathbb{P}(X\neq Y)
+ 2^{p}\mathbb{E}\left[d(X,x_0)^p\mathbf{1}\{d(X,x_0)>R\}+d(Y,x_0)^p\mathbf{1}\{d(Y,x_0)>R\}\right]\\
&\le (2R)^p \delta
+ 2^{p} R^{p-m} M_m(\mu,\nu),
\label{eq:dbound2}
\end{align}
where the final step uses the tail bound
$\int_{\{d>R\}} d^p d\mu \le R^{p-m} \int d^m d\mu$ (and similarly for $\nu$) for $m > p$.
Choose $R$ by balancing the two terms: $(2R)^p\delta=2^{p} R^{p-m}M_m$, giving
$R^m=M_m/\delta$. Then the bound in \eqref{eq:dbound2} gives
$\mathbb{E}[ d(X,Y)^p] \le 2(2R)^p\delta
= 2^{p+1}M_m(\mu,\nu)^{\frac{p}{m}}\delta^{1-\frac{p}{m}}.$
Taking $p$-th roots yields the claim:
\[
\textrm{d}_{\textrm{W}_p}(\mu,\nu) = \Big(\inf_{(X,Y): X\sim \mu; \,Y \sim \nu}\mathbb{E}[d(X,Y)^p]\Big)^{\frac{1}{p}}
\le 2^{1+\frac{1}{p}}M_m(\mu,\nu)^{\frac{1}{m}}\delta^{\frac{1}{p}-\frac{1}{m}}.
\]
\end{proof}

The following corollary considers the setting of random measures, like we need in our setting.
\begin{Corollary}\label{cor:random}
Let $\{\widehat\mu_n\},\{\widehat\nu_n\}$ be random probability measures on $(X,d)$ and fix $m>p\ge 1$. If both
$\textrm{d}_{\textrm{TV}}(\widehat\mu_n,\widehat\nu_n)\xrightarrow{\mathbb{P}}0$ and $M_m(\widehat\mu_n,\widehat\nu_n)=O_{\mathbb{P}}(1)$ defined in \eqref{eq:M_m} for some $x_0 \in X$,
then
$\textrm{d}_{\textrm{W}_p}(\widehat\mu_n,\widehat\nu_n)\xrightarrow{\mathbb{P}}0.$
\end{Corollary}
\begin{proof}
Applying Lemma~\ref{lem:tv-wp} pointwise to $(\widehat\mu_n,\widehat\nu_n)$ we find
\[
\textrm{d}_{\textrm{W}_p}(\widehat\mu_n,\widehat\nu_n)
\le 2^{1+\frac{1}{p}} M_m(\widehat\mu_n,\widehat\nu_n)^{\frac{1}{m}}
\textrm{d}_{\textrm{TV}}(\widehat\mu_n,\widehat\nu_n)^{\frac{1}{p}-\frac{1}{m}}.\]
Since $M_m(\widehat\mu_n,\widehat\nu_n)=O_{\mathbb{P}}(1)$ and $\textrm{d}_{\textrm{TV}}(\widehat\mu_n,\widehat\nu_n) \overset{p}{\to} 0$,
the right-hand side goes to $0$ in probability.
\end{proof}

Let us now turn to the proof of Theorem \ref{thm:BvM_mix_data-dep}.

\begin{proof}

We first note that the triangle inequality gives
\begin{align}
\nonumber\textrm{d}_{\textrm{W}_p} \Big(\pi_{n,\hat{\alpha}_n^{\text{mix}}}, \, q  \phi_n +(1-q)\delta_{\hat\theta}\Big) &= \textrm{d}_{\textrm{W}_p} \Big(q_n\pi_{n,\hat{\alpha}_n} + (1-q_n)\pi_{n,\hat{\gamma}_n} , \, q  \phi_n +(1-q)\delta_{\hat\theta}\Big)\\
&\leq \label{eq:w1} \textrm{d}_{\textrm{W}_p} \Big(q_n\pi_{n,\hat{\alpha}_n} + (1-q_n)\pi_{n,\hat{\gamma}_n} , \, q\pi_{n,\hat{\alpha}_n} + (1-q)\pi_{n,\hat{\gamma}_n}\Big) \\
& \quad \quad +  \label{eq:w2} \textrm{d}_{\textrm{W}_p} \Big( q\pi_{n,\hat{\alpha}_n} + (1-q)\pi_{n,\hat{\gamma}_n} , \, q \phi_n +(1-q)\delta_{\hat\theta}\Big).
\end{align}
% %
We will show that both the terms \eqref{eq:w1} and \eqref{eq:w2} are $o_{f_{0,n}}(1)$, which, by the above bound  establishes the desired result.

\noindent\textbf{Term \eqref{eq:w1}:} To show that \eqref{eq:w1} converges to zero in $f_{0,n}$-probability, we will apply Corollary \ref{cor:random} with $\widehat{\mu}_n = q_n\pi_{n,\hat{\alpha}_n} + (1-q_n)\pi_{n,\hat{\gamma}_n}$ and $\widehat{\nu}_n = q\pi_{n,\hat{\alpha}_n} + (1-q)\pi_{n,\hat{\gamma}_n}$. To do so, we will prove the following as Claim 1 and Claim 2 below:
\begin{enumerate}
    \item $\textrm{d}_{\textrm{TV}}(\widehat{\mu}_n, \widehat{\nu}_n) = o_{f_{0,n}}(1)$.
    \item $M_m(\widehat{\mu}_n, \widehat{\nu}_n) = O_{f_{0,n}}(1)$ for some $m > p$, where $M_m$ is defined in \eqref{eq:M_m}.
\end{enumerate}

\textbf{Claim 1:} By the definition of total variation distance, we have
\begin{align*}
    \nonumber \textrm{d}_{\textrm{TV}}(\widehat{\mu}_n, \widehat{\nu}_n) &= \textrm{d}_{\textrm{TV}}\left(q_n\pi_{n,\hat{\alpha}_n}(\theta|X^n) + (1-q_n)\pi_{n,\hat{\gamma}_n}(\theta|X^n) ,~q\pi_{n,\hat{\alpha}_n}(\theta|X^n) + (1-q)\pi_{n,\hat{\gamma}_n}(\theta|X^n)\right) \\
    \nonumber&= \frac{1}{2}\int_{\mathbb{R}^p}\left|(q_n-q)\pi_{n,\hat{\alpha}_n}(\theta|X^n) - (q_n-q)\pi_{n,\hat{\gamma}_n}(\theta|X^n)\right|d\theta \\
    &\leq \frac{|q_n-q|}{2}\Big(\int_{\mathbb{R}^p}\pi_{n,\hat{\alpha}_n}(\theta|X^n)d\theta + \int_{\mathbb{R}^p}\pi_{n,\hat{\gamma}_n}(\theta|X^n)d\theta\Big) = |q_n-q|.
\end{align*}
Finally, since $q_n\to q$ in $f_{0,n}$-probability, the right side of the above converges to 0, and we conclude that $\textrm{d}_{\textrm{TV}}(\widehat{\mu}_n, \widehat{\nu}_n)$ converges to 0.

\textbf{Claim 2:}  Given  $m>p$, and letting $x_0=\theta^*$ in \eqref{eq:M_m} we can bound $M_m(\cdot,\cdot)$ as follows:
\begin{align}
  \frac{ M_m(\widehat{\mu}_n, \widehat{\nu}_n)}{2}  
    \nonumber&= \frac{q_n+q}{2}\int_{\mathbb{R}^p} \hspace{-2.3pt} \|\theta-\theta^*\|_2^m\pi_{n,\hat{\alpha}_n}(\theta|X^n)d\theta 
    + \frac{2-(q_n+q)}{2}\int_{\mathbb{R}^p} \hspace{-2.3pt}\|\theta-\theta^*\|_2^m\pi_{n,\hat{\gamma}_n}(\theta|X^n)d\theta \\
    &\leq \int_{\mathbb{R}^p}\|\theta-\theta^*\|_2^m\pi_{n,\hat{\alpha}_n}(\theta|X^n)d\theta 
    + \int_{\mathbb{R}^p}\|\theta-\theta^*\|_2^m\pi_{n,\hat{\gamma}_n}(\theta|X^n)d\theta=:T_1+T_2. \label{eq:M_alpha_gamma_hat}
\end{align}
Notice that we have shown
$\frac{1}{2}M_m(\widehat{\mu}_n, \widehat{\nu}_n) = T_1 + T_2 = M_m(\pi_{n,\hat{\alpha}_n},\pi_{n,\hat{\gamma}_n}),$
and to prove the claim, it suffices to argue that the above is $O_{f_{0,n}}(1)$.
We will see that in fact both $T_1$ and $T_2$ are $o_{f_{0,n}}(1)$ in what follows. Indeed, the change of variable $h = \sqrt{\alpha_nn}(\theta-\theta^*)$ as well as Theorem \ref{thm:moments_alt} and  $(\alpha_nn)^{-m/2}\to0$ show that
\begin{align}
    \nonumber 
    T_1 
    &\leq\int_{\mathbb{R}^p} \|\theta-\theta^*\|_2^m \phi\Big(\theta\big|\hat{\theta}, \frac{1}{\alpha_nn}V_{\theta^*}^{-1}\Big)d\theta + \int_{\mathbb{R}^p} \|(\theta-\theta^*)\|_2^m\Big|\pi_{n,\hat{\alpha}_n}(\theta|X^n) - \phi\Big(\theta\big|\hat{\theta}, \frac{1}{\alpha_nn}V_{\theta^*}^{-1}\Big)\Big|d\theta  \\
    \label{eq:T1_bound} & = (\alpha_nn)^{-\frac{m}{2}}\int_{\mathbb{R}^p} (\alpha_nn)^{-\frac{p}{2}}  \|h\|_2^m \phi\Big(h \, \big| \, \sqrt{\alpha_nn}(\hat{\theta}-\theta^*), V_{\theta^*}^{-1}\Big) dh + o_{f_{0,n}}(1) = o_{f_{0,n}}(1),
\end{align}
 The last equality follows from  Lemma \ref{lem:gauss_int_bound2} and $(\alpha_nn)^{-m/2}\to0$.

We now use the bound $\|\theta-\theta^*\|_2^m \leq 2^{m-1}\|\theta-\hat{\theta}\|_2^m + 2^{m-1}\|\hat{\theta}-\theta^*\|_2^m$ to obtain
\begin{align*}
   \frac{T_{2}}{2^{m-1}}  
    &\leq \int_{\mathbb{R}^p} \big( \|\theta-\hat{\theta}\|_2^m  + \|\hat{\theta}-\theta^*\|_2^m\big)\pi_{n,\hat{\gamma}_n}(\theta|X^n)d\theta= d_{\text{W}_{\text{m}}}(\pi_{n,\hat{\gamma}_n},\delta_{\hat{\theta}})^m + \|\hat{\theta}-\theta^*\|_2^m.
\end{align*}
Since we have assumed that the conditions of Proposition \ref{prop:alpha_inf} hold for some $m > p$, the first term of the above converges to $0$ in $f_{0,n}$-probability by Proposition \ref{prop:alpha_inf}. Furthermore, by \textbf{(A0)} of Theorem \ref{thm:moments_alt}, the second term is $o_{f_{0,n}}(1)$. Hence, $T_{2} = o_{f_{0,n}}(1)$.

\noindent\textbf{Term \eqref{eq:w2}:} 
By convexity of the $d_{\textrm{W}_p}(P,Q)^p$ (Lemma \ref{lem:Wasserstein_convex}) and  Proposition \ref{prop:alpha_inf},  we have that
\begin{align}
  \nonumber  \textrm{d}_{\textrm{W}_p}(q\pi_{n,\hat{\alpha
  }_n} + (1-q)\pi_{n,\hat{\gamma}_n},~q \phi_n +(1-q)\delta_{\hat\theta})^p &\leq q\textrm{d}_{\textrm{W}_p}(\pi_{n,\hat{\alpha}_n},\phi_n)^p + (1-q)\textrm{d}_{\textrm{W}_p}(\pi_{n,\hat{\gamma}_n} ,\delta_{\hat\theta})^p\\
    & = q\textrm{d}_{\textrm{W}_p}(\pi_{n,\hat{\alpha}_n},\phi_n)^p+o_{f_{0,n}}(1).\label{eq:dtv_mix_bound}
\end{align}
The remaining term in \eqref{eq:dtv_mix_bound} converges to zero in $f_{0,n}$-probability by Corollary~\ref{cor:random} with $\hat{\mu}_n = \pi_{n,\hat{\alpha}_n}$ and $\hat{\nu}_n = \phi_n$. Indeed, Corollary~\ref{cor:BvM} gives that $\textrm{d}_{\textrm{TV}}(\widehat\mu_n,\widehat\nu_n)\to 0$ in $f_{0,n}$-probability; hence, to apply Corollary~\ref{cor:random} we need to show that $M_m(\widehat\mu_n,\widehat\nu_n)=O_{f_{0,n}}(1),$ for $M_m$ defined in \eqref{eq:M_m}. Notice that
\begin{equation*}
    M_m(\widehat\mu_n,\widehat\nu_n) = \int \|\theta - \theta^*\|^m_2  \pi_{n,\hat{\alpha}_n} (\theta|X^n)d\theta +\int \|\theta - \theta^*\|^m_2  \phi\Big(\theta | \hat{\theta}, \frac{1}{\alpha_nn}V_{\theta^*}^{-1}\Big) d\theta.
\end{equation*}
The first term on the right side of the above is the $T_1$ of \eqref{eq:M_alpha_gamma_hat} and the second term is the same as the first term of \eqref{eq:T1_bound}, both of which are $o_{f_{0,n}}(1)$ as shown above.
\end{proof}

\section{Technical lemmas}\label{app:technical_lemmas}

\begin{Lemma}\label{lem:f(g,h)}
Assume that \textbf{(A0)}--\textbf{(A2)} hold. Then for any $\eta > 0$ and $\epsilon > 0$, there exists an integer $N(\eta, \epsilon)$ and a sequence $r_n \equiv r_n(\eta, \epsilon) \rightarrow \infty$, such that the following hold for any fixed $n$ with $\alpha_nn > N(\eta, \epsilon)$:
\begin{enumerate}
    \item Over $\bar{B}_{0}(r_n)$, the density $\pi_{n,\alpha_n}^{\text{LAN}}(\cdot|X^n)$ given in \eqref{eq:scaled_densities} is positive and the random variable $f_n(g,h)$ given in \eqref{fn+} is well defined with probability at least $1-\epsilon/2$.
    \item We have that 
    \begin{equation}\label{eq:lemma1_results}
    \begin{split}
        \mathbb{P}_{f_{0,n}}\Big(\sup_{g,h\in \bar{B}_{0}(r_n)}f_n(g,h) > \eta \Big) < \epsilon.
    \end{split}
    \end{equation}
\end{enumerate}
\end{Lemma}
\begin{proof}
Our proof consists of three steps, adapting the arguments of \cite[Lemma 5]{avella_medina_robustness_2022}. In Step 1, we check that $f_n(g,h)$ is well defined over balls of  fixed (in $\alpha_nn$) radius $r>0$. In Step 2, 
we prove \eqref{eq:lemma1_results} where the supremum is taken over balls of  fixed radius $r>0$. In Step 3, we construct a sequence $r_n\rightarrow\infty$ where such that both results hold for $g,h\in\bar{B}_0(r_n)$ (i.e., where the results of Step 1 and Step 2 hold to give the final result).\\ 

\noindent\textbf{Step 1: Check that $f_n(g,h)$ is well defined on a ball of fixed radius.} We show that the function $f_n(g,h)$ is well defined on $\bar{B}_{0}(r)$ with high probability for $\alpha_nn$ (which we specify later) sufficiently large. Recall the definition of $f_n(g,h)$:
\begin{equation*}
\begin{split}
    f_n(g,h)
    &= \left\{1-\frac{\phi_n(h)\pi_{n,\alpha_n}^{\text{LAN}}(g|X^n)}{\pi_{n,\alpha_n}^{\text{LAN}}(h|X^n)\phi_n(g)}\right\}^+.\
\end{split}
\end{equation*}
We say $f_n(g,h)$ is well defined, if for $g,h\in\bar{B}_0(r)$, both $\pi_{n,\alpha}^{LAN}(h|X^n)$ and $\phi_n(g)$ are positive (i.e., we are not dividing by 0). 

Recall the definition of the $\alpha_n$-posterior in \eqref{alphaposterior} and the scaled densities in \eqref{eq:scaled_densities}. Then, 
\begin{equation*}   
    \begin{split}    
        \frac{\pi_{n,\alpha_n}^{\text{LAN}}\left(g|X^n\right)}{\pi_{n,\alpha_n}^{\text{LAN}}\left(h|X^n\right)}
        &= \left[\frac{f_n(X^n|\theta^* + \frac{g}{\sqrt{\alpha_nn}})}{f_n(X^n|\theta^* + \frac{h}{\sqrt{\alpha_nn}})}\right]^{\alpha_n}  \hspace{-2pt} \frac{\pi_n(g)}{\pi_n(h)} = \frac{\exp\left(\alpha_n\log f_n\left(X^n|\theta^* + \frac{g}{\sqrt{\alpha_nn}}\right)\right)}{\exp\left(\alpha_n\log f_n\left(X^n|\theta^* + \frac{h}{\sqrt{\alpha_nn}}\right)\right)} \frac{\pi_n(g)}{\pi_n(h)},
    \end{split}
\end{equation*}
where the density of the prior distribution of the transformation $\sqrt{\alpha_nn}(\theta-\theta^*)$ is denoted
\begin{equation}
    \begin{split}
        \pi_n(h) & \equiv (\alpha_nn)^{-p/2}\pi(\theta^* + h/\sqrt{\alpha_nn}).
    \label{eq:scaled_densities2}
    \end{split}
\end{equation}
 By the definition in \eqref{fn+}, with $$s_n(h)=\exp\Big(\alpha_n\log\Big(\frac{f_n(X^n|\theta^* + \frac{h}{\sqrt{\alpha_nn}})}{ f_n(X^n|\theta^*)}\Big)\Big),$$
we have
\begin{equation}\label{fn+s}
\begin{split}
    f_n(g,h)
    &= \Big\{1-\frac{\phi_n(h)\pi_{n,\alpha_n}^{\text{LAN}}(g|X^n)}{\pi_{n,\alpha_n}^{\text{LAN}}(h|X^n)\phi_n(g)}\Big\}^+ = \left\{1-\frac{\phi_n(h)s_n(g)\pi_n(g)}{\phi_n(g)s_n(h)\pi_n(h)}\right\}^+.
\end{split}
\end{equation}
To show that \eqref{fn+s} is well defined on the ball, we justify positivity of $\phi_n(\cdot)$, $s_n(\cdot)$, and $\pi_n(\cdot)$. This is obvious for $\phi_n(\cdot)$, as it is the Gaussian density. 

First we consider the positivity of $s_n(\cdot)$, which follows from \textbf{(A2)}. For any $h\in\bar{B}_0(r)$,
\begin{align}
    s_n(h) 
    \nonumber&= \exp\left(\alpha_n\left[\log f_n\Big(X^n\Big|\theta^*+\frac{h}{\sqrt{\alpha_nn}}\Big)-\log f_n(X^n|\theta^*)\right]\right)  \\
    \nonumber&= \exp\left(\alpha_n \left[\log f_n\Big(X^n\Big|\theta^*+\frac{h}{\sqrt{\alpha_nn}}\Big)-\log f_n(X^n|\theta^*)\right]-R_n(h)\right) \exp\left(R_n(h)\right) \\
    \label{eq:sn(h)_eq}&= \exp\left(\sqrt{\alpha_n}h^\top  V_{\theta^*}\Delta_{n,\theta^*} - \frac{1}{2}h^\top  V_{\theta^*}h\right)\exp\left(R_n(h)\right),
\end{align}
where \eqref{eq:sn(h)_eq} follows from the definition of $R_n(h)$ in Assumption \textbf{(A2)}. To conclude that $s_n(h)$ is positive with high probability on the ball, it suffices to show that neither of the terms inside $\exp(\cdot)$ in \eqref{eq:sn(h)_eq} diverges to $-\infty$ uniformly over $h\in\bar{B}_0(r)$. 

We have that $\Delta_{n,\theta^*} = O_{f_{0,n}}(1)$ by \textbf{(A0)} and $V_{\theta^*}$ is positive-definite by \textbf{(A2)}. From this, we have  $\sup_{h\in\bar{B}_0(r)} |h^\top  V_{\theta^*}\Delta_{n,\theta^*}| = O_{f_{0,n}}(1)$ and therefore, $\sup_{h\in\bar{B}_0(r)} \sqrt{\alpha_n} |h^\top  V_{\theta^*}\Delta_{n,\theta^*}| = o_{f_{0,n}}(1)$. Furthermore, by \textbf{(A2)}, we have that $\sup_{h\in\bar{B}_0(r)}|R_n(h)| = o_{f_{0,n}}(1)$. We conclude that for any $\epsilon > 0$, there exists $t(\epsilon) > 0$ and $N_0\equiv N_0(\epsilon,t(\epsilon),r$) such that the following event
\begin{equation*}
    \mathcal{E} = \Big\{\inf_{h\in\bar{B}_0(r)} \sqrt{\alpha_n}h^\top  V_{\theta^*}\Delta_{n,\theta^*} - \frac{1}{2}h^\top  V_{\theta^*}h > -t \quad \cap \quad  \inf_{h\in\bar{B}_0(r)} R_n(h) > -t\Big\},
\end{equation*}
satisfies, for $\alpha_nn > N_0$,
\begin{equation}\label{eq:sn(h)_prob}
    \begin{split}
        \mathbb{P}_{f_{0,n}}\left(\mathcal{E}\right) \geq 1-{\epsilon}/{2}.
    \end{split}   
\end{equation}

It remains to justify the positivity of $\pi_n(\cdot)$, given in \eqref{eq:scaled_densities2}, on the ball. We verify that  $\theta^* + h/\sqrt{\alpha_nn}$ and $\theta^* + g/\sqrt{\alpha_nn}$ belong to the ball $B_{\theta^*}(\delta)$, for any $g, h\in\bar{B}_0(r)$, where $\pi(\theta)$ is continuous and positive by Assumption \textbf{(A1)}. Indeed, notice that for any $r > 0$, there exists an integer $N_1 \equiv N_1(r,\delta) \equiv \lceil\frac{4r^2}{\delta^2}\rceil > 0$ such that $\theta^* + h/\sqrt{\alpha_nn} \in B_{\theta^*}(\delta)$ whenever $h \in \bar{B}_{0}(r)$ and $\alpha_n n \geq N_1(r, \delta)$. To see this, note that if $\|h\|_2 \leq r$ and $\alpha_n n \geq \frac{4r^2}{\delta^2}$, then 
$$\frac{\|h\|_2}{\sqrt{\alpha_nn}} \leq  \frac{r}{\sqrt{\alpha_nn}} \leq  \frac{r}{\sqrt{\frac{4r^2}{\delta^2}}} = \frac{\delta}{2} < \delta.$$
Then as $\|h\|_2/\sqrt{\alpha_nn} < \delta$ and $\|g\|_2/\sqrt{\alpha_nn} < \delta$, we have that $\theta^* + h/\sqrt{\alpha_nn}$ and $\theta^* + g/\sqrt{\alpha_nn}$ belong to the ball $B_{\theta^*}(\delta)$. Hence, positivity of $\pi_n(\cdot)$ on the ball is immediate when $\alpha_n n > N_1$.

By this conclusion and \eqref{eq:sn(h)_prob}, we have that for $\alpha_nn > \max(N_0, N_1)$, \eqref{fn+s} is well defined over $\bar{B}_0(r)$ with $f_{0,n}$-probability at least $1-\epsilon/2$.\

Notice that we have shown a result that is similar to what we would like to prove in result 1, but for $g,h\in\bar{B}_0(r)$, where $r$ is fixed in $\alpha_nn$. In Step 3, we will show that these same steps can be proven when $g,h\in\bar{B}_0(r_n)$, for a specific growing sequence to complete the proof of result 1.

\noindent \textbf{Step 2: Show the statement in \eqref{eq:lemma1_results} for fixed $r > 0$.}
For any sequence $h_n \in \bar{B}_{0}(r)$, such that $h_n\rightarrow h$, notice that \textbf{(A2)} implies
\begin{align}
    \log(s_n(h_n)) 
    \nonumber&= \alpha_n\Big[\log f_n\big(X^n \big|\theta^*+\frac{h_n}{\sqrt{\alpha_nn}}\big)-\log f_n(X^n|\theta^*)
    \Big] \\
    &= \sqrt{\alpha_n}h_n^\top  V_{\theta^*}\Delta_{n,\theta^*} - \frac{1}{2}h_n^\top  V_{\theta^*}h_n + o_{f_{0,n}}(1) 
    \label{eq:logsn(h)} = - \frac{1}{2}h_n^\top  V_{\theta^*}h_n + o_{f_{0,n}}(1),
\end{align}
where we used $\alpha_n\rightarrow0$ and $V_{\theta^*}\Delta_{n,\theta^*} = O_{f_{0,n}}(1)$ by \textbf{(A0)} as $V_{\theta^*}$ is positive definite by \textbf{(A2)}.

By definition of the density $\phi_n(h_n)$, we have 
$\log(\phi_n(h_n)) = -\frac{p}{2}\log(2\pi) + \frac{1}{2}\log(\det(V_{\theta^*})) - \frac{1}{2}h_n^\top  V_{\theta^*}h_n.$
Thus,
\begin{equation}\label{logLAN}
\log \left[\frac{s_n(h_n)}{\phi_n(h_n)}\right] \hspace{-1pt} =  o_{f_{0,n}}(1) + \frac{p}{2}\log(2\pi) - \frac{1}{2}\log(\det(V_{\theta^*})).\
\end{equation}
In \eqref{eq:sn_phin_ratio}--\eqref{bn+} below, we will implicitly condition on $\mathcal{E}$ and take $\alpha_nn > N_1$ so that each of the objects in \eqref{eq:sn_phin_ratio}--\eqref{eq:pi_ratio} is well defined. By \eqref{eq:logsn(h)}--\eqref{logLAN}, we notice that
\begin{equation}\label{eq:sn_phin_ratio}
\begin{split}
\log\left[\frac{\phi_n(h_n)s_n(g_n)}{\phi_n(g_n)s_n(h_n)}\right] = o_{f_{0,n}}(1).
\end{split}
\end{equation}
By \eqref{eq:scaled_densities2} and continuity of $\pi$, it follows $\pi(\theta^* + g_n/\sqrt{\alpha_nn}), \pi(\theta^* + h_n/\sqrt{\alpha_nn}) \rightarrow \pi(\theta^*)$, 
\begin{equation}\label{eq:pi_ratio}
\begin{split}
\log\left[\frac{\pi_n(g_n)}{\pi_n(h_n)}\right] = \log\left[\frac{\pi(\theta^* + g_n/\sqrt{\alpha_nn})}{\pi(\theta^* + h_n/\sqrt{\alpha_nn})}\right] = o_{f_{0,n}}(1).
\end{split}
\end{equation}
By \eqref{eq:sn_phin_ratio}--\eqref{eq:pi_ratio}, we obtain
\begin{equation}\label{bn+}
\begin{split}
b_n(g_n, h_n) 
&\equiv  \log\Big[\frac{\phi_n(h_n)s_n(g_n)\pi_n(g_n)}{\phi_n(g_n)s_n(h_n)\pi_n(h_n)}\Big] =  o_{f_{0,n}}(1).\
\end{split}
\end{equation}

As $h_n$ and $g_n$ are arbitrary sequences in $\bar{B}_{0}(r)$, result \eqref{bn+} is equivalent to saying that for any fixed $r$, there exists an $N_2 \equiv N_2(r,\epsilon,\eta)$, such that  $\mathbb{P}_{f_{0,n}}(|b_n(g_n, h_n) | > \eta|\mathcal{E}) < \epsilon/2$ for $\alpha_n n > \max(N_1, N_2)$. Note,  we also take $\alpha_nn$ to be larger than $N_1$ so that $f_n(g,h)$ is well defined conditional on $\mathcal{E}$; moreover, when $\alpha_nn> N_0$ we have $P(\mathcal{E}) \geq 1 - \epsilon/2$. Because $f_n(g_n,h_n)$ is a continuous function of $|b_n(g_n, h_n)|$ and $g_n$ and $h_n$ are assumed to be convergent sequences in $\bar{B}_0(r)$, by the continuous mapping theorem, there exists an integer $\tilde{N}_2 \equiv \tilde{N}_2(r, \epsilon, \eta)$ such that for $\alpha_n n >  \max(N_1, \tilde{N}_2$),
\begin{equation}\label{eq:fn_prob_bound_cond}
\begin{split}
    &\mathbb{P}_{f_{0,n}}\Big(\sup_{g,h\in \bar{B}_{0}(r)}f_n(g,h) > \eta \, \Big| \, \mathcal{E}\Big)
    < \frac{\epsilon}{2}.
    \end{split}
\end{equation}
Then by \eqref{eq:sn(h)_prob} and \eqref{eq:fn_prob_bound_cond}, we have for $\alpha_nn > \max(N_0, N_1, \tilde{N}_2)$,
\begin{equation}\label{eq:fn_prob_bound}
\begin{split}
    \mathbb{P}_{f_{0,n}}\Big(\sup_{g,h\in \bar{B}_{0}(r)}f_n(g,h) > \eta\Big)
&\leq\mathbb{P}_{f_{0,n}}\Big(\sup_{g,h\in \bar{B}_{0}(r)}f_n(g,h) > \eta \Big|\mathcal{E}\Big) + \mathbb{P}_{f_{0,n}}(\mathcal{E}^\mathsf{c}) 
    < \frac{\epsilon}{2} + \frac{\epsilon}{2} = \epsilon.
    \end{split}
\end{equation}
Now we have shown a result similar to the second result of the lemma, but for the supremum taken over $\bar{B}_0(r)$, where the radius is fixed in $r$. In step 3, we show that \eqref{eq:fn_prob_bound} holds for our constructed $r_n$.\\

\noindent \textbf{Step 3: Show the statements of Lemma \ref{lem:f(g,h)} for $r_n \rightarrow \infty$.} Recall the definitions of $N_0(\epsilon,t(\epsilon),r)$ from Step 1, $N_1(r, \delta) = \lceil \frac{4r^2}{\delta^2} \rceil$ also from Step 1, and $\tilde{N}_2(r, \epsilon, \eta)$ from Step 2. Let $N^*(\epsilon, \eta) = \max\{N_0(\epsilon,t(\epsilon),1), N_1(1, \delta), \tilde{N}_2(1,\epsilon,\eta)\}$ and for all $\alpha_n n > N^*(\epsilon, \eta)$, let
\[
r_n = \max\left\{r\in\mathbb{R} \, \big| \, r\leq\delta\sqrt{\alpha_nn}/2 \quad \textrm{ and } \quad \alpha_n n > \max\{N_0(\epsilon, t(\epsilon), r), \tilde{N}_2(r,\epsilon,\eta)) \right\}.\
\]

For $\alpha_n n \leq N^*(\epsilon,\eta)$, we can define $r_n$ arbitrarily,  e.g, $r_n=1$.\
We need to check: \\
\textbf{(A) Check that $r_n$ is well defined (i.e.,\, there exists an $r$ that satisfies the definition of $r_n$):} Indeed, $r=1$ satisfies the $r_n$ definition for $\alpha_n n > N^*(\epsilon, \eta)$:
\[
    \frac{\delta\sqrt{\alpha_nn}}{2} > \frac{\delta\sqrt{N^*(\epsilon, \eta)}}{2} \geq \frac{\delta\sqrt{4/\delta^2}}{2} = 1 = r,
\]
where we have used that $N^*(\epsilon, \eta) \geq N_1(1, \delta) \geq {4}/{\delta^2}$. Next, notice that
$\alpha_n n > N^*(\epsilon, \eta) \geq \max\{N_0(\epsilon, t(\epsilon), 1), \Tilde{N}_2(1, \epsilon, \eta)\}  = \max\{N_0(\epsilon, t(\epsilon), r), \Tilde{N}_2(r, \epsilon, \eta)\}$ by the definition of $N^*(\epsilon,\eta)$.

\noindent \textbf{(B) Check that in fact $r_n \rightarrow \infty$:} Indeed, 
\begin{align*}    
    &r_n 
    = \max\Big\{r\in\mathbb{R} \, \big| \, r\leq\frac{\delta\sqrt{\alpha_nn}}{2} \quad \textrm{ and } \quad \alpha_n n > \max\{N_0(\epsilon, t(\epsilon), r), \tilde{N}_2(r,\epsilon,\eta))  \Big\} \\
    &= \min\Big\{ \max\Big\{r\in\mathbb{R}  \big|  r \leq \frac{\delta\sqrt{\alpha_nn}}{2} \Big\}, \max\Big\{r\in\mathbb{R}  \big| \alpha_n n > \max\{N_0(\epsilon, t(\epsilon), r), \tilde{N}_2(r,\epsilon,\eta))  \Big\}\Big\},
\end{align*}
Clearly $\delta\sqrt{\alpha_nn}/2\rightarrow\infty$, so $r_n\rightarrow\infty$ if $\max\big\{r\in\mathbb{R}|\, \alpha_n n > \max\{N_0(\epsilon, t(\epsilon), r), \tilde{N}_2(r,\epsilon,\eta)) \big\} \rightarrow \infty$. Notice that in order for $\max\big\{r\in\mathbb{R}| \, \alpha_n n > \max\{N_0(\epsilon, t(\epsilon), r), \tilde{N}_2(r,\epsilon,\eta))  \big\} \rightarrow \infty$ as $n$ grows, we need that both $N_0(\epsilon, t(\epsilon), r)$ and $\tilde{N}_2(r,\epsilon,\eta)$ remain finite as $r$ grows. Indeed, if the result is not true, there would exist an $r_{\max}$ such that eith $N_0(\epsilon, t(\epsilon), r) = \infty$ or $\tilde{N}_2(r, \epsilon, \eta) = \infty$ for $r > r_{\max}$. However, Step 2 holds for any finite $r$, e.g. $r = r_{\max} + 1$, and, in particular, the results of Step 2 imply that $N_0$ and $\tilde{N}_2$ are finite so this is a contradiction.

\noindent \textbf{(C) Show that for all $\alpha_n n > N^*(\epsilon,\eta)$ we have $\mathbb{P}_{f_{0,n}}(\sup_{g,h\in \bar{B}_{0}(r_n)}f_n(g,h) > \eta) < \epsilon. $}
Doing so amounts to repeating Step 1 and Step 2 for $g,h\in\bar{B}_0(r_n)$: \\

Step 1: First, we check that $f_n(g,h)$ is well defined on $\bar{B}_0(r_n)$. We verify \eqref{eq:sn(h)_prob} for $\bar{B}_0(r_n)$, which implies that $s_n(\cdot)$ is positive on $\bar{B}_0(r_n)$. To do so, we notice that for each $\alpha_nn$, the set $\bar{B}_{0}(r_n)$ is compact. That is, $\bar{B}_0(r_n) = \bar{B}_0(r')$ for some $r' > 0$. Hence, by \textbf{(A2)}, for all $\epsilon > 0$, there exists $t > 0$ and $N_0' \equiv N_0'(\epsilon,t,r')$ such that \eqref{eq:sn(h)_prob} holds for $\alpha_nn > N_0'$ on $\bar{B}_0(r')$. Second, we check that $\pi_n(\cdot)$ is positive for all $g,h\in\bar{B}_0(r_n)$. We recognize that $\theta^* + h/\sqrt{\alpha_nn} \in B_{\theta^*}(\delta)$ whenever $h\in \bar{B}_{0}(r_n)$ since $\|h\|_2/\sqrt{\alpha_nn} \leq r_n/\sqrt{\alpha_nn} \leq \delta/2$, which follows from the definition of $r_n$. This guarantees $f_n(g,h)$ is well defined as discussed in Step 1.\\ 

Step 2: By the $r_n$ definition, $\alpha_n n > \max\{N_0(\epsilon, t(\epsilon), r_n), \tilde{N}_2(r_n,\epsilon,\eta)\}$ for all $\alpha_n n > N^*(\epsilon,\eta)$. Moreover, the condition that $\alpha_n n > N_1(r_n, \delta) = \lceil \frac{4r_n^2}{\delta^2} \rceil$ in \eqref{eq:fn_prob_bound} is to guarantee that $\theta^* + h/\sqrt{\alpha_nn} \in B_{\theta^*}(\delta)$ whenever $h\in \bar{B}_{0}(r_n)$, which we have justified above. Hence, the bound in \eqref{eq:fn_prob_bound} holds and this completes the proof.

\end{proof}

\begin{Lemma}{\cite[Lemma 1]{ray_asymptotics_2023}}\label{lem:TVbound}
    Consider densities $\varphi$ and $\psi$ that are positive on $K$, a given compact set $K \subset \mathbb{R}^p$. Then, for any function, $s(h)$, that is nonnegative on all of $\mathbb{R}^p$,
\begin{align*}
    \int_{\mathbb{R}^p}s(h)|\varphi(h) - \psi(h)|dh  &\leq \Big[ \sup_{g,g'\in K}\tilde{f}(g,g')\Big]\int_{\mathbb{R}^p}s(h)(\psi(h) + \varphi(h) )dh  \\
    &\qquad + \int_{\mathbb{R}^p\setminus K}s(h) (\psi(h) + \varphi(h) )dh,
\end{align*}
where
\begin{equation*}
    \tilde{f}(g,g') \equiv \left\{ 1 - \frac{\varphi(g')\psi(g)}{\psi(g')\varphi(g)}\right\}^+.
\end{equation*}
\end{Lemma}
\begin{Remark}
   Unlike \cite[Lemma 1]{ray_asymptotics_2023}, Lemma~\ref{lem:TVbound} does not use the quantity $\tilde{f}^-(g,h)$ since, by \cite[Eqs.\ (21)-(22)]{ray_asymptotics_2023},  we see that $\sup_{g,h\in K}\tilde{f}^+(g,h) = \sup_{g,h\in K}\tilde{f}^-(g,h)$.
\end{Remark}

\begin{Lemma}\label{lem:gauss_int_bound2}
Suppose Assumption \textbf{(A0)} holds. Then for sequences $\alpha_n\to0$ and $\alpha_nn\to\infty$, there exists a constant $M$ such that
    \begin{equation}\label{eq:gauss_int_bound_result}
    \begin{split}
        \lim_{n\to\infty}&\mathbb{P}_{f_{0,n}}\Big(\int_{\mathbb{R}^p} \|h\|_2^k \phi_n(h) dh > M \Big) \\
        &= \lim_{n\to\infty}\mathbb{P}_{f_{0,n}}\Big(\int_{\mathbb{R}^p} (\alpha_nn)^{-\frac{p}{2}}  \|h\|_2^k \phi\Big(h  \big|  \sqrt{\alpha_nn}(\hat{\theta}-\theta^*), V_{\theta^*}^{-1}\Big) dh > M\Big) = 0.
    \end{split}
    \end{equation}
\end{Lemma}
\begin{proof}
        We notice that by the change of variable $h = \sqrt{\alpha_nn}(\theta + \hat{\theta} - \theta^*)$,
    \begin{equation}
\begin{split}
    \int_{\mathbb{R}^p}  \|h\|_2^k \phi_n(h) dh
    &= \int_{\mathbb{R}^p} (\alpha_nn)^{-\frac{p}{2}}  \|h\|_2^k \phi\Big(h \, \big| \, \sqrt{\alpha_nn}(\hat{\theta}-\theta^*), V_{\theta^*}^{-1}\Big) dh \\
   &= \int_{\mathbb{R}^p} (\alpha_nn)^{\frac{k}{2}} \|\theta - \hat{\theta} + \theta^*\|_2^k\phi\Big(\theta  \big|  0, \frac{V_{\theta^*}^{-1}}{\alpha_nn}\Big)d\theta \label{eq:t2_eq1}  \\
   &= (\alpha_nn)^{\frac{k}{2}} \mathbb{E}_{\mathcal{N}(0, \frac{V_{\theta^*}^{-1}}{\alpha_nn} )} [ \|\theta + \hat{\theta} - \theta^*\|_2^k ].
    \end{split}
    \end{equation}
    Throughout the remainder of the proof, we will use $\mathbb{E}$ rather than $\mathbb{E}_{\mathcal{N}(0, V_{\theta^*}^{-1}/(\alpha_nn))}$ to denote expectations taken with respect to $\theta \sim \mathcal{N}(0, V_{\theta^*}^{-1}/(\alpha_nn))$. It is not difficult to show that $(a + b)^{p} \leq 2^{p} (a^{p} + b^{p})$ for $p \geq 0$. Considering the right side of \eqref{eq:t2_eq1}, notice
    $\mathbb{E}  [ \|\theta + \hat{\theta} - \theta^*\|_2^k ]
        \leq \mathbb{E}  [ (\|\theta\|_2 + \|\hat{\theta} - \theta^*\|_2)^k]
        \leq 2^k \mathbb{E}  [ \|\theta\|_2^k] + 2^k\|\hat{\theta} - \theta^*\|_2^k.$
    We have shown that 
    \begin{equation}
        \label{eq:t2_eq2a}
        \int_{\mathbb{R}^p} \|h\|_2^k \phi_n(h) dh \leq 2^k  (\alpha_nn)^{\frac{k}{2}}  \mathbb{E}  [ \|\theta\|_2^k] + (2 \sqrt{\alpha_n})^{k} \|\sqrt{n}(\hat{\theta} - \theta^*)\|_2^k .
    \end{equation}
    Notice that $\|\theta\|_2^k \leq p^{\frac{k}{2} - 1} \sum_{i=1}^p |\theta_i|^{k}$ by H{\"o}lder's inequality. Next, since $\theta_i \sim \mathcal{N}(0, \frac{1}{\alpha_nn}[V_{\theta^*}^{-1}]_{ii})$,
    \[\mathbb{E}[|\theta_i|^{k}] = \frac{2^{\frac{k}{2}}}{\sqrt{\pi}} \Gamma \left(\frac{k+1}{2}\right) \left[\frac{[V_{\theta^*}^{-1}]_{ii}}{\alpha_nn}\right]^\frac{k}{2},\]
    for each $i \in [p]$.
    Therefore, the first term on the right side of \eqref{eq:t2_eq2a} is upper bounded by a constant: 
    \begin{align}
        2^k  (\alpha_nn)^{\frac{k}{2}}  \mathbb{E}  [ \|\theta\|_2^k ]
        &\leq  2^k  (\alpha_nn)^{\frac{k}{2}}p^{\frac{k}{2}-1}  \frac{2^\frac{k}{2}}{\sqrt{\pi}}  \Gamma \Big(\frac{k+1}{2}\Big)  \sum_{i=1}^p  \Big[\frac{[V_{\theta^*}^{-1}]_{ii}}{\alpha_nn}\Big]^{\frac{k}{2}} 
         \nonumber \\
        &=\frac{2^{\frac{3k}{2}}p^{\frac{k}{2}-1}}{\sqrt{\pi}}  \Gamma \Big(\frac{k+1}{2}\Big) \sum_{i=1}^p  \left[[V_{\theta^*}^{-1}]_{ii}\right]^{\frac{k}{2}}. \label{eq:t2_eq2b}
    \end{align}
    Hence, to show the result in \eqref{eq:gauss_int_bound_result}, it suffices to show that \eqref{eq:t2_eq2a} is bounded in $f_{0,n}$-probability. We do so as follows. By Assumption \textbf{(A0)}, we have that $\sqrt{n}(\hat{\theta} - \theta^*)$ converges in distribution to $\mathcal{N}(0, V)$ for some positive definite matrix, $V$. Let $Q\Lambda Q^\top$ denote the spectral decomposition of $V$. Then we have that $\Lambda^{-\frac{1}{2}}\sqrt{n}(\hat{\theta} - \theta^*)$ converges in distribution to $\mathcal{N}(0, I)$. By the continuous mapping theorem, $\|\Lambda^{-\frac{1}{2}}\sqrt{n}(\hat{\theta} - \theta^*)\|_2^{k} = (\|\Lambda^{-\frac{1}{2}}\sqrt{n}(\hat{\theta} - \theta^*)\|_2^{2})^{\frac{k}{2}}$
    converges in distribution to $(\chi^2_p)^{\frac{k}{2}}$ and, hence, is bounded in $f_{0,n}$-probability. Denote by $\lambda_1,\ldots,\lambda_p$ the diagonal elements of $\Lambda$ (i.e., the eigenvalues of $V$). Notice that
    \begin{equation*}
        \begin{split}
            \|\Lambda^{-\frac{1}{2}}\sqrt{n}(\hat{\theta} - \theta^*)\|_2^{k} 
        = \Big(\sum_{j=1}^p\frac{[\sqrt{n}(\hat{\theta}_j - \theta^*)_j]^2}{\lambda_j} \Big)^{\frac{k}{2}} &\geq \frac{1}{\lambda_{\text{max}}^{k/2}} \Big(\sum_{j=1}^p[\sqrt{n}(\hat{\theta}_j - \theta^*)_j]^2\Big)^{\frac{k}{2}} \\
            &= \frac{\|\sqrt{n}(\hat{\theta} - \theta^*)\|_2^{k}}{\lambda_{\text{max}}^{k/2}},
    \end{split}
    \end{equation*}
    from which it follows $ \|\sqrt{n}(\hat{\theta} - \theta^*)\|_2^{k} \leq (\sqrt{\lambda_{\text{max}}})^k\|\Lambda^{-\frac{1}{2}}\sqrt{n}(\hat{\theta} - \theta^*)\|_2^{k}$. 
   From the fact that $\|\Lambda^{-\frac{1}{2}}\sqrt{n}(\hat{\theta} - \theta^*)\|_2^{k}$ is bounded in $f_{0,n}$-probability and the assumption that $\alpha_n\to0$ we have the following:
    for any $\epsilon > 0$ and \emph{for any} $M_0 > 0$, there exists an integer $N_1 \equiv N(M_0, \epsilon, k)$ such that whenever $n > N_1$,
    \begin{equation}\label{eq:gauss_int_bound_term2}
        \mathbb{P}_{f_{0,n}}\left((2 \sqrt{\alpha_n})^{k} \|\sqrt{n}(\hat{\theta} - \theta^*)\|_2^k > M_0\right)
        \leq \mathbb{P}_{f_{0,n}}\left((2 \sqrt{\lambda_{\text{max}}\alpha_n})^{k} \|\Lambda^{-\frac{1}{2}}\sqrt{n}(\hat{\theta} - \theta^*)\|_2^{k} > M_0\right) 
        \leq \epsilon.\
    \end{equation}
    Using this and the bound in \eqref{eq:t2_eq2b}, taking $M$ to be a constant larger than the constant in \eqref{eq:t2_eq2b}, we have that whenever $n > N_1$,
    \begin{equation}\label{eq:gauss_int_bound_result2}
        \begin{split}
            \mathbb{P}_{f_{0,n}}&\Big(\int_{\mathbb{R}^p} \|h\|_2^k \phi_n(h) dh > M \Big) \\
            &\leq \mathbb{P}_{f_{0,n}}\Big(\frac{2^{\frac{3k}{2}}p^{\frac{k}{2}-1}}{\sqrt{\pi}}  \Gamma \Big(\frac{k+1}{2}\Big)  \sum_{i=1}^p  \left[[V_{\theta^*}^{-1}]_{ii}\right]^{\frac{k}{2}} + (2 \sqrt{\alpha_n})^{k} \|\sqrt{n}(\hat{\theta} - \theta^*)\|_2^k > M\Big) \\
            &= \mathbb{P}_{f_{0,n}}\Big((2 \sqrt{\alpha_n})^{k} \|\sqrt{n}(\hat{\theta} - \theta^*)\|_2^k > M - \frac{2^{\frac{3k}{2}}p^{\frac{k}{2}-1}}{\sqrt{\pi}}  \Gamma \Big(\frac{k+1}{2}\Big)  \sum_{i=1}^p  \left[[V_{\theta^*}^{-1}]_{ii}\right]^{\frac{k}{2}}\Big)
            < \epsilon.
        \end{split}
    \end{equation}
\end{proof}

\begin{Lemma}\label{MomentKc2}
    Suppose a random variable $Z \overset{d}{\equiv} Y|X^n$ has a density $f_Z(\cdot)$ on $\mathbb{R}^p$. For some $k \geq 0$, assume there exists a $\gamma > 0$ such that $\mathbb{E}[\|Z\|_2^{k(1+\gamma)}] = O_{f_{0,n}}(1)$. Then, for any $\epsilon > 0$, any $M > 0$, and any $r_n \rightarrow \infty$, there exists an integer $N > 0$ such that $\mathbb{P}_{f_{0,n}}(\int_{\bar{B}_0(r_n)^\mathsf{c}}\|z\|_2^{k} \, f_Z(z) dz > M) < \epsilon$ for all $\alpha_nn > N$.
\end{Lemma}

\begin{proof}
The proof of Lemma \ref{MomentKc2} follows almost verbatim from the proof of \cite[Lemma 2]{ray_asymptotics_2023}. Let $f_{\|Z\|_2^k}(\cdot)$ denote the density function of $\|Z\|_2^k$. Define the random variable 
\begin{equation*}
\begin{split}
    V 
    &\equiv \|Z\|_2^k\mathbbm{1}\{ \|Z\|_2 > r_n\}
    = \begin{cases}
    0 &\textrm{if } \|Z\|_2 \leq r_n, \\
    \|Z\|_2^k &\textrm{if } \|Z\|_2 > r_n.
    \end{cases}
\end{split}
\end{equation*}
For the remainder of the proof, we will take $\alpha_nn > N_0$, where $N_0$ is an integer large enough such that $r_n > 0$ for all $\alpha_nn > N_0$. Thus, the density of $V$ is given by
\begin{equation*}
    f_V(v) =
    \begin{cases}
    \mathbb{P}(\|Z\|_2 \leq r_n) &\textrm{ if } v = 0, \\
    f_{\|Z\|_2^k}(v) &\textrm{ if } v > r_n^{k}.
    \end{cases}
\end{equation*}
 Hence, we can represent the expectation of $V$ as 
\begin{equation}\label{markov0}
    \begin{split}
        \mathbb{E}[V]
        &= \int_0^{r_n^k} \mathbb{P}(V > t)dt + \int_{r_n^k}^\infty \mathbb{P}(V > t)dt.
    \end{split}
\end{equation}
Using the definition of $V$ and its density function, we study \eqref{markov0}.
Consider the first term on the right side,
\begin{equation}\label{markov}
    \begin{split}
        \int_0^{r_n^k}    \mathbb{P}(V > t)dt  = \int_0^{r_n^k} \int_t^\infty f_V(v)dvdt
        &\overset{(a)}{=} \int_0^{r_n^k} \mathbb{P}(\|Z\|_2^k > r_n^k)dt  \\
        &= r_n^k\mathbb{P}(\|Z\|_2^k > r_n^k) \overset{(b)}{\leq} \frac{ \mathbb{E}[\|Z\|_2^{k(1+\gamma)}]}{r_n^{k\gamma}},
    \end{split}
\end{equation}
where the $\gamma$ is such that $\mathbb{E}[\|Z\|_2^{k(1+\gamma)}] = O_{f_{0,n}}(1)$. 
Step $(a)$ in \eqref{markov} follows because, for $0 < t < r_n^k$, we have
$
\int_t^\infty f_V(v)dv = \int_{r_n^k}^\infty f_{\|Z\|_2^k}(v) dv = \mathbb{P}(\|Z\|_2^k > r_n^k),$
and step $(b)$ by Markov's inequality. For the second term on the right side of \eqref{markov0}, we again use Markov's inequality:
\begin{equation}\label{markov1}
    \begin{split}
       \int_{r_n^k}^\infty \mathbb{P}(V > t)dt &\leq   \int_{r_n^k}^\infty \frac{\mathbb{E}[V^{1+\gamma}]}{t^{1+\gamma}} dt
       \overset{(c)}{\leq} \int_{r_n^k}^\infty    \frac{\mathbb{E}[\|Z\|_2^{k(1+\gamma)}]}{t^{1+\gamma}} dt  
       =   \frac{\mathbb{E}[\|Z\|_2^{k(1+\gamma)}]}{\gamma r_n^{k\gamma}}.
    \end{split}
\end{equation}
Step $(c)$ in \eqref{markov1} follows because $V^{1 + \gamma}  = (\|Z\|_2^k\mathbbm{1}\{\|Z\|_2 > r_n\})^{1 + \gamma}  \leq \|Z\|_2^{k(1+\gamma)}$.

Plugging \eqref{markov} and \eqref{markov1} into \eqref{markov0}, we find
\begin{equation}\label{markov_new}
    \begin{split}
        \int_{\bar{B}_0(r_n)^\mathsf{c}}\|z\|_2^{k} \, f_Z(z) dz = 
        \mathbb{E}[V]
        &\leq \frac{\gamma + 1}{\gamma r_n^{k\gamma}}  \mathbb{E}\left[\|Z\|_2^{k(1+\gamma)}\right]. 
    \end{split}
\end{equation}
By \eqref{markov_new}, we have that
$
    \mathbb{P}_{f_{0,n}}(\int_{\bar{B}_0(r_n)^\mathsf{c}}\|z\|_2^{k} \, f_Z(z) dz > M)
    \leq \mathbb{P}_{f_{0,n}}(\frac{\gamma + 1}{\gamma r_n^{k\gamma}}  \mathbb{E}[\|Z\|_2^{k(1+\gamma)}] > M) 
    \leq \mathbb{P}_{f_{0,n}}(\frac{\gamma + 1}{\gamma r_n^{k\gamma}}  \mathbb{E}[\|Z\|_2^{k(1+\gamma)}] > M)$.
Since $r_n\to\infty$ and $\mathbb{E}[\|Z\|_2^{k(1+\gamma)}] = O_{f_{0,n}}(1)$ by assumption, for all $\epsilon > 0$, there exists $N\equiv N(\epsilon, k, \gamma) > N_0$ such that the right hand side of the above is upper bounded by $\epsilon$ whenever $\alpha_nn > N$. Hence, we have established the desired result. 
\end{proof}

\subsection{Auxiliary lemmas} 

\begin{Lemma}\label{lem:l1_l2_tensor_bound}
   
    Given a vector $v\in\mathbb{R}^p$, we have $
        \|v^{\otimes k}\|_1 \leq p^{\frac{k}{2}} \left\|v\right\|_2^k$ for $\|v^{\otimes k}\|_1$ defined in \eqref{eq:tensor_norm_def}.
\end{Lemma}
\begin{proof}
First, recall that $\left\|v\right\|_1 \leq \sqrt{p}\|v\|_2$, which follows from the Cauchy-Schwarz inequality. 
Now, recalling the definition of $\left\|v^{\otimes k}\right\|_1$ in \eqref{eq:tensor_norm_def}, we obtain 
$\|v^{\otimes k}\|_1 
        = \sum_{i_1=1}^{p} \cdots \sum_{i_k=1}^{p}|v_{i_1}\times\cdots\times v_{i_k}| \leq \sum_{i_1=1}^{p}|v_{i_1}|  \cdots \sum_{i_k=1}^{p}|v_{i_k}|= \|v\|_1^k \leq p^{\frac{k}{2}} \|v\|_2^k.$
\end{proof}

It is well known that the Gaussian concentration inequality of \cite{tsirelson:ibragimov:sudakov1976} leads to the following result (see also \cite{boucheron:lugosi:massart2013}).

\begin{Lemma}\label{lem:Gaussian_concentration}
 Let $X\sim\mathcal{N}(0,\Sigma)$. Then, $\mathbb{P}(\|X\|_2 > s) \leq \exp\big(\frac{-(s - \sqrt{\text{tr}(\Sigma)})^2}{2\|\Sigma\|_{\text{op}}}\big)$ for all $s > 0$.
\end{Lemma}

The following convexity lemma is a direct consequence of the form of the dual of $W_p^p(\mu,\nu)$ \cite[Theorem 1.20]{chewi:nilesweed:rigollet2024}

\begin{Lemma}\label{lem:Wasserstein_convex}
    Let $q\in[0,1]$. For probability measures $\mu_q = q\mu_1 + (1-q)\mu_2$ and $\nu_q = q\nu_1 + (1-q)\nu_2$, we have $d_{\text{W}_p}(\mu_q,\nu_q)^p  \leq qd_{\text{W}_p}(\mu_1,\nu_1)^p + (1-q)d_{\text{W}_p}(\mu_2,\nu_2)^p.$
\end{Lemma}

\begin{Lemma}\cite[Corr.\ 2.51]{folland1999real}\label{corr:spherical}
    If $f$ is a nonnegative or integrable, Lebesgue measurable function on $\mathbb{R}^p$ such that $f(x) = g(|x|)$ for a function $g$ on $(0,\infty)$, then $\int_{\mathbb{R}^p} f(x)dx = \text{vol}(\mathbb{S}^{p-1})\int_0^\infty \rho^{p-1}g(\rho)d\rho.$
\end{Lemma}

\section{Verifying conditions of the main theorems}

\subsection{Sufficient conditions for Assumption \textbf{(A2)}}\label{app:suff_cond_A2}
The conditions in Proposition \ref{prop:A2_suff} are the usual regularity conditions that also imply $\sqrt{n}$-stochastic LAN \cite[Chapter 7]{van_der_vaart_asymptotic_1998} and can be relaxed for \text{i.i.d.}\ data: see \cite[Lemmas 2.1-2.2]{kleijn_bernstein-von-mises_2012}, which provides sufficient conditions to ensure the existence of $\Delta_{n,\theta^*}$ and $V_{\theta^*}$ to satisfiy \eqref{eq:first_deriv_cond}.

\begin{Proposition}\label{prop:A2_suff}
    Suppose there exists an $r_0 > 0$ such that on $\bar{B}_{\theta^*}(r_0)$ the following hold:
    \begin{enumerate}
        \item The log-likelihood, $\log f_n(X^n|\theta)$, is three-times continuously differentiable in $\theta$.
        \item Let $\frac{1}{n}\nabla^3\log f_n(X^n|\theta)\in\mathbb{R}^{p\times p\times p}$ be a third-order tensor of partial derivatives of $\log f_n(X^n|\theta)$. There exists a function $\tilde{M}:\mathcal{X}^n\rightarrow\mathbb{R}$, where $\mathbb{E}_{f_{0,n}}|\frac{1}{n}\tilde{M}(X^n)| < \infty$, such that the entries of $\nabla^3\log f_n(X^n|\theta)$ satisfy 
        $| [\nabla^3\log f_n(X^n\big|\theta)]_{i,j,k}| = |\frac{\partial^3}{\partial\theta_i\partial\theta_j\partial\theta_k}\log f_n(X^n|\theta)| < \tilde{M}(X^n).$
        
        \item Let $\nabla\log f_n(X^n|\theta)$ and $\nabla^2\log f_n(X^n|\theta)$ denote the gradient and Hessian of $\log f_n(X^n|\theta)$, respectively. Assume that in $f_{0,n}$-probability,
        \begin{equation}\label{eq:first_deriv_cond}
            \frac{1}{\sqrt{n}}\nabla\log f_n(X^n|\theta^*) - V_{\theta^*}\Delta_{n,\theta^*} \rightarrow 0,\quad \text{ and } \quad 
            \frac{1}{n}\nabla^2\log f_n(X^n|\theta^*) +V_{\theta^*} \rightarrow 0,
        \end{equation}
        for some vector $\Delta_{n,\theta^*}\in\mathbb{R}^p$ and a positive definite matrix $V_{\theta^*}\in\mathbb{R}^{p\times p}$.\
    \end{enumerate}
    Under the above conditions, for $R_n(h)$ defined in \eqref{eq:Rn(h)_def}, for all $M > 0$, $r > 0$, and $\epsilon > 0$, there exists an integer $N(M, r, \epsilon)$ such that for all $\alpha_nn > N$,
    \begin{equation}\label{eq:A2_condition}
        \mathbb{P}_{f_{0,n}}\Big(\sup_{h\in\bar{B}_0(r)}|R_n(h)| > M\Big) < \epsilon.\
    \end{equation}
\end{Proposition}

\begin{proof}
    
The proof  has three steps. First, we employ condition 3 in a second-order Taylor expansion of $\frac{1}{n}\log f_n(X^n|\theta^*+{h}/{\sqrt{\alpha_nn}})$ around $\frac{1}{n}\log f_n(X^n|\theta^*)$. Second, we study the convergence of the remainder of the Taylor expansion on sets of the form $\bar{B}_0(r)$ for any $r > 0$. Finally, we combine these results to obtain the condition of \textbf{(A2)}, as specified in \eqref{eq:A2_condition}.\\

\noindent\textbf{Step 1: Taylor expansion of the log-likelihood.} Let $g(t) \equiv  \frac{1}{n}\log f_n(X^n|u(t))$, where $u(t) \equiv \theta^*+{th}/{\sqrt{\alpha_nn}}$. A second-order Taylor expansion of $g(1)$ around $g(0)$ shows that
\begin{align}
    g(1) 
    \nonumber&= g(0) + g'(0)(1-0) + \frac{1}{2}g''(0)(1-0)^2 + \frac{1}{2}\int_0^1 (1-t)^2g'''(t)dt\\
    \label{eq:g_expansion}&= g(0) + g'(0) + \frac{1}{2}g''(0) + \frac{1}{2}\int_0^1 (1-t)^2g'''(t)dt.
\end{align}
Note that an expansion of $g(1)$ around $g(0)$ is equivalent to an expansion of $\frac{1}{n}\log f_n(X^n|\theta^*+{h}/{\sqrt{\alpha_nn}})$ around $\frac{1}{n}\log f_n(X^n|\theta^*)$ since
\begin{equation}\label{eq:g1g0}
    g(1) = \frac{1}{n}\log f_n\Big(X^n\Big|\theta^*+\frac{h}{\sqrt{\alpha_nn}}\Big)~~\textrm{and}~~g(0) = \frac{1}{n}\log f_n(X^n|\theta^*).
\end{equation}
Furthermore,
\begin{equation}\label{eq:g_derivs}
    \begin{split}
        g'(0) &= \frac{h^\top}{\sqrt{\alpha_nn}} \Big[\frac{1}{n} \nabla \log f_n(X^n|\theta^*)\Big], \quad \quad \quad g''(0) = \frac{h^\top}{\sqrt{\alpha_nn}} \Big[\frac{1}{n}\nabla ^2\log f_n(X^n|\theta^*)\Big] \frac{h}{\sqrt{\alpha_nn}}, \\
        g'''(t) &= \biggl<\frac{1}{n}\nabla^3\log f_n\Big(X^n\big|\theta^* + \frac{th}{\sqrt{\alpha_nn}}\Big), \Big(\frac{h}{\sqrt{\alpha_nn}}\Big)^{\otimes 3}\biggr>,
    \end{split}
\end{equation}
where we represent $g'''(t)$ in the tensor product notation defined in \eqref{eq:tensor_inner_product_def}. Substituting \eqref{eq:g1g0} and \eqref{eq:g_derivs} into \eqref{eq:g_expansion} yields
\begin{equation}
\begin{split}
\label{eq:Taylor_approx}
\frac{1}{n}\log f_n\Big(X^n\Big|\theta^*+\frac{h}{\sqrt{\alpha_nn}}\Big) &= \frac{1}{n}\log f_n(X^n|\theta^*) +\frac{h^\top}{\sqrt{\alpha_nn}} \Big[\frac{1}{n} \nabla \log f_n(X^n|\theta^*)\Big] \\
&\qquad + \frac{1}{2} \frac{h^\top}{\sqrt{\alpha_nn}} \Big[\frac{1}{n}\nabla ^2\log f_n(X^n|\theta^*)\Big] \frac{h}{\sqrt{\alpha_nn}} + \textrm{rem}_n(h),
\end{split}
\end{equation}
where $\textrm{rem}_n(h)$ is the remainder of the Taylor approximation given by 
\begin{equation*}
    \begin{split}
    \textrm{rem}_n(h) &= \frac{1}{2}\int_0^{1} \hspace{-3.2pt} (1- t)^2 g'''(t) dt = \frac{1}{2(\alpha_n n)^{\frac{3}{2}}}\int_0^{1} \hspace{-3.2pt} (1- t)^2 \Bigl<\frac{1}{n}\nabla^3\log f_n\Big(X^n\big|\theta^* + \frac{th}{\sqrt{\alpha_nn}}\Big), h^{\otimes 3}\Bigr> dt.
    \end{split}
\end{equation*}
Then by the definition of $ R_n(h)$ in Assumption \textbf{(A2)} and the Taylor approximation in \eqref{eq:Taylor_approx},
\begin{equation}
    \begin{split}
    &|R_n(h)| =\Big|\alpha_n\Big[\log f_n\big(X^n\big|\theta^*+\frac{h}{\sqrt{\alpha_nn}}\big)-\log f_n(X^n|\theta^*)\Big] -\sqrt{\alpha_n}h^\top V_{\theta^*}\Delta_{n,\theta^*}+\frac{1}{2}h^\top V_{\theta^*}h\Big| \\
    &=\Big|\alpha_nn\Big(\frac{h^\top}{\sqrt{\alpha_nn}}\Big[\frac{1}{n}\nabla\log f_n(X^n|\theta^*)\Big] + \frac{1}{2}\frac{h^{\top}}{\sqrt{\alpha_nn}}\Big[\frac{1}{n}\nabla^2\log f_n(X^n|\theta^*)\Big]\frac{h}{\sqrt{\alpha_nn}} + \textrm{rem}_n(h)\Big) \\
    &\qquad-\sqrt{\alpha_n}h^\top V_{\theta^*}\Delta_{n,\theta^*}+\frac{1}{2}h^\top V_{\theta^*}h\Big| \\
    \label{eq:Rn(h)_triangle}&\leq \Big|\sqrt{\alpha_n}h^\top\Big[\frac{1}{\sqrt{n}} \nabla\log f_n(X^n|\theta^*) - V_{\theta^*}\Delta_{n,\theta^*}\Big] + \frac{h^\top}{2}\Big[\frac{1}{n}\nabla^2\log f_n(X^n|\theta^*)+V_{\theta^*}\Big]h\Big| \\
    &\qquad + |\alpha_nn\textrm{rem}_n(h)|. 
    \end{split}
\end{equation}
Before proceeding to Step 2, we will study the first term of \eqref{eq:Rn(h)_triangle}. In particular, we will establish that for all $\epsilon > 0$, $M > 0$, and any $r > 0$, there exists $N_0 \equiv N_0(\epsilon, r, M)$ such that for $\alpha_nn > N_0$,
\begin{equation}\label{eq:prop2_step1}
\mathbb{P}_{f_{0,n}} \hspace{-2pt}\Big(\hspace{-1.6pt} \sup_{h\in\bar{B}_0(r)} \hspace{-2pt}\Big|\sqrt{\alpha_n}h^\top\Big[\frac{\nabla\log f_n(X^n|\theta^*)}{\sqrt{n}}  - V_{\theta^*}\Delta_{n,\theta^*}\Big]+\frac{h^\top}{2}\Big[\frac{\nabla^2\log f_n(X^n|\theta^*)}{n}+V_{\theta^*}\Big]h\Big| > \frac{M}{2} \hspace{-2pt} \Big) < \frac{\epsilon}{2}.
\end{equation}
Let $T_1(h) := \sqrt{\alpha_n}h^\top[\frac{1}{\sqrt{n}} \nabla\log f_n(X^n|\theta^*) - V_{\theta^*}\Delta_{n,\theta^*}]$ and $T_2(h) := \frac{1}{2}h^\top[\frac{1}{n} \nabla^2\log f_n(X^n|\theta^*)+V_{\theta^*}]h$.
We establish the result in \eqref{eq:prop2_step1} by first noting that
\begin{align}
\sup_{h} |T_1(h) + T_2(h)| &\leq\sup_{h} |T_1(h)| +\sup_{h}|T_2(h)| \label{eq:prop2_prob1_max_bound_0}\leq2 \max\big\{\sup_{h} |T_1(h)|, \sup_{h} | T_2(h)|\big\}.
\end{align}
Now, to show \eqref{eq:prop2_step1}, notice $ \max\{\sup_{h} |T_1(h)|, \sup_{h} | T_2(h)|\} > \frac{M}{4}$
if $\{\sup_{h} |T_1(h)| > \frac{M}{4}\} \cup \{\sup_{h} | T_2(h)| > \frac{M}{4}\}.$
Therefore, by \eqref{eq:prop2_prob1_max_bound_0} and a union bound, we have
\begin{align}
     \mathbb{P}_{f_{0,n}}&\Big(\sup_{h\in\bar{B}_0(r)}| T_1(h) + T_2(h)| > \frac{M}{2} \Big) \nonumber \\
     &\label{eq:prop2_prob1_union}\leq\mathbb{P}_{f_{0,n}}\Big(\sup_{h\in\bar{B}_0(r)} |T_1(h)| > \frac{M}{4}\Big) + \mathbb{P}_{f_{0,n}}\Big(\sup_{h\in\bar{B}_0(r)} |T_2(h)| > \frac{M}{4} \Big).
\end{align}
We upper bound the first term of \eqref{eq:prop2_prob1_union} using Cauchy-Schwarz as follows:
\begin{align}
\nonumber\mathbb{P}_{f_{0,n}}\hspace{-1pt}&\Big(\sup_{h\in\bar{B}_0(r)} |T_1(h)| > \frac{M}{4}\Big) = \mathbb{P}_{f_{0,n}}\Big(\sup_{h\in\bar{B}_0(r)} \Big|\sqrt{\alpha_n}h^\top\Big[\frac{ 1}{\sqrt{n}}\nabla\log f_n(X^n|\theta^*) - V_{\theta^*}\Delta_{n,\theta^*}\Big]\Big| > \frac{M}{4}\Big) \\
\nonumber&\leq\mathbb{P}_{f_{0,n}}\Big(\sup_{h\in\bar{B}_0(r)} \sqrt{\alpha_n}\|h\|_2\Big\|\frac{ 1}{\sqrt{n}}  \nabla\log f_n(X^n|\theta^*) - V_{\theta^*}\Delta_{n,\theta^*}\Big\|_2 > \frac{M}{4}\Big) \\
\label{eq:prop2_prob11}&=\mathbb{P}_{f_{0,n}}\Big(\sqrt{\alpha_n}\Big\|\frac{1}{\sqrt{n}}  \nabla\log f_n(X^n|\theta^*)  - V_{\theta^*}\Delta_{n,\theta^*}\Big\|_2 > \frac{M}{4r}\Big).
\end{align}
The equality in \eqref{eq:prop2_prob11} follows from taking the supremum over $\bar{B}_0(r)$. By Condition 3 and the assumption that $\alpha_n\to0$, we have that for all $\epsilon > 0$, $M > 0$, and $r > 0$, there exists an integer $N_{01} \equiv N_{01}(\epsilon, r, M)$ such that for $\alpha_nn > N_{01}$, the probability in \eqref{eq:prop2_prob11} is upper bounded by $\epsilon/4$.

We now upper bound the second term of \eqref{eq:prop2_prob1_union}. First, using Lemma \ref{lem:l1_l2_tensor_bound} we have 
\begin{align*}
\nonumber2\sup_{h\in\bar{B}_0(r)} |T_2(h)| &= \sup_{h\in\bar{B}_0(r)} \Big|h^\top \Big[\frac{1}{n} \nabla^2\log f_n(X^n|\theta^*+V_{\theta^*}\Big]h \Big| \\
\nonumber& =\sup_{h\in\bar{B}_0(r)} \Big|\sum_{i,j}\Big[\frac{1}{n} \nabla^2\log f_n(X^n|\theta^*)+V_{\theta^*}\Big]_{ij}h_ih_j\Big|  \\
\nonumber& \leq  \max_{1\leq i,j \leq p}\Big|\Big[\frac{1}{n} \nabla^2\log f_n(X^n|\theta^*)+V_{\theta^*}\Big]_{ij} \Big|\sup_{h\in\bar{B}_0(r)}  \sum_{i,j} |h_ih_j| \\
&  \leq \max_{1\leq i,j \leq p}\Big|\Big[\frac{1}{n} \nabla^2\log f_n(X^n|\theta^*)+V_{\theta^*}\Big]_{ij}  \Big| \sup_{h\in\bar{B}_0(r)}  p\|h\|_2^2 \\
&=  pr^2\max_{1\leq i,j\leq p}\Big|\Big[\frac{1}{n} \nabla^2\log f_n(X^n|\theta^*)+V_{\theta^*}\Big]_{ij}\Big| .
\end{align*}
Hence, we have
\begin{align}
\mathbb{P}_{f_{0,n}}\Big(\sup_{h\in\bar{B}_0(r)} |T_2(h)| > \frac{M}{4}\Big) 
\label{eq:prop2_prob22}& \leq \mathbb{P}_{f_{0,n}}\Big(\max_{1\leq i,j\leq p}\Big|\Big[\frac{1}{n} \nabla^2\log f_n(X^n|\theta^*)+V_{\theta^*}\Big]_{ij}\Big|  > \frac{M}{2pr^2} \Big).
\end{align}
By Condition 3, for all $\epsilon > 0$, $M > 0$, and $r > 0$, there exists an integer $N_{02} \equiv N_{02}(\epsilon,p,r,M)$ such that for $\alpha_nn > N_{02}$, the probability in \eqref{eq:prop2_prob22} is upper bounded by $\epsilon/4$.
Combining the results of \eqref{eq:prop2_prob11} and \eqref{eq:prop2_prob22}, we have that the result in \eqref{eq:prop2_step1} holds for $\alpha_nn > N_0 \equiv \max(N_{01}, N_{02})$.\\

\noindent\textbf{Step 2: Analysis of $\mathbf{\boldsymbol{\alpha}_nn \textrm{rem}_n(h)}$ for $h\in\bar{B}_0(r)$.} In this step, we verify that for all $\epsilon > 0$, $r > 0$, and $M > 0$, and for $\alpha_nn$ sufficiently large,
\begin{equation}\label{eq:prop2_step2}
    \mathbb{P}_{f_{0,n}}\Big(\sup_{h\in \bar{B}_{0}(r)}\left|\alpha_nn \textrm{rem}_n(h)\right| > \frac{M}{2}\Big) < \frac{\epsilon}{2}.
\end{equation}

First, we argue that $\textrm{rem}_n(h)$ is well defined on $\bar{B}_{0}(r)$ for $\alpha_nn$ sufficiently large. That is, we will show that for any $h\in\bar{B}_0(r)$, the vector $\theta^* + {th}/{\sqrt{\alpha_nn}}$ belongs to the neighborhood of $\theta^*$ where $\nabla^3\log f_n(X^n|\theta^* + {th}/{\sqrt{\alpha_nn}})$ is continuous. 
By assumption, there exists $r_0 > 0$ such that $\log f_n(X^n|\theta)$ has continuous third derivatives on $\bar{B}_{\theta^*}(r_0)$. For any $r > 0$, there exists $N_1 \equiv N_1(r_0, r) = \left\lceil r^2/r_0^2\right\rceil$ such that $\theta^*+th/\sqrt{\alpha_nn} \in \bar{B}_{\theta^*}(r_0)$ whenever $h \in \bar{B}_0(r)$ and $\alpha_nn \geq N_1$. To see this, note that if $\|h\|_2 \leq r$ and $\alpha_nn \geq r^2/r_0^2$, then for $0\leq t\leq 1$,
\begin{equation}\label{eq:B(r0)_arg}
    \Big\|\theta^* + \frac{th}{\sqrt{\alpha_nn}} -\theta^*\Big\|_2 \leq \frac{tr}{\sqrt{{r^2}/{r_0^2}}} < tr_0 \leq  r_0.\
\end{equation}
Hence, for any $h\in\bar{B}_0(r)$, the remainder $\textrm{rem}_n(h)$ is well defined for $\alpha_nn$ sufficiently large. 

Next, we can bound $\alpha_nn\textrm{rem}_n(h)$ using the definition of the tensor inner product in \eqref{eq:tensor_inner_product_def}:
\begin{align}
    \nonumber|\alpha_nn\textrm{rem}_n(h)| &= \Big|\frac{1}{2\sqrt{\alpha_nn}}\int_0^1 (1-t)^2\biggl<\frac{1}{n}\nabla^3\log f_n\big(X^n\big|\theta^* + \frac{t h}{\sqrt{\alpha_nn}}\big), h^{\otimes 3}\biggr>dt\Big| \\
    \nonumber &\leq\frac{\|h^{\otimes{3}}\|_1}{2\sqrt{\alpha_nn}}  \int_0^1 (1-t)^2 \max_{1 \leq i,j,k \leq p} \Big| \Big[\frac{1}{n}\nabla^3\log f_n\big(X^n\big|\theta^* + \frac{th}{\sqrt{\alpha_nn}}\big)\Big]_{i,j,k}\Big|  dt \\
   \nonumber &\leq \frac{\|h^{\otimes{3}}\|_1}{2\sqrt{\alpha_nn}}  \sup_{ t'\in[0,1]}   \max_{1 \leq i,j,k \leq p} \Big| \Big[\frac{1}{n}\nabla^3\log f_n\big(X^n\big|\theta^* + \frac{t' h}{\sqrt{\alpha_nn}}\big)\Big]_{i,j,k}\Big| \int_0^1 (1-t)^2  dt  \\
    \label{eq:rem_ineq1} &\leq  \frac{p^{\frac{3}{2}}\|h\|_2^3}{6 \sqrt{\alpha_n n}}  \sup_{ t\in[0,1]} \max_{1 \leq i,j,k \leq p} \Big| \Big[\frac{1}{n}\nabla^3\log f_n\big(X^n\big|\theta^* + \frac{th}{\sqrt{\alpha_nn}}\big)\Big]_{i,j,k}\Big|.
\end{align}
 The final inequality  follows from Lemma \ref{lem:l1_l2_tensor_bound} and evaluating the integral. Using the above,
\begin{align}
    \nonumber\sup_{h\in \bar{B}_{0}(r)}\left|\alpha_nn\textrm{rem}_n(h)\right| &\leq \sup_{h\in \bar{B}_{0}(r)} \frac{p^{\frac{3}{2}}\|h\|_2^3}{6\sqrt{\alpha_nn}} \sup_{ t\in[0,1]} \max_{1 \leq i,j,k \leq p} \Big| \Big[\frac{1}{n}\nabla^3\log f_n\big(X^n\big|\theta^* + \frac{th}{\sqrt{\alpha_nn}}\big)\Big]_{i,j,k}\Big| \\
     \nonumber &\leq \frac{p^{\frac{3}{2}}r^3}{6\sqrt{\alpha_nn}} \sup_{\substack{ h\in \bar{B}_{0}(r) \\ t\in[0,1]}} \max_{1 \leq i,j,k \leq p} \Big| \Big[\frac{1}{n}\nabla^3\log f_n\big(X^n\big|\theta^* + \frac{th}{\sqrt{\alpha_nn}}\big)\Big]_{i,j,k}\Big| \\
     \label{eq:rem_ineq3}&\leq  \frac{p^{\frac{3}{2}}r^3}{6\sqrt{\alpha_nn}} \max_{\theta\in \bar{B}_{\theta^*}(r_0)}  \max_{1 \leq i,j,k \leq p} \Big| \Big[\frac{1}{n}\nabla^3\log f_n(X^n|\theta)\Big]_{i,j,k}\Big|.
\end{align}
  The inequality in \eqref{eq:rem_ineq3} follows from noting that $h\in\bar{B}_0(r)$ implies $\theta^* + {th}/{\sqrt{\alpha_nn}}\in\bar{B}_{\theta^*}(r_0)$  when $\alpha_nn \geq N_1$, by the argument in \eqref{eq:B(r0)_arg}, and the assumed continuity of $\nabla^3\log f_n(X^n|\theta)$ on the compact set $\bar{B}_{\theta^*}(r_0)$ allows us to change $\sup_{\theta\in \bar{B}_{\theta^*}(r_0)}$ to $\max_{\theta\in \bar{B}_{\theta^*}(r_0)}$.

Taking $\alpha_nn \geq N_1$, we apply \eqref{eq:rem_ineq3} and Markov's inequality to the probability in \eqref{eq:prop2_step2}:
\begin{align}
   \nonumber  \mathbb{P}_{f_{0,n}} &\Big(\sup_{h\in \bar{B}_{0}(r)}|\alpha_nn \textrm{rem}_n(h)| > \frac{M}{2}\Big) \\
   &\leq \mathbb{P}_{f_{0,n}} \Big(\frac{p^{\frac{3}{2}}r^3}{6\sqrt{\alpha_nn}} \max_{\theta\in \bar{B}_{\theta^*}(r_0)}  \max_{1 \leq i,j,k \leq p} \Big| \Big[\frac{1}{n}\nabla^3\log f_n(X^n|\theta)\Big]_{i,j,k}\Big| >  \frac{M}{2} \Big)   \nonumber \\
    \nonumber  &\leq \frac{p^{\frac{3}{2}}r^3}{3M\sqrt{\alpha_nn}}\mathbb{E}_{f_{0,n}}\Big[\max_{\theta\in \bar{B}_{\theta^*}(r_0)} \max_{i,j,k} \Big| \Big[\frac{1}{n}\nabla^3\log f_n(X^n\big|\theta)\Big]_{i,j,k}\Big| \Big] \\
    &\leq \frac{p^{\frac{3}{2}}r^3}{3M\sqrt{\alpha_nn}}\mathbb{E}_{f_{0,n}}\Big[\Big|\frac{1}{n}\tilde{M}(X^n)\Big|\Big] 
    \label{eq:A2_R'n_prob_bound_final}< \frac{\epsilon}{2}.
\end{align}
In the above,  the first inequality in \eqref{eq:A2_R'n_prob_bound_final} follows from the condition 2 assumption and the second by further taking $\alpha_nn > \max(N_1, N_2)$, where
\begin{equation*}
    N_2 \equiv N_2(M, r, \epsilon) = \Big\lceil\frac{4p^3r^6}{9M^2\epsilon^2} \mathbb{E}_{f_{0,n}}\Big[\big|\frac{1}{n}\tilde{M}(X^n)\big|\Big]^2\Big\rceil.
\end{equation*}
We note that $N_2$ exists (i.e., $N_2$ is finite), since $\mathbb{E}_{f_{0,n}}[|\frac{1}{n}\tilde{M}(X^n)|]$ is finite by assumption.\\

\noindent\textbf{Step 3: Final bound.} We combine the results in \eqref{eq:Rn(h)_triangle}, \eqref{eq:prop2_step1}, and \eqref{eq:prop2_step2} to obtain \eqref{eq:A2_condition} as follows. Considering \eqref{eq:prop2_step1}, let
\[f(h)  \equiv  \Big|\sqrt{\alpha_n}h^\top\Big[\frac{\nabla\log f_n(X^n|\theta^*)}{\sqrt{n}}  - V_{\theta^*}\Delta_{n,\theta^*}\Big]+\frac{h^\top}{2}\Big[\frac{\nabla^2\log f_n(X^n|\theta^*)}{n}+V_{\theta^*}\Big]h\Big|.\]
Then, for any $M > 0$, $r > 0$, and $\epsilon > 0$, we take $\alpha_nn > \max(N_0, N_1, N_2)$ to obtain the following by applying the bound in \eqref{eq:Rn(h)_triangle}:
\begin{align}
    &\mathbb{P}_{f_{0,n}}\Big(\sup_{h\in\bar{B}_0(r)}|R_n(h)| > M\Big)
    \nonumber \leq \mathbb{P}_{f_{0,n}}\Big(\sup_{h\in\bar{B}_0(r)} f(h)+ \sup_{h\in\bar{B}_0(r)} |\alpha_nn\textrm{rem}_n(h)| > M\Big) \\
    \nonumber&\leq \mathbb{P}_{f_{0,n}}\Big(\max\Big\{\sup_{h\in\bar{B}_0(r)} f(h),~\sup_{h\in\bar{B}_0(r)} |\alpha_nn\textrm{rem}_n(h)|\Big\} > \frac{M}{2}\Big) \\
    \nonumber&= \mathbb{P}_{f_{0,n}}\Big(\Big\{\sup_{h\in\bar{B}_0(r)} f(h) > \frac{M}{2}\Big\} \cup \Big\{\sup_{h\in\bar{B}_0(r)}|\alpha_nn\textrm{rem}_n(h)| > \frac{M}{2}\Big\}\Big) \\
    &\leq \mathbb{P}_{f_{0,n}}\Big(\sup_{h\in\bar{B}_0(r)} f(h) > \frac{M}{2}\Big) + \mathbb{P}_{f_{0,n}}\Big(\sup_{h\in\bar{B}_0(r)}|\alpha_nn\textrm{rem}_n(h)| > \frac{M}{2}\Big) \label{eq:prop2_step3_ineqfinal}\leq \frac{\epsilon}{2} + \frac{\epsilon}{2} = \epsilon. 
\end{align}
The second inequality in \eqref{eq:prop2_step3_ineqfinal} follows from \eqref{eq:prop2_step1} and \eqref{eq:prop2_step2}. 
This completes the proof.

\end{proof}

\subsection{Sufficient conditions for Assumption (A3)}\label{app:post_conc_suff}
The next proposition shows that Conditions \ref{condition1} and \ref{condition2} imply \textbf{(A3)}, even for improper priors.

\begin{Proposition}\label{prop:A3_suff}
    Suppose that assumptions \textbf{(A0)}--\textbf{(A1)} hold and assume that $\pi(\theta) \leq \kappa$ for some $\kappa > 0$. Furthermore, assume Conditions \ref{condition1} and \ref{condition2}. Then, Assumption \textbf{(A3)} holds; that is, for all $\epsilon > 0$ and all $k_0 \geq 0$ there exist a constant $M < \infty$, and an integer $N$ such that for $\alpha_nn\geq N$, we have $\mathbb{P}_{f_{0,n}}(\mathbb{E}_{\pi_{n,\alpha_n}(\theta|X^n)}[\|\sqrt{\alpha_nn}(\theta-\theta^*)\|_2^{k_0}] > M) < \epsilon.$
\end{Proposition}
\begin{proof}
First, we define the following events for constants $L, r, c_3, c_4$ and function $\gamma$ from Conditions \ref{condition1} and \ref{condition2}. We take an arbitrary constant $M_0>0$ to be defined just after.
\begin{equation}\label{eq:A3_events_def}
    \begin{split}
        \mathcal{A} &= \left\{\log f_n(X^n|\theta)-\log f_n(X^n|\theta^*) \leq -Ln\gamma(\|\theta-\theta^*\|_2) \text{ for all } \theta\in B_{\theta^*}(r)^\mathsf{c}\right\}, \\
        \mathcal{B} &= \big\{-nc_3\|\theta-\hat{\theta}\|_2^2 \leq \log f_n(X^n|\theta)-\log f_n(X^n|\hat{\theta}) \leq -nc_4\|\theta-\hat{\theta}\|_2^2 \text{ for all } \theta\in B_{\hat{\theta}}(2r)\big\}, \\
        \mathcal{C} &= \{\|\hat{\theta}-\theta^*\|_2 < r\}, \quad \quad \quad \quad  \mathcal{D} = \{\|\sqrt{n}(\hat{\theta}-\theta^*)\|_2 < M_0\}.
    \end{split}
\end{equation}
Let $\epsilon >0$ be arbitrary.
By Condition \ref{condition1}, there exists $N_0 \equiv N_0(L, c_0, c_1, c_2, r, \epsilon)$ such that $\mathbb{P}_{f_{0,n}}(\mathcal{A}^\mathsf{c}) < \epsilon/4$ whenever $\alpha_nn > N_0$. By Condition \ref{condition2}, there exists $N_1 \equiv N_1(c_0, c_1, c_2, c_3, c_4, r, \epsilon)$ such that $\mathbb{P}_{f_{0,n}}(\mathcal{B}^\mathsf{c}) < \epsilon/4$ whenever $\alpha_nn > N_1$. By Assumption \textbf{(A0)}, there exists $N_2 \equiv N_2(r, \epsilon)$ such that $\mathbb{P}_{f_{0,n}}(\mathcal{C}^\mathsf{c}) < \epsilon/4$ whenever $\alpha_nn > N_2$ and there exists $M_0$ and $N_3 \equiv N_3(M_0, \epsilon)$ such that $\mathbb{P}_{f_{0,n}}(\mathcal{D}^\mathsf{c}) < \epsilon/4$ whenever $\alpha_nn > N_3$. 

Let $\mathcal{E} = \{\mathcal{A}\cap\mathcal{B}\cap\mathcal{C}\cap\mathcal{D}\}$. By a union bound, whenever $\alpha_nn > \max(N_0, N_1, N_2, N_3)$, we have $\mathbb{P}_{f_{0,n}}(\mathcal{E}^\mathsf{c}) \leq \mathbb{P}_{f_{0,n}}(\mathcal{A}^\mathsf{c}) + \mathbb{P}_{f_{0,n}}(\mathcal{B}^\mathsf{c}) + \mathbb{P}_{f_{0,n}}(\mathcal{C}^\mathsf{c}) + \mathbb{P}_{f_{0,n}}(\mathcal{D}^\mathsf{c}) < \epsilon.$
Hence, when $\alpha_nn > \max(N_0, N_1, N_2, N_3)$, using the bound $(a+b)^{k_0}\leq 2^{k_0}(a^{k_0} + b^{k_0})$ for ${k_0}\geq0$,
\begin{equation}
\begin{split}
\label{eq:Pbound1}
\mathbb{P}_{f_{0,n}}&\big(\mathbb{E}_{\pi_{n,\alpha_n}(\theta|X^n)}[\|\sqrt{\alpha_nn}(\theta-\theta^*)\|_2^{k_0}] > M\big) \\
&\leq \mathbb{P}_{f_{0,n}}\big(\mathbb{E}_{\pi_{n,\alpha_n}(\theta|X^n)}[\|\sqrt{\alpha_nn}(\theta-\theta^*)\|_2^{k_0}] > M \cap\mathcal{E}\big) + {\epsilon}
\\
    &\leq \mathbb{P}_{f_{0,n}}\big(2^{k_0}\mathbb{E}_{\pi_{n,\alpha_n}(\theta|X^n)}[\|\sqrt{\alpha_nn}(\theta-\hat{\theta})\|_2^{k_0}] + 2^{k_0}\|\sqrt{\alpha_nn}(\hat{\theta}-\theta^*)\|_2^{k_0} > M\cap\mathcal{E}\big) + {\epsilon} \\
    &\leq \mathbb{P}_{f_{0,n}}\big(2^{k_0}\mathbb{E}_{\pi_{n,\alpha_n}(\theta|X^n)}[\|\sqrt{\alpha_nn}(\theta-\hat{\theta})\|_2^{k_0}] > M - (2\sqrt{\alpha_n}M_0)^{k_0}\cap\mathcal{E}\big) +{\epsilon},
    \end{split}
\end{equation}
where the last  inequality comes from the intersection with event $\mathcal{D}$.

Now, for any $\tau > 0$, define $M' \equiv 2^{-k_0}(M - (2\tau M_0)^{k_0})$ and note that since $\alpha_n\to0$ there exists $N_4 \equiv N_4(\tau)$ such that $\sqrt{\alpha_n} < \tau$ when $\alpha_nn > N_4$. In particular, choose $\tau$ small enough that $M' >0$. Then, $M - (2\sqrt{\alpha_n}M_0)^{k_0} \geq M - (2\tau M_0)^{k_0}$ when $\alpha_nn > N_4$. Hence, from \eqref{eq:Pbound1},
\begin{equation}
\label{eq:Pbound2}
    \mathbb{P}_{f_{0,n}}(\mathbb{E}_{\pi_{n,\alpha_n}(\theta|X^n)}[\|\sqrt{\alpha_nn}(\theta-\theta^*)\|_2^{k_0}] > M) \leq \mathbb{P}_{f_{0,n}}(\mathbb{E}_{\pi_{n,\alpha_n}(\theta|X^n)}[\|\sqrt{\alpha_nn}(\theta-\hat{\theta})\|_2^{k_0}] > M'\cap\mathcal{E}) + \epsilon.
\end{equation}

Moving forward, our aim is to show that there is some $M'$ large enough that
\begin{align}
\label{eq:new_goal}
\mathbb{P}_{f_{0,n}}\big(\mathbb{E}_{\pi_{n,\alpha_n}(\theta|X^n)}[\|\sqrt{\alpha_nn}(\theta-\hat{\theta})\|_2^{k_0}] > M'\cap\mathcal{E}\big) = 0 .
\end{align}
That is, for the remainder of the proof, we will study $\mathbb{E}_{\pi_{n,\alpha_n}(\theta|X^n)}[\|\sqrt{\alpha_nn}(\theta-\hat{\theta})\|_2^{k_0}]$ on the event $\mathcal{E}$ and show that the probability on the right side of \eqref{eq:Pbound2} equals $0$. 
By the tail formula of the expectation, we have the following:
\begin{align}
    k_0^{-1}\mathbb{E}_{\pi_{n,\alpha_n}(\theta|X^n)}[\|\sqrt{\alpha_nn}(\theta-\hat{\theta})\|_2^{k_0}]
  &= k_0^{-1}\int_0^\infty \mathbb{P}_{\pi_{n,\alpha_n}(\theta|X^n)}(\|\sqrt{\alpha_nn}(\theta-\hat{\theta})\|_2^{k_0} > s)ds \nonumber \\
    \label{eq:tail_changeofv}&= \int_0^\infty t^{k_0-1}\mathbb{P}_{\pi_{n,\alpha_n}(\theta|X^n)}(\|\sqrt{\alpha_nn}(\theta-\hat{\theta})\|_2 > t)dt,
\end{align}
where in \eqref{eq:tail_changeofv}, we used the change of variable $s = t^{k_0} \implies ds = k_0t^{k_0-1} dt$. Now from \eqref{eq:tail_changeofv},
\begin{align}
k_0^{-1}\mathbb{E}_{\pi_{n,\alpha_n}(\theta|X^n)}[\|\sqrt{\alpha_nn}(\theta-\hat{\theta})\|_2^{k_0}] &= \int_0^{2r\sqrt{\alpha_nn}} t^{k_0-1}\mathbb{P}_{\pi_{n,\alpha_n}(\theta|X^n)}(\|\sqrt{\alpha_nn}(\theta-\hat{\theta})\|_2 > t)dt \nonumber \\
& \qquad + \int_{2r\sqrt{\alpha_nn}}^\infty t^{k_0-1}\mathbb{P}_{\pi_{n,\alpha_n}(\theta|X^n)}(\|\sqrt{\alpha_nn}(\theta-\hat{\theta})\|_2 > t)dt. \label{eq:Pbound3}
\end{align}
Label the terms on the right side of \eqref{eq:Pbound3} as $I_1$ and $I_2$.
Letting $M'' = M'/(2k_0)$, we showed
\begin{align}
\label{eq:what_we_have}
    \mathbb{P}_{f_{0,n}}\big(\mathbb{E}_{\pi_{n,\alpha_n}(\theta|X^n)}[\|\sqrt{\alpha_nn}(\theta-\hat{\theta})\|_2^{k_0}] > M'\cap\mathcal{E}\big) 
    &= \mathbb{P}_{f_{0,n}}\big(I_1 + I_2 > M'/k_0\cap\mathcal{E}\big).
\end{align}
We next notice,
\begin{align*}
   \mathbb{P}_{f_{0,n}}\big(I_1 + I_2 > M'/k_0\cap\mathcal{E}\big)
   &\leq \mathbb{P}_{f_{0,n}}\big(\max(I_1,I_2) > M''\cap\mathcal{E}\big) \\
   &= \mathbb{P}_{f_{0,n}}\big(\{I_1 > M'' \cup I_2 > M''\}\cap\mathcal{E}\big) \\
    \nonumber&= \mathbb{P}_{f_{0,n}}\big(\{I_1 > M''\cap\mathcal{E}\} \cup \{I_2 > M''\cap\mathcal{E}\}\big) \\
    \nonumber&\leq \mathbb{P}_{f_{0,n}}\big(I_1 > M''\cap\mathcal{E}\big) + \mathbb{P}_{f_{0,n}}\big(I_2 > M''\cap\mathcal{E}\big).
\end{align*}
Hence, from the above and \eqref{eq:what_we_have}, we have that
\begin{align}\label{eq:A3prob_ineq2}
   \mathbb{P}_{f_{0,n}}\big(\mathbb{E}_{\pi_{n,\alpha_n}(\theta|X^n)}[\|\sqrt{\alpha_nn}(\theta-\hat{\theta})\|_2^{k_0}] > M'\cap\mathcal{E}\big)  \leq \mathbb{P}_{f_{0,n}}(I_1 > M''\cap\mathcal{E}) + \mathbb{P}_{f_{0,n}}(I_2 > M''\cap\mathcal{E}).
\end{align}
We will show there exists some $M''$ (more precisely, there exists some $M$ since $2 k_0 M'' = M' = 2^{-k_0}(M - (2 \tau M_0)^{k_0})$) such that both probabilities on the right side of \eqref{eq:A3prob_ineq2} are 0.

\noindent\textbf{First Term of \eqref{eq:A3prob_ineq2}:} We will show that $I_1$ defined in \eqref{eq:Pbound3} is bounded above by a constant for $\alpha_nn$ sufficiently large conditioned on $\mathcal{E}$. Hence, we may choose $M''$ to be larger than this constant to make $\mathbb{P}_{f_{0,n}}(I_1 > M''\cap\mathcal{E}) = 0$. First, considering $I_1$ defined in \eqref{eq:Pbound3}, our goal is to obtain a bound for 
\begin{align*}
    \mathbb{P}_{\pi_{n,\alpha_n}(\theta|X^n)}&(\|\sqrt{\alpha_nn}(\theta-\hat{\theta})\|_2 > t) = \mathbb{P}_{\pi_{n,\alpha_n}(\theta|X^n)}\Big(\|\theta-\hat{\theta}\|_2 > \frac{t}{\sqrt{\alpha_nn}}\Big)\\
     &= \mathbb{P}_{\pi_{n,\alpha_n}(\theta|X^n)}\Big(2r > \|\theta-\hat{\theta}\|_2 > \frac{t}{\sqrt{\alpha_nn}}\Big) + \mathbb{P}_{\pi_{n,\alpha_n}(\theta|X^n)}\Big(\|\theta-\hat{\theta}\|_2 > 2r\Big).
\end{align*}
Hence, for $I_1$ of \eqref{eq:Pbound3}
\begin{align}
\label{eq:I11}
I_1 &= \int_0^{2r\sqrt{\alpha_nn}}t^{k_0-1}  \mathbb{P}_{\pi_{n,\alpha_n}(\theta|X^n)}\Big(2r > \|\theta-\hat{\theta}\|_2 > \frac{t}{\sqrt{\alpha_nn}}\Big)dt \\
&\qquad + \int_0^{2r\sqrt{\alpha_nn}} t^{k_0-1}\mathbb{P}_{\pi_{n,\alpha_n}(\theta|X^n)}\Big(\|\theta-\hat{\theta}\|_2 > 2r\Big)dt. \nonumber
\end{align}
We label the terms of \eqref{eq:I11} as $I_{11}$ and $I_{12}$ and in what follows we show that $I_{11}$ and $I_{12}$ are bounded above by a constant for $\alpha_nn$ sufficiently large, conditional on $\mathcal{E}$.

\noindent\textbf{Term $I_{11}$ of \eqref{eq:I11}:} Let $\ell_n(\theta) = \frac{1}{n}\log f_n(X^n|\theta)$. Then
\begin{equation}
\label{eq:I11_first}
    \mathbb{P}_{\pi_{n,\alpha_n}(\theta|X^n)}\Big(2r > \|\theta-\hat{\theta}\|_2 > \frac{t}{\sqrt{\alpha_nn}}\Big)
    = \frac{\int_{2r > \|\theta-\hat{\theta}\|_2 > t/\sqrt{\alpha_nn}} e^{\alpha_nn(\ell_n(\theta)-\ell_n(\hat{\theta}))}\pi(\theta)d\theta}{\int e^{\alpha_nn(\ell_n(\theta)-\ell_n(\hat{\theta}))}\pi(\theta)d\theta} \equiv \frac{\text{num}}{\text{den}}.
\end{equation}
We will bound the probability in the above by upper bounding $\text{num}$ and lower bounding $\text{den}$. On the event $\mathcal{B}$, we have that
\begin{align*}
   \text{num}
    &\leq \int_{2r > \|\theta-\hat{\theta}\|_2 > \frac{t}{\sqrt{\alpha_nn}}}e^{-\alpha_nnc_4\|\theta-\hat{\theta}\|_2^2}\pi(\theta)d\theta
    \leq \kappa\int_{2r > \|\theta-\hat{\theta}\|_2 > \frac{t}{\sqrt{\alpha_nn}}} e^{-\alpha_nnc_4\|\theta-\hat{\theta}\|_2^2}d\theta,
\end{align*}
where the second inequality follows from the bound on the prior $\pi(\theta)\leq\kappa$. Applying the change of variable $h = \sqrt{\alpha_nn}(\theta-\hat{\theta}) \implies d\theta = (\alpha_nn)^{-p/2}dh$ to the above, we obtain
\begin{align}
    \text{num}
    \nonumber&\leq \kappa(\alpha_nn)^{-\frac{p}{2}}\int_{2r\sqrt{\alpha_nn} > \|h\|_2 > t} e^{-c_4\|h\|_2^2}dh\leq \kappa(\alpha_nn)^{-\frac{p}{2}}\int_{\|h\|_2 > t} e^{-c_4\|h\|_2^2}dh.
\end{align}
In particular, as $\int_{\|h\|_2 >t} (2\pi)^{-\frac{p}{2}}(2c_4)^{\frac{p}{2}}e^{-\frac{2c_4}{2}\|h\|_2^2}dh = \mathbb{P}(\|Z\|_2 >t\sqrt{2 c_4})$ for a standard Gaussian vector $Z \sim \mathcal{N}(0, \mathbb{I})$, we have 
$\int_{\|h\|_2 >t} (2\pi)^{-\frac{p}{2}}(2c_4)^{\frac{p}{2}}\exp\{-\frac{2c_4}{2}\|h\|_2^2\}dh \leq \exp\{-c_4(t - \sqrt{p/(2c_4)})^2\}$ by Lemma \ref{lem:Gaussian_concentration}. Hence, letting $\tilde{C} \equiv \kappa(2\pi)^{p/2}(2c_4)^{-p/2}$, the above bound implies 
\begin{align}
    \text{num} \leq \kappa(\alpha_nn)^{-\frac{p}{2}}(2\pi)^{\frac{p}{2}}(2c_4)^{-\frac{p}{2}} e^{-c_4(t - \sqrt{p/(2c_4)})^2}
    \label{eq:N_bound}&\equiv (\alpha_nn)^{-\frac{p}{2}}\tilde{C}e^{-c_4(t - \sqrt{p/(2c_4)})^2}.
\end{align}

Similarly, we may lower bound $ \text{den}$. Before we do so, note by Assumption \textbf{(A1)} that there exists $r' > 0$ such that $\pi(\theta) > 0$ on $B_{\theta^*}(r')$. Furthermore, $B_{\theta^*}(r)\subseteq B_{\hat{\theta}}(2r)$ on the event $\mathcal{C}$. This follows from the triangle inequality, as $\|\theta-\hat{\theta}\|_2 \leq \|\theta-\theta^*\|_2 + \|\theta^*-\hat{\theta}\|_2 < r + r = 2r$. Hence, $\theta\in B_{\theta^*}(r) \implies \theta\in B_{\hat{\theta}}(2r)$. Let $\tilde{r} = \min(r',r)$. Then, on the event $\mathcal{C}\cap\mathcal{B}$,
\begin{align*}
    \text{den} 
    = \int e^{\alpha_nn(\ell_n(\theta)-\ell_n(\hat{\theta}))}\pi(\theta)d\theta
    &\geq  \int_{B_{\theta^*}(\tilde{r})} e^{\alpha_nn(\ell_n(\theta)-\ell_n(\hat{\theta}))}\pi(\theta)d\theta \\
    &\geq \inf_{\theta'\in B_{\theta^*}(\tilde{r})}\pi(\theta')\int_{B_{\theta^*}(\tilde{r})} e^{\alpha_nn(\ell_n(\theta)-\ell_n(\hat{\theta}))}d\theta \\
    &\geq L(\theta^*, \tilde{r})\int_{B_{\theta^*}(\tilde{r})} e^{-\alpha_nnc_3\|\theta-\hat{\theta}\|_2^2} d\theta,
\end{align*}
where $L(\theta^*, \tilde{r}) \equiv \inf_{\theta\in B_{\theta^*}(\tilde{r})}\pi(\theta)$, which is strictly positive by Assumption \textbf{(A1)} and since $\tilde{r} \leq r'$. As the integral is also positive, the ratio of $\text{num}/\text{den}$ is well defined (i.e., we are not dividing by $0$). Applying the change of variable $h = \sqrt{\alpha_nn}(\theta-\theta^*)$, we have that $\theta=(\alpha_nn)^{-1/2}h + \theta^*$. Hence, $\theta\in B_{\theta^*}(\tilde{r}) \implies (\alpha_nn)^{-1/2}h + \theta^* \in B_{\theta^*}(\tilde{r}) \implies h\in B_0(\tilde{r}\sqrt{\alpha_nn}).$ From the above then, with the change of variables, we find
\begin{align*}
    \text{den} 
    \geq L(\theta^*, \tilde{r})\int_{B_{\theta^*}(\tilde{r})} e^{-\alpha_nnc_3\|\theta-\hat{\theta}\|_2^2} d\theta 
    &= L(\theta^*, \tilde{r})(\alpha_nn)^{-\frac{p}{2}}\int_{B_0(\tilde{r}\sqrt{\alpha_nn})} e^{-\alpha_nnc_3\left\|\theta^* + \frac{h}{\sqrt{\alpha_nn}} - \hat{\theta}\right\|_2^2} dh \\
    &= L(\theta^*, \tilde{r})(\alpha_nn)^{-\frac{p}{2}}\int_{B_0(\tilde{r}\sqrt{\alpha_nn})} e^{-c_3\left\|h + \sqrt{\alpha_nn}(\theta^*-\hat{\theta})\right\|_2^2} dh.
\end{align*}
Notice that since $(a + b)^p \leq 2^p\left(a^p + b^p\right)$ for $p \geq 0$, we have
$\|h + \sqrt{\alpha_nn}(\theta^*-\hat{\theta})\|_2^2 \leq 4\|h\|_2^2 + 4\|\sqrt{\alpha_nn}(\theta^*-\hat{\theta})\|_2^2.$
Next, notice that on the event $\mathcal{D}$, we have that $\|\sqrt{n}(\theta^*-\hat{\theta})\|_2^2 < M_0^2$ and for $\alpha_nn > N_4$, we have that $\alpha_n < \tau^2$. Hence, we can further bound the above
\begin{align}\label{eq:den_norm_bound}
    \left\|h + \sqrt{\alpha_nn}(\theta^*-\hat{\theta})\right\|_2^2 \leq 4\|h\|_2^2 + 4\tau^2M_0^2.
\end{align}
Next, using the bound in \eqref{eq:den_norm_bound}, we have that on $\mathcal{D}$ and with $\alpha_nn$ large enough,
\begin{align}
    \text{den} 
    &\nonumber \geq L(\theta^*, \tilde{r})(\alpha_nn)^{-\frac{p}{2}}\int_{B_0(\tilde{r}\sqrt{\alpha_nn})} e^{-c_3\left\|h + \sqrt{\alpha_nn}(\theta^*-\hat{\theta})\right\|_2^2} dh \\
    &\nonumber \geq L(\theta^*, \tilde{r})(\alpha_nn)^{-\frac{p}{2}}\int_{B_0(\tilde{r}\sqrt{\alpha_nn})} e^{-c_3(4\left\|h\right\|_2^2 + 4\tau^2M_0^2)} dh \\
    &= L(\theta^*, \tilde{r})e^{-4c_3\tau^2M_0^2}(\alpha_nn)^{-\frac{p}{2}}\int_{B_0(\tilde{r}\sqrt{\alpha_nn})}e^{-4c_3\left\|h\right\|_2^2} dh \geq C_0(\alpha_nn)^{-\frac{p}{2}}, \label{eq:D_bound}
\end{align}
where $C_0 = L(\theta^*, \tilde{r})e^{-4c_3\tau^2M_0^2}\int_{B_0(\tilde{r})} e^{-4c_3\left\|h\right\|_2^2} dh$ and the final inequality holds for $\alpha_n n >1$.

Using the bounds from \eqref{eq:N_bound} and \eqref{eq:D_bound} in \eqref{eq:I11_first} yields
\begin{align*}
   \mathbb{P}_{\pi_{n,\alpha_n}(\theta|X^n)}\Big(2r > \|\theta-\hat{\theta}\|_2 > \frac{t}{\sqrt{\alpha_nn}}\Big) \hspace{-.55pt}
    = \frac{\text{num}}{\text{den}}
    \leq \frac{ (\alpha_nn)^{-\frac{p}{2}}\tilde{C}e^{-c_4(t - \sqrt{p/(2c_4)})^2}}{ (\alpha_nn)^{-\frac{p}{2}}C_0} = Ce^{-c_4(t - \sqrt{\frac{p}{2c_4}})^2},
\end{align*}
where $C \equiv \tilde{C}/C_0$. Now, we bound $I_{11}$  of \eqref{eq:I11} using the above as follows:
\begin{align}
   \nonumber I_{11} &= \int_0^{2r\sqrt{\alpha_nn}} \hspace{-.6pt}t^{k_0-1}  \mathbb{P}_{\pi_{n,\alpha_n}(\theta|X^n)}\Big(2r > \|\theta-\hat{\theta}\|_2 > \frac{t}{\sqrt{\alpha_nn}}\Big)dt \\
    &\leq C\int_{0}^{2r\sqrt{\alpha_nn}} t^{k_0-1}e^{-c_4(t - \sqrt{p/(2c_4)})^2}dt 
    \label{eq:I11_bound_cauchy}\leq C\int_{0}^{\infty} t^{k_0-1}e^{-c_4(t - \sqrt{p/(2c_4)})^2}dt, 
\end{align}
which is a finite constant independent of $n$. \\

\noindent\textbf{Term $I_{12}$ of \eqref{eq:I11}:} First notice that
\begin{align}
    I_{12} \nonumber& = \mathbb{P}_{\pi_{n,\alpha_n}(\theta|X^n)}\big(\|\theta-\hat{\theta}\|_2 > 2r\big) \int_0^{2r\sqrt{\alpha_nn}}t^{k_0-1}dt \\
    &= \frac{(2r)^{k_0}(\alpha_nn)^{\frac{k_0}{2}}}{k_0} \cdot \frac{e^{\alpha_nn\ell_n(\theta^*)} \int_{\|\theta-\hat{\theta}\|_2 > 2r} e^{\alpha_nn(\ell_n(\theta)-\ell_n(\theta^*))}\pi(\theta)d\theta}{e^{\alpha_nn\ell_n(\hat{\theta})}\int e^{\alpha_nn(\ell_n(\theta)-\ell_n(\hat{\theta}))}\pi(\theta)d\theta}. 
\end{align}
Then, since $e^{\alpha_nn\ell_n(\theta^*)} \leq e^{\alpha_nn\ell_n(\hat{\theta})}$ as $\hat{\theta}$ is the MLE, we have
\begin{align}
    I_{12} & \leq  \frac{(2r)^{k_0}(\alpha_nn)^{\frac{k_0}{2}} \int_{\|\theta-\hat{\theta}\|_2 > 2r} e^{\alpha_nn(\ell_n(\theta)-\ell_n(\theta^*))}\pi(\theta)d\theta}{k_0 \int e^{\alpha_nn(\ell_n(\theta)-\ell_n(\hat{\theta}))}\pi(\theta)d\theta} \\
    &\leq \frac{\kappa (2r)^{k_0}(\alpha_nn)^{\frac{k_0+p}{2}}}{C_0k_0}\int_{\|\theta-\hat{\theta}\|_2 > 2r} \hspace{-6.18pt} e^{\alpha_nn(\ell_n(\theta)-\ell_n(\theta^*))} d\theta, \label{eq:I12bound1}
\end{align}
where the second inequality above uses  the bound in \eqref{eq:D_bound} and the assumption $\pi(\theta)\leq\kappa$.  Next, using the triangle inequality $\|\theta-\hat{\theta}\|_2 \leq \|\theta-\theta^*\|_2 + \|\hat{\theta}-\theta^*\|_2$ and that we are on the event $\mathcal{C}$, we have $\{\theta: \|\theta-\hat{\theta}\|_2 > 2r\} \subseteq \{\theta: \|\theta-\theta^*\|_2 > r\}$ and therefore, 
\begin{align}
\label{eq:int1}
  \int_{\|\theta-\hat{\theta}\|_2 > 2r} e^{\alpha_nn(\ell_n(\theta)-\ell_n(\theta^*))} d\theta  &\leq \int_{\|\theta-\theta^*\|_2 > r} e^{\alpha_nn(\ell_n(\theta)-\ell_n(\theta^*))} d\theta \leq \int_{\|\theta-\theta^*\|_2 > r}  e^{-L\alpha_nn\gamma(\|\theta-\theta^*\|_2)}d\theta.
\end{align}
The final inequality above follows conditional on the event $\mathcal{A}$ defined in \eqref{eq:A3_events_def}. Now, switching to spherical coordinates (with $\text{vol}(\mathbb{S}^{p-1})$ denoting the volume of the unit sphere -- see Corollary \ref{corr:spherical}),
\begin{align}
   \int_{\|\theta-\theta^*\|_2 > r} e^{-L\alpha_nn\gamma(\|\theta-\theta^*\|_2)}d\theta &= \text{vol}(\mathbb{S}^{p-1}) \int_{r}^{\infty} \rho^{p-1}e^{-L\alpha_nn\gamma(\rho)}d\rho  \nonumber\\
    \label{eq:int2}&\leq \text{vol}(\mathbb{S}^{p-1}) e^{Lc_2\alpha_nn} \int_{r}^{\infty} \rho^{p-1}e^{-Lc_0\alpha_nn\log(1+c_1\rho) }d\rho \\
    \label{eq:int3} &\leq \text{vol}(\mathbb{S}^{p-1})e^{Lc_2\alpha_nn}c_1^{-Lc_0\alpha_nn} \int_{r}^{\infty}\rho^{p-1-Lc_0\alpha_nn}d\rho \\
    \label{eq:int4} &= \text{vol}(\mathbb{S}^{p-1})e^{Lc_2\alpha_nn}(rc_1)^{-Lc_0\alpha_nn} \frac{r^{p}}{Lc_0\alpha_n n - p}.
\end{align}
In \eqref{eq:int2}, we have appealed to the Condition 1 assumption, $\gamma(v) \geq c_0\log(1+c_1v) - c_2$ for $v \geq r$ and in \eqref{eq:int3}, we use $(1+c_1\rho)^{-Lc_0\alpha_nn} \leq c_1^{-Lc_0\alpha_nn}\rho^{-Lc_0\alpha_nn}$ (since $c_1 > 0$). Finally, in \eqref{eq:int4} we use that for $\alpha_n n > (p+1)/(Lc_0)$, we have $p-Lc_0\alpha_nn < -1$  so that 
$$ \int_{r}^{\infty}\rho^{p-1-Lc_0\alpha_nn}d\rho = \frac{1}{p-Lc_0\alpha_n n}\rho^{p-Lc_0\alpha_nn} \Big|^{\infty}_r = \frac{1}{Lc_0\alpha_n n - p}r^{p-Lc_0\alpha_nn}.$$
Plugging the bounds in \eqref{eq:int1}-\eqref{eq:int4} into \eqref{eq:I12bound1},
\begin{align}
    I_{12} \label{eq:I12_final}&\leq \frac{\text{vol}(\mathbb{S}^{p-1})\kappa (2r)^{k_0}(\alpha_nn)^{\frac{k_0+p}{2}}}{C_0k_0} \times \frac{r^p}{Lc_0\alpha_nn-p} \times \Big(\frac{e^{c_2}}{(c_1r)^{c_0}}\Big)^{L\alpha_nn}.
\end{align}
The third term in \eqref{eq:I12_final} decays exponentially in $\alpha_nn$ since we have taken $r > \frac{e^{c_2/c_0}}{c_1}$. Hence, for some $\tau_1 > 0$, there exists $N_5 \equiv N_5(\tau_1)$ such that $I_{12} < \tau_1$ whenever $\alpha_nn > \max((p+1)/(Lc_0),N_5)$.

Now we look back to the first term on the right side of \eqref{eq:A3prob_ineq2}. Considering the bound in \eqref{eq:I11_bound_cauchy} and the argument just above in \eqref{eq:I12_final}, taking $\alpha_nn$ large enough,
\begin{align}
    \mathbb{P}_{f_{0,n}}(I_1 > M'' \cap \mathcal{E}) 
    \nonumber&= \mathbb{P}_{f_{0,n}}(I_{11} + I_{12} > M'' \cap \mathcal{E}) \\
    \label{eq:I1_prob_bound}&\leq \mathbb{P}_{f_{0,n}}\Big(C\int_{0}^{\infty} t^{k_0-1}e^{-c_4(-t - \sqrt{p/(2c_4)})^2}dt + \tau_1 > M'' \cap \mathcal{E}\Big),
\end{align}
and we may take $M'' = C\int_{0}^{\infty} t^{k_0-1}e^{c_4(-t - \sqrt{p/(2c_4)})^2}dt + \tau_1 + 1 < \infty$ so that the probability on the right hand side equals 0 as desired. \\

\noindent\textbf{Second Term of \eqref{eq:A3prob_ineq2}:} We show that  $I_2$ is bounded above by a constant for $\alpha_nn$ sufficiently large conditioned on $\mathcal{E}$. Hence, we may choose $M''$ to be larger than this constant such that $\mathbb{P}_{f_{0,n}}(I_2 > M''\cap\mathcal{E}) = 0$. First, by the change of variable $u = \frac{t}{\sqrt{\alpha_nn}}\implies dt = \sqrt{\alpha_nn}du$,
\begin{align}
    I_{2} &= \int_{2r\sqrt{\alpha_nn}}^\infty t^{k_0-1}\mathbb{P}_{\pi_{n,\alpha_n}(\theta|X^n)}\Big(\|\theta-\hat{\theta}\|_2 > \frac{t}{\sqrt{\alpha_nn}}\Big)dt \nonumber \\
    &= (\alpha_nn)^{\frac{k_0}{2}}\int_{2r}^\infty u^{k_0-1}\mathbb{P}_{\pi_{n,\alpha_n}(\theta|X^n)}(\|\theta-\hat{\theta}\|_2 > u)du. \label{eq:I2}
\end{align}
We express the probability as follows:
\begin{align*}
    \mathbb{P}_{\pi_{n,\alpha_n}(\theta|X^n)}(\|\theta-\hat{\theta}\|_2 > u) 
    &= \frac{e^{\alpha_nn\ell_n(\theta^*)}\int_{\|\theta-\hat{\theta}\|_2 > u}e^{\alpha_nn(\ell_n(\theta)-\ell_n(\theta^*))}\pi(\theta)d\theta}{e^{\alpha_nn\ell_n(\hat{\theta})}\int e^{\alpha_nn(\ell_n(\theta)-\ell_n(\hat{\theta}))}\pi(\theta)d\theta} \\
    &\leq (\alpha_nn)^{\frac{p}{2}}\kappa C_0^{-1}\int_{\substack{\|\theta-\hat{\theta}\|_2 > u}} e^{\alpha_nn(\ell_n(\theta)-\ell_n(\theta^*))}d\theta,
\end{align*}
where we have used that $\ell_n(\theta^*)\leq\ell_n(\hat{\theta})$, the bound in \eqref{eq:D_bound}, and the assumed bound on the prior. Our goal is to obtain a bound for the above that is valid for $u > 2r$. As in \eqref{eq:int1}--\eqref{eq:I12_final}, on the event $\mathcal{A}\cap\mathcal{C}$, we can bound the above as follows: for large enough $\alpha_n n$,
\begin{align}
    \mathbb{P}_{\pi_{n,\alpha_n}(\theta|X^n)}(\|\theta-\hat{\theta}\|_2 > u) 
    \nonumber&\leq (\alpha_nn)^{p/2}\kappa C_0^{-1}\int_{\substack{\|\theta-\hat{\theta}\|_2 > u}}\exp\left(\alpha_nn(\ell_n(\theta)-\ell_n(\theta^*))\right)d\theta \\
    \label{eq:I2_final_prob}&\leq \frac{\text{vol}(\mathbb{S}^{p-1})e^{Lc_2\alpha_nn}c_1^{-Lc_0\alpha_nn}(\alpha_nn)^{\frac{p}{2}}\kappa}{ C_0}\times\frac{(u/2)^{p-Lc_0\alpha_nn}}{Lc_0\alpha_nn-p}.
\end{align}
Plugging the final bound in \eqref{eq:I2_final_prob} into our expression for $I_2$ in \eqref{eq:I2} yields
\begin{align*}
    I_2 
    &= (\alpha_nn)^{\frac{k_0}{2}}\int_{2r}^\infty u^{k_0-1}\mathbb{P}_{\pi_{n,\alpha_n}(\theta|X^n)}(\|\theta-\hat{\theta}\|_2 > u)du \\
    &\leq \frac{\text{vol}(\mathbb{S}^{p-1})e^{Lc_2\alpha_nn}c_1^{-Lc_0\alpha_nn}\kappa (\alpha_nn)^{\frac{k_0+p}{2}}}{2^{p-Lc_0\alpha_nn}C_0(Lc_0\alpha_nn-p)}\int_{2r}^\infty u^{k_0-1+p-Lc_0\alpha_nn}du \\
    &= \frac{\text{vol}(\mathbb{S}^{p-1})r^{k_0+p}\kappa (\alpha_nn)^{\frac{k_0+p}{2}}}{2^p C_0(Lc_0\alpha_nn-p)(Lc_0\alpha_nn-(k_0+p))} \times \Big(\frac{e^{c_2}}{(c_1r)^{c_0}}\Big)^{L\alpha_nn}.
\end{align*}
We see that the second term in the above decays to 0 exponentially in $\alpha_nn$, since we have taken $r > \frac{e^{c_2/c_0}}{c_1}$. Hence, for some $\tau_2 > 0$, there exists $N_6 \equiv N_6(\tau_2)$ such that $I_{2} < \tau_2$ whenever $\alpha_nn > N_6$. Then taking $\alpha_nn > N_6$, we have that
\begin{align}\label{eq:I2_prob_bound}
    \mathbb{P}_{f_{0,n}}(I_2 > M'' \cap \mathcal{E}) \leq \mathbb{P}_{f_{0,n}}(\tau_2 > M'' \cap \mathcal{E}),
\end{align}
and we may take $M'' = \tau_2 + 1 < \infty$ such that the probability on the right side equals 0. \\

\noindent\textbf{Final bound:} Taking $\alpha_nn$ large enough and 
$$M'' = \max\Big\{C\int_{0}^{\infty} t^{k_0-1}e^{-c_4(-t - \sqrt{p/(2c_4)})^2}dt + \tau_1 + 1,~\tau_2 + 1\Big\} < \infty,$$
we have by the bounds in \eqref{eq:A3prob_ineq2}, \eqref{eq:I1_prob_bound}, and \eqref{eq:I2_prob_bound}, that
\begin{align*}
    \mathbb{P}_{f_{0,n}}(\mathbb{E}_{\pi_{n,\alpha_n}(\theta|X^n)}[\|\sqrt{\alpha_nn}(\theta-\theta^*)\|_2^{k_0}] > M) &\leq \mathbb{P}_{f_{0,n}}(I_1 > M''\cap\mathcal{E}) + \mathbb{P}_{f_{0,n}}(I_2 > M''\cap\mathcal{E}) + \epsilon \\
    &= \epsilon.
\end{align*}
Finally, we recover the original constant, $M$, as follows. Recall that $M'' = M'/(2k_0) \implies M'=2k_0M''$ and that $M = 2^{k_0}M' + (2\tau M_0)^{k_0} = 2^{k_0+1}k_0M'' + (2\tau M_0)^{k_0}$. Hence, $$M = 2^{k_0+1}k_0\max\Big\{C\int_{0}^{\infty} t^{k_0-1}e^{-c_4(t - \sqrt{p/(2c_4)})^2}dt + \tau_1 + 1,~\tau_2 + 1\Big\} + (2\tau M_0)^{k_0} < \infty.$$ \end{proof}

\subsection{Sufficient conditions for Assumption \textbf{(A2')}}\label{app:likelihood_suff}
Assumption \textbf{(A2')} can be thought of a higher order LAN condition. The standard LAN condition in \textbf{(A2)} is a quadratic approximation of the log-likelihood ratio centered at the pseudo-true parameter. Assumption \textbf{(A2')} is a cubic approximation of the log-likelihood ratio centered at the MLE. Proposition \ref{prop:A2_suff} allows us to check \textbf{(A2)} by putting regularity conditions on three derivatives of the log-likelihood. The following proposition allows us to check \textbf{(A2')} by imposing regularity conditions on four derivatives of the log-likelihood.

\begin{Proposition}\label{prop:likelihood_suff}
    Assume \textbf{(A0)} holds. Furthermore, assume that there exists some $\gamma> 0$ such that the following hold on $B_{\theta^*}(\gamma)$:
    \begin{enumerate}
        \item The log-likelihood function, $\log f_n(X^n|\theta)$, is four-times continuously differentiable.\
        \item Let $\frac{1}{n}\nabla^4\log f_n(X^n|\theta)\in\mathbb{R}^{p\times p\times p\times p}$ be a fourth-order tensor containing the partial derivatives of $\log f_n(X^n|\theta)$. For all $\theta$ in the closure of the open ball, there exists a function $\tilde{M}:\mathcal{X}^n\to\mathbb{R}$, where $\mathbb{E}_{f_{0,n}}|\frac{1}{n}\tilde{M}(X^n)| < \infty$, such that the entries of $\nabla^4\log f_n(X^n|\theta)$ satisfy
        $|[\nabla^4\log f_n(X^n|\theta)]_{i,j,k,l}| =|\frac{\partial^4}{\partial\theta_i\partial\theta_j\partial\theta_k\partial\theta_l}\log f_n(X^n|\theta)| < \tilde{M}(X^n).$
        \item The Hessian of $-\log f_n(X^n|\theta)$ is positive definite with probability tending to 1. That is, the smallest eigenvalue of the Hessian of $-\log f_n(X^n|\theta)$ is positive with probability tending to 1; namely, $\mathbb{P}_{f_{0,n}}(\lambda_{\min}(-\nabla^2\frac{1}{n}\log f_n(X^n|\theta)) > 0) \rightarrow 1.$
    \end{enumerate}
    Then the condition on $\tilde{R}_n(h)$ in \eqref{eq:tildeR_remainder_control} is satisfied with $\delta = \gamma/2$, $H_n = -\frac{1}{n}\nabla^2\log f_n(X^n|\hat{\theta})$ (which is symmetric and positive definite with probability tending to 1), $S_n = \nabla^3\frac{1}{n}\log f_n(X^n|\hat{\theta})$, and $M = \frac{p^2}{12\epsilon}\mathbb{E}_{f_{0,n}}[\frac{1}{n}\tilde{M}(X^n)] + 1$.
\end{Proposition}

\begin{proof}
We prove Proposition \ref{prop:likelihood_suff} in two steps. First, we take a third-order Taylor expansion of $\frac{1}{n}\log f_n(X^n|\hat{\theta}+{h}/{\sqrt{\alpha_nn}})$ around $\frac{1}{n}\log f_n(X^n|\hat{\theta})$. Second, we study the convergence of the remainder of the Taylor expansion on the set $\mathcal{H}_n$, and show how the choice of $M = \frac{p^2}{12\epsilon}\mathbb{E}_{f_{0,n}}[\frac{1}{n}\tilde{M}(X^n)] + 1$ leads to the condition of \textbf{(A2')}, as specified in \eqref{eq:tildeR_remainder_control}. \\

\noindent\textbf{Step 1: Taylor expansion of the log-likelihood.} Let $g(t) \equiv  \frac{1}{n}\log f_n(X^n|u(t))$, where $u(t) \equiv \hat{\theta}+{th}/{\sqrt{\alpha_nn}}$. Similarly to \eqref{eq:g_expansion}, a third-order Taylor expansion of $g(1)$ around $g(0)$ shows that
\begin{align}
    g(1) \label{eq:g_expansion_tildeR}&= g(0) + g'(0) + \frac{1}{2}g''(0) + \frac{1}{6}g'''(0) + \frac{1}{6}\int_0^1 (1-t)^3g''''(t)dt.
\end{align}
Note that an expansion of $g(1)$ around $g(0)$ is equivalent to an expansion of $\frac{1}{n}\log f_n(X^n|\hat{\theta}+{h}/{\sqrt{\alpha_nn}})$ around $\frac{1}{n}\log f_n(X^n|\hat{\theta})$ since $g(1)$ and $g(0)$ are as in \eqref{eq:g1g0} with $\hat{\theta}$ replacing $\theta^*$.
Furthermore, $g'(0)$, $g''(0)$, and $g'''(0)$ are as in \eqref{eq:g_derivs} with $\hat{\theta}$ replacing $\theta^*$ and
\begin{equation}\label{eq:g_derivs_tildeR}
    \begin{split}
        g''''(t) = \biggl<\frac{1}{n}\nabla^4\log f_n\big(X^n\big|\hat{\theta} + \frac{th}{\sqrt{\alpha_nn}}\big), \Big(\frac{h}{\sqrt{\alpha_nn}}\Big)^{\otimes 4}\biggr>.
    \end{split}
\end{equation}
Substituting these values into \eqref{eq:g_expansion_tildeR} yields
\begin{equation}
\begin{split}
\label{eq:Taylor_approx_tildeR}
\frac{1}{n}\log f_n\Big(X^n\Big|\hat{\theta}+\frac{h}{\sqrt{\alpha_nn}}\Big) &= \frac{1}{n}\log f_n(X^n|\hat{\theta}) + \frac{h^\top}{\sqrt{\alpha_nn}} \left[\frac{1}{n} \nabla \log f_n(X^n|\hat{\theta})\right] \\
&\qquad + \frac{1}{2} \frac{h^\top}{\sqrt{\alpha_nn}} \left[\frac{1}{n}\nabla ^2\log f_n(X^n|\hat{\theta})\right] \frac{h}{\sqrt{\alpha_nn}} \\
&\qquad +\frac{1}{6}\biggl<\frac{1}{n}\nabla^3\log f_n(X^n|\hat{\theta}), \Big(\frac{h}{\sqrt{\alpha_nn}}\Big)^{\otimes 3}\biggr> + \textrm{rem}_n(h),
\end{split}
\end{equation}
where $\textrm{rem}_n(h)$ is the remainder of the Taylor approximation given by 
\begin{equation}
\begin{split}
\label{eq:remdef}
\textrm{rem}_n(h) = \frac{1}{6}\int_0^{1} (1- t)^3 g''''(t) dt = \frac{1}{6(\alpha_n n)^2}\int_0^{1} (1- t)^3 \Bigl<\frac{1}{n}\nabla^4\log f_n\big(X^n\big|\hat{\theta} + \frac{th}{\sqrt{\alpha_nn}}\big), h^{\otimes 4}\Bigr> dt.
\end{split}
\end{equation}
Then by the definition of $\tilde{R}_n(h)$ in Assumption \textbf{(A2')} and the Taylor approximation in \eqref{eq:Taylor_approx_tildeR},
\begin{align}
    \nonumber&\tilde{R}_n(h) =\alpha_n\Big[\log f_n\big(X^n\big|\hat{\theta}+\frac{h}{\sqrt{\alpha_nn}}\big)-\log f_n(X^n|\hat{\theta})\Big]
    +\frac{1}{2}h^\top H_nh - \frac{1}{6\sqrt{\alpha_nn}}\langle S_n, h^{\otimes 3}\rangle \\
    \nonumber&=\alpha_nn\Big(\frac{h^\top}{\sqrt{\alpha_nn}}\Big[\frac{1}{n}\nabla\log f_n(X^n|\hat{\theta})\Big] + \frac{1}{2}\frac{h^{\top}}{\sqrt{\alpha_nn}}\Big[\frac{1}{n}\nabla^2\log f_n(X^n|\hat{\theta})\Big]\frac{h}{\sqrt{\alpha_nn}} \\
    \nonumber&\hspace{50pt}+\frac{1}{6}\Bigl<\frac{1}{n}\nabla^3\log f_n(X^n|\hat{\theta}), \Big(\frac{h}{\sqrt{\alpha_nn}}\Big)^{\otimes 3}\Bigr> + \textrm{rem}_n(h)\Big) +\frac{1}{2}h^\top H_nh - \frac{1}{6\sqrt{\alpha_nn}}\langle S_n, h^{\otimes 3}\rangle \\
    \label{eq:tildeRn(h)_grad}&= \frac{1}{2}h^\top\Big[\frac{1}{n}\nabla^2\log f_n(X^n|\hat{\theta}) + H_n\Big]h + \frac{1}{6\sqrt{\alpha_nn}}\Big\langle\frac{1}{n}\nabla^3\log f_n(X^n|\hat{\theta}) - S_n, h^{\otimes 3}\Big\rangle + \alpha_nn\text{rem}_n(h) \\
    \label{eq:tildeRn(h)_equal}&=\alpha_nn\textrm{rem}_n(h).
\end{align}
The equality in \eqref{eq:tildeRn(h)_grad} follows from \textbf{(A0)}, which implies that $\nabla\log f_n(X^n|\hat{\theta}) = 0$, and by combining common terms. The equality in \eqref{eq:tildeRn(h)_equal} follows with $H_n = -\frac{1}{n}\nabla^2\log f_n(X^n|\hat{\theta})$ and $S_n = \frac{1}{n}\nabla^3\log f_n(X^n|\hat{\theta})$. 

In this step, it remains to show that this choice of $H_n$ is symmetric and positive definite with $f_{0,n}$-probability tending to 1. 
Recall, we have let $\gamma$ denote the radius of the open ball such that conditions 1-3 hold. Given this $\gamma$, define the following events:
\begin{equation}
\label{eq:E0E1_def_likprop}
    \mathcal{E}_1 =\Big\{\lambda_{\min}\Big(-\nabla^2\frac{1}{n}\log f_n(X^n|\hat{\theta})\Big) > 0 \Big\}~~~~~\text{and}~~~~~\mathcal{E}_2 = \{\hat{\theta}\in B_{\theta^*}(\gamma/2)\}.\
\end{equation}
First, we  show that $H_n$ is positive definite with $f_{0,n}$-probability tending to 1. To do so, note that $\mathcal{E}_1$ is the event that we would like to control, but it only makes sense if the log-likelihood is twice differentiable at $\hat{\theta}$. The event $\mathcal{E}_2$ corresponds to $\hat{\theta}$ lying in a neighborhood of $\theta^*$ where $\nabla^2\log f_n(X^n|\theta)$ is continuous. In fact, $\nabla^2\log f_n(X^n|\theta)$ is assumed in condition 1 to be continuous on a larger ball, $B_{\theta^*}(\gamma)$, but the more restricted event $\mathcal{E}_2$ will be useful in Step 2 of the proof. 

We will show that for any $\epsilon > 0$, there exists an integer $N_0$ such that for $\alpha_nn > N_0$,
\begin{equation}\label{eq:HnPSDprob}
    \mathbb{P}_{f_{0,n}}(\mathcal{E}_1 \cap \mathcal{E}_2) \geq 1-\epsilon.
\end{equation}
By Assumption \textbf{(A0)}, there exists an integer $N_{00} \equiv N_{00}(\epsilon,\gamma)$ such that for $\alpha_nn > N_{00}$,
\begin{equation}
\label{eq:E1c_prob_likprop}
    \mathbb{P}_{f_{0,n}}(\mathcal{E}_2^\mathsf{c}) = \mathbb{P}_{f_{0,n}}\left(\|\hat{\theta}-\theta^*\|_2 \geq \frac{\gamma}{2}\right) < \frac{\epsilon}{4} < \frac{\epsilon}{2}.
\end{equation}
Furthermore, by the condition 3 assumption, there exists some $N_{01}\equiv N_{01}(\epsilon,\gamma)$ such that for $\alpha_nn > N_{01}$, we have
$\mathbb{P}_{f_{0,n}}(\mathcal{E}_1^{\mathsf{c}}\cap\mathcal{E}_2) < \frac{\epsilon}{4}.$
Indeed, we assume that $\lambda_{\text{min}}(-\nabla^2\frac{1}{n}\log f_n(X^n|\theta)) > 0$ with $f_{0,n}$-probability tending to 1, for $\theta\in B_{\theta^*}(\gamma)$ (condition 3). In $\mathcal{E}_1$, we are evaluating $\lambda_{\text{min}}(-\nabla^2\frac{1}{n}\log f_n(X^n|\theta))$ at $\theta=\hat{\theta}$, so we split the probability $\mathbb{P}(\mathcal{E}_1^\mathsf{c})$ according to $\hat{\theta}$ either belonging to $B_{\theta^*}(\gamma)$ or not. On the event, $\mathcal{E}_1^\mathsf{c}\cap\{\hat{\theta}\in B_{\theta^*}(\gamma)\}$, we appeal to condition 3.
Hence, with \eqref{eq:E1c_prob_likprop}, whenever $\alpha_nn > N_0 \equiv \max(N_{00}, N_{01})$,
\begin{align}
\label{eq:E0c_likprop}
\mathbb{P}_{f_{0,n}}\left(\mathcal{E}_1^\mathsf{c}\right) 
        \leq \mathbb{P}_{f_{0,n}}(\mathcal{E}_1^{\mathsf{c}}\cap\mathcal{E}_2)  + \mathbb{P}_{f_{0,n}}(\mathcal{E}_2^\mathsf{c})
        < \frac{\epsilon}{4} + \frac{\epsilon}{4} = \frac{\epsilon}{2}.
\end{align}
By \eqref{eq:E1c_prob_likprop} and \eqref{eq:E0c_likprop} with a union bound, for $\alpha_nn > N_0$, we have shown \eqref{eq:HnPSDprob}.
Therefore, we conclude that $-\frac{1}{n}\nabla^2\log f_n(X^n|\hat{\theta})$ is positive definite tending to 1.  Now, we show that $H_n$ is symmetric with $f_{0,n}$-probability tending to 1. Indeed, this is the case on event $\mathcal{E}_2$ defined in \eqref{eq:E0E1_def_likprop}. On $\mathcal{E}_2$, we have $\hat{\theta}\in B_{\theta^*}(\gamma)$. Then, by continuity of $-\frac{1}{n}\nabla^2\log f_n(X^n|\theta)$ on $B_{\theta^*}(\gamma)$,
$$[H_n]_{ij} = -\frac{\partial^2}{\partial\theta_i\partial\theta_j}\frac{1}{n}\nabla^2\log f_n(X^n|\theta)\bigg|_{\theta=\hat\theta} = -\frac{\partial^2}{\partial\theta_j\partial\theta_i}\frac{1}{n}\nabla^2\log f_n(X^n|\theta)\bigg|_{\theta=\hat\theta} = [H_n]_{ji}$$
on $\mathcal{E}_2$, which follows from Clairaut's theorem. By this and \eqref{eq:E1c_prob_likprop}, we have shown that $H_n$ is symmetric with $f_{0,n}$-probability tending to 1.

\noindent\textbf{Step 2: Analysis of $\mathbf{\boldsymbol{\alpha}_nn \textrm{rem}_n(h)}$}. Recall that $\gamma$ is the radius of the open ball such that conditions 1-3 hold. In this step, we aim to show that for all $\epsilon > 0$ and $\alpha_nn$ sufficiently large,
\begin{equation}
\label{eq:prop6_step2}
    \mathbb{P}_{f_{0,n}}(\mathcal{E}') < \epsilon,
\end{equation}
where
$\mathcal{E'} = \{\sup_{h\in\mathcal{H}_n} [\alpha_nn \textrm{rem}_n(h) - \frac{M}{\alpha_nn}\|h\|_2^4] > 0\}$
with $\mathcal{H}_n = \{h\in\mathbb{R}^p|\frac{\|h\|_2}{\sqrt{\alpha_nn}} \leq \delta\}$ where $\delta = \gamma/2$ and $M = \frac{p^2}{12\epsilon}\mathbb{E}_{f_{0,n}}[\frac{1}{n}\tilde{M}(X^n)] + 1$. Recall the event $\mathcal{E}_2$ defined in \eqref{eq:E0E1_def_likprop} and define
$\mathcal{E}_3 = \{\hat{\theta} + \frac{h}{\sqrt{\alpha_nn}}\in B_{\theta^*}(\gamma)\}.$
Notice, $\mathcal{E}_3^\mathsf{c}\subset\mathcal{E}_2^\mathsf{c}$ since   on event $\mathcal{E}_3^\mathsf{c}$, for all $h\in\mathcal{H}_n$,  we also have
\begin{equation*}
    \gamma \leq \Big\|\hat{\theta} + \frac{h}{\sqrt{\alpha_nn}} - \theta^*\Big\|_2\leq \|\hat{\theta}-\theta^*\|_2+\frac{\gamma}{2} \iff \hat\theta\notin B_{\theta^*}(\gamma/2).
\end{equation*}
Using this fact and \eqref{eq:E1c_prob_likprop}, it follows that taking $\alpha_nn > N_{00}$ we obtain
$ \mathbb{P}_{f_{0,n}}(\mathcal{E}_3^\mathsf{c})\leq 
\mathbb{P}_{f_{0,n}}\left(\mathcal{E}_2^\mathsf{c}\right) < \frac{\epsilon}{4}.
$
Now we define $\mathcal{E} \equiv \mathcal{E}_2 \cap \mathcal{E}_3$ and notice that from a union bound with the bound established in \eqref{eq:E1c_prob_likprop}, we have
$\mathbb{P}_{f_{0,n}}(\mathcal{E}^{\mathsf{c}})  = \mathbb{P}_{f_{0,n}}((\mathcal{E}_2\cap\mathcal{E}_3)^\mathsf{c}) \leq \mathbb{P}_{f_{0,n}}(\mathcal{E}_2^\mathsf{c}) + \mathbb{P}_{f_{0,n}}(\mathcal{E}_3^\mathsf{c}) \leq  \frac{\epsilon}{4} +\frac{\epsilon}{4}= \frac{\epsilon}{2}.$
Taking $\alpha_nn > N_{00}$, we can bound the probability in \eqref{eq:prop6_step2} as follows:
\begin{equation}\label{eq:probE0E1E2_arg_likprop}
    \begin{split}
    \quad\mathbb{P}_{f_{0,n}}(\mathcal{E}')
    &\leq \mathbb{P}_{f_{0,n}}(\mathcal{E}'  \cap  \mathcal{E}) + \mathbb{P}_{f_{0,n}}(\mathcal{E}^\mathsf{c})
    \leq \mathbb{P}_{f_{0,n}}(\mathcal{E}'  \cap  \mathcal{E}) + \frac{\epsilon}{2}.
    \end{split}
\end{equation}
It therefore remains to show that $\mathbb{P}_{f_{0,n}}(\mathcal{E}' \cap \mathcal{E}) < \frac{\epsilon}{2}$ for our choice of $M$.

First, we argue that $\textrm{rem}_n(h)$ is well defined over $\mathcal{H}_n$ on the event $\mathcal{E}$. That is, we will show that for all $0\leq t\leq 1$ and all $h\in\mathbb{R}^p$ such that ${\|h\|_2}/{\sqrt{\alpha_nn}} \leq \delta$, the vector $\hat{\theta} + {th}/{\sqrt{\alpha_nn}}$ belongs to $B_{\theta^*}(\gamma)$, the neighborhood of $\theta^*$ where $\log f_n(X^n|\theta)$ is four-times continuously differentiable. First, notice that for all $h\in\mathcal{H}_n$, we have that $\frac{t\|h\|_2}{\sqrt{\alpha_nn}} \leq \delta = \frac{\gamma}{2}$. Therefore, 
\begin{align}\label{eq:in_the_ball2}
    \Big\|\hat{\theta} + \frac{t h}{\sqrt{\alpha_nn}} - \theta^*\Big\|_2
    \leq \|\hat{\theta}-\theta^*\|_2 + \frac{t\|h\|_2}{\sqrt{\alpha_nn}}
    \leq \|\hat{\theta}-\theta^*\|_2 + \frac{\gamma}{2}
    < \frac{\gamma}{2} + \frac{\gamma}{2} = \gamma,
\end{align}
where we have also used that we are on the event $\mathcal{E}_2 = \{\hat{\theta}\in B_{\theta^*}(\gamma/2)\}$.\ We have established that $\hat{\theta} + th/\sqrt{\alpha_nn} \in B_{\theta^*}(\gamma)$. Furthermore, we have that $\nabla^4\log f_n(X^n|\theta)$ is continuous on $B_{\theta^*}(\gamma)$ by condition 1; hence, $\textrm{rem}_n(h)$ defined in \eqref{eq:remdef} is well defined on $\mathcal{E}$. Furthermore,  on $\mathcal{E}$, we can upper bound $|\alpha_nn\textrm{rem}_n(h)|$ for $h\in\mathcal{H}_n$  as  in \eqref{eq:rem_ineq1} using Lemma \ref{lem:l1_l2_tensor_bound}:
\begin{align}
|\alpha_nn\textrm{rem}_n(h)| 
    \label{eq:rem_ineq1_likprop}&\leq  \frac{p^2  \|h\|_2^4}{24\alpha_nn} \sup_{t\in[0,1]} \max_{1 \leq i,j,k,l \leq p} \Big| \Big[\frac{1}{n}\nabla^4\log f_n\big(X^n\big|\hat{\theta} + \frac{th}{\sqrt{\alpha_nn}}\big)\Big]_{i,j,k,l}\Big| .
\end{align}
Next, recall the definition of $\mathcal{E}'$ from \eqref{eq:prop6_step2}; namely, $\mathcal{E'} = \{\sup_{h\in\mathcal{H}_n} [\alpha_nn \textrm{rem}_n(h) - \frac{M}{\alpha_nn}\|h\|_2^4] > 0\}$. Notice that, using the bound in \eqref{eq:rem_ineq1_likprop}, we have
\begin{align}
    \nonumber&\sup_{h\in\mathcal{H}_n} \Big[\alpha_nn\textrm{rem}_n(h) - \frac{M \|h\|_2^4}{\alpha_n n} \Big]  \\
    \nonumber&\leq \sup_{h\in\mathcal{H}_n} \Big[\frac{\|h\|_2^4}{\alpha_nn} \Big( \frac{p^2}{24}\sup_{t\in[0,1]}  \max_{1 \leq i,j,k,l \leq p} \Big| \Big[\frac{1}{n}\nabla^4\log f_n\big(X^n\big|\hat{\theta} + \frac{th}{\sqrt{\alpha_nn}}\big)\Big]_{i,j,k,l}\Big|  - M\Big) \Big]\\
    \label{eq:A2_R'n_bound_ineq2}&\leq  \delta^4\alpha_nn   \Big(\frac{p^2}{24}\sup_{\substack{ t\in[0,1] \\h\in\mathcal{H}_n}} \max_{1 \leq i,j,k,l \leq p} \Big| \Big[\frac{1}{n}\nabla^4\log f_n\big(X^n\big|\hat{\theta} + \frac{th}{\sqrt{\alpha_nn}}\big)\Big]_{i,j,k,l}\Big|  - M\Big),
\end{align}
where the inequality in \eqref{eq:A2_R'n_bound_ineq2} follows from applying the bound $\|h\|_2 \leq \delta \sqrt{\alpha_nn}$, which follows from the definition of $\mathcal{H}_n$. Now we use the bound in \eqref{eq:A2_R'n_bound_ineq2} to upper  bound $\mathbb{P}_{f_{0,n}}(\mathcal{E}'\cap\mathcal{E})$:
\begin{align}
    \nonumber&\mathbb{P}_{f_{0,n}}(\mathcal{E}'\cap\mathcal{E})=\mathbb{P}_{f_{0,n}}\Big(\Big\{\sup_{h\in\mathcal{H}_n} \Big[\alpha_nn\textrm{rem}_n(h) - \frac{M}{\alpha_n n}\|h\|_2^4 \Big] > 0\Big\}\cap\mathcal{E}\Big) \\
    \nonumber &\leq\mathbb{P}_{f_{0,n}}\Big(\Big\{ \delta^4\alpha_nn \Big(\frac{p^2}{24}\sup_{\substack{ t\in[0,1] \\h\in\mathcal{H}_n}} \max_{1 \leq i,j,k,l \leq p} \Big| \Big[\frac{1}{n}\nabla^4\log f_n\big(X^n\big|\hat{\theta} + \frac{th}{\sqrt{\alpha_nn}}\big)\Big]_{i,j,k,l}\Big|  - M\Big) > 0\Big\}\cap\mathcal{E}\Big) \\
    \label{eq:EEbound}&=\mathbb{P}_{f_{0,n}}\Big(\Big\{\frac{p^2}{24}\sup_{\substack{ t\in[0,1] \\h\in\mathcal{H}_n}} \max_{1 \leq i,j,k,l \leq p} \Big| \Big[\frac{1}{n}\nabla^4\log f_n\big(X^n\big|\hat{\theta} + \frac{th}{\sqrt{\alpha_nn}}\big)\Big]_{i,j,k,l}\Big| > M\Big\}\cap\mathcal{E}\Big).
\end{align}
Next, $\hat{\theta} + {th}/{\sqrt{\alpha_nn}}\in B_{\theta^*}(\gamma)$  for all $h\in\mathcal{H}_n$  by the argument in \eqref{eq:in_the_ball2}.
Therefore,
\begin{align*}
    &\sup_{\substack{t\in[0,1] \\ h\in\mathcal{H}_n}}  \max_{1\leq i,j,k,l \leq p} \Big| \Big[\frac{1}{n}\nabla^4\log f_n\big(X^n\big|\hat{\theta} + \frac{th} {\sqrt{\alpha_nn}}\big)\Big]_{i,j,k,l}\Big| \\
    & \qquad \leq  \sup_{\theta\in B_{\theta^*}(\gamma)} \max_{1 \leq i,j,k,l \leq p} \Big| \Big[\frac{1}{n}\nabla^4\log f_n\big(X^n\big|\theta\big)\Big]_{i,j,k,l}\Big|.
\end{align*}
We apply the above to the bound on $\mathbb{P}_{f_{0,n}}(\mathcal{E}'\cap\mathcal{E})$ from \eqref{eq:EEbound} and use Markov's Inequality:
\begin{align}
     \nonumber \mathbb{P}_{f_{0,n}}(\mathcal{E}'\cap\mathcal{E}) 
    \nonumber&\leq\mathbb{P}_{f_{0,n}}\Big(\frac{p^2}{24}\sup_{\theta\in B_{\theta^*}(\gamma)} \max_{1 \leq i,j,k,l \leq p} \Big| \Big[\frac{1}{n}\nabla^4\log f_n\big(X^n\big|\theta\big)\Big]_{i,j,k,l}\Big| > M\Big) \\
     \nonumber &\leq \frac{p^2}{24M}\mathbb{E}_{f_{0,n}}\Big[ \sup_{\theta\in B_{\theta^*}(\gamma)} \max_{1 \leq i,j,k,l \leq p} \Big| \Big[\frac{1}{n}\nabla^4\log f_n\big(X^n\big|\theta\big)\Big]_{i,j,k,l}\Big|\Big] \\
    \label{eq:sup_max_likprop_expectation_bound} &\leq \frac{p^2}{24M}\mathbb{E}_{f_{0,n}}\Big[\frac{1}{n}\tilde{M}(X^n)\Big] < \frac{\epsilon}{2}.
\end{align}
The first inequality in \eqref{eq:sup_max_likprop_expectation_bound} follows from condition 2 and the second  from choosing $M$ such that $M > \frac{p^2}{12\epsilon}\mathbb{E}_{f_{0,n}}[\frac{1}{n}\tilde{M}(X^n)]$. For example, choose $M = \frac{p^2}{12\epsilon}\mathbb{E}_{f_{0,n}}[\frac{1}{n}\tilde{M}(X^n)] + 1$, which is finite by condition 2. The bounds in \eqref{eq:probE0E1E2_arg_likprop} and \eqref{eq:sup_max_likprop_expectation_bound} give the desired result in \eqref{eq:prop6_step2}.

\textbf{Conclusion.} The result in \eqref{eq:prop6_step2} and the equality $\tilde{R}_n(h) = \alpha_nn\text{rem}_n(h)$ shown in Step 1 (see \eqref{eq:tildeRn(h)_equal}) establish that the desired representation for $\tilde{R}_n(h)$ is satisfied with $\delta=\gamma/2$, $H_n = -\frac{1}{n}\nabla^2\log f_n(X^n|\hat{\theta})$ (which is symmetric and positive definite with probability tending to 1, as shown in \eqref{eq:HnPSDprob}, $S_n = \nabla^3\frac{1}{n}\log f_n(X^n|\hat{\theta})$, and $M = \frac{p^2}{12\epsilon}\mathbb{E}_{f_{0,n}}[\frac{1}{n}\tilde{M}(X^n)] + 1$.
\end{proof}

\subsection{Sufficient conditions for Assumption \textbf{(ALap)} (Proof of Proposition \ref{prop:prior_suff})}\label{app:prior_suff}

\begin{proof}
Let $g(t) \equiv  g(u(t))$, where $u(t) \equiv \hat{\theta}+{th}/{\sqrt{\alpha_nn}}$. A first-order Taylor expansion of $g(1) = b(\hat{\theta} + {h}/{\sqrt{\alpha_nn}})$ around $g(0) = b(\hat{\theta})$ shows that
$g(1)= g(0) + g'(0) + \int_0^1 (1-t)g''(t)dt.
$
Using that
\begin{equation*}
    \begin{split}
        g'(0) &= \frac{h^\top}{\sqrt{\alpha_nn}} \nabla b(\hat{\theta}), \,\,\, \text{ and } \,\,\, g''(t) = \frac{h^\top}{\sqrt{\alpha_nn}} \Big[\nabla^2 b\Big(\hat{\theta} + \frac{th}{\sqrt{\alpha_nn}}\Big)\Big] \frac{h}{\sqrt{\alpha_nn}},
    \end{split}
\end{equation*}
we find
\begin{equation}
\begin{split}
\label{eq:Taylor_approx_b}
    b\Big(\hat{\theta} + \frac{h}{\sqrt{\alpha_nn}}\Big) &= b(\hat{\theta}) + \frac{h^\top\nabla b(\hat{\theta}) }{\sqrt{\alpha_nn}} + \textrm{rem}_n(h); \\
    \textrm{rem}_n(h) &:= \int_0^1 \frac{(1-t)}{\alpha_nn}  h^\top \Big[\nabla^2 b\Big(\hat{\theta} + \frac{th}{\sqrt{\alpha_nn}}\Big)\Big] h dt.
\end{split}
\end{equation}
By the definition of $R_b(h)$ in Assumption \textbf{(ALap)} and the Taylor approximation in \eqref{eq:Taylor_approx_b},
\begin{align}
        R_b(h)
        =\Big[b\Big(\hat{\theta} + \frac{h}{\sqrt{\alpha_n n}}\Big) - b(\hat{\theta})\Big] - \frac{v^\top h}{\sqrt{\alpha_n n}} &=\frac{h^\top  \nabla b(\hat{\theta})}{\sqrt{\alpha_nn}} + \textrm{rem}_n(h) - \frac{v^\top h}{\sqrt{\alpha_n n}} =\textrm{rem}_n(h), 
    \label{eq:rem_equiv}
\end{align}
where the final inequality follows by setting $v=\nabla b(\hat{\theta}) = \pi(\hat\theta)\nabla q(\hat\theta) + q(\hat\theta)\nabla\pi(\hat\theta)$.
Because $R_b(h) = \textrm{rem}_n(h)$ as shown above, to complete the proof, we prove that for all $\epsilon > 0$,
        \begin{equation}\label{eq:toshow_prop2}
            \mathbb{P}_{f_{0,n}}\Big(\sup_{h\in\mathcal{H}_{n,b}} \textrm{rem}_n(h) - \frac{M}{\alpha_n n}\|h\|_2^2 > 0\Big) < \epsilon,
        \end{equation}
 when $M = p\sup_{\theta\in B_{\theta^*}(\gamma)}  \max_{1 \leq i,j \leq p} | [\nabla^2b\left(\theta\right)]_{i,j}|$ for $\alpha_nn$ sufficiently large. In what follows we denote $ \mathcal{E}' = \{\sup_{h\in\mathcal{H}_{n,b}}\textrm{rem}_n(h) - \frac{M}{\alpha_n n}\|h\|_2^2 > 0\}$ so that \eqref{eq:toshow_prop2} reads $ \mathbb{P}_{f_{0,n}}(\mathcal{E}') < \epsilon$.

Define events 
    $\mathcal{E}_1\equiv\{\hat{\theta}\in B_{\theta^*}(\gamma/2)\}$ and $\mathcal{E}_2\equiv\{\hat{\theta} + {h}/{\sqrt{\alpha_nn}}\in B_{\theta^*}(\gamma)\}.$
By Assumption \textbf{(A0)}, there exists an $N_{0} \equiv N_{0}(\epsilon,\gamma)$ such that whenever $\alpha_nn > N_{0}$, we have
    $\mathbb{P}_{f_{0,n}}(\mathcal{E}_1^\mathsf{c}) = \mathbb{P}_{f_{0,n}}(\|\hat{\theta}-\theta^*\|_2 \geq \frac{\gamma}{2}) < \frac{\epsilon}{2}.$
Next, note that $\mathcal{E}_2^\mathsf{c} \subset \mathcal{E}_1^\mathsf{c}$, as
$\gamma \leq \|\hat{\theta} + {h}/{\sqrt{\alpha_nn}} - \theta^*\|_2 \leq \|\hat{\theta}-\theta^*\|_2 + \frac{\gamma}{2}$ is true if and only if $\hat{\theta}\not\in B_{\theta^*}(\gamma/2).$
Hence, taking $\alpha_nn > N_{0}$, we have that $\mathbb{P}_{f_{0,n}}(\mathcal{E}_2^\mathsf{c}) \leq \mathbb{P}_{f_{0,n}}(\mathcal{E}_1^\mathsf{c}) < \frac{\epsilon}{2}$ as well. Now we define $\mathcal{E} \equiv \mathcal{E}_1\cap\mathcal{E}_2$ and from a union bound we have $\mathbb{P}_{f_{0,n}}(\mathcal{E}^{\mathsf{c}}) \leq \epsilon$ when $\alpha_nn > N_{0}$.
Hence, taking $\alpha_nn > N_{0}$,  we can bound the probability in \eqref{eq:toshow_prop2}:
\begin{equation*}
    \begin{split}
\mathbb{P}_{f_{0,n}}\hspace{-2pt}\Big(\hspace{-1pt}\sup_{h\in\mathcal{H}_{n,b}} \hspace{-3pt} \textrm{rem}_n(h) - \frac{M\|h\|_2^2 }{\alpha_n n} > \hspace{-1pt} 0 \hspace{-1pt}\Big) \hspace{-2pt} = \hspace{-1pt}\mathbb{P}_{f_{0,n}}(\mathcal{E}')
    &\leq \mathbb{P}_{f_{0,n}}(\mathcal{E}' \cap \mathcal{E}) \hspace{-1pt} +\hspace{-1pt}  \mathbb{P}_{f_{0,n}}(\mathcal{E}^\mathsf{c})
    \leq \mathbb{P}_{f_{0,n}}(\mathcal{E}'  \cap \mathcal{E}) + \epsilon.
    \end{split}
\end{equation*}
To complete the proof, we will therefore show that $\mathbb{P}_{f_{0,n}}(\mathcal{E}'  \cap \mathcal{E})=0$ for our choice of $M$.

First, we notice that by an argument like that around \eqref{eq:in_the_ball2}, on the event $\mathcal{E}$, we have that $\hat{\theta} + {th}/{\sqrt{\alpha_nn}} \in B_{\theta^*}(\gamma)$; hence, $\textrm{rem}_n(h)$ is well defined over $\mathcal{H}_{n,b}$. Next, using Lemma \ref{lem:l1_l2_tensor_bound} and the intersection with event $\mathcal{E}$, we can bound $\textrm{rem}_n(h)$ on $\mathcal{H}_{n,b}$ in a similar way as is done in \eqref{eq:rem_ineq1}, leading to 
\begin{align*}
    \textrm{rem}_n(h)
  &\leq  \frac{p \|h\|_2^2}{2\alpha_nn} \sup_{t\in[0,1]}  \max_{1 \leq i,j \leq p} \Big| \Big[\nabla^2b\Big(\hat{\theta} + \frac{th}{\sqrt{\alpha_nn}}\Big)\Big]_{i,j}\Big| .
\end{align*}
Next, since $\hat{\theta} + {th}/{\sqrt{\alpha_nn}}\in B_{\theta^*}(\gamma)$, by an argument like that in \eqref{eq:rem_ineq3} (or right after \eqref{eq:EEbound}),
\begin{align}
    \sup_{h\in\mathcal{H}_{n,b}} \textrm{rem}_n(h)
    \label{eq:remb_ineq3}&\leq \sup_{h\in\mathcal{H}_{n,b}} \frac{p \|h\|_2^2 }{2\alpha_nn} \sup_{\theta\in B_{\theta^*}(\gamma)}  \max_{1 \leq i,j \leq p} \Big| [\nabla^2b(\theta)]_{i,j}\Big| .
\end{align}
We finally apply \eqref{eq:remb_ineq3} to bound $\mathbb{P}_{f_{0,n}}(\mathcal{E}'\cap\mathcal{E})$.
\begin{align}
    \mathbb{P}_{f_{0,n}}(\mathcal{E}'\cap\mathcal{E}) \nonumber&=\mathbb{P}_{f_{0,n}}\Big(\sup_{h\in\mathcal{H}_{n,b}}\textrm{rem}_n(h) - \frac{M \|h\|_2^2}{\alpha_n n} > 0 \, \cap \, \mathcal{E}\Big) \\
    \nonumber&\leq \mathbb{P}_{f_{0,n}}\Big(\sup_{h\in\mathcal{H}_{n,b}} \frac{\|h\|_2^2 }{\alpha_nn} \Big(\frac{p}{2} \sup_{\theta\in B_{\theta^*}(\gamma)}  \max_{1 \leq i,j \leq p} \left| [\nabla^2b(\theta)]_{i,j}\right| - M\Big) > 0 \, \cap \, \mathcal{E}\Big) \\
    \label{eq:PE'E=0_prop4}&\leq \mathbb{P}_{f_{0,n}}\Big(\frac{\gamma^2}{4}\Big(\frac{p}{2} \sup_{\theta\in B_{\theta^*}(\gamma)}  \max_{1 \leq i,j \leq p} \left|[\nabla^2b(\theta)]_{i,j}\right| - M\Big) > 0 \, \cap \, \mathcal{E}\Big) = 0.
\end{align}
The final inequality follows as $\mathcal{H}_{n,b}=\{h\in\mathbb{R}^p|\frac{\|h\|_2}{\sqrt{\alpha_nn}} \leq \delta\}$ with $\delta = \gamma/2$ and the final equality  follows from plugging in $M = p\sup_{\theta\in B_{\theta^*}(\gamma)}  \max_{1 \leq i,j \leq p} | [\nabla^2b(\theta)]_{i,j}|$. 
\end{proof}

\subsection{Verifying the conditions of our main results for $\tilde{f}_n(X^n|\theta)$}
In this section, we show that if the conditions of Theorems \ref{thm:moments_alt} and \ref{thm:BayesEstimator} are satisfied by $f_n(X^n|\theta)$,
then the tilted likelihood, $\tilde{f}_n(X^n|\theta) = f_n(X^n|\theta)^{\frac{\hat{\alpha}_n}{\alpha_n}}$ -- where $\alpha_n$ is such that $\frac{1}{n}\ll\alpha_n\ll1$ and $\frac{\hat{\alpha}_n}{\alpha_n}\to1$ in $f_{0,n}$-probability -- also satisfies the assumptions with the same values of $\hat{\theta}$, $\theta^*$, and $V_{\theta^*}$. This allows us to show that the results of Theorems \ref{thm:moments_alt} and \ref{thm:BayesEstimator} hold for the $\hat{\alpha}_n$-posterior by replacing $f_n(X^n|\theta)$ with $\tilde{f}_n(X^n|\theta)$. We note that assumptions \textbf{(A1)} and \textbf{(A1')} are assumptions on the prior only, so we do not have to check them.

The first result checks the Theorem \ref{thm:moments_alt} assumptions \textbf{(A0)} and  \textbf{(A2)} and shows that both hold for $\tilde{f}_n(X^n|\theta)$ with the same $\hat{\theta}$, $\theta^*$, and $V_{\theta^*}$ specified in these assumptions for $f_n(X^n|\theta)$.

\begin{Proposition}\label{prop:random_alpha_(A0)(A2)}
    Suppose $f_n(X^n|\theta)$ satisfies assumptions \textbf{(A0)} and \textbf{(A2)}. Furthermore, suppose that for a data-dependent sequence, $\hat{\alpha}_n$, there exists a sequence $\frac{1}{n}\ll\alpha_n\ll1$ such that ${\hat{\alpha}_n}/{\alpha_n}\to1$ in $f_{0,n}$-probability. Then the tilted likelihood $\tilde{f}_n(X^n|\theta) = f_n(X^n|\theta)^{\frac{\hat{\alpha}_n}{\alpha_n}}$ also satisfies \textbf{(A0)} and \textbf{(A2)} with the same $\hat{\theta}$, $\theta^*$, and $V_{\theta^*}$.
\end{Proposition}
\begin{proof}
    We first check that $\tilde{f}_n(X^n|\theta)$ satisfies Assumption \textbf{(A0)}. In particular, we will show that the MLE with respect to $\tilde{f}_n(X^n|\theta)$ (which we denote by $\tilde{\theta}$) is equal to the MLE with respect to $f_n(X^n|\theta)$ (denoted $\hat{\theta}$). Indeed, since $\hat{\alpha}_n/\alpha_n$ does not depend on $\theta$, we have that
    \begin{align*}
        \tilde{\theta} = \arg\max_\theta\log\tilde{f}_n(X^n|\theta)
        = \arg\max_\theta\frac{\hat{\alpha}_n}{\alpha_n} \log f_n(X^n|\theta)
        = \arg\max_\theta\log f_n(X^n|\theta) = \hat{\theta}.
    \end{align*}

    We next check that $\tilde{f}_n(X^n|\theta)$ satisfies Assumption \textbf{(A2)}. In particular, for $V_{\theta^*}$ originally specified in \textbf{(A2)} for $f_n(X^n|\theta)$, we define as $\tilde{R}_n(h)$ the value in \eqref{eq:Rn(h)_def} with $f_n(X^n|\theta)$ replaced with $\tilde{f}_n(X^n|\theta)$. By assumption, $\sup_{h\in\bar{B}_0(r)}| R_n(h)|\to0$ for any $r > 0$ in $f_{0,n}$-probability, and now  we show that $\sup_{h\in\bar{B}_0(r)}| \tilde{R}_n(h)|\to0$ for any $r > 0$, in $f_{0,n}$-probability. Notice that
    \begin{align*}
        |\tilde R_n(h)| &= \Big|\alpha_n\Big[\log \tilde{f}_n\big(X^n\big|\theta^*+\frac{h}{\sqrt{\alpha_nn}}\big)-\log \tilde{f}_n(X^n|\theta^*)\Big] -\sqrt{\alpha_n}h^\top V_{\theta^*}\Delta_{n,\theta^*}+\frac{1}{2}h^\top V_{\theta^*}h\Big| \\
        &= \Big|\frac{\hat{\alpha}_n}{\alpha_n} \Big(\alpha_n\Big[\hspace{-1pt} \log f_n\big(X^n\big|\theta^*+\frac{h}{\sqrt{\alpha_nn}}\big)-\log f_n(X^n|\theta^*)\Big] \Big) -\sqrt{\alpha_n}h^\top V_{\theta^*}\Delta_{n,\theta^*}+\frac{1}{2}h^\top V_{\theta^*}h\Big| \\
        &= \Big|\frac{\hat{\alpha}_n}{\alpha_n} R_n(h)  - \Big(1 - \frac{\hat{\alpha}_n}{\alpha_n}\Big) \sqrt{\alpha_n}h^\top V_{\theta^*}\Delta_{n,\theta^*}+ \Big(1 - \frac{\hat{\alpha}_n}{\alpha_n}\Big)\frac{1}{2}h^\top V_{\theta^*}h\Big| \\
        &\leq \frac{\hat{\alpha}_n}{\alpha_n}|R_n(h)|  +\left|1-\frac{\hat{\alpha}_n}{\alpha_n}\right|\left|\sqrt{\alpha_n}h^\top V_{\theta^*}\Delta_{n,\theta^*}\right|+\left|1-\frac{\hat{\alpha}_n}{\alpha_n}\right|\left|\frac{1}{2}h^\top V_{\theta^*}h\right|.
    \end{align*}
    Now, since $\frac{\hat{\alpha}_n}{\alpha_n}=1+o_{f_{0,n}}(1)$ and $\sup_{h\in\bar{B}_0(r)}|R_n(h)|=o_{f_{0,n}}(1)$, using the above we have
    \begin{align*}
        \sup_{h\in\bar{B}_0(r)}|\tilde R_n(h)|
        &\leq  o_{f_{0,n}}(1)+o_{f_{0,n}}(1)\Big( \sup_{h\in\bar{B}_0(r)} \Big\{\big|\sqrt{\alpha_n}h^\top V_{\theta^*}\Delta_{n,\theta^*}\big|+\big|\frac{1}{2}h^\top V_{\theta^*}h\big| \Big\}\Big)=o_{f_{0,n}}(1).
    \end{align*}
     In the last equality we used that the term in parentheses is $O_{f_{0,n}}(1)$. Indeed,
     by Lemma \ref{lem:l1_l2_tensor_bound},
    \begin{align*}
        \big|\frac{1}{2}h^\top V_{\theta^*}h\big| \leq \frac{1}{2}\sum_{1\leq i,j\leq p}[V_{\theta^*}]_{ij}|h_ih_j| \leq \max_{1\leq i,j\leq p}[V_{\theta^*}]_{ij} \frac{1}{2}\sum_{1\leq i,j\leq p}|h_ih_j| \leq \frac{p}{2}\|h\|_2^2\max_{1\leq i,j\leq p}[V_{\theta^*}]_{ij}.
    \end{align*}
    Using the above and the Cauchy-Schwarz inequality, we find
    \begin{align}
     \nonumber   \sup_{h\in\bar{B}_0(r)} \hspace{-3pt} \Big\{\big|\sqrt{\alpha_n}h^\top V_{\theta^*}\Delta_{n,\theta^*}\big| \hspace{-1pt} + \hspace{-1pt} \big|\frac{h^\top V_{\theta^*}h}{2}\big| \Big\}
        &\leq \sup_{h\in\bar{B}_0(r)} \hspace{-3pt} \Big\{\sqrt{\alpha_n}\|h\|_2\|V_{\theta^*}\Delta_{n,\theta^*}\|_2 + \frac{p\|h\|_2^2}{2} \hspace{-1.05pt} \max_{1\leq i,j\leq p}[V_{\theta^*}]_{ij} \Big\}\\
        \label{eq:A2_sup_bound_random2}&\leq \sqrt{\alpha_n}r\|V_{\theta^*}\Delta_{n,\theta^*}\|_2 + \frac{p r^2}{2}\max_{1\leq i,j\leq p}[V_{\theta^*}]_{ij}.
    \end{align}
     Since $\alpha_n\to0$ and $\Delta_{n,\theta^*}=O_{f_{0,n}}(1)$, we have that \eqref{eq:A2_sup_bound_random2} is $O_{f_{0,n}}(1)$.
\end{proof}

The second result  checks Assumption \textbf{(A3)} when $f_n(X^n|\theta)$ satisfies the conditions of Proposition \ref{prop:A3_suff}, which gives conditions on the likelihood to check \textbf{(A3)}. In the proposition below, we show that the conditions of Proposition \ref{prop:A3_suff} are satisfied by $\tilde{f}_n(X^n|\theta)$ if they are satisfied by $f_n(X^n|\theta)$. This, in turn, ensures that Assumption \textbf{(A3)} holds for $\tilde{f}_n(X^n|\theta)$.

\begin{Proposition}\label{prop:random_alpha(A3)}
    Suppose $f_n(X^n|\theta)$ and $\pi(\theta)$ satisfy the conditions of Proposition \ref{prop:A3_suff}. Furthermore, suppose for a data-dependent sequence, $\hat{\alpha}_n$, there exists a sequence $\frac{1}{n}\ll\alpha_n\ll1$ such that $\frac{\hat{\alpha}_n}{\alpha_n}\to1$ in $f_{0,n}$-probability. Then Assumption \textbf{(A3)} holds for the $\hat{\alpha}_n$-posterior.
\end{Proposition}

\begin{proof}
  Let $\tilde{f}_n(X^n|\theta) = f_n(X^n|\theta)^{\frac{\hat{\alpha}_n}{\alpha_n}}$. We will show that $\tilde{f}_n(X^n|\theta)$ satisfies the conditions of Proposition \ref{prop:A3_suff}, from which it follows that \textbf{(A3)} holds for the $\hat{\alpha}_n$-posterior. Notice that we have already proved that \textbf{(A0)} holds for $\tilde{f}_n(X^n|\theta)$ and we have assumed that the prior satisfies \textbf{(A1)} as well as the bound $\pi(\theta) \leq \kappa$ for some $\kappa >0$. It remains to show 
  that Conditions \ref{condition1} and \ref{condition2} hold for $\tilde{f}_n(X^n|\theta)$.
  We prove this in what follows.

    \noindent\textbf{Checking Condition \ref{condition1}.} 
    To check Condition \ref{condition1}, we will show that for any $\epsilon > 0$ there exists a constant $\tilde{L} > 0$ and a function $\tilde{\gamma}: [0,\infty)\to [0,\infty)$ such that $\mathbb{P}_{f_{0,n}}(\mathcal{A}) \leq \epsilon$, where
    \begin{equation*}
        \mathcal{A} = \Big\{ \sup_{\theta\in B_{\theta^*}(r)^\mathsf{c}} \log \tilde{f}_n(X^n|\theta)-\log \tilde{f}_n(X^n|\theta^*) > -\tilde{L}n\tilde{\gamma}(\|\theta-\theta^*\|_2) \Big\},
    \end{equation*}
    for $n$ large enough and for any $r>0$ satisfying the bound in the condition statement. 
To do so, define the event $\mathcal{E} = \{\hat{\alpha}_n/\alpha_n > 1-t\}$ for some $t\in(0,1)$. By assumption,  there exists $N_0\equiv N_0(t,\epsilon)$ such that $\mathbb{P}_{f_{0,n}}(\mathcal{E}^\mathsf{c}) < \epsilon/2$ whenever $n > N_0$. Then, taking $n > N_0$, we have
    \begin{align}
        \nonumber&\mathbb{P}_{f_{0,n}}(\mathcal{A} ) \leq \mathbb{P}_{f_{0,n}}(\mathcal{A} \cap\mathcal{E}) + \frac{\epsilon}{2} \\
        \nonumber&\leq\mathbb{P}_{f_{0,n}}\Big(\Big\{\sup_{\theta\in B_{\theta^*}(r)^\mathsf{c}}\log f_n(X^n|\theta)-\log f_n(X^n|\theta^*) > -\frac{\tilde{L}n\tilde{\gamma}(\|\theta-\theta^*\|_2)}{1-t}\Big\}\cap\mathcal{E}\Big) + \frac{\epsilon}{2} \\
       \nonumber &\leq\mathbb{P}_{f_{0,n}}\Big(\sup_{\theta\in B_{\theta^*}(r)^\mathsf{c}}\log f_n(X^n|\theta)-\log f_n(X^n|\theta^*) > -\frac{\tilde{L}n\tilde{\gamma}(\|\theta-\theta^*\|_2)}{1-t}\Big) + \frac{\epsilon}{2}.
    \end{align}
    As we assumed Condition \ref{condition1} holds for $f_n(X^n|\theta)$, letting $\tilde{L} := (1-t)L$ (which is positive as $t < 1$) and $\tilde{\gamma}(\cdot) = \gamma(\cdot)$, there exists a $N_1 \equiv N_1(t,\epsilon)$ such that the above probability is upper bounded by $\frac{1}{2}\epsilon$ when $n > N_1$. Hence, for $n > \max(N_0, N_1)$, we have the desired result.
    Notice, we showed Condition \ref{condition1} holds for $\tilde{f}_n(X^n|\theta)$ with the same $\gamma(\cdot)$ as it does for $f_n(X^n|\theta)$.

    \noindent\textbf{Checking Condition \ref{condition2}.} First, note that the value of $r$ is the one such that Condition \ref{condition1} holds (for both $f_n(X^n|\theta)$ and $\tilde{f}_n(X^n|\theta)$, as shown above). Next, let $g_n(\theta, X^n) \equiv \log \tilde{f}_n(X^n|\theta)-\log \tilde{f}_n(X^n|\hat{\theta})$. Then for some constants $\tilde{c}_3, \tilde{c}_4 >0,$
    \begin{align*}
        &\Big\{ -n\tilde{c}_3\|\theta-\hat{\theta}\|_2^2\le g_n(\theta, X^n)  \leq -n\tilde{c}_4\|\theta-\hat{\theta}\|_2^2 \quad \text {for all } \theta\in B_{\hat{\theta}}(2r)\Big\}^\mathsf{c} \\
        &\quad = \Big\{g_n(\theta, X^n) < -n\tilde{c}_3\|\theta-\hat{\theta}\|_2^2  \,\,  \cup \,\,  g_n(\theta, X^n) > -n\tilde{c}_4\|\theta-\hat{\theta}\|_2^2 \quad \text{for some } \theta\in B_{\hat{\theta}}(2r)\Big\}.
    \end{align*}
    Hence, checking Condition \ref{condition2} is equivalent to showing that for all $\epsilon > 0$, there exists $N$ and constants $\tilde{c}_3 > 0$ and $\tilde{c}_4 > 0$ such that whenever $n > N$,
    \begin{equation*}
        \mathbb{P}_{f_{0,n}}\Big(g_n(\theta, X^n) < -n\tilde{c}_3\|\theta-\hat{\theta}\|_2^2  \,\,  \cup \,\,  g_n(\theta, X^n) > -n\tilde{c}_4\|\theta-\hat{\theta}\|_2^2 \quad \text{for some } \theta\in B_{\hat{\theta}}(2r)\Big) < \epsilon.
    \end{equation*}
   By a union bound, we can upper bound the probability above as follows:
    \begin{align*}
        &\mathbb{P}_{f_{0,n}}\Big(g_n(\theta, X^n) < -n\tilde{c}_3\|\theta-\hat{\theta}\|_2^2  \,\,  \cup \,\,  g_n(\theta, X^n) > -n\tilde{c}_4\|\theta-\hat{\theta}\|_2^2 \quad \text{for some } \theta\in B_{\hat{\theta}}(2r)\Big) \\
        &\quad \leq \mathbb{P}_{f_{0,n}}\Big(g_n(\theta, X^n) < -n\tilde{c}_3\|\theta-\hat{\theta}\|_2^2  \quad \text{for some } \theta\in B_{\hat{\theta}}(2r)\Big) \\
        &\qquad + \mathbb{P}_{f_{0,n}}\Big(g_n(\theta, X^n) > -n\tilde{c}_4\|\theta-\hat{\theta}\|_2^2 \quad \text{for some } \theta\in B_{\hat{\theta}}(2r)\Big)  \\
        & \quad = \mathbb{P}_{f_{0,n}}\Big(\frac{\hat{\alpha}_n}{\alpha_n}\big(\log f_n(X^n|\theta)-\log f_n(X^n|\hat{\theta})\big) < -n\tilde{c}_3\|\theta-\hat{\theta}\|_2^2 \quad \text{for some } \theta\in B_{\hat{\theta}}(2r)\Big) \\
        &\qquad+ \mathbb{P}_{f_{0,n}}\Big(\frac{\hat{\alpha}_n}{\alpha_n}\big(\log f_n(X^n|\theta)-\log f_n(X^n|\hat{\theta})\big) > -n\tilde{c}_4\|\theta-\hat{\theta}\|_2^2 \quad \text{for some } \theta\in B_{\hat{\theta}}(2r)\Big) \\
        &\equiv P_1 + P_2.
    \end{align*}
    Now, define $\mathcal{E} = \{1-t \leq \hat{\alpha}_n/\alpha_n \leq 1+t\}$ for some $t\in(0,1)$. By assumption, there exists $N_0 \equiv N_0(t,\epsilon)$ such that $\mathbb{P}_{f_{0,n}}(\mathcal{E}^\mathsf{c}) < \frac{1}{4}\epsilon$ when $n > N_0$. We use this to bound $P_1$ and $P_2$.

    First, consider term $P_1$. Notice that $\log f_n(X^n|\theta)-\log f_n(X^n|\hat{\theta}) \leq 0$ for all $\theta\in B_{\hat{\theta}}(2r)$ by the definition of $\hat\theta$. Furthermore, Condition \ref{condition2} holding for $f_n(X^n|\theta)$ implies there exists an $N_1(\epsilon, c_3, c_4, t)$ such that the event in Condition \ref{condition2} holds with  probability at least $1-\epsilon/4$ whenever $n > N_1$. Hence, taking $n > \max\{N_0,N_1\}$ and $\tilde{c}_3 = (1+t)c_3$, where $c_3 > 0$ is the constant such that Condition \ref{condition2} holds for $f_n(X^n|\theta)$, we have
    \begin{align*}
        &P_1 \hspace{-1pt} \leq \hspace{-1pt} \mathbb{P}_{f_{0,n}}\hspace{-2pt} \Big(\hspace{-2pt}\Big\{\frac{\hat{\alpha}_n}{\alpha_n}\big(\log f_n(X^n|\theta)-\log f_n(X^n|\hat{\theta})\big)\hspace{-2pt} <\hspace{-2pt}  -n\tilde{c}_3\|\theta-\hat{\theta}\|_2^2 \text{ for some } \theta\in B_{\hat{\theta}}(2r) \hspace{-2pt}\Big\}\cap\mathcal{E} \Big) \hspace{-2pt}+ \hspace{-2pt} \frac{\epsilon}{4} \\
        &\leq \mathbb{P}_{f_{0,n}}\Big(\log f_n(X^n|\theta)-\log f_n(X^n|\hat{\theta}) < -\frac{n\tilde{c}_3\|\theta-\hat{\theta}\|_2^2}{1+t} \text{ for some } \theta\in B_{\hat{\theta}}(2r)\Big) + \frac{\epsilon}{4} \\
       &= 1 - \mathbb{P}_{f_{0,n}}\Big(\log f_n(X^n|\theta)-\log f_n(X^n|\hat{\theta}) \geq -nc_3\|\theta-\hat{\theta}\|_2^2 \text{ for all } \theta\in B_{\hat{\theta}}(2r)\Big) + \frac{\epsilon}{4} \\
       & \leq \frac{\epsilon}{4} + \frac{\epsilon}{4} = \frac{\epsilon}{2}.
    \end{align*}
    The proof that $P_2 < \epsilon/2$ is analagous and this gives the desired result.
\end{proof}

The final result in this section checks assumptions \textbf{(A2')} and \textbf{(A3')} for $\tilde{f}_n(X^n|\theta)$, which are required to apply Lemma \ref{lem:lap_int_lem} and Theorem \ref{thm:BayesEstimator} to $\tilde{f}_n(X^n|\theta)$.

\begin{Proposition}\label{prop:random_alpha_(A2')(A3')}
    Suppose $f_n(X^n|\theta)$ satisfies Assumptions \textbf{(A0)}, \textbf{(A2')}, and \textbf{(A3')}  and that  for a data-dependent sequence, $\hat{\alpha}_n$, there exists a sequence $\frac{1}{n}\ll\alpha_n\ll1$ such that ${\hat{\alpha}_n}/{\alpha_n}\to1$ in $f_{0,n}$-probability. Then $\tilde{f}_n(X^n|\theta) = f_n(X^n|\theta)^{\frac{\hat{\alpha}_n}{\alpha_n}}$ also satisfies \textbf{(A2')} and \textbf{(A3')}.
\end{Proposition}

\begin{proof}
    We check that $\tilde{f}_n(X^n|\theta)$ satisfies Assumption \textbf{(A2')}. It suffices to show that there exists a matrix $\tilde{H}_n\in\mathbb{R}^{p\times p}$ (that is symmetric and positive-definite with high probability), a tensor $\tilde{S}_n\in\mathbb{R}^{p\times p\times p}$, and a constant $\tilde{M} < \infty$ such that for all $\epsilon > 0$ and some $\delta > 0$,
    \begin{align}\label{eq:g_random_cond}
        \mathbb{P}_{f_{0,n}}\Big(\sup_{h\in\mathcal{H}_n}\tilde{R}_n(h) - \frac{\tilde{M}}{\alpha_nn}\|h\|_2^4 > 0\Big) < \epsilon,
    \end{align}
    where $\mathcal{H}_n = \{h\in\mathbb{R}^p|\frac{\|h\|_2}{\sqrt{\alpha_nn}} < \delta\}$ and
    \begin{align*}
        \tilde{R}_n(h)
        &\equiv\alpha_n\Big[\log \tilde{f}_n\big(X^n\big|\hat{\theta}+\frac{h}{\sqrt{\alpha_nn}}\big)-\log \tilde{f}_n(X^n|\hat{\theta})\Big] +\frac{1}{2}h^\top \tilde{H}_nh - \frac{1}{6\sqrt{\alpha_n n}}\langle \tilde{S}_n, h^{\otimes 3}\rangle \\
        &=\alpha_n\frac{\hat{\alpha}_n}{\alpha_n}\Big[\log f_n\big(X^n\big|\hat{\theta}+\frac{h}{\sqrt{\alpha_nn}}\big)-\log f_n(X^n|\hat{\theta})\big] +\frac{1}{2}h^\top \tilde{H}_nh - \frac{1}{6\sqrt{\alpha_n n}}\langle \tilde{S}_n, h^{\otimes 3}\rangle.
    \end{align*}
    Define the event $\mathcal{A} = \{\hat{\alpha}_n/\alpha_n < 1 + t\}$ for some $t > 0$.
   By assumption, there exists $N_0 \equiv N_0(t,\epsilon)$ such that $\mathbb{P}_{f_{0,n}}(\mathcal{A}^\mathsf{c}) < \epsilon/2$ whenever $n > N_0$. Then taking $n > N_0$, we can bound the probability in \eqref{eq:g_random_cond} as follows:
    \begin{align*}
        \mathbb{P}_{f_{0,n}}\Big(\sup_{h\in\mathcal{H}_n}\tilde{R}_n(h) - \frac{\tilde{M}}{\alpha_nn}\|h\|_2^4 > 0\Big) \leq \mathbb{P}_{f_{0,n}}\Big(\Big\{\sup_{h\in\mathcal{H}_n}\tilde{R}_n(h) - \frac{\tilde{M}}{\alpha_nn}\|h\|_2^4 > 0\Big\} \cap \mathcal{A}\Big) + \frac{\epsilon}{2}.
    \end{align*}
    It remains to bound the first probability on the right side of the above. By  definition of $\hat\theta$, we know that $ \log f_n(X^n|\hat{\theta}+\frac{h}{\sqrt{\alpha_nn}})-\log f_n(X^n|\hat{\theta}) \leq 0$ for any $h\in\mathcal{H}_n$.
    Then, on event $\mathcal{A}$,
    \begin{equation*}
        \tilde{R}_n(h) \leq (1+t)\alpha_n\Big[\log f_n\big(X^n\big|\hat{\theta}+\frac{h}{\sqrt{\alpha_nn}}\big)-\log f_n(X^n|\hat{\theta})\Big] +\frac{1}{2}h^\top \tilde{H}_nh - \frac{1}{6\sqrt{\alpha_n n}}\langle \tilde{S}_n, h^{\otimes 3}\rangle,
    \end{equation*}
    and therefore on $\mathcal{A}$
    \begin{align*}
        &\frac{1}{1+t} \Big(\tilde{R}_n(h) - \frac{\tilde{M}\|h\|_2^4}{\alpha_nn}\Big) \\
        &\leq  \alpha_n\Big[\log f_n\big(X^n\big|\hat{\theta}+\frac{h}{\sqrt{\alpha_nn}}\big)-\log f_n(X^n|\hat{\theta})\Big] +\frac{h^\top\tilde{H}_n h}{2(1+t)} - \frac{\langle \frac{1}{1+t}\tilde{S}_n, h^{\otimes 3}\rangle }{6\sqrt{\alpha_n n}}- \frac{\tilde{M}\|h\|_2^4}{\alpha_nn(1+t)}.
    \end{align*}
    Then letting $\tilde{H}_n := (1+t)H_n$ (which is positive definite with high probability by choosing $t < 1$), $\tilde{S}_n := (1+t)S_n$, and $\tilde{M} = (1+t)M < \infty$, we find
    \begin{align*}
        &\mathbb{P}_{f_{0,n}}\Big(\Big\{\sup_{h\in\mathcal{H}_n}\tilde{R}_n(h) - \frac{\tilde{M}}{\alpha_nn}\|h\|_2^4 > 0 \Big\}\cap \mathcal{A}\Big) \\
        &\leq \mathbb{P}_{f_{0,n}}\hspace{-1.5pt}\Big( \hspace{-1.5pt} \sup_{h\in\mathcal{H}_n} \hspace{-1.5pt} \alpha_n \hspace{-1.5pt} \Big[ \hspace{-1.5pt} \log f_n\big(X^n\big|\hat{\theta} \hspace{-1pt} + \hspace{-1pt} \frac{h}{\sqrt{\alpha_nn}}\big) \hspace{-1pt} - \hspace{-1pt} \log f_n(X^n|\hat{\theta})\Big] \hspace{-1pt} + \hspace{-1pt} \frac{h^\top H_n h}{2} \hspace{-1pt}  - \hspace{-1pt} \frac{\langle S_n, h^{\otimes 3}\rangle}{6\sqrt{\alpha_n n}} \hspace{-1pt} - \hspace{-1pt} \frac{M\|h\|_2^4}{\alpha_nn} \hspace{-1pt} > \hspace{-1pt} 0\Big) \\
        &\leq \epsilon/2,
    \end{align*}
    where the last line follows for $n$ large enough since we have assumed Assumption \textbf{(A2')} holds for $f_n(X^n|\hat{\theta})$. Hence, \textbf{(A2')} holds for $\tilde{f}_n(X^n|\theta)$ with these choices of $\tilde{H}_n$, $\tilde{S}_n$, and $\tilde{M}$. \\

    \noindent\textbf{Assumption (A3'):} We will show that the condition in \eqref{eq:A3'} holds with $\tilde{f}_n(X^n|\theta)$ replacing $f_n(X^n|\theta)$. More specifically, we show that for all $\epsilon > 0$ and $\delta' > 0$, there exists a constant $\tilde{c}$ such that  for $n$ sufficiently large
    \begin{align*}
        &\mathbb{P}_{f_{0,n}}\Big(\sup_{\theta\in B_{\hat{\theta}}(\delta')^\mathsf{c}}\log \tilde{f}_n(X^n|\theta) - \log \tilde{f}_n(X^n|\hat{\theta}) > -\tilde{c}\Big).
    \end{align*}
 To do so, define the event $\mathcal{E} = \{\hat{\alpha}_n/\alpha_n > 1-t\}$ for $t\in(0,1)$ and note that there exists $N_0\equiv N_0(\epsilon,t)$ such that $\mathbb{P}_{f_{0,n}}(\mathcal{E}^\mathsf{c}) < \epsilon/2$ whenever $n > N_0$. Taking $n > N_0$ then, we have
    \begin{align*}
        &\mathbb{P}_{f_{0,n}}\Big(\sup_{\theta\in B_{\hat{\theta}}(\delta')^\mathsf{c}}\log \tilde{f}_n(X^n|\theta) - \log \tilde{f}_n(X^n|\hat{\theta}) > -\tilde{c}\Big) \\
        &\quad \leq\mathbb{P}_{f_{0,n}}\Big(\Big\{\sup_{\theta\in B_{\hat{\theta}}(\delta')^\mathsf{c}}\log \tilde{f}_n(X^n|\theta) - \log \tilde{f}_n(X^n|\hat{\theta}) > -\tilde{c}\Big\} \cap \mathcal{E}\Big) + \frac{\epsilon}{2} \\
        &\quad = \mathbb{P}_{f_{0,n}}\Big(\Big\{\frac{\hat{\alpha}_n}{\alpha_n}\Big[\sup_{\theta\in B_{\hat{\theta}}(\delta')^\mathsf{c}}\log f_n(X^n|\theta) - \log f_n(X^n|\hat{\theta})\Big] > -\tilde{c}\Big\}\cap\mathcal{E}\Big) + \frac{\epsilon}{2} \\
        &\quad\leq \mathbb{P}_{f_{0,n}}\Big(\Big\{\sup_{\theta\in B_{\hat{\theta}}(\delta')^\mathsf{c}}\log f_n(X^n|\theta) - \log f_n(X^n|\hat{\theta}) > -\frac{\tilde{c}}{1-t}\Big\}\cap\mathcal{E}\Big) + \frac{\epsilon}{2} \\
        &\quad\leq \mathbb{P}_{f_{0,n}}\Big(\sup_{\theta\in B_{\hat{\theta}}(\delta')^\mathsf{c}}\log f_n(X^n|\theta) - \log f_n(X^n|\hat{\theta}) > -\frac{\tilde{c}}{1-t}\Big) + \frac{\epsilon}{2}.
    \end{align*}
    Since we have assumed \textbf{(A3')} holds for $f_n(X^n|\theta)$, if we let $\tilde{c} = (1-t)c$ (which is positive as $t < 1$), then there exists an $N_1 \equiv N_1(\epsilon,t,c)$ such that the above probability is upper bounded by $\epsilon/2$ whenever $n > N_1$. We obtain \textbf{(A3')} with $\tilde{f}_n(X^n|\theta)$ by taking $n > \max(N_0,N_1)$.
\end{proof}

\section{Complementary examples}
\subsection{Proof of Proposition \ref{prop:counter_example_exp_family}}\label{app:counterexample_expfam_proof}

\begin{proof}
It is easy to see that the characteristic function of the prior distribution in \eqref{eq:exponential_family_prior} is
\begin{align}\label{eq:conjugate_chf}
        \varphi_{\pi}(t) 
          &= \int_{H} \exp(it^\top\eta)\exp\left(\eta^\top\xi - \nu A(\eta) - \psi(\xi, \nu)\right)d\eta 
        = \exp\left(\psi(it + \xi, \nu) - \psi(\xi, \nu)\right).\
\end{align}
From \eqref{eq:alpha_post_exponential_short} we see that the posterior belongs to the same family as the prior, \eqref{eq:exponential_family_prior}, and the characteristic function of the $\alpha_n$-posterior is $\varphi_{\pi_{n,\alpha_n}}(t) = \exp(\psi(it + \xi_n, \nu_n) - \psi(\xi_n, \nu_n)).$
Now, suppose that $\alpha_nn \rightarrow \alpha_0$ for some $\alpha_0 \geq 0$ and $\frac{1}{n}T(X^n)$ converges in $f_{0,n}$-probability to some limit $g(\eta^*)$, where $\eta^*$ denotes the true value of the natural parameter. Then, in $f_{0,n}$-probability, $\xi_n = \alpha_nn(\frac{1}{n}T(X^n)) + \xi \rightarrow \alpha_0g(\eta^*) + \xi \equiv \xi',$ and $\nu_n = \alpha_nn + \nu \rightarrow \alpha_0 + \nu \equiv \nu'.$
 Since $\varphi_{\pi_{n,\alpha_n}}(t)$ is continuous in $\xi_n$ and $\nu_n$, the limit in $f_{0,n}$-probability of the characteristic function of $\pi_{n,\alpha_n}$ is $\varphi_{\pi_{n,\alpha_n}}(t) \rightarrow \exp(\psi(it + \xi', \nu') - \psi(\xi', \nu')),$ which has the same form as the characteristic function of the prior, \eqref{eq:conjugate_chf}. By L\'evy's continuity theorem, the $\alpha_n$-posterior converges weakly to a distribution belonging to the family in  \eqref{eq:exponential_family_prior}.
\end{proof}

\subsection{Explicit counterexamples for the BvM result when $\alpha_n \asymp 1/n$}
\label{app:counterexamples}

We work out examples of likelihood models and priors where Proposition \ref{prop:counter_example_exp_family} holds. We compute explicitly the limiting characteristic functions for these examples which demonstrate that the BvM cannot hold.

\noindent\textbf{Exponential-Gamma:} Let $X^n = (X_1,\ldots,X_n) \overset{\text{i.i.d.}}{\sim} \text{Exp}(\lambda)$ and $\lambda \sim \Gamma(a,b)$. Denote the true parameter value by $\lambda^*$. We can write
$f(X^n|\lambda) = \exp(-\lambda\sum_{i=1}^nX_i - n[-\log(\lambda)] - 0).$
Hence, in the form of \eqref{eq:exponential_family_likelihood}, we have taken $\eta = -\lambda$, $A(\eta) = -\log(-\eta)$, $T(X^n) = \sum_{i=1}^nX_i$, and $B(X^n)=0$. We can write
$
    \pi(\lambda) = \exp(-\lambda b - (1-a)(-\log(\lambda)) - \log(\frac{b^{a}}{\Gamma(a)})).$
Hence, writing $\pi(\eta)$ in the form of \eqref{eq:exponential_family_prior} yields $\xi = b$, $\nu = 1-a$, and $\psi(\xi, \nu) = \log(\frac{\xi^{1-\nu}}{\Gamma(1-\nu)})$. The characteristic function of $ \pi(\lambda)$ is given by
$
    \varphi_{\pi}(t) = \exp(\log(\frac{(\xi+it)^{1-\nu}}{\Gamma(1-\nu)}) - \log(\frac{\xi^{1-\nu}}{\Gamma(1-\nu)})).$
The density of the $\alpha_n$-posterior according to \eqref{eq:alpha_post_exponential_short} is
$\pi_{n,\alpha_n}(\lambda|X^n) = \exp(-\lambda\xi_n- \nu_n(-\log(\lambda)) - \psi(\xi_n, \nu_n)).$
where we take $\xi_n \equiv \alpha_n\sum_{i=1}^nX_i + b$ and $ \nu_n \equiv \alpha_nn + 1-a.$
We can write the characteristic function of $\pi_{n,\alpha_n}(\lambda|X^n)$ as
\begin{equation*}
    \begin{split}
        \varphi_{\pi_{n,\alpha_n}} (t) 
       &= \exp\Big(\log  \Big(\frac{(\xi_n+it)^{1-\nu}}{\Gamma(1-\nu_n)}\Big)  -  \log \Big(\frac{\xi_n^{1-\nu_n}}{\Gamma(1-\nu_n)}\Big)\Big) \\
        &\rightarrow  \exp \Big(\log  \Big(\frac{(\xi'+it)^{1-\nu'}}{\Gamma(1-\nu')}\Big)  -   \log  \Big(\frac{\xi'^{1-\nu}}{\Gamma(1-\nu')}\Big)\Big),
    \end{split}
\end{equation*}
where the limit is in $f_{0,n}$-probability and we take $g(\eta^*) = -\frac{1}{\eta^*} = \frac{1}{\lambda^*}$, $\xi'= \alpha_0g(\eta^*) + \xi = \frac{\alpha_0}{\lambda^*} + b$, and $\nu' = \alpha_0 + \nu = \alpha_0 + 1 - a.$   

\noindent\textbf{Pareto-Gamma:} Let $X^n = (X_1,\ldots,X_n) \overset{\text{i.i.d.}}{\sim} \text{Pareto}(x_m, k)$, where $x_m$ is known, and $k \sim \Gamma(a,b)$. Denote the true parameter value by $k^*$.  We can write 
$f(X^n|k) = \exp(-(k+1)\sum_{i=1}^n\log(\frac{X_i}{x_m}) - n\log(\frac{x_m}{k})-0).$
Hence, in the form of \eqref{eq:exponential_family_likelihood} yields $\eta = -(k+1)$, $A(\eta) = \log(\frac{x_m}{-\eta - 1})$, $T(X^n) = \sum_{i=1}^n\log(\frac{X_i}{x_m})$, and $B(X^n)=0$. We can write
$\pi(k) = \exp(-(k+1) b - (1-a)\log(\frac{x_m}{k}) -(-\log(\frac{x_m^{a-1}b^a}{\Gamma(a)})- b)).$
Hence, writing $\pi(\eta)$ in the form of \eqref{eq:exponential_family_prior} yields $\xi = b$, $\nu = 1-a$, and $\psi(\xi, \nu) = -\log(\frac{x_m^{-\nu}\xi^{1-\nu}}{\Gamma(1-\nu)})-\xi$. The characteristic function of $ \pi(k)$ is given by
$\varphi_{\pi}(t) = \exp(-\log(\frac{x_m^{-\nu}(\xi+it)^{1-\nu}}{\Gamma(1-\nu)}) + \log(\frac{x_m^{-\nu}\xi^{1-\nu}}{\Gamma(1-\nu)})-it).$
The density of the $\alpha_n$-posterior according to \eqref{eq:alpha_post_exponential_short} is
$\pi_{n,\alpha_n}(k|X^n) = \exp(-(k+1)\xi_n - \nu_n\log(\frac{x_m}{k}) -\psi(\xi_n, \nu_n )),$
where we take $\xi_n \equiv \alpha_n\sum_{i=1}^n\log(\frac{X_i}{x_m})+b$ and $\nu_n \equiv \alpha_nn + (1-a)$.
We can write the characteristic function of $\pi_{n,\alpha_n}(k|X^n)$ as
\begin{equation*}
    \begin{split}
        \varphi_{\pi_{n,\alpha_n}}(t) 
        &= \exp\Big(-\log\Big(\frac{x_m^{-\nu}(\xi_n+it)^{1-\nu_n}}{\Gamma(1-\nu_n)}\Big) + \log\Big(\frac{x_m^{-\nu}\xi_n^{1-\nu_n}}{\Gamma(1-\nu_n)}\Big)-it\Big) \\
        &\rightarrow \exp\Big(-\log\Big(\frac{x_m^{-\nu'}(\xi'+it)^{1-\nu'}}{\Gamma(1-\nu')}\Big) + \log\Big(\frac{x_m^{-\nu'}\xi'^{1-\nu'}}{\Gamma(1-\nu')}\Big)-it\Big),
    \end{split}
\end{equation*}
where the limit is in $f_{0,n}$-probability with
$g(\eta^*) = -\frac{1}{\eta^* + 1} = \frac{1}{k^*}$,
        $\xi' = \alpha_0g(\eta^*) + \xi = \frac{\alpha_0}{k^*} + b,$ and
        $\nu' = \alpha_0 + \nu = \alpha_0 + 1 - a.$ 

\noindent \textbf{Bernoulli-Beta:} Let $X^n = (X_1,\ldots,X_n) \overset{\text{i.i.d.}}{\sim} \text{Bernoulli}(p)$ and $p \sim \text{beta}(a,b)$. Denote the true parameter value by $p^*$.  We can write $ f(X^n|p) = \exp(\log(\frac{p}{1-p})\sum_{i=1}^n X_i - n\log(\frac{1}{1-p}) - 0).$
Hence, in the form of \eqref{eq:exponential_family_likelihood}, we take $\eta = \log(\frac{p}{1-p})$, $A(\eta) = \log(1 + e^\eta)$, $T(X^n) = \sum_{i=1}^nX_i$, and $B(X^n)=0$. We can write 
$\pi(p) = \exp((a-1)\log(\frac{p}{1-p}) - (a + b + 2)\log(\frac{1}{1-p}) - \log(\frac{\Gamma(a)\Gamma(b)}{\Gamma(a+b)})).$
Hence, in the form of \eqref{eq:exponential_family_prior} we take $\xi = a-1$, $\nu = a + b + 2$, and $\psi(\xi, \nu) = \log(\frac{\Gamma(1+\xi)\Gamma(\nu-\xi-3)}{\Gamma(\nu-2)})$. The characteristic function of $\pi(p)$ is given by
\begin{equation*}
    \varphi_\pi(t) 
    = \exp\Big(\log\Big(\frac{\Gamma(1+\xi+it)\Gamma(\nu-\xi-it-3)}{\Gamma(\nu-2)}\Big) - \log\Big(\frac{\Gamma(1+\xi)\Gamma(\nu-\xi-3)}{\Gamma(\nu-2)}\Big)\Big).\
\end{equation*}
The density of the $\alpha_n$-posterior according to \eqref{eq:alpha_post_exponential_short} is
$\pi_{n,\alpha_n}(p|X^n) 
        = \exp(\xi_n\log(\frac{p}{1-p} - \nu_n\log(\frac{1}{1-p}) - \psi(\xi_n, \nu_n))$
where we have used $\xi_n = \alpha_n\sum_{i=1}^n X_i + a - 1$ and $\nu_n = \alpha_nn + a + b + 2$. We can write the characteristic function of $  \pi_{n,\alpha_n}(p|X^n)$ as
\begin{equation*}
    \begin{split}
        \varphi_{\pi_{n,\alpha_n}}(t) 
        &= \exp\Big(\log\Big(\frac{\Gamma(1+\xi_n+it)\Gamma(\nu_n-\xi_n-it-3)}{\Gamma(\nu_n-2)}\Big) - \log\Big(\frac{\Gamma(1+\xi_n)\Gamma(\nu_n-\xi_n-3)}{\Gamma(\nu_n-2)}\Big)\Big) \\
        &\rightarrow \exp\Big(\log\Big(\frac{\Gamma(1+\xi'+it)\Gamma(\nu'-\xi'-it-3)}{\Gamma(\nu'-2)}\Big) - \log\Big(\frac{\Gamma(1+\xi')\Gamma(\nu'-\xi'-3)}{\Gamma(\nu'-2)}\Big)\Big),
    \end{split}
\end{equation*}
where the limit is in $f_{0,n}$-probability and $g(\eta^*) = \frac{e^{\eta^*}}{1+e^{\eta^*}} = p^*$, $\xi' = \alpha_0g(\eta^*) + \xi = \alpha_0p^* + a-1,$and $\nu' = \alpha_0 + \nu = \alpha_0 + a + b + 2.$   

\noindent\textbf{Gaussian-Gaussian:} Let $X^n = (X_1,\ldots,X_n) \overset{\text{i.i.d.}}{\sim} \mathcal{N}(\mu, \sigma^2)$, where $\sigma^2$ is known, and $\mu \sim \mathcal{N}(\mu_0,\sigma_0^2)$. Denote the true parameter value by $\mu^*$.  We can write
$f(X^n|\mu) = \exp(\mu\frac{\sum_{i=1}^nX_i}{\sigma^2} - n(\frac{\mu^2}{2\sigma^2}) - (\frac{n}{2}\log(2\pi\sigma^2)+\frac{1}{2\sigma^2}\sum_{i=1}^nX_i^2)).$
Hence, writing $f_n(X^n|\eta)$ in the form of \eqref{eq:exponential_family_likelihood} yields $\eta = \mu$, $A(\eta) = \frac{\eta^2}{2\sigma^2}$, $T(X^n) = \frac{1}{\sigma^2} \sum_{i=1}^nX_i$, and $B(X^n) = \frac{n}{2}\log(2\pi\sigma^2)+\frac{1}{2\sigma^2}\sum_{i=1}^nX_i^2$. We can write $\pi(\mu)$ as
$\pi(\mu) = \exp(\mu\frac{\mu_0}{\sigma_0^2}-\frac{\sigma^2}{\sigma_0^2}(\frac{\mu^2}{2\sigma^2})-(\frac{1}{2}\log(2\pi\sigma_0^2)+\frac{1}{2\sigma_0^2}\mu_0^2)).$
Hence, writing $\pi(\eta)$ in the form of \eqref{eq:exponential_family_prior} yields $\xi = \mu_0/\sigma_0^2$, $\nu = \sigma^2/\sigma_0^2$, and $\psi(\xi, \nu) = \frac{1}{2}\log(\frac{2\pi\sigma^2}{\nu}) + \frac{\xi^2\sigma^2}{2\nu}$. The characteristic function of $\pi(\mu)$ is given by
$\varphi_\pi(t) = \exp(\frac{(\xi + it)^2\sigma^2}{2\nu} - \frac{\xi^2\sigma^2}{2\nu}).$
The density of the $\alpha_n$-posterior according to \eqref{eq:alpha_post_exponential_short} is
$f(X^n|\mu) = \exp(\mu\xi_n - \nu_n (\frac{\mu^2}{2\sigma^2})- \psi(\xi_n, \nu_n)).$
where we have used $\xi_n = \frac{\alpha_n\sum_{i=1}^nX_i}{\sigma^2} + \frac{\mu_0}{\sigma_0^2}$ and $
        \nu_n = \alpha_nn + \frac{\sigma^2}{\sigma_0^2}.$ We write the characteristic function of $f(X^n|\mu)$ as
\begin{equation*}
    \begin{split}
        \varphi_{\pi_{n,\alpha_n}}(t) 
        &= \exp\Big(\frac{(\xi_n + it)^2\sigma^2}{2\nu_n} - \frac{\xi_n^2\sigma^2}{2\nu_n}\Big) \rightarrow \exp\Big(\frac{(\xi' + it)^2\sigma^2}{2\nu'} - \frac{\xi'^2\sigma^2}{2\nu'}\Big),
    \end{split}
\end{equation*}
where the  limit is in $f_{0,n}$-probability and
$g(\eta^*) = \frac{\eta^*}{\sigma^2} = \frac{\mu^*}{\sigma^2}$, $\xi' = \alpha_0g(\eta^*) + \xi = \frac{\alpha_0\mu^*}{\sigma^2} + \frac{\mu_0}{\sigma_0^2}$, and $\nu' = \alpha_0 + \nu = \alpha_0 + {\sigma^2}/{\sigma_0^2}.$
Written in the form of \eqref{eq:exponential_family_prior}, the alleged limiting distribution of the $\alpha_n$-posterior according to Theorem \ref{cor:BvM} is
$
    \phi(\mu|\hat{\mu}, \frac{\sigma^2}{\alpha_nn}) 
    = \exp(\mu\frac{\alpha_nn\hat{\mu}}{\sigma^2}-\alpha_nn(\frac{\mu^2}{2\sigma^2})-(\frac{1}{2}\log(\frac{2\pi\sigma^2}{\alpha_nn})+\frac{\alpha_nn\hat{\mu}^2}{2\sigma^2})).$
Defining $\xi_n' = \frac{\alpha_nn\hat{\mu}}{\sigma^2}$ and  $\nu_n' = \alpha_nn,$
we can write the characteristic function of $\phi(\mu|\hat{\mu}, \frac{\sigma^2}{\alpha_nn})$ as
\begin{equation*}
    \begin{split}
        \varphi_{\pi_{n,\alpha_n}}(t) 
        &= \exp\Big(\frac{(\xi_n' + it)^2\sigma^2}{2\nu_n'} - \frac{\xi_n'^2\sigma^2}{2\nu_n'}\Big) \rightarrow \exp\Big(\frac{(\xi'' + it)^2\sigma^2}{2\nu''} - \frac{\xi''^2\sigma^2}{2\nu''}\Big),
    \end{split}
\end{equation*}
where the limit is in in $f_{0,n}$-probability and
$\xi'' = \frac{\alpha_0\mu^*}{\sigma^2} \neq \xi'$ and $\nu'' = \alpha_0 \neq \nu'.$ 

In this example, the characteristic function is that of a Gaussian distribution, which is exactly what we expect from  Proposition \ref{prop:counter_example_exp_family} in this case. However, the BvM stated in \ref{cor:BvM} does not hold since the Gaussian distribution that appears in that statement is not the same.

\subsection{Models where (A3) holds}\label{app:A3_holds}
In this section, we verify Assumption \textbf{(A3)} for models including our linear regression example and exponential families. We show that even with improper priors, which lead to heavier tails, the posterior will still have arbitrarily many moments for $\alpha_nn$ sufficiently large. \\

\noindent\textbf{Linear regression:} We consider a linear regression model similar to the one in the simulations of Section \ref{sec:model_spec}. For $x_i\in\mathbb{R}^p$ and $y_i\in\mathbb{R}$, we specify the linear model $y_i=x_i^\top\beta^* + \epsilon_i$, where $\epsilon_i \overset{\text{\text{i.i.d.}}}{\sim} \mathcal{N}(0,1)$ and $x_{i}\overset{\text{\text{\text{i.i.d.}}}}{\sim}\mathcal{N}(0,I_p)$. Denote the MLE/ordinary least squares estimator by $\hat{\beta}$ and  specify the prior $\pi(\beta)\propto 1$. Letting $Y = [y_1,\ldots,y_n]^\top$ and $X = [x_1^\top \cdots x_n^\top]^\top$ and using the identity $\|Y - X\beta\|_2^2 = \|Y - X\hat{\beta}\|_2^2 + \|X(\hat{\beta}-\beta)\|_2^2$, it is easy to see that 
\begin{equation}\label{eq:beta_posterior_A3}
    \pi_{n,\alpha_n}(\beta|X, Y) = \phi\Big(\beta \, \Big| \, \hat{\beta}, \frac{1}{\alpha_nn}\Big(\frac{1}{n}X^\top X\Big)^{-1}\Big).
\end{equation}
To verify Assumption \textbf{(A3)}, we show that $ \mathbb{E}_{\pi_{n,\alpha_n}(\beta|X,Y)}[\|\sqrt{\alpha_nn}(\beta-\beta^*)\|_2^{k_0}] = O_{f_{0,n}}(1)$ holds for all $k_0 \geq 0$.
First, define $v_n \equiv \sqrt{\alpha_nn}(\beta-\hat{\beta})$ and $S_n \equiv \frac{1}{n}X^\top X$. From \eqref{eq:beta_posterior_A3}, we see that 
\begin{align}
   v_n \, | \, X,Y&\sim\mathcal{N}\left(\mathbf{0}, S_n^{-1}\right) 
    \label{eq:Sn1/2vn}
    \implies\|S_n^{1/2}v_n\|_2^2 \, | \, X,Y\sim\chi^2_p.
\end{align}
Next, by the bound $(a+b)^{k_0}\leq 2^{k_0}(a^{k_0} + b^{k_0})$ for $k_0\geq0$, we have that
\begin{align*}
    2^{-k_0}\mathbb{E}_{\pi_{n,\alpha_n}(\beta|X,Y)}[\|\sqrt{\alpha_nn}(\beta-\beta^*)\|_2^{k_0}]
   &\leq \mathbb{E}_{\pi_{n,\alpha_n}(\beta|X,Y)}[\|\sqrt{\alpha_nn}(\beta-\hat{\beta})\|_2^{k_0} + \|\sqrt{\alpha_nn}(\hat{\beta}-\beta^*)\|_2^{k_0}] \\
    &= \mathbb{E}[\|v_n\|_2^{k_0} \, | \, X,Y] + \|\sqrt{\alpha_nn}(\hat{\beta}-\beta^*)\|_2^{k_0}\\
  &= \mathbb{E}\big[\big(v_n^\top S_n^{1/2}S_n^{-1}S_n^{1/2}v_n\big)^{k_0/2} \, | \, X,Y\big]  + o_{f_{0,n}}(1)\\
  &\leq \lambda_{\text{min}}(S_n)^{-k_0/2} \mathbb{E}_{u\sim\chi^2_p}\big[u^{k_0/2}\big]+ o_{f_{0,n}}(1) =O_{f_{0,n}}(1),
\end{align*}
where the second equality follows from Assumption \textbf{(A0)} and  $\alpha_n\to0$ and the last one from  $S_n = \frac{1}{n}X^\top X$ converging in probability to a positive definite matrix.  Hence,  Assumption \textbf{(A3)} is satisfied. \\

\noindent\textbf{Linear regression (with random $\boldsymbol{\alpha}$):} Cleary \eqref{eq:beta_posterior_A3} also implies that the $\hat{\alpha}_n$-posterior can be written as
\begin{equation}\label{eq:beta_posterior_A3_random}
    \pi_{n,\hat\alpha_n}(\beta|X, Y) = \phi\Big(\beta \, \Big| \, \hat{\beta}, \frac{1}{\hat\alpha_nn}\Big(\frac{1}{n}X^\top X\Big)^{-1}\Big),
\end{equation}
Assuming  $\alpha_n$ is a sequence such that $\frac{\hat{\alpha}_n}{\alpha_n}\to1$ in $f_{0,n}$-probability, by \eqref{eq:Sn1/2vn} we have that
$\mathbb{E}_{\pi_{n,\hat{\alpha}_n}(\beta|X,Y)}[\|\sqrt{\frac{\hat{\alpha}_n}{\alpha_n}}S_n^{1/2}v_n\|_2^2]  \sim \chi^2_p,$
where $v_n = \sqrt{\alpha_nn}(\beta-\hat{\beta})$. Therefore, following similar calculations as in  the determistic $\alpha_n$ case, we have that 
\begin{align*}
    2^{-k_0}\mathbb{E}_{\pi_{n,\hat{\alpha}_n}(\beta|X,Y)}[\|\sqrt{\alpha_nn}(\beta-\beta^*)\|_2^{k_0}]
&\leq\mathbb{E}_{\pi_{n,\hat{\alpha}_n}(\beta|X,Y)}\Big[\big\|\sqrt{\alpha_nn}(\beta-\hat{\beta})\big\|_2^{k_0}\Big] + o_{f_{0,n}}(1) \\
    &=\Big(\frac{\alpha_n}{\hat\alpha_n}\Big)^{k_0}\mathbb{E}_{\pi_{n,\hat{\alpha}_n}(\beta|X,Y)}\Big[\Big(\big\|\sqrt{\frac{\hat{\alpha}_n}{\alpha_n}}v_n\big\|_2^2\Big)^{k_0/2}\Big] + o_{f_{0,n}}(1)\\
 \nonumber   &\leq (1 + o_{f_{0,n}}(1))\lambda_{\text{min}}(S_n)^{-k_0/2} \mathbb{E}_{u\sim\chi^2_p}[u^{k_0/2}]+ o_{f_{0,n}}(1)\\
 &=O_{f_{0,n}}(1),
\end{align*}
Hence, \textbf{(A3)} holds for the $\hat{\alpha}_n$-posterior. \\

\noindent\textbf{Exponential families:} Let $X^n$ have density $f_n(X^n|\eta) = \exp\left(\eta T(X^n) - n(A(\eta) - B(X^n)\right)$, where $\eta\in H\subset \mathbb{R}$ denotes the natural parameter and $A$ is twice differentiable and strictly convex in the interior of $H$. Denote the MLE by $\hat{\eta}$ and the true natural parameter value by $\eta^*$. Let $\pi(\eta)\propto 1$. Then the $\alpha_n$-posterior is proportional to $\exp(\alpha_n\eta T(X^n) - \alpha_nnA(\eta))$. We verify \textbf{(A3)} for the $\alpha_n$-posterior by verifying Conditions \ref{condition1}--\ref{condition2} of Proposition \ref{prop:A3_suff}. By Proposition \ref{prop:random_alpha(A3)}, doing so also amounts to checking that \textbf{(A3)} holds for the $\hat{\alpha}_n$-posterior.

\noindent\textit{Condition \ref{condition1}.} We will show that the constants $c_0 = c_1 = 1$ and $c_2 = 0$ -- with $L$ to be specified later -- lead to \eqref{eq:cond2_part1} in Condition \ref{condition1}.

Consider first the case $\eta > \eta^* + r$, where $r > \frac{e^{c_2/c_0}}{c_1}=1$. Since $\eta > \eta^*$, we have that $\eta-\eta^* = |\eta-\eta^*|$ and 
\begin{align}
    \log f_n(X^n|\eta) - \log f_n(X^n|\eta^*) 
    \nonumber&= |\eta-\eta^*|T(X^n) - n(A(\eta) - A(\eta^*)) \\
    \label{eq:expfam_case1_bound1}&= n|\eta-\eta^*|\Big[\frac{1}{n}T(X^n) - A'(\eta^*) - \Big(\frac{A(\eta)-A(\eta^*)}{|\eta-\eta^*|} - A'(\eta^*)\Big)\Big].
\end{align}
Assuming that $\frac{1}{n}T(X^n)\to A'(\eta^*)$ in $f_{0,n}$-probability (this is true in the i.i.d.\ setting), with probability tending to one, the above is further bounded as follows for some $t > 0$ as
\begin{align*}
    \log f_n(X^n|\eta) - \log f_n(X^n|\eta^*) 
    &\leq -n|\eta-\eta^*|\Big[-t + \frac{A(\eta)-A(\eta^*)}{|\eta-\eta^*|} - A'(\eta^*)\Big] \\
    &= -n|\eta-\eta^*|\left[-t + A'(\tilde{\eta}) - A'(\eta^*)\right]\\
     &\leq -n|\eta-\eta^*|\Big[-t + \inf_{\{\eta|\eta > \eta^* + r\}}\big(A'(\eta) - A'(\eta^*)\big)\Big]\\
     &= -n|\eta-\eta^*|\Big[-t + A'(\eta^* + r) - A'(\eta^*)\Big],
\end{align*}
 where $\tilde{\eta}\in (\eta^*, \eta)$, the last inequality used $\eta > \eta^* + r$ and the last equality follows the strict convexity of $A$.
Choosing $t$ small enough the term in brackets is positive with $f_{0,n}$-probability tending to 1. 

In the case $\eta \leq \eta^* - r$, we have $\eta-\eta^* = -|\eta-\eta^*|$ and it is easy to check that one obtains similar bounds,  which lead to  $\log f_n(X^n|\eta) - \log f_n(X^n|\eta^*) \leq -n|\eta-\eta^*|[-t - (A'(\eta^* - r) - A'(\eta^*))].$ Combining the cases, \eqref{eq:cond2_part1} of Condition \ref{condition1} is satisfied by setting 
\begin{equation*}
    L = -t+\max\{A'(\eta^* + r) - A'(\eta^*), A'(\eta^*)-A'(\eta^* - r) \}
\end{equation*}
and $\gamma(|\eta-\eta^*|) = |\eta-\eta^*|$.  In particular, $\gamma(v)\geq c_0\log(1+c_1v)-c_2$ for $v\geq r > 1$ with the constants $c_0=c_1=1$ and $c_2=0$.

\noindent\textit{Condition \ref{condition2}.} By the mean value theorem,
$$
    \log f_n(X^n|\eta)-\log f_n(X^n|\hat{\eta}) = -\frac{n}{2}A''(\hat{\eta} + \xi_n(\eta-\hat{\eta}))(\eta-\hat{\eta})^2,$$
where $\xi_n\in(0,1)$. Note that $A''(\hat{\eta} + \xi_n(\eta-\hat{\eta}))$ is strictly positive for all $\eta$ by the strict convexity of $A$. For $\eta\in\hat{B}_{\hat{\eta}}(2r)$ (with $r > 1$ from checking condition \ref{condition1}), we have that $B_{\hat{\eta}}(2r)\subseteq B_{\eta^*}(4r)$ with $f_{0,n}$-probability tending to 1, and hence, 
\begin{equation*}
    \sup_{\eta\in B_{\hat{\eta}}(2r)}A''(\eta) \leq \sup_{\eta\in B_{\eta^*}(4r)}A''(\eta) \quad \text{ and } \quad \inf_{\eta\in B_{\hat{\eta}}(2r)}A''(\eta) \geq \inf_{\eta\in B_{\eta^*}(4r)}A''(\eta).
\end{equation*}
Combining the mean value theorem statement  with the above bounds, we have that with probability tending to $1$,
\begin{align}
    -\frac{n}{2}\inf_{\eta\in B_{\eta^*}(4r)}A''(\eta)(\eta-\hat{\eta})^2 \nonumber&< \log f_n(X^n|\eta)-\log f_n(X^n|\hat{\eta}) < -\frac{n}{2}\sup_{\eta\in B_{\eta^*}(4r)}A''(\eta)(\eta-\hat{\eta})^2,
\end{align}
uniformly on $B_{\hat{\eta}}(2r)$. So we can take $c_3 = \frac{1}{2}\inf_{\eta\in B_{\eta^*}(4r)}A''(\eta)$ and $c_4 = \frac{1}{2}\sup_{\eta\in B_{\eta^*}(4r)}A''(\eta)$ to conclude the proof. 

\section{Complementary discussion of related work}
\label{App:related_work}
There is vast body of work in Bayesian statistics that is related to our contribution. Perhaps most obviously related, is the thread of literature developing methodology for  data-driven posterior tempering. The SafeBayesian algorithm of \cite{grunwald_safe_2012, grunwald_inconsistency_2017}  is one of the earliest papers in this topic. It chooses the $\alpha$ that minimizes the cumulative posterior expected log loss. In \cite{holmes_assigning_2017}, the authors proposed to choose $\alpha$  such that the expected gain in information between the prior and $\alpha$-posterior matches the expected gain in information between the prior and standard posterior. In the work of \cite{syring_calibrating_2019, wu_comparison_2023}, the tempering parameter is calibrated to give credible intervals obtaining nominal frequentist coverage probability. In a more restrictive theoretical setting, \cite{bu_towards_2024} proposed closed-form expressions for the optimal inverse temperature parameter in the Gibbs posterior that minimizes the population risk. Recent work proposes choosing $\alpha$ to optimize posterior predictions, but also shows this leads to an ill-defined optimization problem where the tempering parameter does not have a meaningful effect in large samples \cite{mclatchie_predictive_2025}. However, similar approaches for generalized posterior inference outperform standard posteriors \cite{lee_liu_nicholls2025}.

A second important line of related work is the  literature studying the asymptotic properties of  Bayesian posteriors. In particular, our work builds on existing BvM theorems for misspecified models in low-dimensional, parametric models where the asymptotic behavior of the resulting posterior is determined by a ``pseudo-true" likelihood model \cite{kleijn_bernstein-von-mises_2012,avella_medina_robustness_2022}. There is also an extensive literature on the asymptotic properties of posterior distributions in  parametric, high-dimensional settings where the dimension of the parameter increases with the sample size. Asymptotic normality results for Bayesian linear regression with a diverging number of parameters $p$ were established in \cite{ghosal_asymptotic_1999, bontemps_bernsteinvon_2011} and central limit theorems about linear statistics from such posteriors were recently established in \cite{lee_deb_mukherjee_2025} in the regime where $p \ll n^{2/3}$. Furthermore, sufficient conditions for and asymptotic regimes for which valid BvMs/Laplace approximations exist when $p\to \infty$ are given in \cite{katsevich_laplace_2024, katsevich_improved_2025, katsevich_laplace_2025}. Finally, we note that BvMs in semiparametric settings were established for standard posteriors in \cite{castillo_bernstein-von_2015} and for $\alpha_n$-posteriors in \cite{lhuillier_semiparametric_2023}.

There is also a broader class of BvM results for generalized posterior distributions \cite{chernozhukov_mcmc_2003, ghosh_robust_2016, miller_asymptotic_2021, lee_liu_nicholls2025,marusic:avellamedina:rush2025}. These posteriors replace the likelihood in the standard construction of a posterior by some Gibbs distribution defined by a loss function that is exponentiated. We note that  Gibbs posteriors arise as optimal solutions to several statistical problems such as aggregation of estimators \cite{vovk_aggregating_1990, dalalyan_aggregation_2008, rigollet_sparse_2012}, stochastic model selection \cite{mcallester_pac-bayesian_2003}, and minimum complexity density estimation/information complexity minimization \cite{barron_minimum_1991, zhang_epsilon-entropy_2006, zhang_information_2006}.  

The popularity of $\alpha$-posteriors stem largely from the robustness properties that they have been shown to enjoy in both theory and practice. These were initially explored in the context of misspecified linear regression models where $\alpha$-posteriors were shown to outperform standard posteriors \cite{grunwald_inconsistency_2017}. Similar conclusions were obtained in \cite{holmes_assigning_2017} in the context of exponential families and correcting for overdispersion. In an asymptotic analysis, \cite{avella_medina_robustness_2022} showed that $\alpha$-posteriors and their variational approximations are robust to general misspecification of the likelihood, assuming the misspecification occurs with vanishing probability.
The coarsened posterior of \cite{miller_robust_2019} conditions on the data being perturbed -- rather than being observed exactly, as it is in the calculation of the standard posterior. This design of the posterior builds in explicitly some amount of robustness to model misspecification. Putting a prior on the magnitude of the perturbation reveals  that, asymptotically, the coarsened posterior is approximately a power posterior, where the power depends on the sample size. In a slightly different vein, \cite{lin_tarp_evans2025} indicates a way to augment randomized control trial data with biased (i.e., misspecified) observational data using a power likelihood approach while maintaining nominal credible interval coverage. 

There have also been some proposals to robustify standard posterior constructions to the presence of outliers. Among these, the Bayesian data reweighting approach proposed by \cite{wang_robust_2017} is perhaps the most similar to $\alpha$-posteriors. The authors proposed to assign a tempering parameter $\alpha_i$ to each observation (as opposed to a single $\alpha$ applied to the entire likelihood) and jointly model the $\alpha_i$'s and observations by putting a prior on the $\alpha_i$'s. The idea of this construction is that individual weights can downweight the impact of individual outliers on the overall contribution to the likelihood \cite{wang_robust_2017}. Other robustifications to outliers  are  derived from constructing generalized posteriors with robust losses \cite{hooker:vidyashankar2014, ghosh_robust_2016,altamirano_robust_2024,matsubara_generalized_2024,marusic:avellamedina:rush2025}  and test statistics \cite{baraud_robust_2024}. There has also been some work arguing that $\alpha$-posteriors can obtain more robustness with respect to the choice of the prior. This was advocated in the context of model selection with Bayes factors in the work of \cite{ohagan_fractional_1995}. Similarly, power priors construct informative priors from historical data by computing an $\alpha$-posterior from the historical data to use as a prior for the current study \cite{ibrahim_power_2000, ibrahim_power_2015, carvalho_normalized_2021}. 

Finally, our work is also related to some literature on Variational Inference (VI). VI is an optimization-based approach to approximate an intractable posterior distribution \cite{jordan_introduction_1999, blei_variational_2017, wainwright_graphical_2008}. The core idea of VI is to approximate the posterior, $\pi_n(\theta|X^n)$,  with a more tractable distribution. To do so, one specifies a family of computationally-tractable distributions, $\mathcal{Q}$, and chooses the distribution $q^*\in\mathcal{Q}$ that is closest to the posterior in terms of KL-divergence.  A variant of this idea, directly connected to  $\alpha$-posteriors, is the following  regularized version of Evidence Lower Bound (ELBO) maximization: $q^* = \arg\max_{q\in\mathcal{Q}}\int \log f_n(X^n|\theta)\pi(\theta)q(\theta)d\theta - \frac{1}{\alpha}d_{\text{KL}}(q\|\pi),$
where $d_{\text{KL}}(q\|\pi) = \int_{\mathbb{R}^p}q(\theta)\log(\frac{q(\theta)}{\pi(\theta)})d\theta$. Indeed, the $\alpha$-posterior maximizes the above problem when $\mathcal{Q}$ is unrestricted \cite{avella_medina_robustness_2022}. Regularized ELBOs and related quantities have been referred to as $\alpha$-VI in the literature \cite{huang_improving_2018, yang_alpha-variational_2020}, and this objective has been used in training variational autoencoders \cite{higgins_beta-vae_2016, burgess_understanding_2018}. A similar objective that instead upweights the influence of the likelihood was shown to efficiently tune hyperparameters in overparameterized probabilistic models \cite{harvey_petrov_hughes2025}.
The $\alpha$-VI objective can be further generalized by replacing the log-likelihood with another loss function; known as generalized VI \cite{knoblauch_optimization-centric_2022} when $\mathcal{Q}$ is restricted. This objective is maximized by the generalized posterior \cite{bissiri_general_2016} when $\mathcal{Q}$ is unrestricted.

\section{Additional experimental details}
\subsection{Bayesian cross-validation formulas}
 
We work out the explicit formulas we use to compute the two cross-validation losses that we used in our experiments. These are more convenient than the integrals \eqref{eq:elpd} and \eqref{eq:elpd_vi}.

\subsubsection{$\overline{\textrm{lppd}}_{\text{LOO}}(\alpha)$ calculation}\label{sec:elpd_calc}
Recall the definition of leave-one-out $\overline{\text{lppd}}_{\text{LOO}}(\alpha)$,
\begin{equation*}
    \overline{\textrm{lppd}}_{\textrm{LOO}}(\alpha) = \frac{1}{n}\sum_{i=1}^n \log \int p(x_i,y_i|\beta)\pi_{n,\alpha}(\beta|X_{-i}Y_{-i})d\beta,
\end{equation*}
where $X_{-i}$, $Y_{-i}$ denote the design matrix and dependent variable, respectively, with the i$^{th}$ entry removed. In our numerical experiments the prior is $N(0,I_p)$ and the $\alpha_n$-posterior in the above is 
$
    \pi_{n,\alpha}(\beta|{X_{-i}},Y_{-i}) = \phi(\beta  | \mu_{n,-i}(\alpha),~\Sigma_{n,-i}(\alpha)),
$
where
\begin{equation}\label{eq:ridge_obj_i}
    \begin{split}
        \mu_{n,-i}(\alpha) \equiv \Big(X_{-i}^\top  X_{-i} + \frac{\sigma^2}{\alpha}I_p\Big)^{-1}X_{-i}^\top  Y_{-i}
        \quad \mbox{and} \quad
        \Sigma_{n,-i}(\alpha) \equiv \frac{\sigma^2}{\alpha}\Big(X_{-i}^\top  X_{-i} + \frac{\sigma^2}{\alpha}I_p\Big)^{-1}.
    \end{split}
\end{equation}
To obtain the posterior in \eqref{eq:ridge_post}, we set $\sigma^2 = 1$ to obtain the ridge regression estimator, $\mu_{n,-i}(\alpha) \equiv \hat{\beta}^{\text{Ridge}}_{-i}$, as the $\alpha_n$-posterior mean. In the data application example, we plug $\hat{\sigma}^2 = \frac{1}{n-p}\sum_{i=1}^n(Y_i - \hat{Y}_i)^2$ into the above.
Substituting the expression of the $\alpha_n$-posterior in $\overline{\text{lppd}}_{\text{LOO}}(\alpha)$ gives
\begin{equation*}\label{eq:step_0_i}
    \overline{\textrm{lppd}}_{\textrm{LOO}}(\alpha) = \frac{1}{n}\sum_{i=1}^n \log  \frac{\int\exp\left(-\frac{1}{2}\left(\frac{1}{\sigma^2}(y_i-x_i^\top  \beta)^2+(\beta-\mu_{n,-i})^\top  \Sigma_{n,-i}^{-1}(\beta-\mu_{n,-i})\right)\right)d\beta}{(2\pi)^{\frac{p}{2}+\frac{1}{2}}\sigma|\Sigma_{n,-i}|^{\frac{1}{2}}},
\end{equation*}
where we dropped the dependence on $\alpha$ from $\mu_{n,-i}$ and $\Sigma_{n,-i}$ in order to make the notation a little less cumbersome. In the remaining calculations, we will not include the denominator of the above as it is an additive constant in the above and does not affect the optimization of $\overline{\textrm{lppd}}_{\textrm{LOO}}(\alpha)$. 
Completing the square we see that 
\begin{align}
   \nonumber &\frac{1}{\sigma^2}(y_i-x_i^\top  \beta)^2+(\beta-\mu_{n,-i})^\top  \Sigma_{n,-i}^{-1}(\beta-\mu_{n,-i}) \\
   \nonumber &= (\beta-\tilde{\mu}_{n,-i})\tilde{\Sigma}_{n,-i}^{-1}(\beta-\tilde{\mu}_{n,-i}) +\Big(\frac{1}{\sigma^2}y_i^2+\mu_{n,-i}^\top\Sigma_{n,-i}^{-1}\mu_{n,-i} -\tilde{\mu}_{n,-i}^\top \tilde{\Sigma}_{n,-i}^{-1}\tilde{\mu}_{n,-i}\Big) \\
   &= (\beta-\tilde{\mu}_{n,-i})\tilde{\Sigma}_{n,-i}^{-1}(\beta-\tilde{\mu}_{n,-i})+C_{n,-i},
    \label{eq:loo_square}
\end{align}
where 
\begin{equation}
\label{eq:ridge_obj_i2}
    \tilde{\mu}_{n,-i} =\tilde{\Sigma}_{n,-i}\Big(\frac{1}{\sigma^2}x_iy_i+\Sigma_{n,-i}^{-1}\mu_{n,-i}\Big)\quad \mbox{and} \quad \tilde{\Sigma}_{n,-i}=\Big(\frac{1}{\sigma^2}x_ix_i^\top+\Sigma_{n,-i}^{-1}\Big)^{-1},
\end{equation}
and in $C_{n,-i}$ we have collected the terms that do not depend on $\beta$.
Therefore,

\begin{align}
 \nonumber &\int\exp\Big(-\frac{1}{2}\Big(\frac{1}{\sigma^2}(y_i-x_i^\top  \beta)^2+(\beta-\mu_{n,-i})^\top  \Sigma_{n,-i}^{-1}(\beta-\mu_{n,-i})\Big)\Big)d\beta\\
 &=e^{-C_{n,-i}/2}\int\exp\Big(\frac{-1}{2}(\beta-\tilde\mu_{n,-i})^\top  \tilde\Sigma_{n,-i}^{-1}(\beta-\tilde\mu_{n,-i})\Big)d\beta  =e^{-C_{n,-i}/2}(2\pi)^{\frac{p}{2}}|\tilde\Sigma_{n,-i}|^{\frac{1}{2}}.
 \label{eq:loo_int}
\end{align}
It follows from \eqref{eq:ridge_obj_i}, \eqref{eq:loo_square}, \eqref{eq:ridge_obj_i2} and \eqref{eq:loo_int} that
$
       \mbox{argmin}_\alpha \{\overline{\textrm{lppd}}_{\textrm{LOO}}(\alpha)\} = \mbox{argmin}_\alpha\{T_1(\alpha) + T_2(\alpha) + T_3(\alpha)\},
$
where
\begin{align*}
    T_1(\alpha) 
    = -\frac{1}{2n}\sum_{i=1}^n\log\Big(\frac{|\tilde\Sigma_{n,-i}(\alpha)^{-1}|}{|\Sigma_{n,-i}(\alpha)^{-1}|}\Big) &=-\frac{1}{2n}\sum_{i=1}^n\log\Big(\frac{|\frac{1}{\sigma^2}x_ix_i^\top+\frac{\alpha }{\sigma^2}(X_{-i}^\top  X_{-i}+\frac{\sigma^2}{\alpha } I_p)|}{|\frac{\alpha }{\sigma^2}(X_{-i}^\top  X_{-i}+\frac{\sigma^2}{\alpha } I_p)|}\Big) \\
    &= -\frac{1}{2n}\sum_{i=1}^n \log\Big|\frac{1}{\alpha }\Big(X_{-i}^\top  X_{-i} + \frac{\sigma^2}{\alpha }I_p\Big)^{-1}x_ix_i^\top   + I_p\Big|,
\end{align*}
where the last expression follows from tedious but straightforward manipulations. 
\begin{align*}
        T_2(\alpha) 
        &=-\frac{1}{2n}\sum_{i=1}^n\mu_{n,-i}^\top\Sigma_{n,-i}^{-1}\mu_{n,-i} =-\frac{1}{2n}\sum_{i=1}^n\frac{\alpha}{\sigma^2}Y_{-i}^\top X_{-i}  \Big(X_{-i}^\top  X_{-i} + \frac{\sigma^2}{\alpha }I_p\Big)^{-1}X_{-i}^\top  Y_{-i},
\end{align*}
and
\begin{equation*}
    \begin{split}
        T_3(\alpha) 
        &= \frac{1}{2n}\sum_{i=1}^n\tilde{\mu}_{n,-i}^\top \tilde{\Sigma}_{n,-i}^{-1}\tilde{\mu}_{n,-i}\\
        &=\frac{1}{2n}\sum_{i=1}^n \Big[\frac{x_iy_i}{\sigma^2} + \frac{\alpha }{\sigma^2}X_{-i}^\top  Y_{-i} \Big]^\top   \Big(\frac{x_ix_i^\top }{\sigma^2}  + \frac{\alpha }{\sigma^2}\Big(X_{-i}^\top  X_{-i} + \frac{\sigma^2}{\alpha }I_p\Big)\Big)^{-1} \Big[\frac{x_iy_i}{\sigma^2} + \frac{\alpha }{\sigma^2}X_{-i}^\top  Y_{-i} \Big].
    \end{split}
\end{equation*}
In our numerical experiments, we optimize $\overline{\text{lppd}}_{\text{LOO}}$ with respect to $\lambda$ and ``convert" the selected $\lambda$ value to its selected $\alpha$ value using the relation $\alpha = \frac{1}{(n-1)\lambda}$ (see Section \ref{app:grid_searches}). As a result, we rescale the leave-one-out matrices by $\frac{1}{n-1}$ (e.g., we use $\frac{1}{n-1}X_{-i}^\top X_{-i}$ in lieu of $X_{-i}^\top X_{-i}$). In the above formulas, we have factored and canceled out the $\frac{1}{n-1}$. The same applies to $\overline{\text{lppd}}_{\text{LOO-VI}}$, which is calculated in the next subsection.

\subsubsection{$\overline{\textrm{lppd}}_{\textrm{VI-LOO}}$ calculation}
\label{sec:elpdvi_calc}

We also consider the $\overline{\textrm{lppd}}_{\textrm{VI-LOO}}$, where density of the Gaussian mean field variational approximation of \eqref{eq:ridge_post} is substituted in the calculation of \eqref{eq:elpd}. This density has parameters
\begin{equation*}
    \mu_{n,-i}(\alpha) \equiv \Big(X_{-i}^\top  X_{-i} + \frac{\sigma^2}{\alpha}I_p\Big)^{-1}X_{-i}^\top  Y_{-i}
    \quad \mbox{and} \quad 
    \Sigma_{n,-i}(\alpha) \equiv \frac{\sigma^2}{\alpha}\Big[\text{diag}\Big(X_{-i}^\top  X_{-i} + \frac{\sigma^2}{\alpha}I_p\Big)\Big]^{-1}.
\end{equation*}
We obtain
$
    \mbox{argmin}_\alpha \{\overline{\textrm{lppd}}_{\textrm{LOO-VI}}(\alpha)\}
        =\mbox{argmin}_\alpha \{\tilde{T}_1(\alpha) + \tilde{T}_2(\alpha) + \tilde{T}_3(\alpha)\},
$ using a similar approach to the calculation of $\overline{\textrm{lppd}}$ (Section \ref{sec:elpd_calc}),
where
\begin{align*}
        &\tilde{T}_1(\alpha) 
        = \frac{-1}{2n}\sum_{i=1}^n\log\Big(\frac{|\tilde\Sigma_{n,-i}(\alpha)^{-1}|}{|\Sigma_{n,-i}(\alpha)^{-1}|}\Big) = -\frac{1}{2n}\sum_{i=1}^n \log\Big|\frac{1}{\alpha }\Big[\text{diag}\Big(X_{-i}^\top  X_{-i} + \frac{\sigma^2}{\alpha }I_p\Big)\Big]^{-1}x_ix_i^\top  + I_p\Big|, \\
        &\tilde{T}_2(\alpha)
        =\frac{-1}{2n}\sum_{i=1}^n\mu_{n,-i}^\top\Sigma_{n,-i}^{-1}\mu_{n,-i} \\
        &= \frac{-1}{2n}\sum_{i=1}^n\frac{\alpha }{\sigma^2}Y_{-i}^\top X_{-i}  \Big(X_{-i}^\top  X_{-i} + \frac{\sigma^2}{\alpha }I_p\Big)^{-1}  \hspace{-.1pt}\text{diag}\Big(X_{-i}^\top  X_{-i} + \frac{\sigma^2}{\alpha }I_p\Big) \Big(X_{-i}^\top  X_{-i} + \frac{\sigma^2}{\alpha }I_p\Big)^{-1}
        X_{-i}^\top  Y_{-i} ,
\end{align*}
and
\begin{align*}
        &T_3 
        =\frac{1}{2n}\sum_{i=1}^n\tilde{\mu}_{n,-i}^\top \tilde{\Sigma}_{n,-i}^{-1}\tilde{\mu}_{n,-i} \\
        &= \frac{1}{2n}\sum_{i=1}^n 
        \Big[\frac{x_iy_i}{\sigma^2} + \frac{\alpha }{\sigma^2}\text{diag}\Big(X_{-i}^\top X_{-i} + \frac{\sigma^2}{\alpha}I_p\Big)\Big(X_{-i}^\top X_{-i} + \frac{\sigma^2}{\alpha}I_p\Big)^{-1}X_{-i}^\top  Y_{-i} \Big]^\top \\
        &\Big(\frac{x_ix_i^\top}{\sigma^2} + \frac{\alpha }{\sigma^2}\text{diag}\Big(X_{-i}^\top  X_{-i} + \frac{\sigma^2}{\alpha }I_p\Big)\Big)^{-1} \Big[\frac{x_iy_i}{\sigma^2} + \frac{\alpha }{\sigma^2}\text{diag}\Big(X_{-i}^\top X_{-i} + \frac{\sigma^2}{\alpha}I_p\Big)\Big(X_{-i}^\top X_{-i} + \frac{\sigma^2}{\alpha}I_p\Big)^{-1}X_{-i}^\top  Y_{-i} \Big].
\end{align*}

\subsection{Grids used in grid searches}\label{app:grid_searches}
We list the grids used in the grid searches in our simulation example (Section \ref{sec:sim-results}, Table \ref{tab:sim_grid}) and our data application example (Section \ref{sec:data_example}, Table \ref{tab:data_grid}). Each table contains the following columns:
\begin{itemize}
    \item Method: selection method/loss function. See Section \ref{sec:sel_methods} for details on each method. 
    \item Parameter: the parameter with respect to which the loss function is minimized. All methods except for Safebayes -- which optimizes for the tempering parameter, $\alpha$ -- optimize for $\lambda$ (the regularization parameter in \eqref{eq:ridge_obj}).
    \item Spacing: whether the grid is linearly spaced or logarithmically spaced.
    \item Interval: lower and upper bounds of the grid.
    \item Density: number of gridpoints.
    \item Mapping to $\alpha$: How the selected parameter is mapped back to $\hat{\alpha}_n$ in the $\alpha_n$-posterior.
\end{itemize}
After $\hat{\alpha}_n$ is computed from Tables \ref{tab:sim_grid} and \ref{tab:data_grid}, we recode $\hat{\alpha}_n$ that are greater than $10^6$ to be ``effectively" infinity, and label such values as $\infty$ in Figures \ref{fig:lim_dist_constant}--\ref{fig:lim_dist_mix} and Figure \ref{fig:lim-dist-data}.

\begin{table}
\begin{center}
    \begin{tabular}{cccccc}
    \toprule
    Method & Parameter & Spacing & Interval & Density & Mapping to $\alpha$ \\
    \midrule
    BCV & $\lambda$ & Logarithmic & [$10^{-12}$, 10] & 200 & $((n-1)\lambda)^{-1}$\\
    BCV + VI & $\lambda$ & Logarithmic & [$10^{-12}$, 10] & 200 & $((n-1)\lambda)^{-1}$ \\
    LOOCV & $\lambda$ & Linear & [$10^{-12}$, 30] & 200 & $\lambda^{-1}$ \\
    Train-test & $\lambda$ & Linear & [$10^{-12}$, 5] & 200& $(n\lambda)^{-1}$ \\
    Safebayes & $\alpha$ & Linear & [0, 1] & 30 & $\alpha$\\
    \bottomrule
    \end{tabular}
\end{center}
\caption{Grids used in the simulation example described in Section \ref{sec:motivating_examples}. }
    \label{tab:sim_grid}
\end{table}

\begin{table}
\begin{center}
    \begin{tabular}{cccccc}
    \toprule
    Method & Parameter& Spacing & Interval & Density & Mapping to $\alpha$ \\
    \midrule
    BCV & $\lambda$ & Logarithmic & [$10^{-12}$, 10] & 200 & $((n-1)\lambda)^{-1}$ \\
    BCV + VI & $\lambda$ & Logarithmic & [$10^{-12}$, 10] & 200 & $((n-1)\lambda)^{-1}$ \\
    LOOCV & $\lambda$ & Linear & [$10^{-12}$, 0.5] & 200 & $\lambda^{-1}$ \\
    Train-test & $\lambda$ & Linear & [$10^{-12}$, 0.05] & 200 & $(n\lambda)^{-1}$ \\
    Safebayes & $\alpha$ & Linear & [0, 1] & 30 & $\alpha$ \\
    \bottomrule
    \end{tabular}
\end{center}
\caption{Grids used in the data application example described in Section \ref{sec:data_example}. }
    \label{tab:data_grid}
\end{table}

\subsection{Relationship between $\hat{\alpha}_n$ and $n$}\label{sec:alpha_lambda}
In this section, we explain how we estimate the rates of convergence of $\hat{\alpha}_n$ to $0$ that appear in Tables \ref{tab:conv_rates} and \ref{tab:conv_rates-data}. To obtain a curve of best fit of $\hat{\alpha}_n$ vs. $n$, we assume $\alpha = Cn^\gamma$, and our goal is to estimate the parameters $C$ and $\gamma$. We regress $\log(\alpha)$ on $\log(n)$, since
$
    \log(\alpha) = \log(C) + \gamma\log n \equiv \beta_0 + \beta_1\log n,$
the estimated curve of best fit is $\hat{\alpha}_n = e^{\hat{\beta}_0}n^{\hat{\beta}_1}$ and $\hat\gamma=\hat\beta_1$. We report the estimates and $95\%$ confidence intervals for $\beta_1$ for our simulation example in Table \ref{tab:conv_rates} and for our data application example in Table \ref{tab:conv_rates-data}. We report the corresponding curves of best fit in Figures \ref{fig:conv-rate} (simulation example) and \ref{fig:conv-rate-data-non-corner} (data application example).

\begin{figure}
    \centering
    \includegraphics{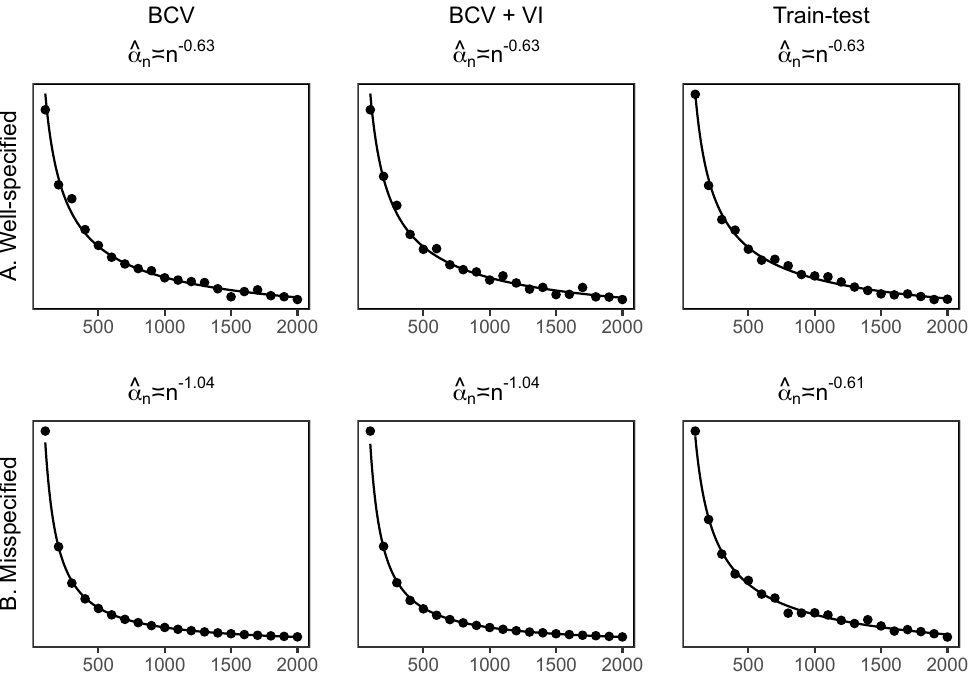}
    \caption{Panel A and B: dependence of $\hat{\alpha}_n$ (y-axis) on $n$ (x-axis) in the well-specified model setting (A) and misspecified model setting (B). Points correspond to $\hat{\alpha}_n$ values averaged over 1,000 replications with overlaid curves of best fit. Plot titles denote the estimated stochastic order of $\hat{\alpha}_n$. Corner solutions at $\hat{\alpha}_n=\infty$ are discarded in the plotting and estimation in the following settings: BCV/well-specified, BCV+VI/well-specified, train-test/well-specified and misspecified.
    }
    \label{fig:conv-rate}
\end{figure}

\begin{figure}
    \centering
    \includegraphics{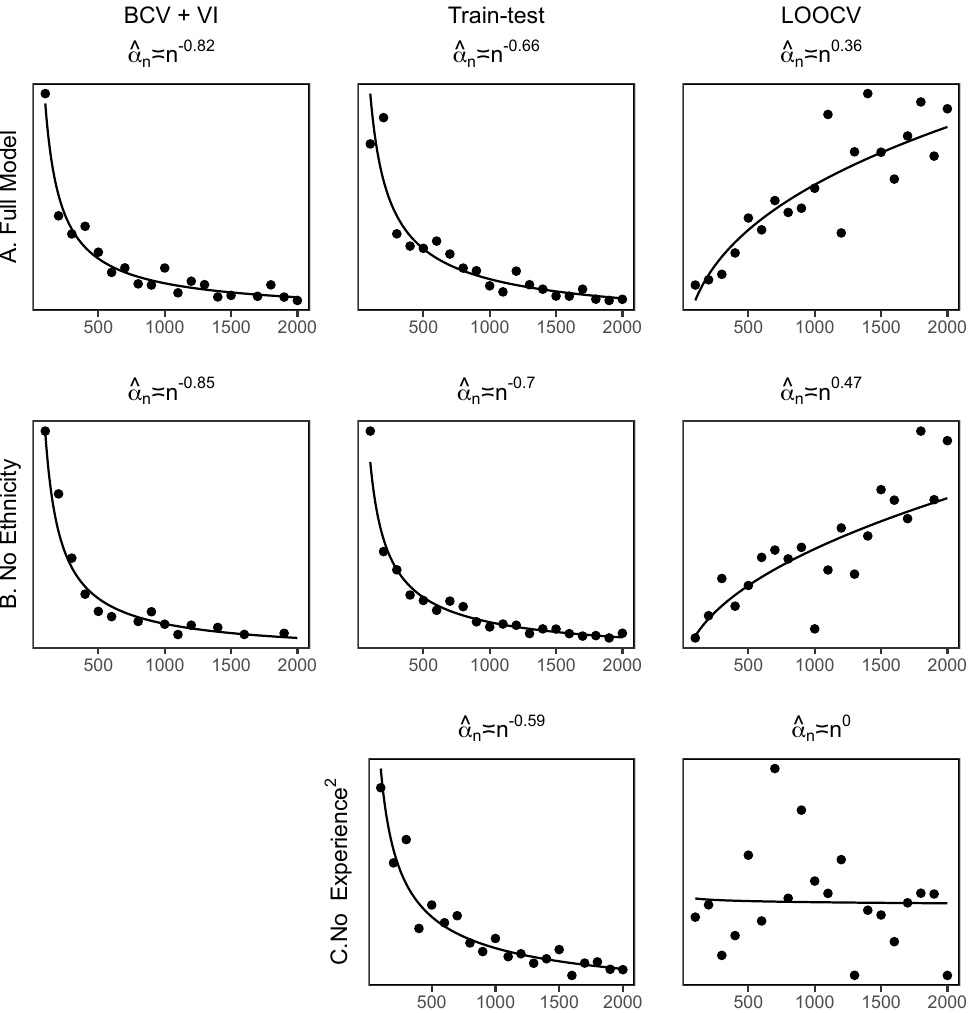}
    \caption{Panel A, B, and C: dependence of $\hat{\alpha}_n$ (y-axis) on $n$ (x-axis) under the full model setting (A), model excluding ethnicity (B), and model excluding squared experience. Points correspond to $\hat{\alpha}_n$ values (with corner solutions discarded) averaged over $100$ replications with overlaid curves of best fit. Plot titles denote the estimated stochastic order of $\hat{\alpha}_n$. The plot for BCV + VI under the model excluding squared experience is not plotted because ``non-corner solutions" only appear for $n = 100$. We also not that the ``non-corner" solutions according to LOOCV appear to be growing -- rather than shrinking -- with $n$. We include this result for completeness.}
    \label{fig:conv-rate-data-non-corner}
\end{figure}

\subsection{Supplementary Figures}
\subsubsection{Figures \ref{fig:lim_dist_quick}--\ref{fig:lim_dist_mix} plotted on a linear scale}\label{app:lim_dist_linear}
In figures \ref{fig:lim_dist_quick} and \ref{fig:lim_dist_mix}, we plot the $\hat{\alpha}_n$ values selected with BCV, BCV + VI, and Train-test on a log scale. We do so because the order of magnitude of $\hat{\alpha}_n$ selected according to these methods changes as $n$ grows. Plotting $\hat{\alpha}_n$ on a common log scale allows us to compare the distributions of $\hat{\alpha}_n$ as $n$ grows while also showcasing the convergence of $\hat{\alpha}_n$ to 0, which we further investigate in Table \ref{tab:conv_rates} and Figure \ref{fig:conv-rate}. In this appendix, we also show the histograms of $\hat{\alpha}_n$ from Figures \ref{fig:lim_dist_quick} and \ref{fig:lim_dist_mix} with $\hat{\alpha}_n$ plotted on a linear scale (Figures \ref{fig:lim_dist_quick_linear} and \ref{fig:lim_dist_mix_linear}). We use a different x-axis for each method and sample size to show the scale/order of magnitude of $\hat{\alpha}_n$ as $n$ grows.

\begin{SCfigure}[0.95][h]
    \centering
    \includegraphics{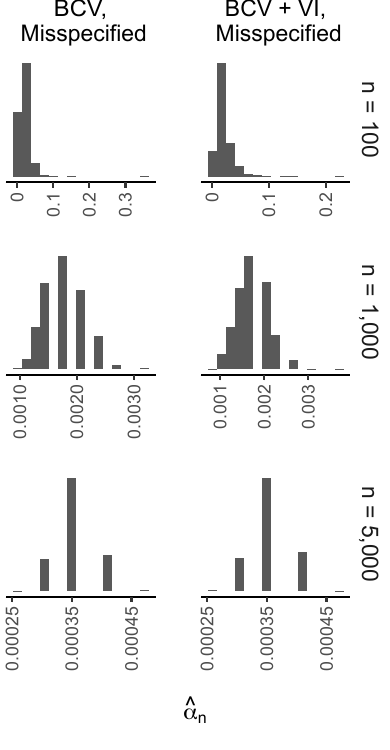}
    \caption{We replicate Figure \ref{fig:lim_dist_quick} with the $\hat{\alpha}_n$ values plotted on a linear scale. Different x-axis scales are used for $n=100$, $n=1,000$, and $n=5,000$ to showcase the different scales of the selected $\hat{\alpha}_n$ values as $n$ grows.} 
    \label{fig:lim_dist_quick_linear}
\end{SCfigure}

\begin{figure}
    \centering
    \includegraphics{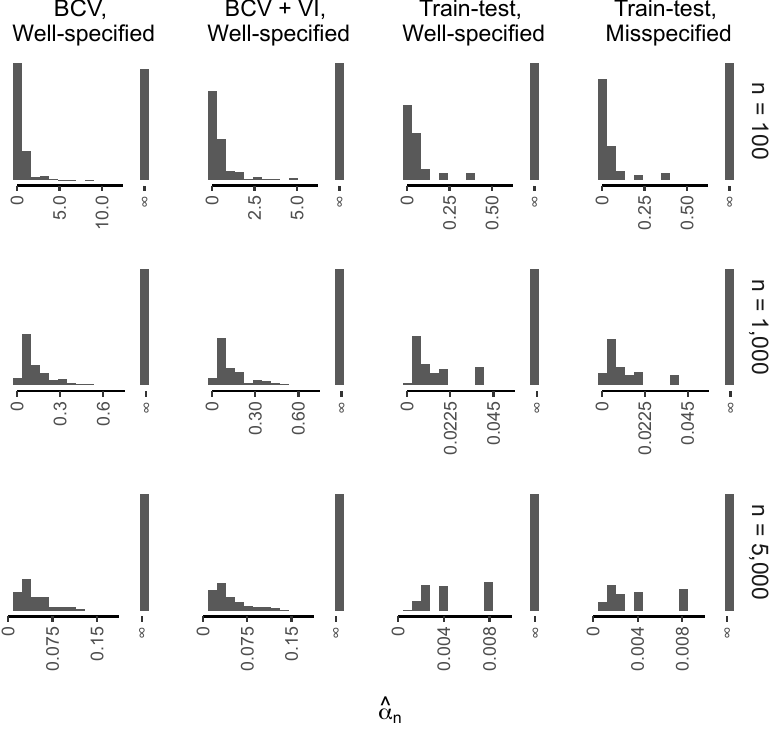}
    \caption{Figure \ref{fig:lim_dist_mix} with the $\hat{\alpha}_n$ values plotted on a linear scale. Different x-axis scales are used for $n=100$, $n=1,000$, and $n=5,000$ to showcase the different scales of the selected $\hat{\alpha}_n$ values as $n$ grows.
    }
    \label{fig:lim_dist_mix_linear}
\end{figure}

\subsubsection{Distribution of selected $\hat{\alpha}_n$ values in CPS1988 example}\label{app:lim_dist_data}
In Section \ref{sec:data_example}, we repeat the experiment in Figures \ref{fig:lim_dist_constant}--\ref{fig:lim_dist_mix} on subsamples from the CPS1988 dataset. We show the distributions of the selected $\hat{\alpha}_n$ in Figure \ref{fig:lim-dist-data}.

\subsubsection{Stability of posterior mean estimates under subsampling}\label{app:data_experiment}
\begin{figure}
    \centering
    \includegraphics[width=0.9\linewidth]{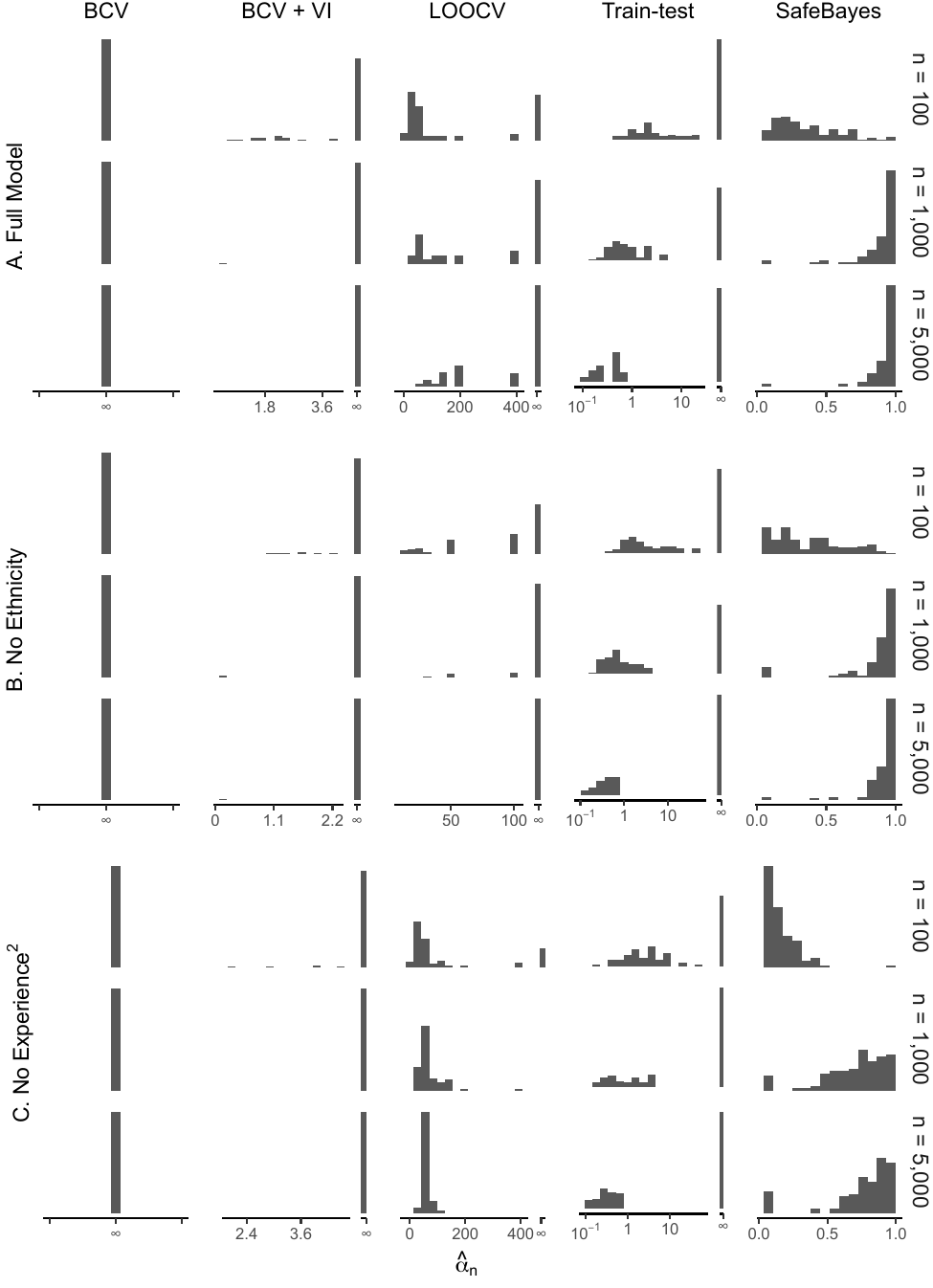}
    \caption{Distribution of $\hat{\alpha}_n$ over 100 subsamples from the CPS1988 dataset under the full model (Panel A), model excluding ethnicity (Panel B), and model excluding squared experience (Panel C). Within each panel, columns correspond to a selection method of $\alpha$ 
    (Bayesian cross-validation, BCV + VI, LOOCV, train-test split, SafeBayes), and rows correspond to sample size per replication ($n=100$, $n=1,000$, $n=5,000$).}
    \label{fig:lim-dist-data}
\end{figure}

In Section \ref{sec:data_example}, we consider the CPS1988 dataset and study the stability to subsampling of the selected $\hat{\alpha}_n$ and the resulting $\hat{\alpha}_n$-posterior mean estimates under the models considered in \eqref{eq:full_model} -- \eqref{eq:noexp2}. We report the results for the $\hat{\alpha}_n$-posterior mean estimates in this section in Figure \ref{fig:data_fig_full} (full model), Figure \ref{fig:data_fig_noeth} (model excluding Ethnicity), and Figure \ref{fig:data_fig_noexp2} (model excluding Experience$^2$).\

\begin{figure}
    \centering
    \begin{subfigure}{0.75\textwidth}
        \includegraphics[width=1\linewidth]{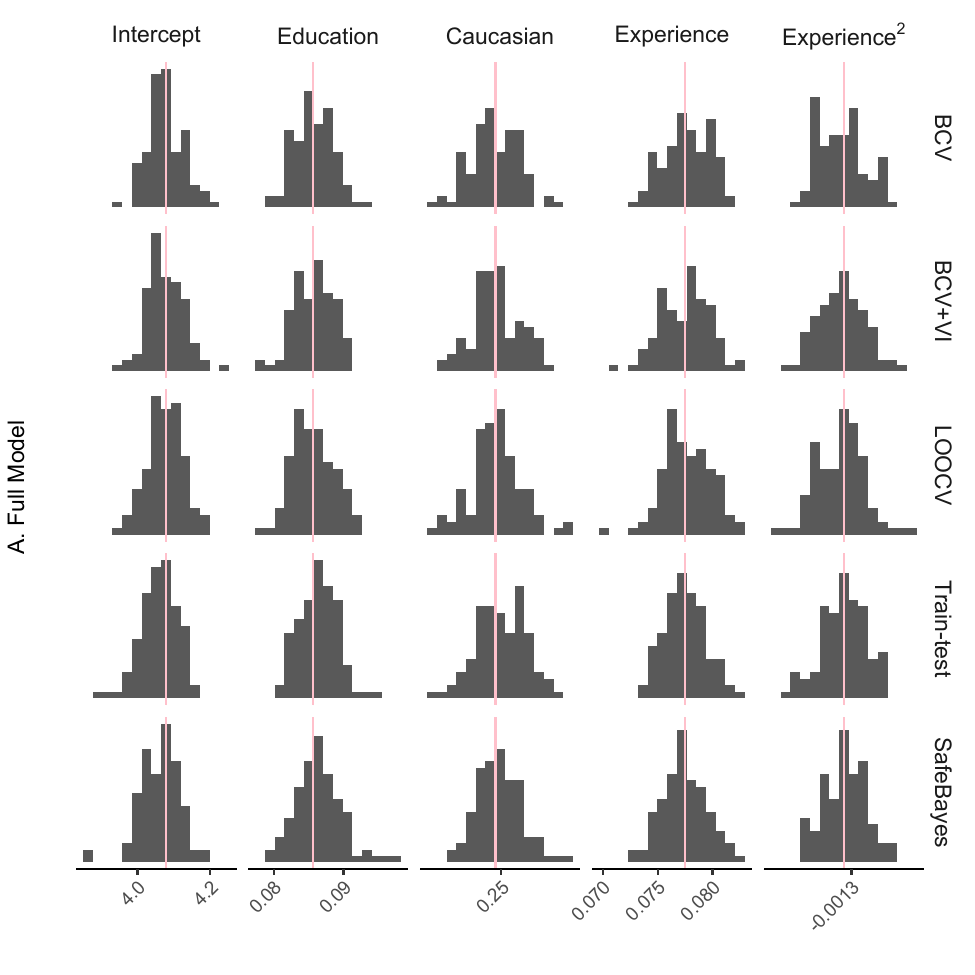}
        \caption{Full model}
    \label{fig:data_fig_full}
    \end{subfigure}
    \label{fig:data_fig}
\end{figure}
\begin{figure}\ContinuedFloat
    \centering
    \begin{subfigure}{0.75\textwidth}
        \includegraphics[width=1\linewidth]{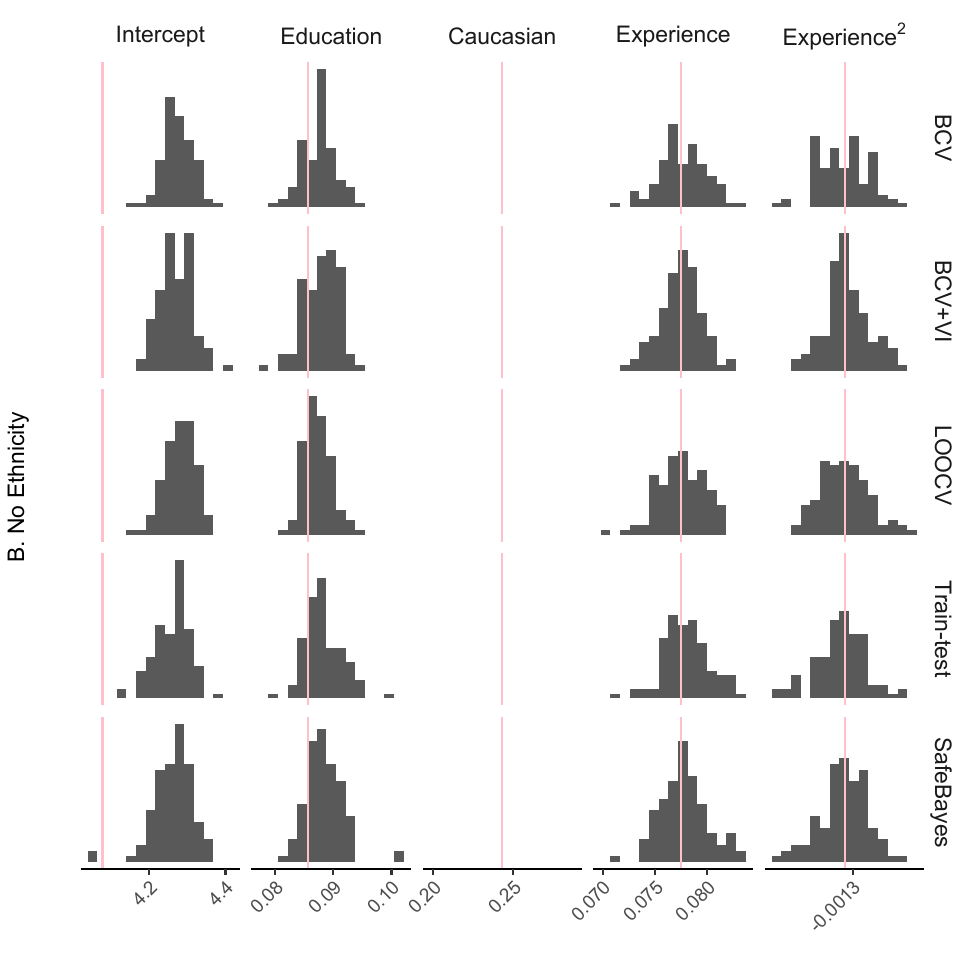}
        \caption{Model excluding ethnicity }
    \label{fig:data_fig_noeth}
    \end{subfigure}
\end{figure}
\begin{figure}\ContinuedFloat
    \centering
    \begin{subfigure}{0.75\textwidth}
        \includegraphics[width=1\linewidth]{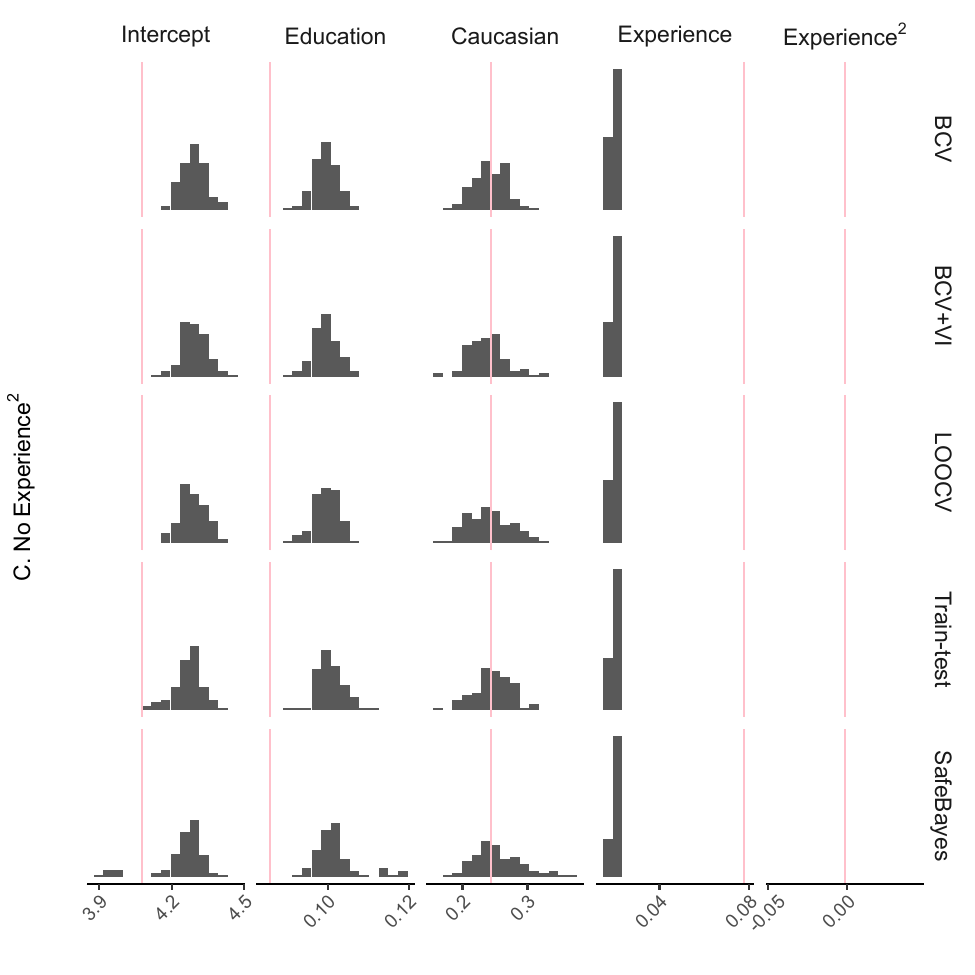}
        \caption{Model excluding squared experience    }
    \label{fig:data_fig_noexp2}
    \end{subfigure}
    \caption{Plots of the distributon of resulting posterior means of each coefficient over 100 subsamples of size $n=5,000$. Each row corresponds to a selection method. The pink vertical lines correspond to the least squares estimates of the coefficients of the full model estimated on the full dataset. }
\end{figure}

\end{document}